\documentclass{article}
\usepackage{mathrsfs}  
\usepackage[english]{babel}
\usepackage[utf8]{inputenc}
\usepackage{amsmath}
\usepackage[title]{appendix}
\usepackage{amsfonts}
\usepackage{graphicx}
\usepackage[colorinlistoftodos]{todonotes}
\usepackage{blindtext}
\usepackage{amssymb}
\usepackage{algorithm}
\usepackage{amsthm}
\usepackage{bbm}
\usepackage{appendix}
\usepackage{authblk}
\numberwithin{equation}{section}
\DeclareMathOperator\supp{supp}

\makeatletter
\newcommand*{\rom}[1]{\expandafter\@slowromancap\romannumeral #1@}

\newtheorem{theorem}{Theorem}[section]

\newtheorem{lemma}[theorem]{Lemma}
\newtheorem{proposition}{Proposition}[section]

\theoremstyle{remark}

\theoremstyle{definition}
\newtheorem{definition}{Definition}[section]
\title{The length of the longest increasing subsequence of Mallows permutation models with $L^1$ and $L^2$ distances}

\author{Chenyang Zhong}
\affil{Department of Statistics, Columbia University}

\date{\today}

\begin{document}
\maketitle
\begin{abstract}
Introduced by Mallows in statistical ranking theory, Mallows permutation model is a class of non-uniform probability measures on the symmetric group $S_n$ that depend on a distance metric $d(\sigma,\tau)$ on $S_n$ and a scale parameter $\beta$. Taking the distance metric to be the $L^1$ and $L^2$ distances--which are respectively known as Spearman's footrule and Spearman's rank correlation in the statistics literature--leads to Mallows permutation models with $L^1$ and $L^2$ distances.

In this paper, we study the length of the longest increasing subsequence of random permutations drawn from Mallows permutation models with $L^1$ and $L^2$ distances. For both models and various regimes of the scale parameter $\beta$, we determine the typical order of magnitude of the length of the longest increasing subsequence and establish a law of large numbers for this length. For Mallows permutation model with the $L^1$ distance, when $\beta\sim \theta n^{-1}$ for some fixed $\theta>0$, the typical length of the longest increasing subsequence is of order $\sqrt{n}$; when $n^{-1}\ll \beta \ll 1$, this typical length is of order $n\sqrt{\beta}$. For Mallows permutation model with the $L^2$ distance, when $\beta\sim \theta n^{-2}$ for some fixed $\theta >0$, the typical length of the longest increasing subsequence is of order $\sqrt{n}$; when $n^{-2}\ll \beta \ll 1$, this typical length is of order $n\beta^{1\slash 4}$.
\end{abstract}

\section{Introduction}\label{Sect.1}

The length of the longest increasing subsequence of random permutations has received much recent interest in various fields including probability, combinatorics, and mathematical physics. For any permutation $\sigma\in S_n$, the length of the longest increasing subsequence of $\sigma$, denoted by $LIS(\sigma)$, is defined as
\begin{eqnarray*}
    LIS(\sigma)&:=& \max\{k\in \{1,2,\cdots,n\}: \sigma(i_1)<\cdots<\sigma(i_k)\text{ for some }\\
    &&\quad\quad i_1,\cdots,i_k\in \{1,2,\cdots,n\}\text{ such that }i_1<\cdots<i_k\}.
\end{eqnarray*}
For uniformly random permutations, there have been extensive investigations on the distribution of the length of the longest increasing subsequence (see e.g. \cite{Ham,LS,VK,AD3,AD2,Cor}), which culminate with the work of Baik, Deift, and Johansson \cite{BDJ} who showed that the limiting distribution of this length is given by the Tracy-Widom distribution arising from random matrix theory. In this paper, we investigate the distribution of the length of the longest increasing subsequence of random permutations drawn from two classes of \emph{non-uniform} probability measures on permutations called Mallows permutation models with $L^1$ and $L^2$ distances, which we introduce as follows.


Mallows permutation model, introduced by Mallows \cite{Mal} in statistical ranking theory, is a class of non-uniform probability measures on the symmetric group $S_n$. The model depends on a distance metric $d(\sigma,\tau)$ on $S_n$, a scale parameter $\beta$, and a location parameter $\sigma_0\in S_n$. Under the model, the probability of picking a permutation $\sigma\in S_n$ is proportional to $\exp(-\beta d(\sigma,\sigma_0))$. For a reasonable choice of the distance metric $d(\sigma,\tau)$, when $\beta>0$, the model is biased towards $\sigma_0$. In this paper, we consider the case where $\beta>0$ and $\sigma_0=Id$--the identity permutation. The reader is referred to \cite[Section 1]{Zho1} for an overview of Mallows permutation model and related literature. 

The distance metric $d(\sigma,\tau)$ as mentioned above can be chosen from a host of metrics on permutations. Several widely used choices are listed as follows (see \cite[Chapter 6]{D} for further discussions on metrics on permutations):
\begin{itemize}
    \item $L^1$ distance, or Spearman's footrule: $d(\sigma,\tau)=\sum_{i=1}^n|\sigma(i)-\tau(i)|$;
    \item $L^2$ distance, or Spearman's rank correlation: $d(\sigma,\tau)=\sum_{i=1}^n (\sigma(i)-\tau(i))^2$;
    \item Kendall's $\tau$: $d(\sigma,\tau)=$ minimum number of pairwise adjacent transpositions taking $\sigma^{-1}$ to $\tau^{-1}$;
    \item Cayley distance: $d(\sigma,\tau)=$ minimum number of transpositions taking $\sigma$ to $\tau$; 
    \item Hamming distance: $d(\sigma,\tau)=\#\{i\in\{1,\cdots,n\}:\sigma(i)\neq \tau(i)\}$;
    \item Ulam's distance: $d(\sigma,\tau)=n-$ the length of the longest increasing subsequence in $\tau\sigma^{-1}$.
\end{itemize}
In this paper, we consider Mallows permutation models with $L^1$ and $L^2$ distances. These models carry a spatial structure, and are also known as ``spatial random permutations'' in the mathematical physics literature \cite{FM}. In the rest of this paper, we also refer to these two models as ``the $L^1$\slash $L^2$ model''. Throughout the paper, for any two permutations $\sigma,\tau\in S_n$, we denote the $L^1$ and $L^2$ distances between $\sigma$ and $\tau$ by
\begin{equation*}
    H(\sigma,\tau)=\sum_{i=1}^n |\sigma(i)-\tau(i)|,\quad \tilde{H}(\sigma,\tau)=\sum_{i=1}^n (\sigma(i)-\tau(i))^2,
\end{equation*}
respectively. The probability measures that correspond to the $L^1$ and $L^2$ models are respectively denoted by $\mathbb{P}_{n,\beta}$ and $\tilde{\mathbb{P}}_{n,\beta}$: For any $\sigma\in S_n$,
\begin{equation*}
    \mathbb{P}_{n,\beta}(\sigma)=Z_{n,\beta}^{-1}\exp(-\beta H(\sigma,Id)), \quad \tilde{\mathbb{P}}_{n,\beta}(\sigma)= \tilde{Z}_{n,\beta}^{-1}\exp(-\beta \tilde{H}(\sigma,Id)),
\end{equation*}
where $Z_{n,\beta}$ and $\tilde{Z}_{n,\beta}$ are the normalizing constants. 


The behavior of the $L^1$ and $L^2$ models depends crucially on the scale parameter $\beta$. If $\beta$ is too small, the models are indistinguishable from the uniform distribution on $S_n$. For the $L^1$ model, ``too small'' means that $\beta$ is much smaller than $n^{-1}$; for the $L^2$ model, ``too small'' means that $\beta$ is much smaller than $n^{-2}$. For $\sigma$ drawn from the $L^1$ model, if $\beta\sim \theta n^{-1}$ for some fixed $\theta>0$, the random probability measure $\nu_{n,\sigma}:=n^{-1}\sum_{i=1}^n \delta_{(i\slash n,\sigma(i)\slash n)}$ converges weakly in probability to a deterministic probability measure on $[0,1]^2$; if $n^{-1}\ll \beta\ll 1$, with high probability, most of the points $\{(i,\sigma(i))\}_{i=1}^n$ are concentrated in a band with order $\beta^{-1}$ width around the diagonal of the plane. Parallel results hold for the $L^2$ model: For $\sigma$ drawn from the $L^2$ model, if $\beta\sim \theta n^{-2}$ for some fixed $\theta>0$, $\nu_{n,\sigma}$ converges weakly in probability to a deterministic probability measure on $[0,1]^2$; if $n^{-2}\ll \beta\ll 1$, with high probability, most of the points $\{(i,\sigma(i))\}_{i=1}^n$ are concentrated in a band with order $\beta^{-1\slash 2}$ width around the diagonal of the plane. Mathematical backups for these properties can be found in \cite{FM,M1,Zhong2}. We also review relevant results in Sections \ref{Sect.1.1} and \ref{Sect.2.2} below.

In this paper, for both the $L^1$ and $L^2$ models and the above two regimes of $\beta$, we determine the typical order of magnitude of the length of the longest increasing subsequence and establish a law of large numbers for this length. For the $L^1$ model, when $\beta\sim \theta n^{-1}$ for some fixed $\theta>0$, the typical length of the longest increasing subsequence is of order $\sqrt{n}$; when $n^{-1}\ll \beta \ll 1$, this typical length is of order $n\sqrt{\beta}$. For the $L^2$ model, when $\beta\sim \theta n^{-2}$ for some fixed $\theta >0$, the typical length of the longest increasing subsequence is of order $\sqrt{n}$; when $n^{-2}\ll \beta \ll 1$, this typical length is of order $n\beta^{1\slash 4}$. The concrete results for the $L^1$ and $L^2$ models are presented in Sections \ref{Sect.1.2} and \ref{Sect.1.3}, respectively.

There have been previous works on the length of the longest increasing subsequence of Mallows permutation models with Kendall's $\tau$ and Cayley distance. The latter model is also known as ``Ewens sampling formula'' in the literature (see e.g. \cite{Crane}). These two Mallows models possess several special\slash exactly solvable structures: Both of them have explicit normalizing constants and can be exactly sampled in an efficient manner, and the latter model is invariant under conjugations (meaning that $\sigma$ has the same distribution as $\tau^{-1}\sigma\tau$ for $\sigma$ drawn from the model and any fixed $\tau\in S_n$). For Mallows permutation model with Kendall's $\tau$, Mueller and Starr \cite{MS} showed a law of large numbers for the regime $\beta\sim \theta n^{-1}$ (where $\theta\in\mathbb{R}$ is fixed), and Bhatnagar and Peled \cite{BP} established a law of large numbers for the regime $n^{-1}\ll \beta \ll 1$. Bhatnagar and Peled \cite{BP} also gave large deviation bounds and concentration inequalities for this Mallows model. Later Basu and Bhatnagar \cite{BB} established a central limit theorem for the regime where $\beta>0$ is fixed. For Mallows permutation model with Cayley distance, Kammoun \cite{Ka3,Ka2} showed that the limiting distribution of the length of the longest increasing subsequence is given by the Tracy-Widom distribution under certain parameter regimes. The proofs of these results rely heavily on the special\slash exactly solvable structures of Mallows permutation models with Kendall's $\tau$ and Cayley distance as mentioned above.

For Mallows permutation models with $L^1$ and $L^2$ distances as considered in this paper, however, there is a lack of exactly solvable structures. For both models, the normalizing constants do not have an explicit form and are hard to compute in general, and there is no known efficient algorithm for exactly sampling from them. Moreover, neither of the models is invariant under conjugations. Due to the lack of exactly solvable structures, there is no previous result in the literature on the distribution of the length of the longest increasing subsequence of these models. In this paper, we develop a novel set of tools to overcome such difficulties. In particular, we utilize hit and run algorithms--which are a unifying class of Markov chain Monte Carlo algorithms--that sample from the $L^1$ and $L^2$ models as a crucial tool in our analysis. A review of these hit and run algorithms is given in Section \ref{Sect.1.5} below.

In the following, we introduce some notations that will be used throughout this paper. We denote $[0]:=\emptyset$ and $[n]:=\{1,2,\cdots,n\}$ for any $n\in\mathbb{N}^{*}$. For any finite set $A$, we denote by $|A|$ the cardinality of $A$. For any $(x_0,y_0)\in\mathbb{R}^2$, $\alpha>0$, and $A\subseteq \mathbb{R}^2$, we denote $(x_0,y_0)+\alpha A:=\{(x_0+\alpha x,y_0+\alpha y):(x,y)\in A\}$. For any two sets $A$ and $B$, we denote by $A\Delta B$ their symmetric difference. 

Throughout the paper, we use $C,c$ to denote positive absolute constants. The values of these constants may change from line to line.  

In Definition \ref{Def1.1} below, we extend the definition of the length of the longest increasing subsequence to bijections. This extension will be useful in proving our main results.

\begin{definition}\label{Def1.1}
For any two sets $S,T\subseteq [n]$ such that $|S|=|T|$ and any bijection $\sigma: S\rightarrow T$, we define the length of the longest increasing subsequence of $\sigma$ by
\begin{eqnarray*}
    LIS(\sigma)&:=& \max\{k\in \{0\}\cup [n]: \sigma(i_1)<\cdots<\sigma(i_k)\text{ for some }i_1,\cdots,i_k\in S\\
    && \quad\quad \text{ such that }i_1<\cdots<i_k\}.
\end{eqnarray*}
\end{definition}

We also introduce the following two definitions.

\begin{definition}\label{Def1.2}
For any set $S\subseteq \mathbb{R}$, we let $\phi(S,n):=S\cap [n]$. For any permutation $\sigma\in S_n$ and any two sets $S,T\subseteq \mathbb{R}$, we define $\sigma|_{S\times T}$ to be the bijection from $\phi(S,n)\cap\sigma^{-1}(\phi(T,n))$ to $\sigma(\phi(S,n))\cap \phi(T,n)$, such that for any $i\in \phi(S,n)\cap \sigma^{-1}(\phi(T,n))$, $(\sigma|_{S\times T})(i)=\sigma(i)$.
\end{definition}

\begin{definition}\label{Def1.5}
For any $n\in\mathbb{N}^{*}$ and any $\sigma\in S_n$, we define 
\begin{equation*}
    S(\sigma):=\{(i,\sigma(i)):i\in [n]\}, \quad \nu_{n,\sigma}:=n^{-1}\sum_{i=1}^n\delta_{(i\slash n,\sigma(i)\slash n)}.
\end{equation*}
\end{definition}

As mentioned before, when $\beta\sim \theta n^{-1}$ for the $L^1$ model or $\beta\sim \theta n^{-2}$ for the $L^2$ model (where $\theta>0$ is fixed), for $\sigma$ drawn the $L^1$ or $L^2$ model, the random probability measure $\nu_{n,\sigma}$ converges weakly in probability to a deterministic probability measure on $[0,1]^2$. In Section \ref{Sect.1.1}, we review results on the density of this limiting probability measure. These results are used in the statement and proof of Theorems \ref{limit_l1_1} and \ref{limit_l2_1} in Sections \ref{Sect.1.2} and \ref{Sect.1.3}. The main results for the $L^1$ and $L^2$ models are presented in Sections \ref{Sect.1.2} and \ref{Sect.1.3}, respectively.

\subsection{Limiting density of $\nu_{n,\sigma}$ for the $L^1$ and $L^2$ models}\label{Sect.1.1}

The length of the longest increasing subsequence of the $L^1$ or $L^2$ model for certain parameter regime ($\beta\sim \theta n^{-1}$ for the $L^1$ model and $\beta\sim \theta n^{-2}$ for the $L^2$ model, where $\theta>0$ is fixed) is closely related to the limiting density of $\nu_{n,\sigma}$ (see Definition \ref{Def1.5}) with $\sigma$ drawn from the corresponding model. In this subsection, we review relevant results on this limiting density. We start with the following definition.

\begin{definition}\label{Defm}
We define $\mathcal{M}$ to be the set of all Borel probability measures on $[0,1]^2$ with uniform marginals.
\end{definition}

The following result for the $L^1$ model follows by adapting the proofs of \cite[Theorem 1.5]{M1} and \cite[Corollary 1.12]{M2}. The detailed proof is given in the appendix.

\begin{proposition}\label{Densi.l1}
Let $(\beta_n)_{n=1}^{\infty}$ be an arbitrary sequence of positive numbers such that $\lim_{n\rightarrow\infty} n\beta_n=\theta >0$. Let $\sigma$ be drawn from $\mathbb{P}_{n,\beta_n}$. Then the random probability measure $\nu_{n,\sigma}$ defined in Definition \ref{Def1.5} converges weakly in probability to a probability measure $\mu_{\theta}\in\mathcal{M}$ that only depends on $\theta$. Moreover, with respect to the Lebesgue measure on $[0,1]^2$, $\mu_{\theta}$ has a continuous density $\rho_{\theta}(\cdot,\cdot)$ given by
\begin{equation*}
    \rho_{\theta}(x,y)=e^{-\theta|x-y|+a_{\theta}(x)+a_{\theta}(y)}, \quad \forall (x,y)\in [0,1]^2,
\end{equation*}
where the function $a_{\theta}(\cdot)\in L^1([0,1])$ satisfies $a_{\theta}(x)=a_{\theta}(1-x), \forall x\in [0,1]$. Moreover, there exist positive constants $m_{\theta}$ and $M_{\theta}$ that only depend on $\theta$, such that $m_{\theta}\leq \rho_{\theta}(x,y)\leq M_{\theta}$ for every $(x,y)\in [0,1]^2$. 
\end{proposition}

The parallel result for the $L^2$ model is given below. It can be proved in a similar manner as Proposition \ref{Densi.l1}.

\begin{proposition}\label{Densi.l2}
Let $(\beta_n)_{n=1}^{\infty}$ be an arbitrary sequence of positive numbers such that $\lim_{n\rightarrow\infty} n^2 \beta_n=\theta >0$. Let $\sigma$ be drawn from $\tilde{\mathbb{P}}_{n,\beta_n}$. Then the random probability measure $\nu_{n,\sigma}$ defined in Definition \ref{Def1.5} converges weakly in probability to a probability measure $\tilde{\mu}_{\theta}\in\mathcal{M}$ that only depends on $\theta$. Moreover, with respect to the Lebesgue measure on $[0,1]^2$, $\tilde{\mu}_{\theta}$ has a continuous density $\tilde{\rho}_{\theta}(\cdot,\cdot)$ given by
\begin{equation*}
    \tilde{\rho}_{\theta}(x,y)=e^{-\theta(x-y)^2+\tilde{a}_{\theta}(x)+\tilde{a}_{\theta}(y)}, \quad \forall (x,y)\in [0,1]^2,
\end{equation*}
where the function $\tilde{a}_{\theta}(\cdot)\in L^1([0,1])$ satisfies $\tilde{a}_{\theta}(x)=\tilde{a}_{\theta}(1-x),\forall x\in [0,1]$. Moreover, there exist positive constants $\tilde{m}_{\theta}$ and $\tilde{M}_{\theta}$ that only depend on $\theta$, such that $\tilde{m}_{\theta}\leq \tilde{\rho}_{\theta}(x,y)\leq \tilde{M}_{\theta}$ for every $(x,y)\in [0,1]^2$.  
\end{proposition}


\subsection{Main results for the $L^1$ model}\label{Sect.1.2}

In this subsection, we present the main results for the $L^1$ model. The following theorem implies a law of large numbers for the parameter regime $\beta\sim \theta n^{-1}$ with fixed $\theta>0$.

\begin{theorem}\label{limit_l1_1}
Let $(\beta_n)_{n=1}^{\infty}$ be an arbitrary sequence of positive numbers such that $\lim_{n\rightarrow\infty} n\beta_n=\theta>0$. Let $\sigma$ be drawn from $\mathbb{P}_{n,\beta_n}$. Then we have
\begin{equation}
    \frac{LIS(\sigma)}{\sqrt{n}}\xrightarrow[]{L^1} 2\int_0^1\sqrt{\rho_{\theta}(x,x)}dx,
\end{equation}
where $\rho_{\theta}(\cdot,\cdot)$ is defined in Proposition \ref{Densi.l1}.
\end{theorem}





The following theorem implies a law of large numbers for the parameter regime $n^{-1}\ll \beta\ll 1$. 

\begin{theorem}\label{limit_l1_2}
Let $(\beta_n)_{n=1}^{\infty}$ be an arbitrary sequence of positive numbers such that $\lim_{n\rightarrow\infty}\beta_n=0$ and $\lim_{n\rightarrow\infty} n\beta_n=\infty$. Let $\sigma$ be drawn from $\mathbb{P}_{n,\beta_n}$. Then we have
\begin{equation}
    \frac{LIS(\sigma)}{n\sqrt{\beta_n}}\xrightarrow[]{L^1} \sqrt{2}.
\end{equation}
\end{theorem}


\subsection{Main results for the $L^2$ model}\label{Sect.1.3}

In this subsection, we present the main results for the $L^2$ model. The following theorem implies a law of large numbers for the parameter regime $\beta\sim \theta n^{-2}$ with fixed $\theta>0$. 

\begin{theorem}\label{limit_l2_1}
Let $(\beta_n)_{n=1}^{\infty}$ be an arbitrary sequence of positive numbers such that $\lim_{n\rightarrow\infty} n^2\beta_n=\theta>0$. Let $\sigma$ be drawn from $\tilde{\mathbb{P}}_{n,\beta_n}$. Then we have
\begin{equation}
    \frac{LIS(\sigma)}{\sqrt{n}}\xrightarrow[]{L^1} 2\int_0^1\sqrt{\tilde{\rho}_{\theta}(x,x)}dx,
\end{equation}
where $\tilde{\rho}_{\theta}(\cdot,\cdot)$ is defined in Proposition \ref{Densi.l2}.
\end{theorem}

The proof of Theorem \ref{limit_l2_1} is similar to that of Theorem \ref{limit_l1_1} and is therefore omitted.

The following theorem implies a law of large numbers for the parameter regime $n^{-2}\ll \beta\ll 1$. 

\begin{theorem}\label{limit_l2_2}
Let $(\beta_n)_{n=1}^{\infty}$ be an arbitrary sequence of positive numbers such that $\lim_{n\rightarrow\infty}\beta_n=0$ and $\lim_{n\rightarrow\infty} n^2\beta_n=\infty$. Let $\sigma$ be drawn from $\tilde{\mathbb{P}}_{n,\beta_n}$. Then we have
\begin{equation}
    \frac{LIS(\sigma)}{n \beta_n^{1\slash 4}}\xrightarrow[]{L^1}  2\pi^{-1\slash 4}.
\end{equation}
\end{theorem}


\bigskip

The rest of this paper is organized as follows. In Section \ref{Sect.2}, we present background materials and preliminary results that will be used in the proofs of the main results. The proofs of Theorems \ref{limit_l1_1}, \ref{limit_l1_2}, and \ref{limit_l2_2} are given in Sections \ref{Sect.3}, \ref{Sect.4}, and \ref{Sect.5}, respectively.

\subsection{Acknowledgement}

The author wishes to thank his PhD advisor, Persi Diaconis, for encouragement, support, and many helpful conversations. The author also thanks Sumit Mukherjee and Wenpin Tang for their helpful comments.

\section{Background and preliminary results}\label{Sect.2}

In this section, we present background materials and preliminary results that will be used in the proofs of our main results. In Section \ref{Sect.2.1}, we review and adapt the notion of ``refined paths'' from \cite{MS}. Then we review hit and run algorithms for sampling from the $L^1$ and $L^2$ models in Section \ref{Sect.1.5}. Finally, we present several preliminary results in Section \ref{Sect.2.2}. 

\subsection{Refined paths and the length of the longest increasing subsequence}\label{Sect.2.1}

In this subsection, we review and adapt the notion of ``refined paths'' introduced in \cite{MS}. Consider any $A_1$, $A_2$, $B_1$, $B_2$, $T_1$, $T_2$, $K_0$ such that $0\leq A_1<A_2 \leq 1$, $0\leq B_1<B_2\leq 1$, and $T_1,T_2,K_0\in\mathbb{N}^{*}$. Let 
\begin{equation*}
    \delta_1:=(A_2-A_1)\slash T_1,\quad \delta_2:=(B_2-B_1)\slash T_2.
\end{equation*}
We assume that $\min\{T_1,T_2\}\geq 2$ throughout this subsection.

Now we decompose the rectangle $(A_1,A_2]\times (B_1,B_2]$ into $T_1T_2$ sub-rectangles. For any $k\in [T_1],k'\in [T_2]$, let
\begin{equation*}
    R_{k,k'}:=(A_1+(k-1)\delta_1,A_1+k\delta_1]\times (B_1+(k'-1)\delta_2,B_1+k'\delta_2].
\end{equation*}
Note that $\{R_{k,k'}\}_{k\in [T_1],k'\in [T_2]}$ are disjoint and 
\begin{equation}\label{Uni}
    (A_1,A_2]\times (B_1,B_2]=\bigcup_{k\in [T_1],k'\in[T_2]} R_{k,k'}.
\end{equation}
We define a \textbf{basic path} to be a sequence $(i_1,j_1),\cdots, (i_{T_1+T_2-1},j_{T_1+T_2-1})$ such that $(i_1,j_1)=(1,1)$, $(i_{T_1+T_2-1},j_{T_1+T_2-1})=(T_1,T_2)$, and for any $l\in [T_1+T_2-2]$, $(i_{l+1}-i_{l},j_{l+1}-j_{l})\in\{(1,0),(0,1)\}$. We note that for any $l\in [T_1+T_2-2]$:
\begin{itemize}
    \item If $(i_{l+1}-i_l,j_{l+1}-j_l)=(1,0)$, then
    \begin{equation*}
        \overline{R}_{i_l,j_l}\cap \overline{R}_{i_{l+1},j_{l+1}}=\{A_1+i_l\delta_1\}\times [B_1+(j_l-1)\delta_2,B_1+j_l\delta_2];
    \end{equation*}
    \item If $(i_{l+1}-i_l,j_{l+1}-j_l)=(0,1)$, then
    \begin{equation*}
        \overline{R}_{i_l,j_l}\cap \overline{R}_{i_{l+1},j_{l+1}}= [A_1+(i_l-1)\delta_1,A_1+i_l\delta_1]\times \{B_1+j_l\delta_2\}.
    \end{equation*}
\end{itemize}
Hereafter, for any set $A\subseteq \mathbb{R}^2$, we denote by $\bar{A}$ the closure of $A$.

In the following, we define \textbf{refined paths}, which are refined versions of basic paths. The set of refined paths, denoted by $\Pi_{A_1,A_2;B_1,B_2}^{T_1,T_2,K_0}$, is defined as the set of sequences $\Gamma$ of the following form:
\begin{equation}\label{G}
    (i_1,j_1),r_1,(i_2,j_2),r_2,\cdots,r_{T_1+T_2-2}, (i_{T_1+T_2-1},j_{T_1+T_2-1}),
\end{equation}
where $r_1,\cdots,r_{T_1+T_2-2}\in [K_0]$, the sequence $(i_1,j_1),\cdots,(i_{T_1+T_2-1},j_{T_1+T_2-1})$ forms a basic path, and the following condition holds: For any $l\in [T_1+T_2-3]$, if $i_l=i_{l+1}=i_{l+2}$ or $j_l=j_{l+1}=j_{l+2}$, then $r_{l+1}\geq r_l$. We also denote
\begin{equation}
    \Pi^{T_1,T_2,K_0}:=\Pi_{0,1;0,1}^{T_1,T_2,K_0}.
\end{equation}

Now we define several quantities that are associated with the refined path $\Gamma$ as given in (\ref{G}). For every $l\in [T_1+T_2-2]$, if $(i_{l+1}-i_l,j_{l+1}-j_l)=(1,0)$, we define
\begin{equation}\label{Interval}
    I_l(\Gamma):=\{A_1+i_l\delta_1\}\times \Big(B_1+(j_l-1)\delta_2+\frac{(r_l-1)\delta_2}{K_0},B_1+(j_l-1)\delta_2+\frac{r_l\delta_2}{K_0}\Big];
\end{equation}
if $(i_{l+1}-i_l,j_{l+1}-j_l)=(0,1)$, we define
\begin{equation}\label{Interval2}
    I_l(\Gamma):=\Big(A_1+(i_l-1)\delta_1+\frac{(r_l-1)\delta_1}{K_0},A_1+(i_l-1)\delta_1+\frac{r_l\delta_1}{K_0}\Big]\times \{B_1+j_l\delta_2\}.
\end{equation}
For every $l\in   [T_1+T_2-2]$, we define $(x_l(\Gamma),y_l(\Gamma))$ to be the midpoint of the interval $I_l(\Gamma)$, and define $(a_l(\Gamma),b_l(\Gamma))$, $(c_l(\Gamma),d_l(\Gamma))$ to be the two endpoints of $I_l(\Gamma)$ such that $a_l(\Gamma)\leq c_l(\Gamma)$ and $b_l(\Gamma)\leq d_l(\Gamma)$. Moreover, we let 
\begin{equation*}
    x_0(\Gamma)=a_0(\Gamma)=c_0(\Gamma)=A_1, \quad y_0(\Gamma)=b_0(\Gamma)=d_0(\Gamma)=B_1;
\end{equation*}
\begin{eqnarray*}
   && x_{T_1+T_2-1}(\Gamma)=a_{T_1+T_2-1}(\Gamma)=c_{T_1+T_2-1}(\Gamma)=A_2, \\ && y_{T_1+T_2-1}(\Gamma)=b_{T_1+T_2-1}(\Gamma)=d_{T_1+T_2-1}(\Gamma)=B_2.
\end{eqnarray*}

The following lemma gives upper and lower bounds for the length of the longest increasing subsequence of a permutation based on refined paths. It is adapted from \cite[Lemma 5.1]{MS}.

\begin{lemma}\label{L1}
Suppose that $A_1,A_2,B_1,B_2,T_1,T_2,K_0$ satisfy the conditions as stated in the preceding. Then for any $n\in\mathbb{N}^{*}$, any $\sigma\in S_n$, any $\alpha,\gamma>0$, any $\kappa\in\mathbb{R}$, and any refined path $\Gamma\in \Pi_{A_1,A_2;B_1,B_2}^{T_1,T_2,K_0}$, we have
\begin{eqnarray}
  &&  LIS(\sigma|_{(\kappa+\alpha A_1,\kappa+\alpha A_2]\times (\kappa+\gamma B_1, \kappa+\gamma B_2]}) \nonumber \\
    &\geq& \sum_{l=1}^{T_1+T_2-1}LIS(\sigma|_{(\kappa+\alpha x_{l-1}(\Gamma),\kappa+\alpha x_l(\Gamma)]\times (\kappa+\gamma y_{l-1}(\Gamma), \kappa+\gamma y_l(\Gamma)]}).
\end{eqnarray} 
Moreover, for any $n\in\mathbb{N}^{*}$, any $\sigma\in S_n$, any $\alpha,\gamma>0$, and any $\kappa\in\mathbb{R}$, we have
\begin{eqnarray}
    && LIS(\sigma|_{(\kappa+\alpha A_1, \kappa+\alpha A_2]\times (\kappa+\gamma B_1,\kappa+\gamma B_2]})    \nonumber\\
    &\leq& \max_{\Gamma\in \Pi_{A_1,A_2;B_1,B_2}^{T_1,T_2,K_0}}\sum_{l=1}^{T_1+T_2-1}LIS(\sigma|_{[\kappa+\alpha a_{l-1}(\Gamma),\kappa+\alpha c_l(\Gamma)]\times [\kappa+\gamma b_{l-1}(\Gamma),\kappa+\gamma d_l(\Gamma)]}).\nonumber\\
    &&
\end{eqnarray}
\end{lemma}
\begin{proof}
Note that by definition,
\begin{equation*}
    x_0(\Gamma)\leq x_1(\Gamma)\leq\cdots\leq x_{T_1+T_2-1}(\Gamma), \quad y_0(\Gamma)\leq y_1(\Gamma)\leq\cdots\leq y_{T_1+T_2-1}(\Gamma).
\end{equation*}
For any $l\in [T_1+T_2-1]$, let
\begin{equation*}
    t_l:=LIS(\sigma|_{(\kappa+\alpha x_{l-1}(\Gamma),\kappa+\alpha x_l(\Gamma)]\times (\kappa+\gamma y_{l-1}(\Gamma),\kappa+\gamma y_l(\Gamma)]}).
\end{equation*}
Suppose that $i_{l,1},\cdots,i_{l,t_l}\in [n]$ satisfy
\begin{equation*}
    \kappa+\alpha x_{l-1}(\Gamma)<i_{l,1}<\cdots<i_{l,t_l}\leq \kappa+\alpha x_l(\Gamma),
\end{equation*}
\begin{equation*}
    \kappa+\gamma y_{l-1}(\Gamma)<\sigma(i_{l,1})<\cdots<\sigma(i_{l,t_l})\leq \kappa+\gamma y_l(\Gamma).
\end{equation*}
Now we have 
\begin{eqnarray*}
    && \kappa+\alpha A_1=\kappa+\alpha x_0(\Gamma)<i_{1,1}<\cdots<i_{1,t_1}\leq  \kappa+\alpha x_1(\Gamma)<i_{2,1}<\cdots<i_{2,t_2}\nonumber\\
    && \leq\cdots 
     <i_{T_1+T_2-1,1}<\cdots<i_{T_1+T_2-1,t_{T_1+T_2-1}}\leq \kappa+\alpha x_{T_1+T_2-1}(\Gamma)=\kappa+\alpha A_2,
\end{eqnarray*}
\begin{eqnarray*}
    && \kappa+\gamma B_1=\kappa+\gamma y_0(\Gamma)<\sigma(i_{1,1})<\cdots<\sigma(i_{1,t_1})\leq \kappa+\gamma y_1(\Gamma)<\sigma(i_{2,1})<\cdots<\sigma(i_{2,t_2})\nonumber\\
    &&\leq \cdots  <\sigma(i_{T_1+T_2-1,1})<\cdots<\sigma(i_{T_1+T_2-1,t_{T_1+T_2-1}})\leq \kappa+\gamma y_{T_1+T_2-1}(\Gamma)=\kappa+\gamma B_2.
\end{eqnarray*}
Hence
\begin{eqnarray*}
   && LIS(\sigma|_{(\kappa+\alpha A_1,\kappa+\alpha A_2]\times (\kappa+\gamma B_1,\kappa+\gamma B_2]})\geq \sum_{l=1}^{T_1+T_2-1}t_l\nonumber\\
   &=&\sum_{l=1}^{T_1+T_2-1}LIS(\sigma|_{(\kappa+\alpha x_{l-1}(\Gamma), \kappa+\alpha x_l(\Gamma)]\times (\kappa+\gamma y_{l-1}(\Gamma), \kappa+\gamma y_l(\Gamma)]}).
\end{eqnarray*}

Now let 
\begin{equation*}
    N:=LIS(\sigma|_{(\kappa+\alpha A_1, \kappa+\alpha A_2]\times (\kappa+\gamma B_1, \kappa+\gamma B_2]}).
\end{equation*}
Suppose that $k_1,\cdots,k_N\in [n]$ satisfy
\begin{equation*}
 \kappa+ \alpha A_1<k_1<\cdots<k_N\leq \kappa+\alpha A_2,\quad \kappa+\gamma B_1<\sigma(k_1)<\cdots<\sigma(k_N)\leq \kappa+\gamma B_2.
\end{equation*}
We construct a refined path $\Gamma\in \Pi_{A_1,A_2;B_1,B_2}^{T_1,T_2,K_0}$ by the following procedure. 

For each $t\in [N]$, we take $p_t\in [T_1]$ and $q_t\in [T_2]$ such that
\begin{equation*}
    (\alpha^{-1}(k_t-\kappa),\gamma^{-1}(\sigma(k_t)-\kappa))\in R_{p_t,q_t}.
\end{equation*}
Note that if $N\geq 2$, for any $t\in [N-1]$, we have $p_{t+1}\geq p_t,q_{t+1}\geq q_t$. We take $\delta\in (0,\min\{\delta_1\slash 2,\delta_2\slash 2, \alpha^{-1}(k_1-\kappa)-A_1,\gamma^{-1}(\sigma(k_1)-\kappa)-B_1\})$, and draw a path (moving first rightward and then upward) between the two points $(A_1+\delta,B_1+\delta)$ and $(\alpha^{-1}(k_1-\kappa),\gamma^{-1}(\sigma(k_1)-\kappa))$. If $N\geq 2$, for every $t\in [N-1]$, we do the following:
\begin{itemize}
    \item If $p_{t+1}>p_t$ and $q_{t+1}>q_t$, we draw a path (moving first rightward and then upward) between the two points $(\alpha^{-1}(k_t-\kappa),\gamma^{-1}(\sigma(k_t)-\kappa))$ and $(\alpha^{-1}(k_{t+1}-\kappa),\gamma^{-1}(\sigma(k_{t+1})-\kappa))$.
    \item If $p_{t+1}=p_t$ or $q_{t+1}=q_t$, we draw a straight line between the two points $(\alpha^{-1}(k_t-\kappa),\gamma^{-1}(\sigma(k_t)-\kappa))$ and $(\alpha^{-1}(k_{t+1}-\kappa),\gamma^{-1}(\sigma(k_{t+1})-\kappa))$.
\end{itemize}
Finally, we draw a path (moving first rightward and then upward) between the two points $(\alpha^{-1}(k_N-\kappa),\gamma^{-1}(\sigma(k_N)-\kappa))$ and $(A_2,B_2)$. 

Now we take the basic path as the sequence formed by those $(k,k')$ with $k\in [T_1],k'\in [T_2]$ such that $R_{k,k'}$ has a non-empty intersection with the path described in the previous paragraph (with the obvious ordering of the sequence). Below we denote this basic path by $(i_1,j_1),\cdots,(i_{T_1+T_2-1},j_{T_1+T_2-1})$. For each $l \in[T_1+T_2-2]$, we pick the largest $r_l\in [K_0]$ such that the corresponding interval as defined in (\ref{Interval}) and (\ref{Interval2}) (where we take the aforementioned basic path; note that the interval only depends on the basic path and $r_l$) has a non-empty intersection with the path specified in the previous paragraph. Let $\Gamma$ be the following sequence:
\begin{equation*}
    (i_1,j_1), r_1, (i_2,j_2),r_2,\cdots, r_{T_1+T_2-2}, (i_{T_1+T_2-1},j_{T_1+T_2-1}).
\end{equation*}
It can be checked that $\Gamma$ is a refined path in $\Pi_{A_1,A_2;B_1,B_2}^{T_1,T_2,K_0}$. We also note that for any $l\in [T_1+T_2-1]$ and $t\in [N]$ such that $(\alpha^{-1}(k_t-\kappa),\gamma^{-1}(\sigma(k_t)-\kappa))\in R_{i_l,j_l}$, we have
\begin{equation*}
    (\alpha^{-1}(k_t-\kappa),\gamma^{-1}(\sigma(k_t)-\kappa))\in [a_{l-1}(\Gamma),c_l(\Gamma)]\times [b_{l-1}(\Gamma),d_l(\Gamma)].
\end{equation*}
Hence 
\begin{eqnarray}\label{Eq1.4.2}
   && \#\{t\in [N]:(\alpha^{-1}(k_t-\kappa),\gamma^{-1}(\sigma(k_t)-\kappa))\in R_{i_l,j_l}\} \nonumber \\
   &\leq&LIS(\sigma|_{[\kappa+\alpha a_{l-1}(\Gamma),\kappa+\alpha c_l(\Gamma)]\times [\kappa+\gamma b_{l-1}(\Gamma),\kappa+\gamma d_l(\Gamma)]}).
\end{eqnarray}
Note that for any $t\in [N]$, we have
\begin{equation}\label{Eq1.4.4}
    (\alpha^{-1}(k_t-\kappa),\gamma^{-1}(\sigma(k_t)-\kappa))\in \bigcup_{l=1}^{T_1+T_2-1} R_{i_l,j_l}.
\end{equation}
By (\ref{Eq1.4.2}) and (\ref{Eq1.4.4}), we conclude that
\begin{eqnarray*}
   &&  LIS(\sigma|_{(\kappa+\alpha A_1,\kappa+\alpha A_2]\times (\kappa+\gamma B_1,\kappa+\gamma  B_2]})=N\nonumber\\
   &\leq&
     \sum_{l=1}^{T_1+T_2-1}LIS(\sigma|_{[\kappa+\alpha a_{l-1}(\Gamma),\kappa+\alpha c_l(\Gamma)]\times [\kappa+\gamma b_{l-1}(\Gamma),\kappa+\gamma d_l(\Gamma)]})\nonumber  \\
   &\leq& \max_{\Gamma\in \Pi_{A_1,A_2;B_1,B_2}^{T_1,T_2,K_0}}\sum_{l=1}^{T_1+T_2-1}LIS(\sigma|_{[\kappa+\alpha a_{l-1}(\Gamma),\kappa+\alpha c_l(\Gamma)]\times [\kappa+\gamma b_{l-1}(\Gamma),\kappa+\gamma d_l(\Gamma)]}).
\end{eqnarray*}

\end{proof}

\subsection{Hit and run algorithms for sampling from the $L^1$ and $L^2$ models}\label{Sect.1.5}

Hit and run algorithms are a broad class of Markov chain Monte Carlo algorithms that includes the celebrated Swendsen-Wang algorithm for sampling from the Ising model. We refer the reader to \cite{AD} for a comprehensive overview of hit and run algorithms. 

The proofs of the main results of this paper are based on hit and run algorithms for sampling from Mallows permutation models with $L^1$ and $L^2$ distances. The algorithm for the $L^2$ model is introduced in \cite{AD}, and the algorithm for the $L^1$ model is introduced in \cite{Zho1}. In this subsection, we briefly review both algorithms, and refer to \cite[Section 2]{Zho1} for further details.



In the proof of Theorem \ref{limit_l2_2}, a resampling algorithm for the $L^2$ model is used. The resampling algorithm preserves the probability measure $\tilde{\mathbb{P}}_{n,\beta}$, and is related to the hit and run algorithm for the $L^2$ model. We introduce the resampling algorithm at the end of this subsection.

\paragraph{Hit and run algorithm for the $L^1$ model}

For any $n\in\mathbb{N}^{*}$ and $\beta>0$, the hit and run algorithm for sampling from the $L^1$ model $\mathbb{P}_{n,\beta}$ is a Markov chain on $S_n$ whose each step consists of the following two sequential parts:
\begin{itemize}
  \item Starting from $\sigma$, for each $i\in [n]$, independently sample $u_i$ from the uniform distribution on $[0, e^{-2\beta (\sigma(i)-i)_{+} }]$. Let $b_i=i-\log(u_i)\slash (2\beta)$ for every $i\in [n]$.
  \item Sample $\sigma'$ uniformly from the set $\{\tau\in S_n: \tau(i)\leq b_i\text{ for every }i\in [n]\}$, and move to the new state $\sigma'$.
\end{itemize}

The sampling problem in the second part can be efficiently implemented as follows: Look at places $i$ where $b_i\geq n$, and place the symbol $n$ at a uniform choice among these places; look at places where $b_i\geq n-1$, and place the symbol $n-1$ at a uniform choice among these places (with the place where the symbol $n$ was placed deleted); and so on. This gives the permutation $\sigma'$. Here, we say that the symbol $j$ is placed at the place $i$ if $\sigma'(i)=j$.

The stationary distribution of the above Markov chain is $\mathbb{P}_{n,\beta}$.

\paragraph{Hit and run algorithm for the $L^2$ model}

For any $n\in\mathbb{N}^{*}$ and $\beta>0$, the hit and run algorithm for sampling from the $L^2$ model $\tilde{\mathbb{P}}_{n,\beta}$ is a Markov chain on $S_n$ whose each step consists of the following two sequential parts:
\begin{itemize}
  \item Starting from $\sigma$, for each $i\in [n]$, independently sample $u_i$ from the uniform distribution on $[0, e^{2\beta i\sigma(i)}]$. Let $b_i=\log(u_i)\slash (2\beta i)$ for every $i\in [n]$.
  \item Sample $\sigma'$ uniformly from the set $\{\tau\in S_n: \tau(i)\geq b_i\text{ for every }i\in [n]\}$, and move to the new state $\sigma'$.
\end{itemize}

Again, the sampling problem in the second part can be efficiently implemented: Look at places $i$ where $b_i\leq 1$, and place the symbol $1$ at a uniform choice among these places; look at places where $b_i\leq 2$, and place the symbol $2$ at a uniform choice among these places (with the place where the symbol $1$ was placed deleted); and so on. This gives the permutation $\sigma'$. 

The stationary distribution of the above Markov chain is $\tilde{\mathbb{P}}_{n,\beta}$.

\paragraph{A resampling algorithm for the $L^2$ model} 

In the following, we introduce a resampling algorithm for the $L^2$ model. The resampling algorithm is related to the hit and run algorithm for the $L^2$ model.

We assume that $n\in\mathbb{N}^{*}$ and $\beta>0$. The inputs of the resampling algorithm are given by a permutation $\sigma\in S_n$, two sets $S_X,S_Y\subseteq [n]$, and a real number $t_0<\min\{i\in [n]: i\in S_X\}$. The output of the resampling algorithm is a permutation $\sigma'\in S_n$ obtained by the following two sequential steps:
\begin{itemize}
    \item Suppose that $\{i\in S_X:\sigma(i)\in S_Y\}=\{i_1,\cdots,i_k\}$ (with $i_1<\cdots<i_k$) and $\{j\in S_Y: \sigma^{-1}(j)\in S_X\}=\{j_1,\cdots,j_k\}$ (with $j_1<\cdots<j_k$). For each $t\in [k]$, we independently sample $u_t$ from the uniform distribution on $[0,e^{2\beta (i_t-t_0) \sigma(i_t)}]$, and let $b_t=\log(u_t)\slash (2\beta (i_t-t_0))$.
    \item Sample $\sigma'$ uniformly from the set
    \begin{equation*}
        \{\tau\in S_n: \tau(i_t)\geq b_t \text{ for every }t\in [k], \tau(i)=\sigma(i) \text{ for every }i\in [n]\backslash \{i_1,\cdots,i_k\}\}.
    \end{equation*}
\end{itemize}

The second step can be implemented as follows: Look at places $i_t$ (where $t\in [k]$) such that $b_t\leq j_1$, and place the symbol $j_1$ at a uniform choice among these places; look at the remaining places $i_t$ (where $t\in [k]$) such that $b_t\leq j_2$ (with the place where $j_1$ was placed deleted), and place the symbol $j_2$ at a uniform choice among these places; and so on. We further take $\sigma'(i)=\sigma(i)$ for every $i\in [n]\backslash\{i_1,\cdots,i_k\}$. This gives the permutation $\sigma'$.

The following lemma shows that the above resampling algorithm preserves the probability measure $\tilde{\mathbb{P}}_{n,\beta}$.

\begin{lemma}\label{L2.2}
Assume that $n\in\mathbb{N}^{*}$ and $\beta>0$. For any two non-random sets $S_X,S_Y\subseteq [n]$ and any fixed $t_0<\min\{i\in [n]: i\in S_X\}$, the following holds. Let $\sigma$ be drawn from $\tilde{\mathbb{P}}_{n,\beta}$, and let $\sigma'$ be the output of the above resampling algorithm with inputs $\sigma,S_X,S_Y, t_0$. Then the distribution of $\sigma'$ is given by $\tilde{\mathbb{P}}_{n,\beta}$.
\end{lemma}

\begin{proof}

For any $\tau,\tau'\in S_n$, we denote by $K(\tau,\tau')$ the probability that the resampling algorithm with inputs $\tau,S_X,S_Y,t_0$ outputs $\tau'$. Note that if $K(\tau,\tau')\neq 0$, then necessarily 
\begin{equation}\label{Eq4.1.1}
    S(\tau)\cap ([n]^2\backslash(S_X\times S_Y))=S(\tau')\cap ([n]^2\backslash(S_X\times S_Y)).
\end{equation}
Below we assume that (\ref{Eq4.1.1}) holds, and let $i_1,\cdots,i_k$ and $j_1,\cdots,j_k$ be defined as in the first step of the resampling algorithm (with inputs $\tau,S_X,S_Y,t_0$). We have
\begin{eqnarray*}
    &&K(\tau,\tau')\\
    &=& e^{-2\beta\sum_{t=1}^k (i_t-t_0)\tau(i_t)}\int_{\prod_{i=1}^n[0,e^{2\beta(i_t-t_0)\tau(i_t)}]}du_1\cdots du_n\\
    &&\frac{\mathbbm{1}_{\tau'(i_t)\geq b_t,\forall t\in [k] \text{ and } \tau'(i)=\tau(i), \forall i\in [n]\backslash \{i_1,\cdots,i_k\}}}{|\{\kappa\in S_n:\kappa(i_t)\geq b_t,\forall t\in [k] \text{ and }  \kappa(i)=\tau(i),\forall i\in [n]\backslash \{i_1,\cdots,i_k\}\}|}\\
   &=& e^{-2\beta\sum_{t=1}^k (i_t-t_0)\tau(i_t)}\int_{\prod_{i=1}^n[0,e^{2\beta(i_t-t_0)\min\{\tau(i_t),\tau'(i_t)\}}]}du_1\cdots du_n\\
    &&\frac{1}{|\{\kappa\in S_n:\kappa(i_t)\geq \log(u_t)\slash (2\beta(i_t-t_0)),\forall t\in [k]\text{ and } \kappa(i)=\tau(i),\forall i\in [n]\backslash \{i_1,\cdots,i_k\}\}|},
\end{eqnarray*}
where $b_t=\log(u_t)\slash (2\beta (i_t-t_0))$ for every $t\in [k]$. Similarly,
\begin{eqnarray*}
&&K(\tau',\tau)\\
&=& e^{-2\beta\sum_{t=1}^k (i_t-t_0)\tau'(i_t)}\int_{\prod_{i=1}^n[0,e^{2\beta(i_t-t_0)\min\{\tau(i_t),\tau'(i_t)\}}]}du_1\cdots du_n\\
&&\frac{1}{|\{\kappa\in S_n:\kappa(i_t)\geq \log(u_t)\slash (2\beta(i_t-t_0)),\forall t\in [k]\text{ and } \kappa(i)=\tau'(i),\forall i\in [n]\backslash \{i_1,\cdots,i_k\}\}|}\\
&=& e^{-2\beta\sum_{t=1}^k (i_t-t_0)\tau'(i_t)}\int_{\prod_{i=1}^n[0,e^{2\beta(i_t-t_0)\min\{\tau(i_t),\tau'(i_t)\}}]}du_1\cdots du_n\\
    &&\frac{1}{|\{\kappa\in S_n:\kappa(i_t)\geq \log(u_t)\slash (2\beta(i_t-t_0)), \forall t\in [k] \text{ and } \kappa(i)=\tau(i), \forall i\in [n]\backslash \{i_1,\cdots,i_k\}\}|}.
\end{eqnarray*}
Hence 
\begin{equation}\label{E1}
   e^{2\beta\sum_{t=1}^k (i_t-t_0)\tau(i_t)} K(\tau,\tau')=e^{2\beta\sum_{t=1}^k (i_t-t_0)\tau'(i_t)} K(\tau',\tau).
\end{equation}
Now note that
\begin{eqnarray*}
 && \tilde{H}(\tau,Id)= \sum_{i=1}^n i^2+\sum_{i=1}^n \tau(i)^2-2\sum_{i=1}^n i\tau(i)=2 \sum_{i=1}^n i^2  -2\sum_{i=1}^n i\tau(i)\\
 &=& 2 \sum_{i=1}^n i^2  -2\sum_{i\in [n]\backslash \{i_1,\cdots,i_k\}} i\tau(i)-2t_0\sum_{t=1}^k \tau(i_t)-2\sum_{t=1}^k (i_t-t_0)\tau(i_t) \nonumber\\
 &=& 2 \sum_{i=1}^n i^2  -2\sum_{i\in [n]\backslash \{i_1,\cdots,i_k\}} i\tau(i)-2t_0\sum_{t=1}^k j_t-2\sum_{t=1}^k (i_t-t_0)\tau(i_t).
\end{eqnarray*}
Similarly, we have 
\begin{eqnarray*}
&&\tilde{H}(\tau',Id)=2\sum_{i=1}^n i^2-2\sum_{i=1}^n i\tau'(i)\nonumber\\
&=&  2 \sum_{i=1}^n i^2  -2\sum_{i\in [n]\backslash \{i_1,\cdots,i_k\}} i\tau'(i)-2t_0\sum_{t=1}^k \tau'(i_t)-2\sum_{t=1}^k (i_t-t_0)\tau'(i_t) \nonumber\\
 &=& 2 \sum_{i=1}^n i^2  -2\sum_{i\in [n]\backslash \{i_1,\cdots,i_k\}} i\tau(i)-2t_0\sum_{t=1}^k j_t-2\sum_{t=1}^k (i_t-t_0)\tau'(i_t).
\end{eqnarray*}
Hence
\begin{equation}\label{E2}
    \frac{\tilde{\mathbb{P}}_{n,\beta}(\tau)}{\tilde{\mathbb{P}}_{n,\beta}(\tau')}=\frac{e^{2\beta\sum_{t=1}^k(i_t-t_0)\tau(i_t)}}{e^{2\beta \sum_{t=1}^k(i_t-t_0)\tau'(i_t)}}.
\end{equation}
Combining (\ref{E1}) and (\ref{E2}), we obtain that
\begin{equation}\label{E3}
    \tilde{\mathbb{P}}_{n,\beta}(\tau)K(\tau,\tau')=\tilde{\mathbb{P}}_{n,\beta}(\tau')K(\tau',\tau).
\end{equation}
Note that when $S(\tau)\cap ([n]^2\backslash(S_X\times S_Y))\neq S(\tau')\cap ([n]^2\backslash(S_X\times S_Y))$, we have $K(\tau,\tau')=K(\tau',\tau)=0$, and (\ref{E3}) still holds.

Now let $\sigma,\sigma'$ be given as in the statement of the lemma. For any $\tau'\in S_n$, noting (\ref{E3}), we obtain that
\begin{eqnarray*}
    \mathbb{P}(\sigma'=\tau')&=&\sum_{\tau\in S_n}\mathbb{P}(\sigma=\tau)K(\tau,\tau')=\sum_{\tau\in S_n}\tilde{\mathbb{P}}_{n,\beta}(\tau)K(\tau,\tau')\nonumber\\
    &=& \sum_{\tau\in S_n} \tilde{\mathbb{P}}_{n,\beta}(\tau')K(\tau',\tau)=\tilde{\mathbb{P}}_{n,\beta}(\tau').
\end{eqnarray*}
Hence the distribution of $\sigma'$ is given by $\tilde{\mathbb{P}}_{n,\beta}$.

\end{proof}

\subsection{Preliminary results}\label{Sect.2.2}

In this subsection, we present several preliminary results, which will be used in the proofs of the main results.

The following tail bound on the length of the longest increasing subsequence of a uniformly random permutation follows from \cite[Theorem 1.1]{LM} and \cite[Theorem 1.1]{LMR}.

\begin{proposition}\label{P1}
For any $\delta_0\in (0,1\slash 3)$, there exists a positive constant $C_{\delta_0}$ that only depends on $\delta_0$, such that the following holds. For any $n\in\mathbb{N}^{*}$, when $\sigma$ is drawn from the uniform distribution on $S_n$, we have
\begin{equation}
    \mathbb{P}(|LIS(\sigma)-2\sqrt{n}|>n^{1\slash 2-\delta_0})\leq C_{\delta_0} \exp(-n^{(1-3\delta_0)\slash 2}).
\end{equation}
\end{proposition}

In the following, we recall several results from \cite{Zhong2}. These results describe the behavior of the $L^1$ model (when $n^{-1}\ll \beta \ll 1$) and the $L^2$ model (when $n^{-2}\ll \beta \ll 1$). We assume that $n\in\mathbb{N}^{*}$ throughout the rest of this subsection. 

We start with the following three definitions.

\begin{definition}\label{Def2.2}
For every $i\in [n]$ and every $\sigma\in S_n$, we let
\begin{equation}\label{Defnd}
     \mathcal{D}_i(\sigma):=\{j\in [n]: j\leq i, \sigma(j)\geq i+1\}, 
\end{equation}
\begin{equation}
    \mathcal{D}_i'(\sigma):=\{j\in [n]: j\geq i+1, \sigma(j)\leq i\}.
\end{equation}
Note that
\begin{equation}\label{Eq5.1.1}
    |\mathcal{D}_i(\sigma)|=i-|\{j\in [n]:j\leq i,\sigma(j)\leq i\}|=|\mathcal{D}'_i(\sigma)|.
\end{equation}
\end{definition}

\begin{definition}\label{Def2.1}
For any $t_0\in [n]$ and any $\sigma\in S_n$, we define
\begin{equation}
      \mu_{n,t_0}=\beta\sum_{i=1}^n\delta_{(\beta(i-t_0),\beta(\sigma(i)-t_0))},
\end{equation}
\begin{equation}
      \tilde{\mu}_{n,t_0}=\beta^{1\slash 2}\sum_{i=1}^n\delta_{(\beta^{1\slash 2}(i-t_0),\beta^{1\slash 2}(\sigma(i)-t_0))}.
\end{equation}
We also define
\begin{equation}
    d\mu=\frac{1}{2}e^{-|x-y|}dxdy, \quad d\tilde{\mu}=\frac{1}{\sqrt{\pi}}e^{-(x-y)^2}dxdy.
\end{equation}
\end{definition}

\begin{definition}\label{Def2.3}
For any $K>0$, we define $\mathbb{B}_K$ to be the set of Borel measurable functions $f(x,y)$ on $\mathbb{R}^2$ such that $\supp(f) \subseteq [-K,K]^2$ and $\|f\|_{Lip},\|f\|_{\infty}\leq 1$. Here, $\|f\|_{\infty}:=\sup_{\mathbf{x}\in\mathbb{R}^2}|f(\mathbf{x})|$.
\end{definition}

The following two propositions give tail bounds on $|\mathcal{D}_i(\sigma)|$ for any $i\in [n]$ when $\sigma$ is drawn from the $L^1$ or $L^2$ model.

\begin{proposition}[\cite{Zhong2}, Proposition 5.3.1]\label{P2.2}
Assume that $0< \beta\leq C_0$ for a fixed positive constant $C_0$ (independent of $n$). Let $\sigma$ be drawn from $\mathbb{P}_{n,\beta}$. Then there exists a positive constant $C$ that only depends on $C_0$, such that for any $u\geq C\beta^{-1}$ and any $i\in [n]$, 
\begin{equation}
     \mathbb{P}(|\mathcal{D}_i(\sigma)|\geq u)\leq 3\exp(-u\slash 4). 
\end{equation}
\end{proposition}

\begin{proposition}[\cite{Zhong2}, Proposition 5.4.1]\label{P2.2.2}
Assume that $0<\beta\leq C_0$ for a fixed positive constant $C_0$ (independent of $n$). Let $\sigma$ be drawn from $\tilde{\mathbb{P}}_{n,\beta}$. Then there exists a positive constant $C$ that only depends on $C_0$, such that for any $u\geq C\beta^{-1\slash 2}$ and any $i\in [n]$,
\begin{equation}
    \mathbb{P}(|\mathcal{D}_i(\sigma)|\geq u)\leq 3\exp(-u\slash 4). 
\end{equation}
\end{proposition}

The following two propositions describe the behavior of the measures $\mu_{n,t_0}$ and $\tilde{\mu}_{n,t_0}$ defined in Definition \ref{Def2.1}.

\begin{proposition}[\cite{Zhong2}, Theorem 4.2.2]\label{P2.3}
For any $\delta_0\in (0,1)$ and $K>0$, there exist positive constants $C_0,c_0,C_1,C_2$ that only depend on $\delta_0, K$, such that the following holds. For any $\beta>0$, any $C_1\leq r\leq \log(1+\beta^{-1})^{8}$, and any $t_0\in [n]$ such that $r\beta^{-1}+1\leq t_0\leq n-r\beta^{-1}$, when $\sigma$ is drawn from $\mathbb{P}_{n,\beta}$, we have 
\begin{equation}
    \mathbb{P}\Big(\sup_{f\in\mathbb{B}_K}\Big|\int f d\mu_{n,t_0}-\int fd\mu\Big|>C_2(\log{r})^{1\slash 4}r^{-1\slash 8}\Big)\leq C_0 \exp(-c_0\beta^{-(1-\delta_0)}).
\end{equation}
\end{proposition}

\begin{proposition}[\cite{Zhong2}, Theorem 4.2.4]\label{P2.3.2}
For any $\delta_0\in (0,1)$ and $K>0$, there exist positive constants $C_0,c_0,C_1,C_2$ that only depend on $\delta_0, K$, such that the following holds. For any $\beta>0$, any $C_1\leq r\leq \log(1+\beta^{-1\slash 2})^{4}$, and any $t_0\in [n]$ such that $r\beta^{-1\slash 2}+1\leq t_0\leq n-r\beta^{-1\slash 2}$, when $\sigma$ is drawn from $\tilde{\mathbb{P}}_{n,\beta}$, we have 
\begin{equation}
    \mathbb{P}\Big(\sup_{f\in\mathbb{B}_K}\Big|\int f d\tilde{\mu}_{n,t_0}-\int fd\tilde{\mu}\Big|>C_2(\log{r})^{1\slash 4}r^{-1\slash 20}\Big)\leq C_0 \exp(-c_0\beta^{-(1-\delta_0)\slash 2}).
\end{equation}
\end{proposition}

\section{Proof of Theorem \ref{limit_l1_1}}\label{Sect.3}

In this section, we give the proof of Theorem \ref{limit_l1_1}. The proof uses the notion of refined paths as discussed in Section \ref{Sect.2.1} together with the hit and run algorithm for sampling from the $L^1$ model. We first establish a preliminary proposition in Section \ref{Sect.3.1}, and then finish the proof of Theorem \ref{limit_l1_1} in Section \ref{Sect.3.2}.

\subsection{A preliminary proposition}\label{Sect.3.1}

In this subsection, we establish the following proposition, which will be used in the proof of Theorem \ref{limit_l1_1}. We recall the setup in Section \ref{Sect.2.1}. 

\begin{proposition}\label{P3.1}
Let $(\beta_n)_{n=1}^{\infty}$ be an arbitrary sequence of positive numbers such that $\lim_{n\rightarrow\infty} n\beta_n=\theta>0$, and let $\sigma$ be drawn from $\mathbb{P}_{n,\beta_n}$. Consider any $T,K_0\in\mathbb{N}^{*}$ such that $T\geq 4$, any refined path $\Gamma\in \Pi^{T,T,K_0}$, and any $l\in [2T-1]$. There exist positive constants $C_1,T_0\geq 4$ that only depend on $\theta$ and positive constants $C_2,c_2,N_0$ that only depend on $T,K_0$ and the sequence $\{\beta_n\}$, such that the following holds. 

Let
\begin{eqnarray*}
   && Q_{\Gamma,l}:=(x_{l-1}(\Gamma),x_l(\Gamma)]\times ( y_{l-1}(\Gamma), y_l(\Gamma)],\\
   &&Q_{\Gamma,l}':=[ a_{l-1}(\Gamma),   c_l(\Gamma)]\times [ b_{l-1}(\Gamma), d_l(\Gamma)],
\end{eqnarray*}
where the endpoints are defined in Section \ref{Sect.2.1}. Let $\mathscr{A}_{\Gamma,l}$ be the event that 
\begin{eqnarray}\label{Eq10.2}
 &&\Big|LIS(\sigma|_{nQ_{\Gamma,l}})-2\sqrt{n}\Big(\int_{Q_{\Gamma,l}}\rho_{\theta}(x,y)dxdy\Big)^{1\slash 2}\Big|\nonumber\\
 &\leq& C_1 T^{-1\slash 2} n^{1\slash 2} (x_l(\Gamma)-x_{l-1}(\Gamma)+y_l(\Gamma)-y_{l-1}(\Gamma))\nonumber\\
 && +C_1(T^{-5} n^{1\slash 2}+T^{-2\slash 3}n^{1\slash 3}),
\end{eqnarray}
and let $\mathscr{B}_{\Gamma,l}$ be the event that
\begin{eqnarray}\label{Eq10.18}
 &&\Big|LIS(\sigma|_{nQ_{\Gamma,l}'})-2\sqrt{n}\Big(\int_{Q_{\Gamma,l}'}\rho_{\theta}(x,y)dxdy\Big)^{1\slash 2}\Big|\nonumber\\
 &\leq& C_1 T^{-1\slash 2} n^{1\slash 2} (c_l(\Gamma)-a_{l-1}(\Gamma)+d_l(\Gamma)-b_{l-1}(\Gamma))\nonumber\\
 && +C_1(T^{-5} n^{1\slash 2}+T^{-2\slash 3}n^{1\slash 3}).
\end{eqnarray}
When $T\geq T_0$ and $n\geq N_0$, we have
\begin{equation}
    \mathbb{P}((\mathscr{A}_{\Gamma,l})^c)\leq C_2n\exp(-c_2 n^{1\slash 4}), \quad  \mathbb{P}((\mathscr{B}_{\Gamma,l})^c)\leq C_2n\exp(-c_2 n^{1\slash 4}).
\end{equation}
\end{proposition}

The rest of this subsection is devoted to the proof of Proposition \ref{P3.1}. We present the proof for $\mathbb{P}((\mathscr{A}_{\Gamma,l})^c)$, and the proof for $\mathbb{P}((\mathscr{B}_{\Gamma,l})^c)$ is similar.

Throughout the rest of this subsection, we fix any sequence of positive numbers $(\beta_n)_{n=1}^{\infty}$ such that $\lim_{n\rightarrow\infty} n\beta_n=\theta>0$. Note that there exists a positive constant $n_0$ that only depends on the sequence $\{\beta_n\}$, such that for any $n\in \mathbb{N}^{*}$ with $n\geq n_0$, 
\begin{equation}\label{Eq4.4}
    \theta\slash 2\leq n \beta_n\leq 2\theta.
\end{equation}
We assume that $n\in\mathbb{N}^{*}$ and $n\geq n_0$. We also fix any $T,K_0\in \mathbb{N}^{*}$ such that $T\geq 4$, any refined path $\Gamma\in \Pi^{T,T,K_0}$, and any $l\in [2T-1]$. We denote $Q_{l}:=Q_{\Gamma,l}$ and $Q_l':=Q_{\Gamma,l}'$ to simplify the notations. We denote by $C',c'$ positive constants that only depend on $\theta$, and denote by $\tilde{C},\tilde{c}$ positive constants that only depend on $T,K_0$ and the sequence $\{\beta_n\}$. The values of these constants may change from line to line.

\subsubsection{Preliminary estimates}\label{Sect.3.1.1}

In this part, we present some preliminary estimates that will be used in Section \ref{Sect.3.1.2}. We start with the following elementary lemma.

\begin{lemma}\label{Lemma3.1}
For any $m,d\in\mathbb{N}^{*}$ such that $d\leq m$, we have 
\begin{equation}
    \binom{m}{d}\leq \Big(\frac{em}{d}\Big)^d.
\end{equation}
\end{lemma}
\begin{proof}
We have
\begin{equation*}
    \binom{m}{d}=\frac{m(m-1)\cdots (m-d+1)}{d!}\leq \frac{m^d}{d^d}\frac{d^d}{d!}.
\end{equation*}
Note that 
\begin{equation*}
    e^d=\sum_{k=0}^{\infty}\frac{d^k}{k!}\geq \frac{d^d}{d!}.
\end{equation*}
Hence
\begin{equation*}
    \binom{m}{d}\leq \Big(\frac{em}{d}\Big)^d.
\end{equation*}
\end{proof}



The following lemma bounds the number of points from $\{(i,\sigma(i))\}_{i=1}^n$ that lie in the rectangle $n Q_l$ when $\sigma$ is drawn from $\mathbb{P}_{n,\beta_n}$. The proof of this lemma is similar to that of \cite[Theorem 1.5]{M1} and is presented in the appendix.

\begin{lemma}\label{Lemma2}
Assume the setup as given in the preceding and recall the definition of $\rho_{\theta}(\cdot,\cdot)$ from Proposition \ref{Densi.l1}. Let $\sigma$ be drawn from $\mathbb{P}_{n,\beta_n}$. For any $\delta>0$, there exist positive constants $C_0, c_0$ that only depend on $T, K_0,\delta$ and the sequence $\{\beta_n\}$, such that for any $\Gamma\in \Pi^{T,T,K_0}$ and $l\in [2T-1]$,
\begin{equation}
    \mathbb{P}\Big(\Big|n^{-1}|S(\sigma)\cap n Q_l|-\int_{Q_l}\rho_{\theta}(x,y)dxdy\Big|\geq \delta \Big)\leq C_0\exp(-c_0 n).
\end{equation}
\end{lemma}

\subsubsection{Analysis using the hit and run algorithm}\label{Sect.3.1.2}

In this part, based on the hit and run algorithm for the $L^1$ model as introduced in Section \ref{Sect.1.5}, we give the proof of Proposition \ref{P3.1}. Let $M_{\theta}$ and $m_{\theta}$ be defined as in Proposition \ref{Densi.l1}. In the following, we assume that
\begin{equation}\label{Eq7.4}
    n\geq \max\{8 K_0 T,K_0^2 T^3\}, \quad T\geq \max\Big\{\frac{1000 e^{5\theta}}{m_{\theta}},4\Big\}.
\end{equation}
If $x_{l-1}(\Gamma)=x_l(\Gamma)$ or $y_{l-1}(\Gamma)=y_l(\Gamma)$, then $Q_l=\emptyset$ and $LIS(\sigma|_{n Q_l})=0$ for any $\sigma\in S_n$. In the following, we assume that $x_{l-1}(\Gamma)<x_l(\Gamma)$ and $y_{l-1}(\Gamma)<y_l(\Gamma)$. Note that 
\begin{equation}\label{Eq4.1}
    (2 K_0 T)^{-1}\leq x_l(\Gamma)-x_{l-1}(\Gamma)\leq T^{-1},\quad  (2 K_0 T)^{-1}\leq y_l(\Gamma)-y_{l-1}(\Gamma)\leq T^{-1},
\end{equation}
which implies
\begin{equation}\label{Eq3.4}
   \min\{n(x_l(\Gamma)-x_{l-1}(\Gamma)),n(y_l(\Gamma)-y_{l-1}(\Gamma))\}\geq \frac{n}{2K_0 T}\geq  4.
\end{equation}
In the following, we assume that
\begin{eqnarray}\label{Eq1}
 && (n x_{l-1}(\Gamma), n x_l(\Gamma)]\cap \mathbb{N}^{*}=\{s_1,s_1+1,\cdots,s_2\},\nonumber\\
 && (n y_{l-1}(\Gamma),n y_l(\Gamma)]\cap\mathbb{N}^{*}=\{s_1',s_1'+1,\cdots,s_2'\}.
\end{eqnarray}

We consider the two cases $y_{l-1}(\Gamma)\geq 1\slash 3$ and $y_{l-1}(\Gamma)< 1\slash 3$ in \textbf{Cases 1-2} as follows.

\paragraph{Case 1: $y_{l-1}(\Gamma)\geq 1\slash 3$}

We generate $\sigma\in S_n$ through the following procedure. We sample $\sigma_0\in S_n$ from $\mathbb{P}_{n,\beta_n}$, and then run one step of the hit and run algorithm for the $L^1$ model to obtain $\sigma$. As $\mathbb{P}_{n,\beta_n}$ is the stationary distribution of the hit and run algorithm, the distribution of $\sigma$ is given by $\mathbb{P}_{n,\beta_n}$. 

We recall that in the hit and run algorithm, starting from $\sigma_0$, for every $i\in [n]$, we independently sample $u_i$ from the uniform distribution on $[0, e^{-2\beta_n(\sigma_0(i)-i)_{+}}]$ and take $b_i=i-\log(u_i)\slash (2\beta_n)$. For every $i\in [n]$, let
\begin{equation*}
    N_i:=|\{j\in [n]: b_j\geq i\}|-n+i.
\end{equation*}
Then we sample $\sigma$ uniformly from the set 
\begin{equation*}
    \{\tau\in S_n: \tau(i)\leq b_i\text{ for every }i\in [n]\}
\end{equation*}
through the following procedure. Look at the $N_n$ integers $i\in [n]$ with $b_i\geq n$, and pick $Y_n$ uniformly from these integers; then look at the $N_{n-1}$ remaining integers $i\in[n]$ with $b_i\geq n-1$ (with $Y_n$ deleted from the list), and pick $Y_{n-1}$ uniformly from these integers; and so on. In this way we obtain $\{Y_i\}_{i=1}^n$. Finally, we let $\sigma\in S_n$ be such that $\sigma(Y_i)=i$ for every $i\in [n]$.

We bound $N_i$ for each $i\in [n]$ as follows. As $b_j\geq j$ for every $j\in [n]$, we have
\begin{equation*}
    N_i=1+\sum_{j=1}^{i-1}\mathbbm{1}_{b_j\geq i}.
\end{equation*}
If $i=1$, we have $N_i=1$. Below we assume that $i\geq 2$. Let $X_j:=\mathbbm{1}_{b_j\geq i}$ for every $j\in[i-1]$. Note that conditional on $\sigma_0$, $X_1,\cdots,X_{i-1}$ are mutually independent, and for any $j\in [i-1]$, $X_j$ follows the Bernoulli distribution with 
\begin{eqnarray*}
   && \mathbb{P}(X_j=1|\sigma_0)=\mathbb{P}(b_j\geq i|\sigma_0)=\mathbb{P}(u_j\leq e^{-2\beta_n(i-j)}|\sigma_0)\\
   &=& \min\{1,e^{-2\beta_n((i-j)-(\sigma_0(j)-j)_{+})}\}\geq e^{-2\beta_n(i-j)}.
\end{eqnarray*}
By Hoeffding's inequality (see e.g. \cite[Theorem 2.8]{BLM}), for any $t\geq 0$,
\begin{equation*}
    \mathbb{P}\Big(N_i\leq 1+\sum_{j=1}^{i-1}e^{-2\beta_n (i-j)}-it\Big|\sigma_0\Big)\leq e^{-2i t^2}. 
\end{equation*}
Hence
\begin{equation*}
    \mathbb{P}\Big(N_i\leq 1+\sum_{j=1}^{i-1}e^{-2\beta_n (i-j)}-it\Big)=\mathbb{E}\Big[\mathbb{P}\Big(N_i\leq 1+\sum_{j=1}^{i-1}e^{-2\beta_n (i-j)}-it\Big|\sigma_0\Big)\Big]\leq e^{-2i t^2}. 
\end{equation*}
By (\ref{Eq4.4}), $2\beta_n (i-j)\leq 2\beta_n n\leq 4\theta$ for any $j\in [i-1]$, hence
\begin{equation*}
    \mathbb{P}(N_i\leq (e^{-4\theta}-t)i)\leq e^{-2i t^2}. 
\end{equation*}
Setting $t=e^{-4\theta}\slash 2$, we obtain that for every $i\in [n]$,
\begin{equation}\label{Eq7}
    \mathbb{P}(N_i\leq  e^{-4\theta}i\slash 2)\leq \exp(-e^{-8\theta} i\slash 2).
\end{equation}
Note that (\ref{Eq7}) holds trivially for $i=1$.

Recall the definitions of $s_1,s_2,s_1',s_2'$ from (\ref{Eq1}). We set
\begin{eqnarray}\label{Eqs}
 && \mathcal{S}_{1,l}:=\{i\in \{s_1,\cdots,s_2\}\backslash \{Y_{s_2'+1},\cdots,Y_n\}:b_i> s_2'\}, \nonumber\\
 && \mathcal{S}_{2,l}:=\{i\in \{s_1,\cdots,s_2\}\backslash \{Y_{s_2'+1},\cdots,Y_n\}:s_1'\leq b_i\leq s_2'\},\nonumber\\
 && \mathcal{S}_l':=\{i\in \{s_1,\cdots,s_2\}: s_1'\leq   b_i \leq s_2' \},\quad W_l:=|\mathcal{S}_l'|.
\end{eqnarray}
Note that $\mathcal{S}_{2,l}\subseteq \mathcal{S}_l'$. We also let
\begin{eqnarray}\label{Eq3}
 && D_l:=|\{i\in [n]: (i,\sigma(i))\in n Q_l\}|, \nonumber\\
 &&  D_l':=|\{i\in [n]: (i,\sigma(i))\in n Q_l, i\in \mathcal{S}_{2,l}\}|.
\end{eqnarray}

We bound $W_l$ as follows. Note that
\begin{equation*}
    W_l=\sum_{i=s_1}^{s_2} \mathbbm{1}_{s_1'\leq b_i\leq s_2'}.
\end{equation*}
For any $i\in\{s_1,\cdots,s_2\}$,
\begin{eqnarray*}
&& \mathbb{P}(s_1'\leq b_i\leq s_2'|\sigma_0)=e^{-2\beta_n(s_1'-\max\{i,\sigma_0(i)\})_{+}}-e^{-2\beta_n(s_2'-\max\{i,\sigma_0(i)\})_{+}}\\
&\leq& 1-e^{-2\beta_n((s_2'-\max\{i,\sigma_0(i)\})_{+}-(s_1'-\max\{i,\sigma_0(i)\})_{+})}\leq 1-e^{-2\beta_n(s_2'-s_1')}\\
&\leq& 2\beta_n(s_2'-s_1')\leq 2n\beta_n(y_l(\Gamma)-y_{l-1}(\Gamma)).
\end{eqnarray*}
For any $i\in\{s_1,\cdots,s_2\}$, let $Z_i:=\mathbbm{1}_{s_1'\leq b_i\leq s_2'}$. Conditional on $\sigma_0$, $Z_{s_1},\cdots,Z_{s_2}$ are mutually independent, and for every $i\in\{s_1,\cdots,s_2\}$, $Z_i$ follows the Bernoulli distribution with parameter $\mathbb{P}(s_1'\leq b_i\leq s_2'|\sigma_0)$. Hence by Hoeffding's inequality, for any $t\geq 0$, we have
\begin{equation*}
    \mathbb{P}(W_l\geq (s_2-s_1+1)(2n\beta_n (y_l(\Gamma)-y_{l-1}(\Gamma))+t)|\sigma_0)\leq e^{-2(s_2-s_1+1)t^2}.
\end{equation*}
Taking $t=2n\beta_n(y_l(\Gamma)-y_{l-1}(\Gamma))$, we obtain that 
\begin{eqnarray}\label{Eq8}
  &&  \mathbb{P}(W_l\geq 4n\beta_n(s_2-s_1+1)(y_l(\Gamma)-y_{l-1}(\Gamma)))\nonumber\\
  &=& \mathbb{E}[\mathbb{P}(W_l\geq 4n\beta_n(s_2-s_1+1)(y_l(\Gamma)-y_{l-1}(\Gamma))|\sigma_0)] \nonumber\\
  &\leq& \exp(-8n^2\beta_n^2(s_2-s_1+1)(y_l(\Gamma)-y_{l-1}(\Gamma))^2).
\end{eqnarray}

Let
\begin{eqnarray}\label{Eq11}
  L_{1,l}:=LIS(\sigma|_{\mathcal{S}_{1,l}\times (n y_{l-1}(\Gamma), n y_l(\Gamma)]}), \nonumber\\
  L_{2,l}:=LIS(\sigma|_{\mathcal{S}_{2,l}\times (n y_{l-1}(\Gamma),n y_l(\Gamma)]}).
\end{eqnarray}
Below we show that
\begin{equation}\label{Eq2}
    L_{1,l}\leq LIS(\sigma|_{n Q_l}) \leq L_{1,l}+L_{2,l}.
\end{equation}
We denote $LIS(\sigma|_{n Q_l})$ by $L$. By the definition of $LIS(\sigma|_{n Q_l})$, there exist indices $i_1,\cdots,i_L\in [n]$, such that $i_1<\cdots<i_L$, $\sigma(i_1)<\cdots<\sigma(i_L)$, and for every $j\in [L]$, $(i_j,\sigma(i_j))\in n Q_l$. Now note that for any $j\in [L]$, $s_1\leq i_j\leq s_2$ and $s_1'\leq \sigma(i_j)\leq s_2'$, hence $b_{i_j}\geq \sigma(i_j)\geq s_1'$, $i_j\in \{s_1,\cdots,s_2\}\backslash \{Y_{s_2'+1},\cdots,Y_n\}$, and $i_j\in \mathcal{S}_{1,l}\cup\mathcal{S}_{2,l}$. Assume that $\{i_1,\cdots,i_L\}=\{k_1,\cdots,k_q\}\cup \{k_1',\cdots,k_{L-q}'\}$, where $q\in\{0\}\cup  [L]$, $k_1,\cdots,k_q\in\mathcal{S}_{1,l}$, $k_1<\cdots<k_q$, $k_1',\cdots,k_{L-q}'\in\mathcal{S}_{2,l}$, and $k_1'<\cdots<k_{L-q}'$. As $(k_1,\sigma(k_1)),\cdots,(k_q,\sigma(k_q)) \in  \mathcal{S}_{1,l}\times (n y_{l-1}(\Gamma), n y_l(\Gamma)]$ and $\sigma(k_1)<\cdots<\sigma(k_q)$, we have $L_{1,l}\geq q$. Similarly, $L_{2,l}\geq L-q$. Hence $LIS(\sigma|_{n Q_l})=L\leq L_{1,l}+L_{2,l}$. The inequality $L_{1,l}\leq LIS(\sigma|_{n Q_l})$ follows from the fact that $S_{1,l}\times (n y_{l-1}(\Gamma), n y_l(\Gamma)]\subseteq n Q_l$. We conclude that (\ref{Eq2}) holds. 

In the following, we bound $D_l'$, $L_{2,l}$, $D_l$, $L_{1,l}$ in \textbf{Steps 1-4}, respectively. Recall the definitions of these quantities in (\ref{Eq3}) and (\ref{Eq11}).

\subparagraph{Step 1}
In this step, we bound $D_l'$. Note that
\begin{eqnarray}\label{Eq4}
 D_l'\leq \sum_{i=s_1'}^{s_2'} \mathbbm{1}_{\sigma^{-1}(i)\in \mathcal{S}_{2,l}}\leq \sum_{i=s_1'}^{s_2'} \mathbbm{1}_{\sigma^{-1}(i)\in \mathcal{S}_{l}'}=\sum_{i=s_1'}^{s_2'}\mathbbm{1}_{Y_i\in\mathcal{S}_l'}.
\end{eqnarray}

Let $\mathcal{B}_l$ be the $\sigma$-algebra generated by $\sigma_0$, $\{b_i\}_{i=1}^n$, and $\{Y_i\}_{i=s_2'+1}^n$. Conditional on $\mathcal{B}_l$, we couple $\{Y_i\}_{i=s_1'}^{s_2'}$ with mutually independent Bernoulli random variables $\{Y_i'\}_{i=s_1'}^{s_2'}$ with parameters (note that $W_l$ is $\mathcal{B}_l$-measurable)
\begin{equation}\label{Eq5}
    \mathbb{P}(Y_i'=1|\mathcal{B}_l)=\min\Big\{\frac{W_l}{N_i},1\Big\}, \quad \forall i\in\{s_1',\cdots,s_2'\}
\end{equation}
as follows. Sequentially for $i=s_2',\cdots,s_1'$, we do the following. Assume that $Y_{i+1},\cdots,Y_n$ have been sampled and that $b_{Y_j}\geq j$ for any $j\in  \{i+1,\cdots,n\}$. Let  
\begin{eqnarray}\label{Eq3.3}
    \mathcal{S}''_{l,i}&:=&\mathcal{S}'_l\cap(\{j\in [n]:b_j\geq i\}\backslash \{Y_{i+1},\cdots,Y_n\})\nonumber\\
    &=&\{j\in\{s_1,\cdots,s_2\}: i\leq b_j\leq s_2'\}\backslash \{Y_{i+1},\cdots,Y_n\}.
\end{eqnarray}
As $N_i=|\{j\in [n]: b_j\geq i\}\backslash \{Y_{i+1},\cdots, Y_n\}|$, we have $|\mathcal{S}_{l,i}''|\leq N_i$ and 
\begin{equation}\label{Eq3.1}
|\{j\in [n]: b_j\geq i\}\backslash (\{Y_{i+1},\cdots,Y_n\}\cup \mathcal{S}''_{l,i})|=N_i-|\mathcal{S}_{l,i}''|\geq  \min\{W_l,N_i\}-|\mathcal{S}''_{l,i}|.
\end{equation}
Moreover, as $|\mathcal{S}_{l,i}''| \leq |\mathcal{S}_l'|=W_l$, we have
\begin{equation}\label{Eq3.2}
 \min\{W_l,N_i\}-|\mathcal{S}_{l,i}''|\geq 0.
\end{equation}
Noting (\ref{Eq3.1}) and (\ref{Eq3.2}), we let $\mathcal{S}'''_{l,i}$ be the set that consists of the smallest $\min\{W_l,N_i\}-|\mathcal{S}''_{l,i}|$ elements in the set $\{j\in [n]: b_j\geq i\}\backslash (\{Y_{i+1},\cdots,Y_n\}\cup \mathcal{S}''_{l,i})$. If $Y_i'=1$, we pick $Y_i$ uniformly from the set
$\mathcal{S}''_{l,i}\cup \mathcal{S}'''_{l,i}$. If $Y_i'=0$, we pick $Y_i$ uniformly from the set $\{j\in [n]: b_j\geq i\}\backslash(\{Y_{i+1},\cdots,Y_n\}\cup\mathcal{S}''_{l,i}\cup\mathcal{S}'''_{l,i})$. Note that $b_{Y_i}\geq i$.

It can be checked that $\{Y_i\}_{i=s_1'}^{s_2'}$ has the desired conditional distribution given $\mathcal{B}_l$ as specified by the hit and run algorithm. Therefore, the above procedure gives a valid coupling between $\{Y_i\}_{i=s_1'}^{s_2'}$ and $\{Y_i'\}_{i=s_1'}^{s_2'}$ conditional on $\mathcal{B}_l$.

Now for any $i\in\{s_1',\cdots,s_2'\}$ such that $Y_i'=0$, we have $Y_i\notin \mathcal{S}''_{l,i}$; as $Y_i\in \{j\in [n]:b_j\geq i\}\backslash \{Y_{i+1},\cdots,Y_n\}$, by (\ref{Eq3.3}), we have $Y_i\notin \mathcal{S}_l'$. Hence for any $i\in\{s_1',\cdots,s_2'\}$, we have $\mathbbm{1}_{Y_i\in\mathcal{S}_l'}\leq Y_i'$. By (\ref{Eq4}), we have
\begin{equation}\label{Eq6}
    D_l'\leq \sum_{i=s_1'}^{s_2'} Y_i'.
\end{equation}

By (\ref{Eq5}), (\ref{Eq6}), and Hoeffding's inequality, we obtain that for any $t\geq 0$, 
\begin{equation}\label{Eq9}
    \mathbb{P}\Big(D_l'\geq \sum_{i=s_1'}^{s_2'}\frac{W_l}{N_i}+(s_2'-s_1'+1)t\Big|\mathcal{B}_l\Big)\leq e^{-2(s_2'-s_1'+1)t^2}.
\end{equation}
Let $\mathcal{C}_l$ be the event that for any $i\in [n]$ such that $i\geq n\slash 3$, we have
\begin{equation}
N_i\geq \frac{1}{2}e^{-4\theta}i.
\end{equation}
Let $\mathcal{E}_l$ be the event that
\begin{equation}\label{Eq3.8}
W_l\leq 4n\beta_n(s_2-s_1+1)(y_l(\Gamma)-y_{l-1}(\Gamma)).
\end{equation}
By (\ref{Eq7}), (\ref{Eq8}), and the union bound, we have
\begin{equation}\label{Eq10.1}
    \mathbb{P}(\mathcal{C}_l^c)\leq n\exp(-e^{-8\theta}n \slash 6),\quad \mathbb{P}(\mathcal{E}_l^c)\leq \exp(-8n^2\beta_n^2(s_2-s_1+1)(y_l(\Gamma)-y_{l-1}(\Gamma))^2).  
\end{equation}
When $\mathcal{C}_l$ holds, for any $i\in \{s_1',\cdots,s_2'\}$ (note that $i\geq s_1'\geq n y_{l-1}(\Gamma)\geq n\slash 3$), 
\begin{equation}\label{Eq3.5}
N_i\geq \frac{1}{6}e^{-4\theta}n.
\end{equation}

Let $\mathcal{D}_l$ be the event that
\begin{equation}\label{E4.2}
    D_l'\geq 25 e^{4\theta} \beta_n  (s_2-s_1+1)(s_2'-s_1'+1)(y_l(\Gamma)-y_{l-1}(\Gamma)).
\end{equation}
Taking $t=\beta_n(s_2-s_1+1)(y_l(\Gamma)-y_{l-1}(\Gamma))$ in (\ref{Eq9}) and noting (\ref{Eq3.8}) and (\ref{Eq3.5}), we obtain that
\begin{eqnarray*}
   && \mathbb{P}(\mathcal{D}_l\cap\mathcal{C}_l\cap\mathcal{E}_l|\mathcal{B}_l)\nonumber\\
   & \leq  & \exp(-2\beta_n^2 (s_2-s_1+1)^2(y_l(\Gamma)-y_{l-1}(\Gamma))^2(s_2'-s_1'+1) ).
\end{eqnarray*}
Hence 
\begin{eqnarray}\label{Eq4.3}
&& \mathbb{P}(\mathcal{D}_l\cap\mathcal{C}_l\cap\mathcal{E}_l)=\mathbb{E}[\mathbb{P}(\mathcal{D}_l\cap\mathcal{C}_l\cap\mathcal{E}_l|\mathcal{B}_l)]\nonumber\\
   & \leq  & \exp(-2\beta_n^2  (s_2-s_1+1)^2(y_l(\Gamma)-y_{l-1}(\Gamma))^2(s_2'-s_1'+1) ).
\end{eqnarray}

Combining (\ref{Eq10.1}) and (\ref{Eq4.3}), by the union bound, we have
\begin{eqnarray}\label{Eq4.5}
    \mathbb{P}(\mathcal{D}_l)&\leq& n\exp(-e^{-8\theta}n \slash 6)+\exp(-8n^2\beta_n^2(s_2-s_1+1)(y_l(\Gamma)-y_{l-1}(\Gamma))^2)\nonumber\\
    &&+\exp(-2\beta_n^2  (s_2-s_1+1)^2(y_l(\Gamma)-y_{l-1}(\Gamma))^2(s_2'-s_1'+1) ).
\end{eqnarray}
By (\ref{Eq7.4}), (\ref{Eq4.1}), and (\ref{Eq1}), we have
\begin{equation}\label{Eq4.6}
    s_2-s_1\geq n(x_l(\Gamma)-x_{l-1}(\Gamma))-2\geq \frac{n}{2 K_0 T}-2\geq \frac{n}{4 K_0 T},\quad s_2'-s_1'\geq \frac{n}{4 K_0 T}. 
\end{equation}
By (\ref{Eq4.4}), (\ref{Eq4.1}), (\ref{Eq4.5}), and (\ref{Eq4.6}), we have
\begin{equation}\label{Eq4.16}
    \mathbb{P}(\mathcal{D}_l)\leq n\exp(-e^{-8\theta} n \slash 6)+2\exp(-\theta^2 n \slash (512 K_0^5 T^5)). 
\end{equation}


\subparagraph{Step 2}

Now we bound $L_{2,l}$. For any $q\in \mathbb{N}^{*}$, we define
\begin{equation}\label{Eq3.6}
    \Lambda_{l,q}:=\sum_{\substack{i_1<\cdots<i_q,j_1<\cdots<j_q\\ i_1,\cdots,i_q\in \{s_1,\cdots,s_2\}\\j_1,\cdots,j_q\in (n y_{l-1}(\Gamma), n y_l(\Gamma)]\cap\mathbb{N}^{*}}} \mathbbm{1}_{\sigma(i_1)=j_1,\cdots,\sigma(i_q)=j_q}\mathbbm{1}_{i_1,\cdots,i_q\in\mathcal{S}_{2,l}}.
\end{equation}
For any $k\in [n]$, let $\mathcal{F}_k$ be the $\sigma$-algebra generated by $\sigma_0$, $\{b_i\}_{i=1}^n$, and $\{Y_i\}_{i=k+1}^n$. For any $i_1,\cdots,i_q\in \{s_1,\cdots,s_2\}$ and $j_1,\cdots,j_q\in (n y_{l-1}(\Gamma), n y_l(\Gamma)]\cap\mathbb{N}^{*}$ such that $i_1<\cdots<i_q$ and $j_1<\cdots<j_q$, we have 
\begin{eqnarray*}
  &&\mathbb{E}[\mathbbm{1}_{\sigma(i_1)=j_1,\cdots,\sigma(i_q)=j_q}\mathbbm{1}_{i_1,\cdots,i_q\in\mathcal{S}_{2,l}}|\mathcal{B}_l]\nonumber\\
  &=& \mathbb{E}[\mathbbm{1}_{\sigma(i_1)=j_1,\cdots,\sigma(i_q)=j_q}|\mathcal{B}_l] \mathbbm{1}_{i_1,\cdots,i_q\in\mathcal{S}_{2,l}}\nonumber\\
  &=& \mathbb{E}[\mathbb{E}[\mathbbm{1}_{\sigma(i_1)=j_1}|\mathcal{F}_{j_1}]\mathbbm{1}_{\sigma(i_2)=j_2,\cdots,\sigma(i_q)=j_q}|\mathcal{B}_l]\mathbbm{1}_{i_1,\cdots,i_q\in\mathcal{S}_{2,l}}\\
  &\leq & \frac{\mathbbm{1}_{i_1,\cdots,i_q\in\mathcal{S}_{2,l}}}{N_{j_1}}\mathbb{E}[\mathbbm{1}_{\sigma(i_2)=j_2,\cdots,\sigma(i_q)=j_q}|\mathcal{B}_l] \leq \cdots\leq \frac{\mathbbm{1}_{i_1,\cdots,i_q\in\mathcal{S}_{2,l}}}{N_{j_1}N_{j_2}\cdots N_{j_q}}.
\end{eqnarray*}
Hence by (\ref{Eq3.5}), we have
\begin{eqnarray}\label{Eq3.7}
\mathbb{E}[\mathbbm{1}_{\sigma(i_1)=j_1,\cdots,\sigma(i_q)=j_q}\mathbbm{1}_{i_1,\cdots,i_q\in\mathcal{S}_{2,l}}|\mathcal{B}_l] \mathbbm{1}_{\mathcal{C}_l\cap\mathcal{E}_l}\leq \Big(\frac{6 e^{4\theta}}{n}\Big)^q \mathbbm{1}_{\mathcal{C}_l\cap\mathcal{E}_l} \mathbbm{1}_{i_1,\cdots,i_q\in\mathcal{S}_{2,l}}.
\end{eqnarray}
By (\ref{Eq3.8}), (\ref{Eq3.6}), (\ref{Eq3.7}), and Lemma \ref{Lemma3.1}, we obtain that
\begin{eqnarray*}
    &&\mathbb{E}[\Lambda_{l,q}|\mathcal{B}_l] \mathbbm{1}_{\mathcal{C}_l\cap\mathcal{E}_l} \leq  
     \Big(\frac{6 e^{4\theta}}{n}\Big)^q\binom{|\mathcal{S}_{2,l}|}{q}\binom{s_2'-s_1'+1}{q} \mathbbm{1}_{\mathcal{C}_l\cap\mathcal{E}_l}\\
    &\leq& \Big(\frac{6e^2 e^{4\theta}|\mathcal{S}_{2,l}|(s_2'-s_1'+1)}{nq^2}\Big)^q \mathbbm{1}_{\mathcal{C}_l\cap\mathcal{E}_l}\leq\Big(\frac{6e^{2+4\theta} W_l(s_2'-s_1'+1)}{nq^2}\Big)^q \mathbbm{1}_{\mathcal{C}_l\cap\mathcal{E}_l}\\
    &\leq& \Big(\frac{24e^{2+4\theta}\beta_n(s_2-s_1+1)(s_2'-s_1'+1)(y_l(\Gamma)-y_{l-1}(\Gamma))}{q^2}\Big)^q.
\end{eqnarray*}
Hence 
\begin{eqnarray}\label{Eq10}
 &&   \mathbb{P}(\{\Lambda_{l,q}\geq 1\}\cap\mathcal{C}_l\cap\mathcal{E}_l)=\mathbb{E}[\mathbb{E}[\mathbbm{1}_{\Lambda_{l,q}\geq 1}|\mathcal{B}_l]\mathbbm{1}_{\mathcal{C}_l\cap\mathcal{E}_l}]  \leq \mathbb{E}[\mathbb{E}[\Lambda_{l,q}|\mathcal{B}_l] \mathbbm{1}_{\mathcal{C}_l\cap\mathcal{E}_l}]  \nonumber\\
 && \leq \Big(\frac{24e^{2+4\theta}\beta_n(s_2-s_1+1)(s_2'-s_1'+1)(y_l(\Gamma)-y_{l-1}(\Gamma))}{q^2}\Big)^q.
\end{eqnarray}
Let
\begin{equation}
    q_0:=8 e^{1+2\theta}\beta_n^{1\slash 2} (s_2-s_1+1)^{1\slash 2}(s_2'-s_1'+1)^{1\slash 2}(y_l(\Gamma)-y_{l-1}(\Gamma))^{1\slash 2}.
\end{equation}
Taking $q=\lceil q_0\rceil$ in (\ref{Eq10}), we obtain that 
\begin{equation*}
    \mathbb{P}(\{\Lambda_{l,\lceil q_0\rceil}\geq 1\}\cap\mathcal{C}_l\cap\mathcal{E}_l)\leq 2^{-q_0},
\end{equation*}
which leads to
\begin{equation}\label{Eq7.1}
    \mathbb{P}(\{L_{2,l}\geq q_0+1\}\cap\mathcal{C}_l\cap\mathcal{E}_l)\leq 2^{-q_0}.
\end{equation}
By (\ref{Eq4.4}), (\ref{Eq7.4})-(\ref{Eq1}), (\ref{Eq4.6}), and the AM-GM inequality, 
\begin{equation}\label{Eq7.2}
 q_0 \geq c \theta^{1\slash 2} (K_0T)^{-3\slash 2} n^{1\slash 2},
\end{equation}
\begin{eqnarray}\label{Eq7.3}
    q_0+1 &\leq& C' T^{-1\slash 2} n^{1\slash 2} (x_l(\Gamma)-x_{l-1}(\Gamma))^{1\slash 2}(y_l(\Gamma)-y_{l-1}(\Gamma))^{1\slash 2}+1 \nonumber\\
    &\leq& C_{\theta} T^{-1\slash 2} n^{1\slash 2} (x_l(\Gamma)-x_{l-1}(\Gamma)+y_l(\Gamma)-y_{l-1}(\Gamma)),
\end{eqnarray}
where $C_{\theta}$ is a positive constant that only depends on $\theta$. Let $\mathscr{E}_l$ be the event that 
\begin{equation}\label{Eq4.24}
    L_{2,l}\leq C_{\theta} T^{-1\slash 2} n^{1\slash 2} (x_l(\Gamma)-x_{l-1}(\Gamma)+y_l(\Gamma)-y_{l-1}(\Gamma)).
\end{equation} 
By (\ref{Eq7.1})-(\ref{Eq7.3}), we have
\begin{equation}\label{Eq4.19}
    \mathbb{P}(\mathscr{E}_l^c\cap\mathcal{C}_l\cap\mathcal{E}_l)\leq \exp(-c\theta^{1\slash 2}(K_0T)^{-3\slash 2}n^{1\slash 2}). 
\end{equation}
By (\ref{Eq10.1}), (\ref{Eq4.19}), and the union bound, we have
\begin{eqnarray}
    \mathbb{P}(\mathscr{E}_l^c)&\leq& \exp(-c\theta^{1\slash 2}(K_0T)^{-3\slash 2}n^{1\slash 2})+n\exp(-e^{-8\theta}n \slash 6)\nonumber\\
    && +\exp(-8n^2\beta_n^2(s_2-s_1+1)(y_l(\Gamma)-y_{l-1}(\Gamma))^2).
\end{eqnarray}
Noting (\ref{Eq4.4}), (\ref{Eq4.1}), and (\ref{Eq4.6}), we obtain that
\begin{eqnarray}\label{Eq4.22}
    \mathbb{P}(\mathscr{E}_l^c)&\leq&  \exp(-c\theta^{1\slash 2}(K_0T)^{-3\slash 2}n^{1\slash 2})+n\exp(-e^{-8\theta}n \slash 6)\nonumber\\
    &&+\exp(-\theta^2n\slash (8K_0^3T^3))\leq C n\exp(-\tilde{c} n^{1\slash 2}).
\end{eqnarray}


\subparagraph{Step 3}

Now we bound $D_l$. Note that $D_l=|S(\sigma)\cap n Q_l|$ and that the distribution of $\sigma$ is given by $\mathbb{P}_{n,\beta_n}$. For any $\delta>0$, let $\mathcal{H}_{l,\delta}$ be the event that 
\begin{equation}\label{E4.1}
    \Big|D_l-n\int_{Q_l}\rho_{\theta}(x,y)dxdy\Big|<n\delta.
\end{equation}
By Lemma \ref{Lemma2}, there exist positive constants $C_0,c_0$ that only depend on $T,K_0,\delta$ and the sequence $\{\beta_n\}$, such that
\begin{equation}\label{Eq4.17}
    \mathbb{P}((\mathcal{H}_{l,\delta})^c)\leq C_0\exp(-c_0 n). 
\end{equation}

\subparagraph{Step 4}

Finally, we bound $L_{1,l}$. Recall the definition of $\mathcal{S}_{1,l}$ in (\ref{Eqs}). Let
\begin{equation}
 R:=|\{i\in [n]: (i,\sigma(i))\in \mathcal{S}_{1,l}\times (n y_{l-1}(\Gamma), n y_l(\Gamma)]\}|.
\end{equation}
We also let $I_1,\cdots,I_n\in \{0\}\cup [n]$ and $J_1,\cdots,J_n\in \{0\}\cup [n]$ be such that
\begin{equation*}
I_{R+1}=\cdots=I_n=0, \quad J_{R+1}=\cdots=J_n=0,
\end{equation*}
\begin{equation*}
1\leq I_1<\cdots<I_R, \quad 1\leq J_1<\cdots<J_R,
\end{equation*}
\begin{equation*}
 \{I_1,\cdots,I_R\}=\{i\in [n]: (i,\sigma(i))\in \mathcal{S}_{1,l}\times (n y_{l-1}(\Gamma), n y_l(\Gamma)]\},
\end{equation*}
\begin{equation*}
 \{J_1,\cdots,J_R\}=\{i\in [n]: (\sigma^{-1}(i),i)\in \mathcal{S}_{1,l}\times (n y_{l-1}(\Gamma), n y_l(\Gamma)]\}.
\end{equation*} 

Note that for any $i\in [n]$ such that $(i,\sigma(i))\in \mathcal{S}_{1,l}\times (ny_{l-1}(\Gamma),ny_l(\Gamma)]$, we have $(i,\sigma(i))\in nQ_l$ and $i\notin \mathcal{S}_{2,l}$. Hence we have
\begin{equation}\label{Eq3.9}
   R\leq D_l-D_l'.
\end{equation}
Now consider any $i\in [n]$ such that $(i,\sigma(i))\in nQ_l$ and $i\notin \mathcal{S}_{2,l}$. We have
\begin{equation}\label{Eq3.11}
  i\in (nx_{l-1}(\Gamma),nx_l(\Gamma)]\cap\mathbb{N}^{*}=\{s_1,\cdots,s_2\},
\end{equation}
\begin{equation}
\sigma(i)\in (ny_{l-1}(\Gamma),ny_l(\Gamma)]\cap\mathbb{N}^{*}=\{s_1',\cdots,s_2'\}.
\end{equation}
If $i=Y_j$ for some $j\in\{s_2'+1,\cdots,n\}$, then 
\begin{equation*}
  \sigma(i)=\sigma(Y_j)=j \notin (ny_{l-1}(\Gamma),n y_l(\Gamma)]\cap\mathbb{N}^{*},
\end{equation*}
which leads to a contradiction. Hence
\begin{equation}\label{Eq3.13}
    i\notin \{Y_{s_2'+1},\cdots,Y_n\}.
\end{equation}
By the construction of $\{Y_j\}_{j=1}^n$ and $\sigma$, we have $b_{Y_j}\geq j$ for any $j\in [n]$ and $Y_{\sigma(i)}=i$. Hence 
\begin{equation}\label{Eq3.12}
    b_i=b_{Y_{\sigma(i)}}\geq \sigma(i)\geq s_1'.
\end{equation}
As $i\notin \mathcal{S}_{2,l}$, by (\ref{Eq3.11}), (\ref{Eq3.13}), and (\ref{Eq3.12}), we have $b_i>s_2'$, hence $i\in\mathcal{S}_{1,l}$. Therefore, we have
\begin{equation}\label{Eq3.10}
R\geq D_l-D_l'.
\end{equation}
Combining (\ref{Eq3.9}) and (\ref{Eq3.10}), we conclude that
\begin{equation}\label{E4.3}
    R=D_l-D_l'.
\end{equation}


Throughout the rest of this subsection, we let $S_0$ be the set that consists solely of the empty mapping $\tau_0:\emptyset\rightarrow\emptyset$, and let $LIS(\tau_0):=0$. If $R\geq 1$, we let $\tau\in S_R$ be such that $\sigma(I_s)=J_{\tau(s)}$ for every $s\in [R]$. If $R=0$, we let $\tau$ be the empty mapping. In the following, we condition on $\mathcal{B}_l$, and consider any $r\in [n]$, $i_1,\cdots,i_r\in [n]$, and $j_1,\cdots,j_r\in [n]$ such that 
\begin{equation*}
\mathbb{P}(R=r, I_1=i_1,\cdots,I_r=i_r,J_1=j_1,\cdots,J_r=j_r|\mathcal{B}_l)>0.
\end{equation*}
By the sampling process of the hit and run algorithm, conditional on $\mathcal{B}_l$, the distribution of $\sigma$ is given by the uniform distribution on the following set:
\begin{equation*}
    \{\kappa\in S_n: \kappa(s)\leq b_s\text{ for every }s\in [n], \kappa^{-1}(s)=Y_s\text{ for every }s\in\{s_2'+1,\cdots,n\}\},
\end{equation*}
which has cardinality $\prod_{s=1}^{s_2'}N_s$. For any $\eta\in S_r$, let $M_{r,\eta}$ be the following set (recall Definition \ref{Def1.5}):
\begin{eqnarray*}
    && \{\kappa\in S_n: \kappa(s)\leq b_s\text{ for every }s\in [n],\kappa^{-1}(s)=Y_s\text{ for every }s\in\{s_2'+1,\cdots,n\},\nonumber\\
    &&\quad\kappa(i_s)=j_{\eta(s)}\text{ for every }s\in [r],\nonumber\\
    && \quad S(\kappa)\cap(\mathcal{S}_{1,l}\times (ny_{l-1}(\Gamma),ny_l(\Gamma)])=\{(i_s,j_{\eta(s)}):s\in [r]\}\}.
\end{eqnarray*}
Then for any $\eta\in S_r$, we have
\begin{eqnarray}\label{E3.2}
   && \mathbb{P}(\{\tau=\eta\}\cap \{R=r,I_1=i_1,\cdots,I_r=i_r,J_1=j_1,\cdots,J_r=j_r\}|\mathcal{B}_l) \nonumber  \\
   && =\frac{|M_{r,\eta}|}{\prod_{s=1}^{s_2'}N_s}.
\end{eqnarray}

Now for any $\eta_1,\eta_2\in S_r$, we define a mapping $\psi_{\eta_1,\eta_2}: M_{r,\eta_1} \rightarrow M_{r,\eta_2}$ as follows. Let $\iota_{\eta_1,\eta_2}\in S_n$ be the unique permutation that maps $j_s$ to $j_{\eta_2\eta_1^{-1}(s)}$ for every $s\in [r]$ and fixes every element in $[n]\backslash \{j_1,\cdots,j_r\}$. For every $\kappa\in M_{r,\eta_1}$, we let $\psi_{\eta_1,\eta_2}(\kappa):=\iota_{\eta_1,\eta_2}\kappa$. Below we verify that $\psi_{\eta_1,\eta_2}(\kappa)\in M_{r,\eta_2}$. For every $s\in [n]\backslash\{i_1,\cdots,i_r\}$, we have $\kappa(s)\in [n]\backslash \{j_1,\cdots,j_r\}$, hence
\begin{equation}\label{E3.1.1}
\iota_{\eta_1,\eta_2}\kappa(s)=\kappa(s)\leq b_s.
\end{equation}
For every $s\in [r]$, we have
\begin{equation}\label{E3.1}
    \iota_{\eta_1,\eta_2}   \kappa (i_s)=\iota_{\eta_1,\eta_2} (j_{\eta_1(s)})=j_{\eta_2(s)}.
\end{equation}
Note that for any $s\in [r]$, $j_{\eta_2(s)}\in (ny_{l-1}(\Gamma),ny_l(\Gamma)]\cap\mathbb{N}^{*}=\{s_1',\cdots,s_2'\}$. Now for any $s\in [r]$, as $i_s\in\mathcal{S}_{1,l}$, by (\ref{E3.1}), we have $b_{i_s}>s_2'\geq j_{\eta_2(s)}=\iota_{\eta_1,\eta_2}   \kappa (i_s)$. Combining this with (\ref{E3.1.1}), we obtain that for every $s\in [n]$, 
\begin{equation}\label{E3.1.2}
\iota_{\eta_1,\eta_2}\kappa(s)\leq b_s.
\end{equation}
For any $s\in \{s_2'+1,\cdots,n\}$, we have $s\notin \{j_1,\cdots,j_r\}$, hence 
\begin{equation}\label{E3.1.3}
\iota_{\eta_1,\eta_2}\kappa(Y_s)=\iota_{\eta_1,\eta_2} (s)=s.
\end{equation}
Moreover, it can be checked that
\begin{equation}\label{E3.1.4}
S(\iota_{\eta_1,\eta_2}\kappa)\cap(\mathcal{S}_{1,l}\times (ny_{l-1}(\Gamma),ny_l(\Gamma)])=\{(i_s,j_{\eta_2(s)}):s\in [r]\}.
\end{equation}
By (\ref{E3.1})-(\ref{E3.1.4}), $\iota_{\eta_1,\eta_2}\kappa\in M_{r,\eta_2}$. We can also verify that for any $\eta_1,\eta_2\in S_n$,
\begin{equation*}
\psi_{\eta_2,\eta_1}\psi_{\eta_1,\eta_2}=Id_{M_{r,\eta_1}}, \quad \psi_{\eta_1,\eta_2}\psi_{\eta_2,\eta_1}=Id_{M_{r,\eta_2}},
\end{equation*}
where for any set $A$, $Id_A$ denotes the identity map on $A$. We conclude that for any $\eta_1,\eta_2\in S_r$, $\psi_{\eta_1,\eta_2}$ is a bijection from $M_{r,\eta_1}$ to $M_{r,\eta_2}$, hence
\begin{equation}\label{Eq3.14}
|M_{r,\eta_1}|=|M_{r,\eta_2}|.
\end{equation}

By (\ref{E3.2}) and (\ref{Eq3.14}), we conclude that for any $\eta\in S_r$,
\begin{eqnarray}\label{Eq3.18}
  &&  \frac{\mathbb{P}(\{\tau=\eta\}\cap \{R=r,I_1=i_1,\cdots,I_r=i_r,J_1=j_1,\cdots,J_r=j_r\}|\mathcal{B}_l)}{\mathbb{P}(R=r,I_1=i_1,\cdots,I_r=i_r,J_1=j_1,\cdots,J_r=j_r|\mathcal{B}_l)}\nonumber\\
  && = \frac{|M_{r,\eta}|}{\sum_{\eta'\in S_r}|M_{r,\eta'}|}=\frac{1}{r!}. 
\end{eqnarray}

Now let $\mathcal{B}'_l$ be the $\sigma$-algebra generated by $\sigma_0$, $\{b_i\}_{i=1}^n$, $\{Y_i\}_{i=s_2'+1}^n$, $R$, $\{I_i\}_{i=1}^n$, and $\{J_i\}_{i=1}^n$. In the following, we consider an arbitrary $A\in\mathcal{B}_l'$. As $\mathbbm{1}_A$ is $\mathcal{B}_l'$-measurable, there exists a Borel measurable function $g:\mathbb{R}^{5n-s_2'+1}\rightarrow \mathbb{R}$ , such that 
\begin{equation}\label{Eq3.17}
   \mathbbm{1}_A=g(\sigma_0,\{b_s\}_{s=1}^n, \{Y_s\}_{s=s_2'+1}^n, R, \{I_s\}_{s=1}^n, \{J_s\}_{s=1}^n ),
\end{equation}
where we identify $\sigma_0$ with $(\sigma_0(1),\cdots,\sigma_0(n))\in \mathbb{R}^n$. Without loss of generality, we assume that $\|g\|_{\infty}\leq 1$ (otherwise we replace $g$ by $\max\{0,\min\{g,1\}\}$). Consider any $r_0\in [n]$ and any $\eta\in S_{r_0}$. We have
\begin{eqnarray}\label{Eq3.19}
&&\mathbb{E}[\mathbbm{1}_{\tau=\eta}\mathbbm{1}_A]\nonumber\\
&=&\sum_{\substack{r\in \{0\}\cup [n],\\i_1,\cdots,i_r\in [n],\\j_1,\cdots,j_r\in [n]}}\mathbb{P}(A\cap\{R=r,I_1=i_1,\cdots,I_r=i_r,J_1=j_1,\cdots,J_r=j_r\}\cap\{\tau=\eta\})\nonumber\\
&=&  \sum_{\substack{i_1,\cdots,i_{r_0}\in [n],\\j_1,\cdots,j_{r_0}\in [n]}}\mathbb{P}(A\cap\{R=r_0,I_1=i_1,\cdots,I_{r_0}=i_{r_0},J_1=j_1,\cdots,J_{r_0}=j_{r_0}\}\cap\{\tau=\eta\}).\nonumber\\
&&
\end{eqnarray}
For any $i_1,\cdots,i_{r_0}\in [n]$ and $j_1,\cdots,j_{r_0}\in [n]$, we have
\begin{eqnarray}\label{Eq3.20}
&& \mathbb{P}(A\cap\{R=r_0,I_1=i_1,\cdots,I_{r_0}=i_{r_0},J_1=j_1,\cdots,J_{r_0}=j_{r_0}\}\cap\{\tau=\eta\}) \nonumber\\
&=& \mathbb{E}[g(\sigma_0,\{b_s\}_{s=1}^n, \{Y_s\}_{s=s_2'+1}^n, R, \{I_s\}_{s=1}^n, \{J_s\}_{s=1}^n )\mathbbm{1}_{\tau=\eta} \nonumber\\
&& \quad\quad\quad \quad\quad\quad\times\mathbbm{1}_{R=r_0,I_1=i_1,\cdots,I_{r_0}=i_{r_0}, J_1=j_1,\cdots,J_{r_0}=j_{r_0}}]\nonumber\\
&=& \mathbb{E}[g(\sigma_0,\{b_s\}_{s=1}^n, \{Y_s\}_{s=s_2'+1}^n, r_0, \{i_s\}_{s=1}^n, \{j_s\}_{s=1}^n )\mathbbm{1}_{\tau=\eta} \nonumber\\
&& \quad\quad\quad \quad\quad\quad\times\mathbbm{1}_{R=r_0,I_1=i_1,\cdots,I_{r_0}=i_{r_0}, J_1=j_1,\cdots,J_{r_0}=j_{r_0}}]\nonumber\\
&=& \mathbb{E}[g(\sigma_0,\{b_s\}_{s=1}^n, \{Y_s\}_{s=s_2'+1}^n, r_0, \{i_s\}_{s=1}^n, \{j_s\}_{s=1}^n ) \nonumber\\
&& \quad \times\mathbb{P}(\{\tau=\eta\}\cap\{R=r_0,I_1=i_1,\cdots,I_{r_0}=i_{r_0}, J_1=j_1,\cdots,J_{r_0}=j_{r_0}\}|\mathcal{B}_l)]\nonumber\\
&=& \frac{1}{r_0!}\mathbb{E}[g(\sigma_0,\{b_s\}_{s=1}^n, \{Y_s\}_{s=s_2'+1}^n, r_0, \{i_s\}_{s=1}^n, \{j_s\}_{s=1}^n ) \nonumber\\
&&  \quad\quad \times \mathbb{P}(R=r_0,I_1=i_1,\cdots,I_{r_0}=i_{r_0},J_1=j_1,\cdots,J_{r_0}=j_{r_0}|\mathcal{B}_l) ]\nonumber\\
&=& \frac{1}{r_0!} \mathbb{E}[g(\sigma_0,\{b_s\}_{s=1}^n, \{Y_s\}_{s=s_2'+1}^n, r_0, \{i_s\}_{s=1}^n, \{j_s\}_{s=1}^n ) \nonumber\\
&& \quad\quad\quad \quad\quad\quad\times\mathbbm{1}_{R=r_0,I_1=i_1,\cdots,I_{r_0}=i_{r_0}, J_1=j_1,\cdots,J_{r_0}=j_{r_0}}]\nonumber\\
&=& \frac{1}{r_0!} \mathbb{E}[g(\sigma_0,\{b_s\}_{s=1}^n, \{Y_s\}_{s=s_2'+1}^n, R, \{I_s\}_{s=1}^n, \{J_s\}_{s=1}^n ) \nonumber\\
&& \quad\quad\quad \quad\quad\quad\times\mathbbm{1}_{R=r_0,I_1=i_1,\cdots,I_{r_0}=i_{r_0}, J_1=j_1,\cdots,J_{r_0}=j_{r_0}}]\nonumber\\
&=& \frac{1}{r_0!} \mathbb{P}(A\cap \{R=r_0, I_1=i_1,\cdots,I_{r_0}=i_{r_0},J_1=j_1,\cdots,J_{r_0}=j_{r_0}\}),
\end{eqnarray}
where we take $i_{r_0+1}=\cdots=i_n=j_{r_0+1}=\cdots=j_n=0$ in the second equality, use (\ref{Eq3.17}) in the first and the last equalities, and use (\ref{Eq3.18}) in the fourth equality. By (\ref{Eq3.19}) and (\ref{Eq3.20}), for any $A\in\mathcal{B}'_l$, $r_0\in [n]$, and $\eta\in S_{r_0}$, we have
\begin{equation}\label{Eq3.21}
\mathbb{E}[\mathbbm{1}_{\tau=\eta}\mathbbm{1}_A]=\frac{1}{r_0!}\mathbb{E}[\mathbbm{1}_{R=r_0}\mathbbm{1}_A]=\mathbb{E}\Big[\frac{1}{R!}\mathbbm{1}_{R=r_0}\mathbbm{1}_A\Big]=\mathbb{E}\Big[\frac{1}{R!}\mathbbm{1}_{\eta\in S_R}\mathbbm{1}_A\Big].
\end{equation}
Now if $\eta\in S_0$, for any $A\in\mathcal{B}_l'$, we have
\begin{equation}\label{Eq3.22}
\mathbb{E}[\mathbbm{1}_{\tau=\eta}\mathbbm{1}_A]=\mathbb{E}[\mathbbm{1}_{R=0}\mathbbm{1}_A]=\mathbb{E}\Big[\frac{1}{R!}\mathbbm{1}_{\eta\in S_R}\mathbbm{1}_A\Big].
\end{equation}
By (\ref{Eq3.21}) and (\ref{Eq3.22}), for any $\eta\in\bigcup_{r=0}^n S_r $, we have
\begin{equation}\label{Eq3.23}
  \mathbb{P}(\tau=\eta|\mathcal{B}_l')=\frac{\mathbbm{1}_{\eta\in S_R}}{R!}. 
\end{equation}





Recalling the definition of $L_{1,l}$ in (\ref{Eq11}), we obtain that 
\begin{equation}\label{Eq3.24}
    L_{1,l}=LIS(\tau).
\end{equation}
Below we fix an arbitrary $\delta_0\in (0,1\slash 3)$. By (\ref{Eq3.23}) and Proposition \ref{P1}, we have
\begin{eqnarray}\label{E3.1.6}
&& \mathbb{P}(|LIS(\tau)-2\sqrt{R}|>R^{1\slash 2-\delta_0}|\mathcal{B}'_l)=\mathbb{E}\Big[\mathbbm{1}_{|LIS(\tau)-2\sqrt{R}|>R^{1\slash 2-\delta_0}}\Big|\mathcal{B}_l'\Big] \nonumber\\
&=& \sum_{r=0}^n\sum_{\eta\in S_r}\mathbb{E}\Big[\mathbbm{1}_{\tau=\eta} \mathbbm{1}_{|LIS(\tau)-2\sqrt{R}|>R^{1\slash 2-\delta_0}}\Big|\mathcal{B}_l'\Big]\nonumber\\
&=& \sum_{r=0}^n\sum_{\eta\in S_r}\mathbb{E}\Big[\mathbbm{1}_{\tau=\eta} \mathbbm{1}_{|LIS(\eta)-2\sqrt{R}|>R^{1\slash 2-\delta_0}}\Big|\mathcal{B}_l'\Big]\nonumber\\
&=& \sum_{r=0}^n\sum_{\eta\in S_r}\mathbbm{1}_{|LIS(\eta)-2\sqrt{R}|>R^{1\slash 2-\delta_0}} \mathbb{P}(\tau=\eta|\mathcal{B}_l')\nonumber\\
&=& \sum_{r=0}^n \sum_{\eta\in S_r} \mathbbm{1}_{|LIS(\eta)-2\sqrt{R}|>R^{1\slash 2-\delta_0}} \frac{\mathbbm{1}_{\eta\in S_R}}{R!}\nonumber\\
&=& \frac{1}{R!}\sum_{\eta\in S_R}\mathbbm{1}_{|LIS(\eta)-2\sqrt{R}|>R^{1\slash 2-\delta_0}}\leq C_{\delta_0} \exp(-R^{(1-3\delta_0)\slash 2}),
\end{eqnarray}
where $C_{\delta_0}$ is a positive constant that only depends on $\delta_0$.


By (\ref{E4.2}), (\ref{E4.1}), and (\ref{E4.3}), for any $\delta>0$, when $\mathcal{D}_l^c\cap\mathcal{H}_{l,\delta}$ holds, we have
\begin{eqnarray}\label{Eqn1.1}
&& n\int_{Q_l}\rho_{\theta}(x,y)dxdy-n\delta\nonumber\\
&& -25 e^{4\theta}\beta_n  (s_2-s_1+1)(s_2'-s_1'+1)(y_l(\Gamma)-y_{l-1}(\Gamma)) \nonumber\\
&\leq & R\leq  n\int_{Q_l}\rho_{\theta}(x,y)dxdy+n\delta.
\end{eqnarray}
Note that by (\ref{Eq1}), we have
\begin{equation}\label{Eqn1.2}
   s_2-s_1\leq n(x_{l}(\Gamma)-x_{l-1}(\Gamma)),\quad s_2'-s_1'\leq n(y_{l}(\Gamma)-y_{l-1}(\Gamma)). 
\end{equation}
In the following, we denote
\begin{equation}
  \Upsilon_l:=50\theta e^{4\theta}(x_{l}(\Gamma)-x_{l-1}(\Gamma)+n^{-1})(y_l(\Gamma)-y_{l-1}(\Gamma)+n^{-1})^2.
\end{equation}
By (\ref{Eq4.1}) and (\ref{Eq3.4}), we have 
\begin{equation}\label{Eq4.2}
  \Upsilon_l \leq 400e^{5\theta} T^{-1}(x_l(\Gamma)-x_{l-1}(\Gamma))(y_l(\Gamma)-y_{l-1}(\Gamma))\leq 400e^{5\theta} T^{-3}.
\end{equation}
By (\ref{Eq4.4}), (\ref{Eqn1.1}), and (\ref{Eqn1.2}), taking $\delta=(K_0 T)^{-10}$, we obtain that when the event $\mathcal{D}_l^c\cap\mathcal{H}_{l,(K_0T)^{-10}}$ holds, 
\begin{equation}\label{E3.1.5}
 n\int_{Q_l}\rho_{\theta}(x,y)dxdy-n((K_0T)^{-10}+\Upsilon_l) \leq R \leq  n\int_{Q_l}\rho_{\theta}(x,y)dxdy+n (K_0T)^{-10}.
\end{equation}
By (\ref{Eq7.4})-(\ref{Eq4.1}) and (\ref{Eq4.2})-(\ref{E3.1.5}), when the event $\mathcal{D}_l^c\cap\mathcal{H}_{l,(K_0T)^{-10}}$ holds, we have
\begin{eqnarray}\label{Eq4.7}
    R&\geq& n m_{\theta}(x_l(\Gamma)-x_{l-1}(\Gamma))(y_l(\Gamma)-y_{l-1}(\Gamma))-n((K_0T)^{-10}+\Upsilon_l)\nonumber\\
    &\geq& n(x_l(\Gamma)-x_{l-1}(\Gamma))(y_l(\Gamma)-y_{l-1}(\Gamma))(m_{\theta}-400 e^{5\theta}T^{-1})-n (K_0T)^{-10}\nonumber\\
    &\geq& \frac{1}{8}n m_{\theta}(K_0T)^{-2}-n(K_0T)^{-10}\geq \frac{1}{16}m_{\theta}(K_0T)^{-2}n.
\end{eqnarray}
By (\ref{E3.1.6}) and (\ref{Eq4.7}), taking $\delta_0=1\slash 6$, we have
\begin{eqnarray}\label{Eq4.14}
  &&  \mathbb{P}(\{|LIS(\tau)-2\sqrt{R}|>R^{1\slash 3}\}\cap\mathcal{D}_l^c\cap\mathcal{H}_{l,(K_0T)^{-10}})\nonumber\\
  &\leq&   \mathbb{P}(\{|LIS(\tau)-2\sqrt{R}|>R^{1\slash 3}\}\cap\{R\geq m_{\theta}(K_0T)^{-2}n\slash 16\})  \nonumber\\
  &=& \mathbb{E}[\mathbb{P}(|LIS(\tau)-2\sqrt{R}|>R^{1\slash 3}|\mathcal{B}_l')\mathbbm{1}_{R\geq m_{\theta}(K_0T)^{-2}n\slash 16}]\nonumber\\
  &\leq& C\mathbb{E}[\exp(-R^{1\slash 4})\mathbbm{1}_{R\geq m_{\theta}(K_0T)^{-2}n\slash 16}]\nonumber\\
  &\leq& C\exp(-c m_{\theta}^{1\slash 4} K_0^{-1\slash 2}   T^{-1\slash 2} n^{1\slash 4})\leq C\exp(-\tilde{c} n^{1\slash 4}).
\end{eqnarray}
Note that by (\ref{Eq4.1}) and (\ref{Eq4.7}), we have
\begin{equation}\label{Eq4.12}
    \int_{Q_l}\rho_{\theta}(x,y)dxdy\leq M_{\theta} (x_l(\Gamma)-x_{l-1}(\Gamma))(y_l(\Gamma)-y_{l-1}(\Gamma))\leq M_{\theta} T^{-2},
\end{equation}
\begin{equation}\label{Eq4.11}
    \int_{Q_l}\rho_{\theta}(x,y)dxdy \geq m_{\theta} (x_l(\Gamma)-x_{l-1}(\Gamma))(y_l(\Gamma)-y_{l-1}(\Gamma))\geq  (K_0T)^{-10}+\Upsilon_l.
\end{equation}
Throughout the rest of this subsection, without loss of generality, we assume that $M_{\theta}\geq 1$. When the event $\{|LIS(\tau)-2\sqrt{R}|\leq  
R^{1\slash 3}\}\cap\mathcal{D}_l^c\cap\mathcal{H}_{l,(K_0T)^{-10}}$ holds, by (\ref{Eq3.24}), (\ref{E3.1.5}), (\ref{Eq4.12}), and (\ref{Eq4.11}), we have
\begin{eqnarray}
  &&  L_{1,l}=LIS(\tau) \leq 2\sqrt{R}+R^{1\slash 3} \nonumber\\
  &\leq& 2\sqrt{n}\Big(\int_{Q_l}\rho_{\theta}(x,y) dx dy+(K_0T)^{-10}\Big)^{1\slash 2}\nonumber\\
  && +n^{1\slash 3}\Big(\int_{Q_l}\rho_{\theta}(x,y) dx dy+(K_0T)^{-10}\Big)^{1\slash 3} \nonumber\\
  &\leq& 2\sqrt{n}\Big(\int_{Q_l}\rho_{\theta}(x,y)dxdy\Big)^{1\slash 2}\Big(1+\frac{(K_0T)^{-10}}{\int_{Q_l}\rho_{\theta}(x,y)dxdy}\Big)+CM_{\theta}^{1\slash 3} T^{-2\slash 3} n^{1\slash 3}\nonumber\\
  &\leq& 2\sqrt{n}\Big(\int_{Q_l}\rho_{\theta}(x,y)dxdy\Big)^{1\slash 2}+C(K_0^{-5}T^{-5}n^{1\slash 2}+M_{\theta}^{1\slash 3}T^{-2\slash 3}n^{1\slash 3}),
\end{eqnarray}
\begin{eqnarray}
&& L_{1,l}=LIS(\tau)\geq 2\sqrt{R}-R^{1\slash 3}\nonumber\\
&\geq& 2\sqrt{n}\Big(\int_{Q_l}\rho_{\theta}(x,y) dx dy-(K_0T)^{-10}-\Upsilon_l\Big)^{1\slash 2}\nonumber\\
&&-n^{1\slash 3}\Big(\int_{Q_l}\rho_{\theta}(x,y) dx dy+(K_0T)^{-10}\Big)^{1\slash 3} \nonumber\\
&\geq& 2\sqrt{n}\Big(\int_{Q_l}\rho_{\theta}(x,y)dxdy\Big)^{1\slash 2}\Big(1-\frac{(K_0T)^{-10}+\Upsilon_l}{\int_{Q_l}\rho_{\theta}(x,y)dxdy}\Big)-CM_{\theta}^{1\slash 3} T^{-2\slash 3} n^{1\slash 3}\nonumber\\
&\geq& 2\sqrt{n}\Big(\int_{Q_l}\rho_{\theta}(x,y)dxdy\Big)^{1\slash 2}-C((K_0^{-10}T^{-10}+\Upsilon_l)^{1\slash 2}n^{1\slash 2}+M_{\theta}^{1\slash 3}T^{-2\slash 3}n^{1\slash 3})\nonumber\\
&\geq& 2\sqrt{n}\Big(\int_{Q_l}\rho_{\theta}(x,y)dxdy\Big)^{1\slash 2}-C((K_0^{-5}T^{-5}+\Upsilon_l^{1\slash 2})n^{1\slash 2}+M_{\theta}^{1\slash 3}T^{-2\slash 3}n^{1\slash 3}),\nonumber\\
&&
\end{eqnarray}
which by (\ref{Eq4.2}) and the AM-GM inequality lead to
\begin{eqnarray}\label{Eq9.1}
&& \Big|L_{1,l}-2\sqrt{n}\Big(\int_{Q_l}\rho_{\theta}(x,y)dxdy\Big)^{1\slash 2}\Big|\nonumber\\
&\leq& C((K_0^{-5}T^{-5}+\Upsilon_l^{1\slash 2})n^{1\slash 2}+M_{\theta}^{1\slash 3}T^{-2\slash 3}n^{1\slash 3}) \nonumber\\
&\leq& C_{\theta}'(T^{-5}n^{1\slash 2}+T^{-2\slash 3}n^{1\slash 3})\nonumber\\
&&+C_{\theta}' T^{-1\slash 2} n^{1\slash 2} (x_l(\Gamma)-x_{l-1}(\Gamma)+y_l(\Gamma)-y_{l-1}(\Gamma)),
\end{eqnarray}
where $C_{\theta}'$ is a positive constant that only depends on $\theta$. Let $\mathscr{E}_l'$ be the event that (\ref{Eq9.1}) holds. We have
\begin{equation*}
   \{|LIS(\tau)-2\sqrt{R}|\leq  
R^{1\slash 3}\}\cap\mathcal{D}_l^c\cap\mathcal{H}_{l,(K_0T)^{-10}} \subseteq \mathscr{E}_l'\cap\mathcal{D}_l^c\cap\mathcal{H}_{l,(K_0T)^{-10}},
\end{equation*}
hence
\begin{equation}\label{Eq8.1}
  (\mathscr{E}_l')^c\cap\mathcal{D}_l^c\cap\mathcal{H}_{l,(K_0T)^{-10}}\subseteq \{|LIS(\tau)-2\sqrt{R}|>  
R^{1\slash 3}\}\cap\mathcal{D}_l^c\cap\mathcal{H}_{l,(K_0T)^{-10}}.
\end{equation}
By (\ref{Eq4.14}) and (\ref{Eq8.1}), we have
\begin{equation}\label{Eq4.18}
\mathbb{P}((\mathscr{E}_l')^c\cap \mathcal{D}_l^c\cap\mathcal{H}_{l,(K_0T)^{-10}}) \leq C\exp(-\tilde{c} n^{1\slash 4}).
\end{equation}
By (\ref{Eq4.16}), (\ref{Eq4.17}), (\ref{Eq4.18}), and the union bound, we have
\begin{eqnarray}\label{Eq4.23}
    \mathbb{P}((\mathscr{E}_l')^c)&\leq& \mathbb{P}((\mathscr{E}_l')^c\cap \mathcal{D}_l^c\cap\mathcal{H}_{l,(K_0T)^{-10}})+\mathbb{P}(\mathcal{D}_l)+\mathbb{P}((\mathcal{H}_{l,(K_0T)^{-10}})^c)\nonumber\\
    &\leq& \tilde{C} n\exp(-\tilde{c} n^{1\slash 4} ).
\end{eqnarray}

\medskip

Recall (\ref{Eq4.24}) and (\ref{Eq9.1}), and take $C_3=C_{\theta}+C_{\theta}'$ (note that $C_3$ is a positive constant that only depends on $\theta$). Let $\mathscr{C}_l$ be the event that 
\begin{eqnarray}\label{Eq5.1}
 &&\Big|LIS(\sigma|_{nQ_l})-2\sqrt{n}\Big(\int_{Q_l}\rho_{\theta}(x,y)dxdy\Big)^{1\slash 2}\Big|\nonumber\\
 &\leq& C_3(T^{-5}n^{1\slash 2}+T^{-2\slash 3}n^{1\slash 3})\nonumber\\
 && +C_3 T^{-1\slash 2} n^{1\slash 2} (x_l(\Gamma)-x_{l-1}(\Gamma)+y_l(\Gamma)-y_{l-1}(\Gamma)).
\end{eqnarray}
By (\ref{Eq2}), (\ref{Eq4.24}), and (\ref{Eq9.1}), we have $\mathscr{E}_l\cap\mathscr{E}_l'\subseteq \mathscr{C}_l$. Hence by (\ref{Eq4.22}), (\ref{Eq4.23}), and the union bound, we have
\begin{equation}\label{Eq5.2}
    \mathbb{P}(\mathscr{C}_l^c)\leq \mathbb{P}(\mathscr{E}_l^c)+\mathbb{P}((\mathscr{E}_l')^c)\leq \tilde{C} n\exp(-\tilde{c} n^{1\slash 4}),
\end{equation}

\paragraph{Case 2: $y_{l-1}(\Gamma)< 1\slash 3$}

We generate $\sigma\in S_n$ through the following procedure. We sample $\sigma_0\in S_n$ from $\mathbb{P}_{n,\beta_n}$, and then run one step of the hit and run algorithm for the $L^1$ model to obtain $\bar{\sigma}\in S_n$. Finally, we let $\sigma\in S_n$ be such that $\sigma(i)=n+1-\bar{\sigma}(n+1-i)$ for every $i\in [n]$. As $\mathbb{P}_{n,\beta_n}$ is the stationary distribution of the hit and run algorithm, the distribution of $\bar{\sigma}$ is given by $\mathbb{P}_{n,\beta_n}$. For any $\tau\in S_n$, 
\begin{eqnarray}\label{Eq1.3.7}
  \mathbb{P}(\sigma=\tau)&=& \mathbb{P}(\bar{\sigma}(i)=n+1-\tau(n+1-i) \text{ for every }i\in [n])\nonumber\\
  &=& Z_{n,\beta_n}^{-1} \exp\Big(-\beta_n \sum_{i=1}^n |n+1-\tau(n+1-i)-i|\Big)\nonumber\\
  &=& Z_{n,\beta_n}^{-1}\exp\Big(-\beta_n  \sum_{i=1}^n|\tau(i)-i|\Big)=\mathbb{P}_{n,\beta_n}(\tau).
\end{eqnarray}
Hence the distribution of $\sigma$ is given by $\mathbb{P}_{n,\beta_n}$.

Let 
\begin{equation*}
    \tilde{Q}_l:=[1+n^{-1}-x_l(\Gamma),1+n^{-1}-x_{l-1}(\Gamma))\times [1+n^{-1}-y_l(\Gamma),1+n^{-1}-y_{l-1}(\Gamma)),
\end{equation*}
\begin{equation*}
    \bar{Q}_l:=[1-x_l(\Gamma),1-x_{l-1}(\Gamma))\times [1-y_l(\Gamma),1-y_{l-1}(\Gamma)).
\end{equation*}
Note that $LIS(\sigma|_{nQ_l})=LIS(\bar{\sigma}|_{n\tilde{Q}_l})$. As $|(S(\bar{\sigma})\cap n\tilde{Q}_l)\Delta (S(\bar{\sigma})\cap n\bar{Q}_l)|\leq 4$, 
 we have
\begin{equation}\label{Eq5.5}
    |LIS(\sigma|_{nQ_l})-LIS(\bar{\sigma}|_{n\bar{Q}_l})|=|LIS(\bar{\sigma}|_{n\tilde{Q}_l})-LIS(\bar{\sigma}|_{n\bar{Q}_l)}|\leq 4.
\end{equation}
By (\ref{Eq4.1}), as $T\geq 4$, we have
\begin{equation}\label{Eq5.4}
    y_l(\Gamma)=(y_{l}(\Gamma)-y_{l-1}(\Gamma))+y_{l-1}(\Gamma)< T^{-1}+\frac{1}{3}<\frac{2}{3}, \text{ hence }1-y_l(\Gamma)> \frac{1}{3}. 
\end{equation}
Recall (\ref{Eq5.1}) and (\ref{Eq5.2}), and note that the distribution of $\bar{\sigma}$ is given by $\mathbb{P}_{n,\beta_n}$. By (\ref{Eq5.4}), following the argument in \textbf{Case 1}, we can deduce that there exists a positive constant $C_4\geq 8$ that only depends on $\theta$, such that the following holds. Letting $\mathscr{D}_l$ be the event that 
\begin{eqnarray}\label{Eq5.3}
 &&\Big|LIS(\bar{\sigma}|_{n\bar{Q}_l})-2\sqrt{n}\Big(\int_{\bar{Q}_l}\rho_{\theta}(x,y)dxdy\Big)^{1\slash 2}\Big|\nonumber\\
 &\leq& \frac{1}{2} C_4(T^{-5}n^{1\slash 2}+T^{-2\slash 3} n^{1\slash 3})\nonumber\\
 && +C_4 T^{-1\slash 2} n^{1\slash 2} (x_l(\Gamma)-x_{l-1}(\Gamma)+y_l(\Gamma)-y_{l-1}(\Gamma)),
\end{eqnarray}
we have
\begin{equation}\label{Eq5.7}
    \mathbb{P}(\mathscr{D}_l^c)\leq \tilde{C} n \exp(-\tilde{c} n^{1\slash 4}).
\end{equation}

Now let $\tilde{\mathscr{C}}_l$ be the event that
\begin{eqnarray}\label{Eq5.8}
 &&\Big|LIS(\sigma|_{nQ_l})-2\sqrt{n}\Big(\int_{Q_l}\rho_{\theta}(x,y)dxdy\Big)^{1\slash 2}\Big|\nonumber\\
 &\leq& C_4(T^{-5}n^{1\slash 2}+T^{-2\slash 3} n^{1\slash 3})\nonumber\\
 && +C_4 T^{-1\slash 2} n^{1\slash 2} (x_l(\Gamma)-x_{l-1}(\Gamma)+y_l(\Gamma)-y_{l-1}(\Gamma)). 
\end{eqnarray}
By Proposition \ref{Densi.l1}, for any $x\in [0,1]$, $a_{\theta}(x)=a_{\theta}(1-x)$. Hence 
\begin{equation}\label{Eq5.6}
    \int_{\bar{Q}_l}\rho_{\theta}(x,y)dxdy=\int_{Q_l}\rho_{\theta}(x,y)dxdy. 
\end{equation}
By (\ref{Eq5.5}), (\ref{Eq5.3}), and (\ref{Eq5.6}), as $n\geq T^2$, we have $\mathscr{D}_l\subseteq \tilde{\mathscr{C}}_l$. Hence by (\ref{Eq5.7}), 
\begin{equation}\label{Eq10.3}
    \mathbb{P}((\tilde{\mathscr{C}}_l)^c)\leq \mathbb{P}(\mathscr{D}_l^c)  \leq \tilde{C} n \exp(-\tilde{c} n^{1\slash 4}).
\end{equation}

\bigskip

Now let 
\begin{equation*}
    N_0=\max\{8K_0 T,K_0^2T^3,n_0\}, \quad T_0=\max\Big\{\frac{1000 e^{5\theta}}{m_{\theta}},4\Big\},
\end{equation*}
and take $C_1=\max\{C_3,C_4\}$ (recall (\ref{Eq5.1}) and (\ref{Eq5.8}); note that $C_1$ is a positive constant that only depends on $\theta$). Let $\mathscr{A}_{\Gamma,l}$ be defined as in (\ref{Eq10.2}). By (\ref{Eq5.2}) and (\ref{Eq10.3}), we conclude that when $T\geq T_0$ and $n  \geq N_0$, 
\begin{equation}
    \mathbb{P}((\mathscr{A}_{\Gamma,l})^c)\leq \tilde{C} n \exp(-\tilde{c} n^{1\slash 4}).
\end{equation}

\subsection{Proof of Theorem \ref{limit_l1_1}}\label{Sect.3.2}

In this subsection, we finish the proof of Theorem \ref{limit_l1_1} based on Proposition \ref{P3.1}. We assume the assumptions that are stated in Theorem \ref{limit_l1_1}.

In the following, we consider any $T,K_0\in \mathbb{N}^{*}$ such that $T\geq 4$ and $K_0$ is odd. We let $T_0$ and $N_0$ be defined as in Proposition \ref{P3.1}, and assume that $T\geq T_0$ and $n\geq N_0$. We denote by $C',c'$ positive constants that only depend on $\theta$, and denote by $\tilde{C},\tilde{c}$ positive constants that only depend on $T,K_0$ and the sequence $\{\beta_n\}$. The values of these constants may change from line to line. 

For any $\Gamma\in \Pi^{T,T,K_0}$ and $l\in [2T-1]$, let $Q_{\Gamma,l}, Q_{\Gamma,l}',\mathscr{A}_{\Gamma,l},\mathscr{B}_{\Gamma,l}$ be defined as in Proposition \ref{P3.1}. We let
\begin{equation*}
    \mathscr{A}:=\bigcap_{\Gamma\in \Pi^{T,T,K_0}}\bigcap_{l=1}^{2T-1}\mathscr{A}_{\Gamma,l}, \quad  \mathscr{B}:=\bigcap_{\Gamma\in \Pi^{T,T,K_0}}\bigcap_{l=1}^{2T-1}\mathscr{B}_{\Gamma,l}.
\end{equation*}
By Proposition \ref{P3.1} and the union bound, we have
\begin{equation}\label{Eq11.3}
    \mathbb{P}(\mathscr{A}^c)\leq \tilde{C}n\exp(-\tilde{c} n^{1\slash 4}), \quad   \mathbb{P}(\mathscr{B}^c)\leq \tilde{C}n\exp(-\tilde{c} n^{1\slash 4}).
\end{equation}

Let $\Gamma_0\in \Pi^{T,T,K_0}$ be
\begin{equation*}
    (1,1), \frac{K_0+1}{2}, (2,1), \frac{K_0+1}{2}, (2,2), \frac{K_0+1}{2}, \cdots, (T,T-1), \frac{K_0+1}{2}, (T,T). 
\end{equation*}
We have $(x_0(\Gamma_0),y_0(\Gamma_0))=(0,0)$, $(x_{2T-1}(\Gamma_0),y_{2T-1}(\Gamma_0))=(1,1)$. For any $l\in [2T-2]$,
\begin{equation*}
    (x_l(\Gamma_0),y_l(\Gamma_0))=\Big(\frac{l+1}{2T},\frac{l}{2T}\Big).
\end{equation*}
By Lemma \ref{L1}, we have 
\begin{equation}\label{Eq10.4}
    LIS(\sigma)\geq \sum_{l=1}^{2T-1} LIS(\sigma|_{nQ_{\Gamma_0,l}}). 
\end{equation}
When the event $\mathscr{A}$ holds, by (\ref{Eq10.2}) and (\ref{Eq10.4}), we have
\begin{eqnarray}\label{Eq10.5}
&&LIS(\sigma)\geq 2\sqrt{n}\sum_{l=1}^{2T-1} \Big(\int_{Q_{\Gamma_0,l}}\rho_{\theta}(x,y)dxdy\Big)^{1\slash 2}-C'(T^{-4}n^{1\slash 2}+T^{1\slash 3}n^{1\slash 3})\nonumber\\
&&\quad\quad-C'T^{-1\slash 2}n^{1\slash 2}\sum_{l=1}^{2T-1}(x_l(\Gamma_0)-x_{l-1}(\Gamma_0)+y_l(\Gamma_0)-y_{l-1}(\Gamma_0))\nonumber\\
&&\geq 2\sqrt{n}\sum_{l=1}^{2T-1} \Big(\int_{Q_{\Gamma_0,l}}\rho_{\theta}(x,y)dxdy\Big)^{1\slash 2}-C'(T^{-1\slash 2}n^{1\slash 2}+T^{1\slash 3}n^{1\slash 3}).
\end{eqnarray}
As the function $\rho_{\theta}(\cdot,\cdot)$ is continuous on the compact set $[0,1]^2$, it is uniformly continuous. Hence for any $\epsilon>0$, there exists $\delta(\epsilon)>0$, such that for any $(x_1,y_1),(x_2,y_2)\in [0,1]^2$ satisfying $\|(x_1,y_1)-(x_2,y_2)\|_2< \delta(\epsilon)$ (where $\|\cdot\|_2$ is the Euclidean distance), we have $|\rho_{\theta}(x_1,y_1)-\rho_{\theta}(x_2,y_2)|<\epsilon$. It can be checked that for any $l\in [2T-1]$ and any $(x,y)\in Q_{\Gamma_0,l}$, 
\begin{equation*}
    \Big|x-\frac{l}{2T}\Big|\leq \frac{1}{T}, \quad  \Big|y-\frac{l}{2T}\Big|\leq \frac{1}{T}, \quad \text{hence } \Big\|(x,y)-\Big(\frac{l}{2T},\frac{l}{2T}\Big)\Big\|_2\leq \frac{2}{T}.
\end{equation*}
Below we consider any $\epsilon>0$ and assume that $T> 2\delta(\epsilon)^{-1}$. For any $l\in [2T-1]$ and any $(x,y)\in Q_{\Gamma_0,l}$, 
\begin{equation*}
    \Big|\rho_{\theta}(x,y)-\rho_{\theta}\Big(\frac{l}{2T},\frac{l}{2T}\Big)\Big|<\epsilon,
\end{equation*}
which leads to
\begin{eqnarray*}\label{Eq10.6}
   && \Big|\int_{Q_{\Gamma_0,l}}\rho_{\theta}(x,y)dxdy-\rho_{\theta}\Big(\frac{l}{2T},\frac{l}{2T}\Big)(x_l(\Gamma_0)-x_{l-1}(\Gamma_0))(y_l(\Gamma_0)-y_{l-1}(\Gamma_0))\Big|\nonumber\\
   &&\leq \epsilon (x_l(\Gamma_0)-x_{l-1}(\Gamma_0))(y_l(\Gamma_0)-y_{l-1}(\Gamma_0))\leq \epsilon T^{-2}.
\end{eqnarray*}
Hence we have
\begin{eqnarray}\label{Eq10.7}
  &&  \Big|\Big(\int_{Q_{\Gamma_0,1}}\rho_{\theta}(x,y)dxdy\Big)^{1\slash 2}-\Big(\frac{1}{2}T^{-2}\rho_{\theta}\Big(\frac{1}{2T},\frac{1}{2T}\Big)\Big)^{1\slash 2}\Big|\nonumber\\
  &\leq& \frac{\epsilon T^{-2}}{\Big(\frac{1}{2}T^{-2}\rho_{\theta}\Big(\frac{1}{2T},\frac{1}{2T}\Big)\Big)^{1\slash 2}} \leq 2\epsilon T^{-1}m_{\theta}^{-1\slash 2}\leq C'\epsilon T^{-1},
\end{eqnarray}
\begin{equation}\label{Eq10.8}
    \Big|\Big(\int_{Q_{\Gamma_0,2T-1}}\rho_{\theta}(x,y)dxdy\Big)^{1\slash 2}-\Big(\frac{1}{2}T^{-2}\rho_{\theta}\Big(\frac{2T-1}{2T},\frac{2T-1}{2T}\Big)\Big)^{1\slash 2}\Big|\leq C'\epsilon T^{-1},
\end{equation}
and for any $l\in \{2,3,\cdots,2T-2\}$, 
\begin{equation}\label{Eq10.9}
    \Big|\Big(\int_{Q_{\Gamma_0,l}}\rho_{\theta}(x,y)dxdy\Big)^{1\slash 2}-\Big(\frac{1}{4}T^{-2}\rho_{\theta}\Big(\frac{l}{2T},\frac{l}{2T}\Big)\Big)^{1\slash 2}\Big|\leq C'\epsilon T^{-1}.
\end{equation}
By (\ref{Eq10.5})-(\ref{Eq10.9}), when the event $\mathscr{A}$ holds, for any $\epsilon>0$, if $T>2\delta(\epsilon)^{-1}$, then
\begin{eqnarray}\label{Eq10.21}
&& LIS(\sigma)\geq \sqrt{2}T^{-1}\sqrt{n}\Big(\rho_{\theta}\Big(\frac{1}{2T},\frac{1}{2T}\Big)^{1\slash 2}+\rho_{\theta}\Big(\frac{2T-1}{2T},\frac{2T-1}{2T}\Big)^{1\slash 2}\Big) \nonumber\\
&&\quad\quad +T^{-1}\sqrt{n}\sum_{l=2}^{2T-2}\rho_{\theta}\Big(\frac{l}{2T},\frac{l}{2T}\Big)^{1\slash 2}-C'((\epsilon+ T^{-1\slash 2})n^{1\slash 2}+T^{1\slash 3}n^{1\slash 3})\nonumber\\
&&\geq T^{-1}\sqrt{n}\sum_{l=1}^{2T}\rho_{\theta}\Big(\frac{l}{2T},\frac{l}{2T}\Big)^{1\slash 2}-C'((\epsilon+T^{-1\slash 2})n^{1\slash 2}+T^{1\slash 3}n^{1\slash 3}).
\end{eqnarray}

Below we consider any $\Gamma\in \Pi^{T,T,K_0}$ and $\epsilon>0$, and assume that $T>2\delta(\epsilon)^{-1}$. For any $l\in [2T-1]$ such that $a_{l-1}(\Gamma)=c_l(\Gamma)$ or $b_{l-1}(\Gamma)=d_l(\Gamma)$, we have
\begin{equation}\label{Eq10.13}
    \int_{Q_{\Gamma,l}'}\rho_{\theta}(x,y)dxdy=\rho_{\theta}(x_{l-1}(\Gamma),y_{l-1}(\Gamma))(c_l(\Gamma)-a_{l-1}(\Gamma))(d_l(\Gamma)-b_{l-1}(\Gamma))=0.
\end{equation}
Now consider any $l\in [2T-1]$ such that $a_{l-1}(\Gamma)<c_l(\Gamma)$ and $b_{l-1}(\Gamma)<d_l(\Gamma)$. We have
\begin{equation}\label{Eq10.11}
   (2K_0T)^{-1}\leq c_l(\Gamma)-a_{l-1}(\Gamma)\leq T^{-1}, \quad (2K_0T)^{-1}\leq d_l(\Gamma)-b_{l-1}(\Gamma) \leq T^{-1}. 
\end{equation}
Hence for any $(x,y)\in Q_{\Gamma,l}'$, $||(x,y)-(x_{l-1}(\Gamma),y_{l-1}(\Gamma))||_2\leq 2T^{-1}<\delta(\epsilon)$, and 
\begin{eqnarray}\label{Eq10.12}
 &&  \Big| \int_{Q_{\Gamma,l}'}\rho_{\theta}(x,y)dxdy-\rho_{\theta}(x_{l-1}(\Gamma),y_{l-1}(\Gamma))(c_l(\Gamma)-a_{l-1}(\Gamma))(d_l(\Gamma)-b_{l-1}(\Gamma))\Big| \nonumber\\
 && \leq \epsilon (c_l(\Gamma)-a_{l-1}(\Gamma))(d_l(\Gamma)-b_{l-1}(\Gamma))\leq \epsilon T^{-2}. 
\end{eqnarray}
By (\ref{Eq10.11}) and (\ref{Eq10.12}), we have
\begin{eqnarray}\label{Eq10.14}
&& \Big|\Big(\int_{Q_{\Gamma,l}'}\rho_{\theta}(x,y)dxdy\Big)^{1\slash 2} \nonumber\\
&&\quad-\rho_{\theta}(x_{l-1}(\Gamma),y_{l-1}(\Gamma))^{1\slash 2}(c_l(\Gamma)-a_{l-1}(\Gamma))^{1\slash 2}(d_l(\Gamma)-b_{l-1}(\Gamma))^{1\slash 2}\Big|\nonumber\\
&\leq& \frac{\epsilon T^{-2}}{\rho_{\theta}(x_{l-1}(\Gamma),y_{l-1}(\Gamma))^{1\slash 2}(c_l(\Gamma)-a_{l-1}(\Gamma))^{1\slash 2}(d_l(\Gamma)-b_{l-1}(\Gamma))^{1\slash 2}}\nonumber\\
&\leq& \frac{\epsilon T^{-2}}{m_{\theta}^{1\slash 2} (2K_0T)^{-1}}= 2 \epsilon K_0 T^{-1} m_{\theta}^{-1\slash 2}\leq C' \epsilon K_0 T^{-1}.
\end{eqnarray}
By (\ref{Eq10.13}) and (\ref{Eq10.14}), we have
\begin{eqnarray}\label{Eq10.15}
&& \sum_{l=1}^{2T-1} \Big(\int_{Q_{\Gamma,l}'}\rho_{\theta}(x,y)dxdy\Big)^{1\slash 2}\nonumber\\
&\leq& \sum_{l=1}^{2T-1}\rho_{\theta}(x_{l-1}(\Gamma),y_{l-1}(\Gamma))^{1\slash 2}(c_l(\Gamma)-a_{l-1}(\Gamma))^{1\slash 2}(d_l(\Gamma)-b_{l-1}(\Gamma))^{1\slash 2}\nonumber\\
&&+C'\epsilon K_0.
\end{eqnarray}
By the AM-GM inequality, recalling Proposition \ref{Densi.l1}, we have
\begin{eqnarray}\label{Eq10.16}
&& \sum_{l=1}^{2T-1} \rho_{\theta}(x_{l-1}(\Gamma),y_{l-1}(\Gamma))^{1\slash 2}(c_l(\Gamma)-a_{l-1}(\Gamma))^{1\slash 2}(d_l(\Gamma)-b_{l-1}(\Gamma))^{1\slash 2}\nonumber\\
&\leq & \sum_{l=1}^{2T-1} (e^{a_{\theta}(x_{l-1}(\Gamma))}(c_l(\Gamma)-a_{l-1}(\Gamma)))^{1\slash 2} (e^{a_{\theta}(y_{l-1}(\Gamma))}(d_l(\Gamma)-b_{l-1}(\Gamma)))^{1\slash 2}\nonumber\\
&\leq& \frac{1}{2}\sum_{l=1}^{2T-1} e^{a_{\theta}(x_{l-1}(\Gamma))}(c_l(\Gamma)-a_{l-1}(\Gamma))+\frac{1}{2}\sum_{l=1}^{2T-1} e^{a_{\theta}(y_{l-1}(\Gamma))}(d_l(\Gamma)-b_{l-1}(\Gamma)).\nonumber\\
&&
\end{eqnarray}
Note that for any $x\in [0,1]$, $e^{a_{\theta}(x)}=\sqrt{\rho_{\theta}(x,x)}\leq M_{\theta}^{1\slash 2}$. Moreover, for any $l\in [2T-1]$,
\begin{eqnarray*}
  &&  |c_l(\Gamma)-x_l(\Gamma)|\leq (2K_0T)^{-1},\quad |a_{l-1}(\Gamma)-x_{l-1}(\Gamma)|\leq (2K_0T)^{-1},\nonumber\\
  &&|d_l(\Gamma)-y_l(\Gamma)|\leq (2K_0T)^{-1},\quad |b_{l-1}(\Gamma)-y_{l-1}(\Gamma)|\leq (2K_0T)^{-1}.
\end{eqnarray*}
Hence by (\ref{Eq10.15}) and (\ref{Eq10.16}), we have
\begin{eqnarray}\label{Eq10.17}
  &&  \sum_{l=1}^{2T-1} \Big(\int_{Q_{\Gamma,l}'}\rho_{\theta}(x,y)dxdy\Big)^{1\slash 2}\leq \frac{1}{2}\sum_{l=1}^{2T-1}e^{a_{\theta}(x_{l-1}(\Gamma))}(x_l(\Gamma)-x_{l-1}(\Gamma))\nonumber\\
  &&\quad\quad\quad    +\frac{1}{2}\sum_{l=1}^{2T-1}e^{a_{\theta}(y_{l-1}(\Gamma))}(y_l(\Gamma)-y_{l-1}(\Gamma))+C'(\epsilon K_0+K_0^{-1}).
\end{eqnarray}
Moreover,
\begin{eqnarray}\label{Eq10.20}
  &&  \sum_{l=1}^{2T-1}(c_l(\Gamma)-a_{l-1}(\Gamma)+d_l(\Gamma)-b_{l-1}(\Gamma))\nonumber\\
  &\leq& \sum_{l=1}^{2T-1}(x_l(\Gamma)-x_{l-1}(\Gamma)+y_l(\Gamma)-y_{l-1}(\Gamma))+CK_0^{-1}\leq C.
\end{eqnarray}

For any $\epsilon>0$, when $T>2\delta(\epsilon)^{-1}$ and the event $\mathscr{B}$ holds, by Lemma \ref{L1}, (\ref{Eq10.18}), (\ref{Eq10.17}), and (\ref{Eq10.20}), we have
\begin{eqnarray}\label{Eq10.22}
&& LIS(\sigma)\leq \max_{\Gamma\in \Pi^{T,T,K_0}}\Big\{\sum_{l=1}^{2T-1}LIS(\sigma|_{n Q_{\Gamma,l}'})\Big\} \nonumber\\
&\leq& 2\sqrt{n}\max_{\Gamma\in \Pi^{T,T,K_0}}\Big\{\sum_{l=1}^{2T-1}\Big(\int_{Q_{\Gamma,l}'}\rho_{\theta}(x,y)dxdy\Big)^{1\slash 2}\Big\}+C'(T^{-4}n^{1\slash 2}+T^{1\slash 3} n^{1\slash 3})\nonumber\\
&&+C' T^{-1\slash 2} n^{1\slash 2}\max_{\Gamma\in \Pi^{T,T,K_0}}\Big\{\sum_{l=1}^{2T-1}(c_l(\Gamma)-a_{l-1}(\Gamma)+d_l(\Gamma)-b_{l-1}(\Gamma))\Big\}\nonumber\\
&\leq&\sqrt{n}\max_{\Gamma\in \Pi^{T,T,K_0}}\Big\{\sum_{l=1}^{2T-1}\sqrt{\rho_{\theta}(x_{l-1}(\Gamma),x_{l-1}(\Gamma))}(x_l(\Gamma)-x_{l-1}(\Gamma))\Big\}\nonumber\\
&&+\sqrt{n}\max_{\Gamma\in \Pi^{T,T,K_0}}\Big\{\sum_{l=1}^{2T-1}\sqrt{\rho_{\theta}(y_{l-1}(\Gamma),y_{l-1}(\Gamma))}(y_l(\Gamma)-y_{l-1}(\Gamma))\Big\}\nonumber\\
&&+C'((T^{-1\slash 2}+\epsilon K_0+K_0^{-1})n^{1\slash 2}+T^{1\slash 3}n^{1\slash 3}).
\end{eqnarray}

By (\ref{Eq10.21}) and (\ref{Eq10.22}), for any $\epsilon>0$, when the event $\mathscr{A}\cap\mathscr{B}$ holds and $T>2\delta(\epsilon)^{-1}$, we have
\begin{eqnarray}
  &&   \Big|\frac{LIS(\sigma)}{\sqrt{n}}-2\int_0^1\sqrt{\rho_{\theta}(x,x)}dx\Big|\nonumber\\
  &\leq& 2\Big|\sum_{l=1}^{2T}\rho_{\theta}\Big(\frac{l}{2T},\frac{l}{2T}\Big)^{1\slash 2}\frac{1}{2T}-\int_0^1\sqrt{\rho_{\theta}(x,x)}dx\Big|\nonumber\\
  &&+\max_{\Gamma\in \Pi^{T,T,K_0}}\Big\{\Big|\sum_{l=1}^{2T-1}\sqrt{\rho_{\theta}(x_{l-1}(\Gamma),x_{l-1}(\Gamma))}(x_l(\Gamma)-x_{l-1}(\Gamma))-\int_0^1\sqrt{\rho_{\theta}(x,x)}dx\Big|\Big\}\nonumber\\
  &&+\max_{\Gamma\in \Pi^{T,T,K_0}}\Big\{\Big|\sum_{l=1}^{2T-1}\sqrt{\rho_{\theta}(y_{l-1}(\Gamma),y_{l-1}(\Gamma))}(y_l(\Gamma)-y_{l-1}(\Gamma))-\int_0^1\sqrt{\rho_{\theta}(x,x)}dx\Big|\Big\}\nonumber\\
  &&+C'(T^{-1\slash 2}+\epsilon K_0+K_0^{-1}+T^{1\slash 3}n^{-1\slash 6}).
\end{eqnarray}
Hence for any $\epsilon>0$, if $T>2\delta(\epsilon)^{-1}$, we have
\begin{eqnarray}\label{Eq11.1}
  && \mathbb{E}\Big[\Big|\frac{LIS(\sigma)}{\sqrt{n}}-2\int_0^1\sqrt{\rho_{\theta}(x,x)}dx\Big|\mathbbm{1}_{\mathscr{A}\cap\mathscr{B}}\Big]\nonumber\\
  &\leq& 2\Big|\sum_{l=1}^{2T}\rho_{\theta}\Big(\frac{l}{2T},\frac{l}{2T}\Big)^{1\slash 2}\frac{1}{2T}-\int_0^1\sqrt{\rho_{\theta}(x,x)}dx\Big|\nonumber\\
  &&+\max_{\Gamma\in \Pi^{T,T,K_0}}\Big\{\Big|\sum_{l=1}^{2T-1}\sqrt{\rho_{\theta}(x_{l-1}(\Gamma),x_{l-1}(\Gamma))}(x_l(\Gamma)-x_{l-1}(\Gamma))-\int_0^1\sqrt{\rho_{\theta}(x,x)}dx\Big|\Big\}\nonumber\\
  &&+\max_{\Gamma\in \Pi^{T,T,K_0}}\Big\{\Big|\sum_{l=1}^{2T-1}\sqrt{\rho_{\theta}(y_{l-1}(\Gamma),y_{l-1}(\Gamma))}(y_l(\Gamma)-y_{l-1}(\Gamma))-\int_0^1\sqrt{\rho_{\theta}(x,x)}dx\Big|\Big\}\nonumber\\
  &&+C'(T^{-1\slash 2}+\epsilon K_0+K_0^{-1}+T^{1\slash 3}n^{-1\slash 6}).
\end{eqnarray}
As $LIS(\sigma)\leq n$, by (\ref{Eq11.3}) and the union bound, we have 
\begin{eqnarray}\label{Eq11.4}
  &&  \mathbb{E}\Big[\Big|\frac{LIS(\sigma)}{\sqrt{n}}-2\int_0^1\sqrt{\rho_{\theta}(x,x)}dx\Big|\mathbbm{1}_{(\mathscr{A}\cap\mathscr{B})^c}\Big]\leq (\sqrt{n}+2M_{\theta}^{1\slash 2})\mathbb{P}((\mathscr{A}\cap\mathscr{B})^c) \nonumber\\
  &&\leq C'\sqrt{n}(\mathbb{P}(\mathscr{A}^c)+\mathbb{P}(\mathscr{B}^c))\leq \tilde{C}n^{3\slash 2}\exp(-\tilde{c} n^{1\slash 4}).
\end{eqnarray}
By (\ref{Eq11.1}) and (\ref{Eq11.4}), for any $\epsilon>0$, if $T>2\delta(\epsilon)^{-1}$, we have
\begin{eqnarray}\label{Eq11.6}
&& \mathbb{E}\Big[\Big|\frac{LIS(\sigma)}{\sqrt{n}}-2\int_0^1\sqrt{\rho_{\theta}(x,x)}dx\Big|\Big]\nonumber\\
  &\leq& 2\Big|\sum_{l=1}^{2T}\rho_{\theta}\Big(\frac{l}{2T},\frac{l}{2T}\Big)^{1\slash 2}\frac{1}{2T}-\int_0^1\sqrt{\rho_{\theta}(x,x)}dx\Big|\nonumber\\
  &&+\max_{\Gamma\in \Pi^{T,T,K_0}}\Big\{\Big|\sum_{l=1}^{2T-1}\sqrt{\rho_{\theta}(x_{l-1}(\Gamma),x_{l-1}(\Gamma))}(x_l(\Gamma)-x_{l-1}(\Gamma))-\int_0^1\sqrt{\rho_{\theta}(x,x)}dx\Big|\Big\}\nonumber\\
  &&+\max_{\Gamma\in \Pi^{T,T,K_0}}\Big\{\Big|\sum_{l=1}^{2T-1}\sqrt{\rho_{\theta}(y_{l-1}(\Gamma),y_{l-1}(\Gamma))}(y_l(\Gamma)-y_{l-1}(\Gamma))-\int_0^1\sqrt{\rho_{\theta}(x,x)}dx\Big|\Big\}\nonumber\\
  &&+C'(T^{-1\slash 2}+\epsilon K_0+K_0^{-1}+T^{1\slash 3}n^{-1\slash 6})+\tilde{C}n^{3\slash 2}\exp(-\tilde{c} n^{1\slash 4}).
\end{eqnarray}

Note that as $T\rightarrow\infty$, the first three terms on the right-hand side of (\ref{Eq11.6}) converges to $0$. In (\ref{Eq11.6}), first letting $n\rightarrow\infty$, and then letting $T\rightarrow\infty$, we obtain that for any $\epsilon>0$,
\begin{equation}\label{Eq11.2}
    \limsup_{n\rightarrow\infty} \mathbb{E}\Big[\Big|\frac{LIS(\sigma)}{\sqrt{n}}-2\int_0^1\sqrt{\rho_{\theta}(x,x)}dx\Big|\Big]\leq C'(\epsilon K_0+K_0^{-1}).
\end{equation}
In (\ref{Eq11.2}), first taking $\epsilon\rightarrow 0$, and then taking $K_0\rightarrow\infty$, we conclude that
\begin{equation}
    \lim_{n\rightarrow\infty} \mathbb{E}\Big[\Big|\frac{LIS(\sigma)}{\sqrt{n}}-2\int_0^1\sqrt{\rho_{\theta}(x,x)}dx\Big|\Big]=0.
\end{equation}




\section{Proof of Theorem \ref{limit_l1_2}}\label{Sect.4}

In this section, we give the proof of Theorem \ref{limit_l1_2}. We first establish three preliminary propositions in Section \ref{Sect.4.1}. Based on these propositions, we finish the proof of Theorem \ref{limit_l1_2} in Section \ref{Sect.4.2}.

\subsection{Three preliminary propositions}\label{Sect.4.1}

In this subsection, we establish three preliminary propositions. These propositions will be used in the proof of Theorem \ref{limit_l1_2}. 

Throughout this subsection, we fix an arbitrary sequence of positive numbers $(\beta_n)_{n=1}^{\infty}$ such that $\lim_{n\rightarrow\infty}\beta_n=0$ and $\lim_{n\rightarrow\infty} n\beta_n=\infty$. We also fix any $L\in \mathbb{N}^{*}$ such that $L\geq 4$.

Below we consider any $n\in\mathbb{N}^{*}$ such that $n\beta_n\geq 4L$ and $\beta_n\leq 1\slash 10$. For any $s\in [1,\lfloor n\beta_n\slash L\rfloor-1]\cap\mathbb{N}$, we let
\begin{equation}
    \mathcal{I}_{n,s}:=((s-1)L\beta_n^{-1},sL\beta_n^{-1}].
\end{equation}
We also let
\begin{equation}
   \mathcal{I}_{n,\lfloor n\beta_n\slash L\rfloor}:=((\lfloor n\beta_n\slash L\rfloor -1)L\beta_n^{-1},n].
\end{equation}
For any $s\in [\lfloor n\beta_n\slash L\rfloor]$, we let
\begin{equation}
    \mathcal{R}_s:=\mathcal{I}_{n,s}\times \mathcal{I}_{n,s}.
\end{equation}
For any $s\in [2,\lfloor n\beta_n\slash L\rfloor]\cap\mathbb{N}$, we let
\begin{equation}
   \mathcal{R}_s':=(0,(s-1)L\beta_n^{-1}]\times \mathcal{I}_{n,s},   \quad  \mathcal{R}_s'':=\mathcal{I}_{n,s}\times (0,(s-1)L\beta_n^{-1}].
\end{equation}
Note that $\Big(\bigcup_{s=1}^{\lfloor n\beta_n\slash L\rfloor}\mathcal{R}_s\Big) \bigcup \Big( \bigcup_{s=2}^{\lfloor n\beta_n\slash L\rfloor}\mathcal{R}_s' \Big)\bigcup \Big( \bigcup_{s=2}^{\lfloor n\beta_n\slash L\rfloor}\mathcal{R}_s'' \Big)=(0,n]^2$.
Hence for any $\sigma\in S_n$,
\begin{equation}\label{Eq3.2.1}
    LIS(\sigma)\leq \sum_{s=1}^{\lfloor n\beta_n\slash L\rfloor}LIS(\sigma|_{\mathcal{R}_s})+\sum_{s=2}^{\lfloor n\beta_n\slash L\rfloor}LIS(\sigma|_{\mathcal{R}_s'})+\sum_{s=2}^{\lfloor n\beta_n\slash L\rfloor}LIS(\sigma|_{\mathcal{R}_s''}).
\end{equation}
Moreover,
\begin{equation}\label{Eq3.2.2}
    LIS(\sigma)\geq \sum_{s=1}^{\lfloor n\beta_n\slash L\rfloor}LIS(\sigma|_{\mathcal{R}_s}).
\end{equation}

The following proposition bounds $LIS(\sigma|_{\mathcal{R}_s'})$ and $LIS(\sigma|_{\mathcal{R}_s''})$ for $\sigma$ drawn from $\mathbb{P}_{n,\beta_n}$ and any $s\in [2,\lfloor n\beta_n\slash L\rfloor]\cap\mathbb{N}$.  

\begin{proposition}\label{P4.1}
Assume that $n\beta_n\geq 4L$ and $\beta_n\leq 1\slash 10$, and let $\sigma$ be drawn from $\mathbb{P}_{n,\beta_n}$. Then there exist positive absolute constants $C,c$, such that for any $s\in [2,\lfloor n\beta_n\slash L\rfloor]\cap\mathbb{N}$, we have
\begin{equation}\label{Eq1.2.1}
    \mathbb{E}[LIS(\sigma|_{\mathcal{R}_s'})]\leq CL^{1\slash 2}\beta_n^{-1\slash 2}+CL^2\exp(-c\beta_n^{-1\slash 2}),
\end{equation}
\begin{equation}\label{Eq1.2.2}
    \mathbb{E}[LIS(\sigma|_{\mathcal{R}_s''})]\leq CL^{1\slash 2}\beta_n^{-1\slash 2}+CL^2\exp(-c\beta_n^{-1\slash 2}).
\end{equation}
\end{proposition}

\begin{proof}

In the following, we fix an arbitrary $s\in [2,\lfloor n\beta_n\slash L\rfloor]\cap\mathbb{N}$. 

We first show (\ref{Eq1.2.1}). We sample $\sigma_0$ from $\mathbb{P}_{n,\beta_n}$, and then run one step of the hit and run algorithm for the $L^1$ model to obtain $\sigma$. As $\mathbb{P}_{n,\beta_n}$ is the stationary distribution of the hit and run algorithm, the distribution of $\sigma$ is given by $\mathbb{P}_{n,\beta_n}$. 

Recall that in the hit and run algorithm, starting from $\sigma_0$, for every $i\in [n]$, we independently sample $u_i$ from the uniform distribution on $[0, e^{-2\beta_n(\sigma_0(i)-i)_{+}}]$, and take $b_i=i-\log(u_i)\slash (2\beta_n)$. For every $i\in [n]$, let
\begin{equation}\label{Eq1.1.9}
    N_i:=|\{j\in [n]: b_j\geq i\}|-n+i.
\end{equation}
Then we sample $\sigma$ uniformly from the set 
\begin{equation}\label{Eq1.1.14}
    \{\tau\in S_n: \tau(i)\leq b_i\text{ for every }i\in [n]\}
\end{equation}
through the following procedure. Look at the $N_n$ integers $i\in [n]$ with $b_i\geq n$, and pick $Y_n$ uniformly from these integers; then look at the $N_{n-1}$ remaining integers $i\in[n]$ with $b_i\geq n-1$ (with $Y_n$ deleted from the list), and pick $Y_{n-1}$ uniformly from these integers; and so on. In this way we obtain $\{Y_i\}_{i=1}^n$. Finally, we let $\sigma\in S_n$ be such that $\sigma(Y_i)=i$ for every $i\in [n]$.

Let
\begin{equation}\label{Eq1.1.15}
    W_s:=\{     i \in [1,(s-1)L\beta_n^{-1}]\cap\mathbb{N}:  b_i\geq (s-1)L\beta_n^{-1}    \}.
\end{equation}
Recall Definition \ref{Def2.2}. Note that 
\begin{eqnarray}\label{Eq1.1.8}
    |W_s|&=& \sum_{i \in [1,(s-1)L\beta_n^{-1}]\cap\mathbb{N}} \mathbbm{1}_{b_i\geq (s-1) L\beta_n^{-1}}\nonumber\\
    &\leq& \big|\mathcal{D}_{\lfloor (s-1)L\beta_n^{-1}\rfloor}(\sigma_0)\big|+\sum_{\substack{i \in [1,(s-1)L\beta_n^{-1}]\cap\mathbb{N}:\\\sigma_0(i)\leq \lfloor (s-1)L\beta_n^{-1}\rfloor}}\mathbbm{1}_{b_i\geq (s-1) L\beta_n^{-1}}.
\end{eqnarray}

Let $\mathscr{I}_{s}$ be the set of $i\in [1,(s-1)L\beta_n^{-1}]\cap\mathbb{N}$ such that $\sigma_0(i)\leq \lfloor (s-1)L\beta_n^{-1}\rfloor$. For any $i\in [1,(s-1)L\beta_n^{-1}]\cap\mathbb{N}$, if $i\in \mathscr{I}_{s}$, we let $X_i:=\mathbbm{1}_{b_i\geq (s-1) L\beta_n^{-1}}$; otherwise we let $X_i:=0$. Note that conditional on $\sigma_0$, $\{X_i\}_{i=1}^{\lfloor (s-1)L\beta_n^{-1}  \rfloor}$ are mutually independent, and for any $i\in \mathscr{I}_{s}$, $X_i$ follows the Bernoulli distribution with
\begin{equation}\label{Eq1.1.1}
     \mathbb{P}(X_i=1|\sigma_0)= \mathbb{P}(b_i\geq (s-1)L\beta_n^{-1}|\sigma_0)=e^{-2\beta_n ((s-1)L\beta_n^{-1}-\max\{i,\sigma_0(i)\})_{+}}.
\end{equation}
For any $l\in [\lceil\log_2((s-1)L)\rceil]$, we define $u_l:=2^{l-1}\beta_n^{-1}$, and let $\mathcal{U}_l$ be the set of $i\in \mathscr{I}_{s}$ such that $u_l\leq (s-1)L\beta_n^{-1}-\max\{i,\sigma_0(i)\}\leq 2u_l$. Below we consider any $l\in [\lceil\log_2((s-1)L)\rceil]$. Note that $u_l\geq \beta_n^{-1}\geq 1$ and $|\mathcal{U}_l|\leq 2(u_l+1)\leq 4u_l$. Moreover, by (\ref{Eq1.1.1}), for any $i\in \mathcal{U}_l$, we have $\mathbb{P}(X_i=1|\sigma_0)\leq e^{-2\beta_n u_l}$. Hence for any $t\in\mathbb{N}^{*}$, by Lemma \ref{Lemma3.1}, we have
\begin{eqnarray}\label{Eq1.1.3}
  && \mathbb{P}\Big(\sum_{i\in\mathcal{U}_l}\mathbbm{1}_{b_i\geq (s-1) L\beta_n^{-1}}\geq t|\sigma_0\Big)=\mathbb{P}\Big(\sum_{i\in\mathcal{U}_l}X_i\geq t|\sigma_0\Big) \nonumber\\
  &\leq& \sum_{\substack{i_1,i_2,\cdots,i_t\in \mathcal{U}_l:\\ i_1<i_2<\cdots<i_t}} \prod_{l=1}^t \mathbb{P}(X_{i_l}=1|\sigma_0)\leq \binom{|\mathcal{U}_l|}{t}e^{-2\beta_n u_l t} \leq \Big( \frac{e|\mathcal{U}_l|}{t}\Big)^t e^{-2\beta_n u_l t}\nonumber\\
  &\leq&\Big(\frac{4e u_l}{t}\Big)^t e^{-2\beta_n u_l t}=\Big(\frac{e\cdot 2^{l+1}\beta_n^{-1}}{t}\Big)^t e^{-2^l t}.
\end{eqnarray}
Note that there exists a positive absolute constant $C_1$, such that for any $l'\in \mathbb{N}^{*}$, $2^{-l'}\log(e\cdot 2^{1+3l'\slash 2}\cdot C_1^{-1})\leq 1\slash 2$. Taking $t=\lceil C_1 2^{-l\slash 2} \beta_n^{-1} \rceil\geq C_1 2^{-l\slash 2} \beta_n^{-1}$ in (\ref{Eq1.1.3}), we obtain that
\begin{eqnarray}\label{Eq1.1.4}
  &&  \mathbb{P}\Big(\sum_{i\in\mathcal{U}_l}\mathbbm{1}_{b_i\geq(s-1) L\beta_n^{-1}}\geq C_1 2^{-l\slash 2} \beta_n^{-1}\Big)\nonumber\\
  &=& \mathbb{P}\Big(\sum_{i\in\mathcal{U}_l}\mathbbm{1}_{b_i\geq(s-1) L\beta_n^{-1}}\geq \lceil C_1 2^{-l\slash 2} \beta_n^{-1}\rceil\Big)\nonumber\\
  &=&\mathbb{E}\Big[\mathbb{P}\Big(\sum_{i\in\mathcal{U}_l}\mathbbm{1}_{b_i\geq(s-1) L\beta_n^{-1}}\geq \lceil C_1 2^{-l\slash 2} \beta_n^{-1}\rceil \Big|\sigma_0\Big)\Big]\nonumber\\
  &\leq& \exp(-\lceil C_1 2^{-l\slash 2}\beta_n^{-1}\rceil 2^l (1-2^{-l}\log(e\cdot 2^{1+3l\slash 2}\cdot C_1^{-1})))\nonumber\\
  &\leq& \exp(-C_1 2^{l\slash 2-1} \beta_n^{-1}). 
\end{eqnarray}
Let $\mathcal{V}_l$ be the event that 
\begin{equation}\label{Eq1.1.5}
      \sum_{i\in\mathcal{U}_l} \mathbbm{1}_{b_i\geq (s-1)L\beta_n^{-1}}\leq C_1 2^{-l\slash 2} \beta_n^{-1}.
\end{equation}
By (\ref{Eq1.1.4}), we have
\begin{equation}\label{Eq1.1.6}
    \mathbb{P}(\mathcal{V}_l^c)\leq \exp(-C_1 2^{l\slash 2-1} \beta_n^{-1}).
\end{equation}
When the event $\bigcap\limits_{l=1}^{\lceil\log_2((s-1)L)\rceil}\mathcal{V}_l$ holds, by (\ref{Eq1.1.5}), we have
\begin{eqnarray}
    &&\sum_{i\in\mathscr{I}_{s}}\mathbbm{1}_{b_i\geq (s-1) L\beta_n^{-1}}\leq |\{i\in \mathscr{I}_{s}: (s-1)L\beta_n^{-1}-\max\{i,\sigma_0(i)\}\leq \beta_n^{-1}\}|\nonumber\\
    &&\quad\quad\quad\quad\quad\quad\quad\quad\quad\quad+\sum_{l=1}^{\lceil\log_2((s-1)L)\rceil} \sum_{i\in\mathcal{U}_l} \mathbbm{1}_{b_i\geq (s-1)L\beta_n^{-1}}\nonumber\\
    &&\quad\quad\quad\leq 2(\beta_n^{-1}+1)+\sum_{l=1}^{\lceil\log_2((s-1)L)\rceil} C_1 2^{-l\slash 2} \beta_n^{-1}\leq C_2\beta_n^{-1},
\end{eqnarray}
where $C_2$ is a positive absolute constant. Hence letting $\mathcal{V}$ be the event that
\begin{equation}
    \sum_{i\in\mathscr{I}_{s}}\mathbbm{1}_{b_i\geq (s-1) L\beta_n^{-1}}\leq C_2\beta_n^{-1},
\end{equation}
we have $\bigcap\limits_{l=1}^{\lceil\log_2((s-1)L)\rceil}\mathcal{V}_l\subseteq \mathcal{V}$. By (\ref{Eq1.1.6}) and the union bound, we have
\begin{equation}\label{Eq1.1.7}
    \mathbb{P}(\mathcal{V}^c)\leq \sum_{l=1}^{\lceil\log_2((s-1)L)\rceil} \exp(-C_1 2^{l\slash 2-1}\beta_n^{-1})\leq C\exp(-c \beta_n^{-1}).
\end{equation}
By (\ref{Eq1.1.8}), (\ref{Eq1.1.7}), and Proposition \ref{P2.2}, there exists a positive absolute constant $C_0$, such that the event $\mathcal{W}=\{|W_s|\leq C_0\beta_n^{-1}\}$ satisfies
\begin{equation}\label{Eq1.1.18}
    \mathbb{P}(\mathcal{W}^c)\leq C\exp(-c \beta_n^{-1}).
\end{equation}

Let $\mathcal{B}_n$ be the $\sigma$-algebra generated by $\sigma_0$ and $\{b_i\}_{i=1}^n$. For any $l\in [n-1]$, let $\mathcal{B}_l$ be the $\sigma$-algebra generated by $\sigma_0$, $\{b_i\}_{i=1}^n$, and $\{Y_i\}_{i=l+1}^n$. For any $q\in \mathbb{N}^{*}$, we let
\begin{equation}\label{Eqq1.1}
    \Lambda_{s,q}:= \sum_{\substack{i_1<\cdots<i_q,j_1<\cdots<j_q\\i_1,\cdots,i_q\in [1,(s-1)L\beta_n^{-1}]\cap\mathbb{N}\\j_1,\cdots,j_q\in \mathcal{I}_{n,s}\cap\mathbb{N}^{*}}}\mathbbm{1}_{\sigma(i_1)=j_1,\cdots,\sigma(i_q)=j_q}\mathbbm{1}_{i_1,\cdots,i_q\in W_s} 
\end{equation}
For any $i_1,\cdots,i_q\in [1,(s-1)L\beta_n^{-1}]\cap\mathbb{N}$ and $j_1,\cdots,j_q\in \mathcal{I}_{n,s}\cap\mathbb{N}^{*}$ such that $i_1<\cdots<i_q$ and $j_1<\cdots<j_q$, we have
\begin{eqnarray}\label{Eq1.1.11}
  &&  \mathbb{E}[\mathbbm{1}_{\sigma(i_1)=j_1,\cdots,\sigma(i_q)=j_q}\mathbbm{1}_{i_1,\cdots,i_q\in W_s} |\mathcal{B}_n]\nonumber\\
  &=&\mathbb{E}[\mathbbm{1}_{\sigma(i_1)=j_1,\cdots,\sigma(i_q)=j_q}|\mathcal{B}_n] \mathbbm{1}_{i_1,\cdots,i_q\in W_s}\nonumber\\
  &=& \mathbb{E}[\mathbb{E}[\mathbbm{1}_{\sigma(i_1)=j_1}|\mathcal{B}_{j_1}]\mathbbm{1}_{\sigma(i_2)=j_2,\cdots,\sigma(i_q)=j_q}|\mathcal{B}_n]\mathbbm{1}_{i_1,\cdots,i_q\in W_s} \nonumber\\
  &\leq & \frac{\mathbbm{1}_{i_1,\cdots,i_q\in W_s}}{N_{j_1}}\mathbb{E}[\mathbbm{1}_{\sigma(i_2)=j_2,\cdots,\sigma(i_q)=j_q}|\mathcal{B}_n]\leq \cdots \leq \frac{\mathbbm{1}_{i_1,\cdots,i_q\in W_s}}{N_{j_1} N_{j_2}\cdots N_{j_q}}. 
\end{eqnarray}

We bound $N_i$ (recall (\ref{Eq1.1.9})) for any $i\in \mathcal{I}_{n,s}\cap\mathbb{N}^{*}$ as follows. As $b_j\geq j$ for every $j\in [n]$, we have
\begin{equation}\label{Eq1.3.1}
    N_i=1+\sum_{j=1}^{i-1}\mathbbm{1}_{b_j\geq i}. 
\end{equation}
Let $Z_j:=\mathbbm{1}_{b_j\geq i}$ for every $j\in [i-1]$. Note that conditional on $\sigma_0$, $Z_1,\cdots,Z_{i-1}$ are mutually independent, and for any $j\in [i-1]$, $Z_j$ follows the Bernoulli distribution with 
\begin{eqnarray*}
    \mathbb{P}(Z_j=1|\sigma_0)&=&\mathbb{P}(b_j\geq i|\sigma_0)=\mathbb{P}(u_j\leq e^{-2\beta_n(i-j)}|\sigma_0)\nonumber\\
    &=&e^{-2\beta_n (i-\max\{j,\sigma_0(j)\})_{+}}\geq e^{-2\beta_n(i-j)}. 
\end{eqnarray*}
Hence for any $j\in [i-\beta_n^{-1},i-1]\cap\mathbb{N}^{*}$ (note that $i-\beta_n^{-1}\geq \beta_n^{-1}\geq 1$), we have $\mathbb{P}(Z_j=1|\sigma_0)\geq e^{-2}$. For any $t\geq 0$, by Hoeffding's inequality, we have
\begin{eqnarray*}
  &&  \mathbb{P}(N_i\leq (e^{-2}-t)\lfloor \beta_n^{-1}\rfloor |\sigma_0)\nonumber\\
  &\leq& \mathbb{P}\Big(\sum_{j=i-\lfloor \beta_n^{-1}\rfloor}^{i-1} Z_j\leq  \sum_{j=i-\lfloor \beta_n^{-1}\rfloor}^{i-1}\mathbb{P}(Z_j=1|\sigma_0)-t\lfloor \beta_n^{-1}\rfloor \Big|\sigma_0 \Big)\leq\exp(-2\lfloor \beta_n^{-1}\rfloor t^2).
\end{eqnarray*}
Taking $t=e^{-2}\slash 2$, as $\lfloor \beta_n^{-1}\rfloor\geq \beta_n^{-1}-1\geq \beta_n^{-1}\slash 2$, we have
\begin{eqnarray}\label{Eq1.1.10}
   \mathbb{P}(N_i\leq e^{-2}\beta_n^{-1}\slash 4)&\leq&
    \mathbb{P}(N_i\leq e^{-2}\lfloor \beta_n^{-1}\rfloor\slash 2)=\mathbb{E}[\mathbb{P}(N_i\leq e^{-2}\lfloor \beta_n^{-1}\rfloor\slash 2|\sigma_0)]\nonumber\\
    &\leq& \exp(-c\lfloor \beta_n^{-1} \rfloor)  \leq \exp(-c\beta_n^{-1}).
\end{eqnarray}

Let $\mathcal{C}_s$ be the event that $N_i\geq e^{-2}\beta_n^{-1}\slash 4$ for every $i\in \mathcal{I}_{n,s}\cap\mathbb{N}^{*}$. By (\ref{Eq1.1.10}) and the union bound, we have
\begin{equation}\label{Eq1.1.12}
    \mathbb{P}(\mathcal{C}_s^c)\leq |\mathcal{I}_{n,s}\cap\mathbb{N}^{*}|\exp(-c\beta_n^{-1})\leq (2L\beta_n^{-1}+1) \exp(-c\beta_n^{-1})\leq CL\exp(-c\beta_n^{-1}).
\end{equation}
For any $q\in\mathbb{N}^{*}$, $i_1,\cdots,i_q\in [1,(s-1)L\beta_n^{-1}]\cap\mathbb{N}$, and $j_1,\cdots,j_q\in \mathcal{I}_{n,s}\cap\mathbb{N}^{*}$ such that $i_1<\cdots<i_q$ and $j_1<\cdots<j_q$, by (\ref{Eq1.1.11}), we have
\begin{eqnarray}
  &&  \mathbb{E}[\mathbbm{1}_{\sigma(i_1)=j_1,\cdots,\sigma(i_q)=j_q}\mathbbm{1}_{i_1,\cdots,i_q\in W_s}|\mathcal{B}_n]\mathbbm{1}_{\mathcal{C}_s\cap\mathcal{W}}\nonumber\\
  &\leq& \frac{\mathbbm{1}_{i_1,\cdots,i_q\in W_s}\mathbbm{1}_{\mathcal{C}_s\cap\mathcal{W}}}{N_{j_1} N_{j_2}\cdots N_{j_q}}\leq (4e^2\beta_n)^q \mathbbm{1}_{i_1,\cdots,i_q\in W_s}
  \mathbbm{1}_{\mathcal{C}_s\cap\mathcal{W}}.
\end{eqnarray}
Hence by (\ref{Eqq1.1}) and Lemma \ref{Lemma3.1}, for any $q\in\mathbb{N}^{*}$, we have
\begin{eqnarray*}
   && \mathbb{E}[\Lambda_{s,q}|\mathcal{B}_n]\mathbbm{1}_{\mathcal{C}_s\cap\mathcal{W}}\leq (4e^2\beta_n)^q \binom{|W_s|}{q}\binom{|\mathcal{I}_{n,s}\cap\mathbb{N}^{*}|}{q}\mathbbm{1}_{\mathcal{C}_s\cap\mathcal{W}}\nonumber\\
  &\leq& \Big(\frac{4e^4\beta_n|W_s||\mathcal{I}_{n,s}\cap\mathbb{N}^{*}|}{q^2}\Big)^q \mathbbm{1}_{\mathcal{C}_s\cap\mathcal{W}}\leq \Big(\frac{4e^4\beta_n (C_0\beta_n^{-1})(2L\beta_n^{-1}+1)}{q^2}\Big)^q \nonumber\\
  &\leq& (CL\beta_n^{-1} q^{-2})^q.
\end{eqnarray*}
Hence
\begin{equation}\label{Eq1.1.16}
    \mathbb{E}[\Lambda_{s,q} \mathbbm{1}_{\mathcal{C}_s\cap\mathcal{W}} ]=\mathbb{E}[\mathbb{E}[\Lambda_{s,q}|\mathcal{B}_n] \mathbbm{1}_{\mathcal{C}_s\cap\mathcal{W}}]\leq (CL\beta_n^{-1} q^{-2})^q.
\end{equation}

Now for any $q\in \mathbb{N}^{*}$, if $LIS(\sigma|_{\mathcal{R}_s'})\geq q$, then there exist 
\begin{equation*}
    i_1,\cdots,i_q\in [1,(s-1)L\beta_n^{-1}]\cap\mathbb{N}, \quad j_1,\cdots,j_q\in \mathcal{I}_{n,s}\cap\mathbb{N}^{*},
\end{equation*}
such that $i_1<\cdots<i_q$, $j_1<\cdots<j_q$, and $\sigma(i_l)=j_l$ for every $l\in [q]$. As $\sigma$ is sampled from the set (\ref{Eq1.1.14}), for every $l\in [q]$, we have
\begin{equation*}
  b_{i_l}\geq \sigma(i_l)=j_l\geq (s-1)L\beta_n^{-1}.
\end{equation*}
Hence $i_1,\cdots,i_q\in W_s$ (recall (\ref{Eq1.1.15})), and $\Lambda_{s,q}\geq 1$. Hence by (\ref{Eq1.1.16}), for any $q\in \mathbb{N}^{*}$, we have 
\begin{eqnarray}
   && \mathbb{P}(\{LIS(\sigma|_{\mathcal{R}_s'})\geq q\}\cap\mathcal{C}_s\cap\mathcal{W})\leq \mathbb{P}(\{\Lambda_{s,q}\geq 1\}\cap\mathcal{C}_s\cap\mathcal{W}) \nonumber\\
   &=& \mathbb{E}[\mathbbm{1}_{\Lambda_{s,q}\geq 1}\mathbbm{1}_{\mathcal{C}_s\cap\mathcal{W}}]\leq \mathbb{E}[\Lambda_{s,q}\mathbbm{1}_{\mathcal{C}_s\cap\mathcal{W}}]\leq (C_0'L\beta_n^{-1} q^{-2})^q,
\end{eqnarray}
where $C_0'\geq 1$ is a positive absolute constant. Taking $q=\lceil \sqrt{2C_0'} L^{1\slash 2}\beta_n^{-1\slash 2}\rceil$, we obtain that
\begin{eqnarray}\label{Eq1.1.17}
  &&  \mathbb{P}(\{LIS(\sigma|_{\mathcal{R}_s'})\geq 2\sqrt{2C_0'}L^{1\slash 2}\beta_n^{-1\slash 2}\}\cap \mathcal{C}_s\cap\mathcal{W}) \nonumber\\
  &\leq& \mathbb{P}(\{LIS(\sigma|_{\mathcal{R}_s'})\geq \lceil\sqrt{2C_0'}L^{1\slash 2}\beta_n^{-1\slash 2}\rceil \}\cap \mathcal{C}_s\cap\mathcal{W})\nonumber\\
  &\leq& 2^{-\lceil \sqrt{2C_0'} L^{1\slash 2}\beta_n^{-1\slash 2}\rceil}\leq \exp(-c L^{1\slash 2}\beta_n^{-1\slash 2}).
\end{eqnarray}
By (\ref{Eq1.1.18}), (\ref{Eq1.1.12}), (\ref{Eq1.1.17}), and the union bound, we have
\begin{eqnarray}
    \mathbb{P}(LIS(\sigma|_{\mathcal{R}_s'})\geq 2\sqrt{2C_0'}L^{1\slash 2}\beta_n^{-1\slash 2})&\leq& \exp(-c L^{1\slash 2}\beta_n^{-1\slash 2})+CL\exp(-c\beta_n^{-1})\nonumber\\
    &\leq& CL\exp(-c \beta_n^{-1\slash 2}).
\end{eqnarray}
Note that $LIS(\sigma|_{\mathcal{R}_s'})\leq |\mathcal{I}_{n,s}\cap\mathbb{N}^{*}|\leq 2L\beta_n^{-1}+1\leq 4L\beta_n^{-1}$. Hence 
\begin{eqnarray}
    \mathbb{E}[LIS(\sigma|_{\mathcal{R}_s'})]&\leq& (4L\beta_n^{-1}) \cdot (CL\exp(-c \beta_n^{-1\slash 2}))+2\sqrt{2C_0'}L^{1\slash 2}\beta_n^{-1\slash 2}\nonumber\\
    &\leq& CL^{1\slash 2}\beta_n^{-1\slash 2}+CL^2\exp(-c\beta_n^{-1\slash 2}).
\end{eqnarray}

In the following, we show (\ref{Eq1.2.2}). Let $\sigma$ be drawn from $\mathbb{P}_{n,\beta_n}$. Note that the distribution of $\sigma^{-1}$ is given by $\mathbb{P}_{n,\beta_n}$, and $LIS(\sigma^{-1}|_{\mathcal{R}_s'})=LIS(\sigma|_{\mathcal{R}_s''})$. Hence by (\ref{Eq1.2.1}),
\begin{equation}
    \mathbb{E}[LIS(\sigma|_{\mathcal{R}_s''})]=\mathbb{E}[LIS(\sigma^{-1}|_{\mathcal{R}_s'})]\leq CL^{1\slash 2}\beta_n^{-1\slash 2}+CL^2\exp(-c\beta_n^{-1\slash 2}).
\end{equation}

\end{proof}

The following proposition bounds $LIS(\sigma|_{\mathcal{R}_s})$ for $\sigma$ drawn from $\mathbb{P}_{n,\beta_n}$ and any $s\in [\lfloor n\beta_n\slash L \rfloor ]$.

\begin{proposition}\label{P4.2}
Assume that $n\beta_n\geq 4 L$ and $\beta_n\leq 1\slash 10$, and let $\sigma$ be drawn from $\mathbb{P}_{n,\beta_n}$. Then there exist positive absolute constants $C,c$, such that for any $s\in [\lfloor n\beta_n\slash L\rfloor]$, 
\begin{equation}\label{Eq1.3.5}
    \mathbb{E}[LIS(\sigma|_{\mathcal{R}_s})]\leq CL\beta_n^{-1\slash 2}+CL^2\exp(-c\beta_n^{-1\slash 2}).
\end{equation}
\end{proposition}

\begin{proof}

We start by showing (\ref{Eq1.3.5}) for any $s\in [2,\lfloor n\beta_n\slash L\rfloor]\cap\mathbb{N}$. We sample $\sigma_0$ from $\mathbb{P}_{n,\beta_n}$, and then run one step of the hit and run algorithm for the $L^1$ model to obtain $\sigma$ following the procedure described at the beginning of the proof of Proposition \ref{P4.1} (with $\{b_i\}_{i=1}^n$, $\{N_i\}_{i=1}^n$, and $\{Y_i\}_{i=1}^n$ defined as there).

Below we fix an arbitrary $s\in [2,\lfloor n\beta_n\slash L\rfloor]\cap\mathbb{N}$. Let $\mathcal{C}_s$ be the event that $N_i\geq  e^{-2}\beta_n^{-1}\slash 4$ for every $i\in\mathcal{I}_{n,s}\cap\mathbb{N}^{*}$. We recall from (\ref{Eq1.1.12}) that
\begin{equation}\label{Eq1.3.3}
    \mathbb{P}(\mathcal{C}_s^c)\leq CL\exp(-c\beta_n^{-1}).
\end{equation}


Let $\mathcal{B}_n$ be the $\sigma$-algebra generated by $\sigma_0$ and $\{b_i\}_{i=1}^n$. For any $l\in [n-1]$, let $\mathcal{B}_l$ be the $\sigma$-algebra generated by $\sigma_0$, $\{b_i\}_{i=1}^n$, and $\{Y_i\}_{i=l+1}^n$. For any $q\in \mathbb{N}^{*}$, let 
\begin{equation}
    \Lambda_{s,q}:= \sum_{\substack{i_1<\cdots<i_q,j_1<\cdots<j_q\\i_1,\cdots,i_q\in \mathcal{I}_{n,s} \cap\mathbb{N}^{*}\\j_1,\cdots,j_q\in \mathcal{I}_{n,s }\cap\mathbb{N}^{*}}}\mathbbm{1}_{\sigma(i_1)=j_1,\cdots,\sigma(i_q)=j_q}.
\end{equation}
For any $i_1,\cdots,i_q,j_1,\cdots,j_q\in \mathcal{I}_{n,s} \cap\mathbb{N}^{*}$ such that $i_1<\cdots<i_q$ and $j_1<\cdots<j_q$, 
\begin{eqnarray}
&& \mathbb{E}[\mathbbm{1}_{\sigma(i_1)=j_1,\cdots,\sigma(i_q)=j_q}|\mathcal{B}_n] = \mathbb{E}[\mathbb{E}[\mathbbm{1}_{\sigma(i_1)=j_1}|\mathcal{B}_{j_1}] \mathbbm{1}_{\sigma(i_2)=j_2,\cdots,\sigma(i_q)=j_q}|\mathcal{B}_n] \nonumber\\
&&  \leq \frac{1}{N_{j_1}} \mathbb{E}[\mathbbm{1}_{\sigma(i_2)=j_2,\cdots,\sigma(i_q)=j_q}|\mathcal{B}_n]\leq \cdots \leq \frac{1}{N_{j_1}N_{j_2} \cdots N_{j_q}},
\end{eqnarray}
which leads to
\begin{equation}
\mathbb{E}[\mathbbm{1}_{\sigma(i_1)=j_1,\cdots,\sigma(i_q)=j_q}|\mathcal{B}_n]\mathbbm{1}_{\mathcal{C}_s}\leq \frac{1}{N_{j_1}N_{j_2} \cdots N_{j_q}} \mathbbm{1}_{\mathcal{C}_s}\leq (4e^2\beta_n)^q.
\end{equation}
Hence by Lemma \ref{Lemma3.1}, we have
\begin{eqnarray}
\mathbb{E}[\Lambda_{s,q}|\mathcal{B}_n]\mathbbm{1}_{\mathcal{C}_s}&\leq& (4e^2\beta_n)^q \binom{|\mathcal{I}_{n,s }\cap\mathbb{N}^{*}|}{q}^2\leq \Big(\frac{4e^4\beta_n |\mathcal{I}_{n,s }\cap\mathbb{N}^{*}|^2}{q^2}\Big)^q \nonumber\\
&\leq& \Big(\frac{4e^4\beta_n(2L\beta_n^{-1}+1)^2}{q^2}\Big)^q \leq \Big(\frac{36 e^4  L^2\beta_n^{-1}}{q^2}\Big)^q.
\end{eqnarray}
Hence
\begin{eqnarray}
         \mathbb{P}(\{LIS(\sigma|_{\mathcal{R}_s})\geq q\}\cap\mathcal{C}_s)&\leq& \mathbb{P}(\{\Lambda_{s,q}\geq 1\}\cap\mathcal{C}_s)=\mathbb{E}[\mathbb{E}[\mathbbm{1}_{\Lambda_{s,q}\geq 1}|\mathcal{B}_n]\mathbbm{1}_{\mathcal{C}_s}] \nonumber\\
  &\leq& \mathbb{E}[\mathbb{E}[\Lambda_{s,q}|\mathcal{B}_n]\mathbbm{1}_{\mathcal{C}_s}]\leq  \Big(\frac{36 e^4  L^2\beta_n^{-1}}{q^2}\Big)^q.
\end{eqnarray}
Taking $q=\lceil 12e^2L\beta_n^{-1\slash 2} \rceil$, we obtain that
\begin{eqnarray}\label{Eq1.3.4}
  \mathbb{P}(\{LIS(\sigma|_{\mathcal{R}_s})\geq 24 e^2 L\beta_n^{-1\slash 2}\}\cap\mathcal{C}_s)&\leq& \mathbb{P}(\{LIS(\sigma|_{\mathcal{R}_s})\geq \lceil 12 e^2 L\beta_n^{-1\slash 2}\rceil\}\cap\mathcal{C}_s) \nonumber\\
  &\leq &  2^{-q}\leq \exp(-c L \beta_n^{-1\slash 2}).
\end{eqnarray}

By (\ref{Eq1.3.3}), (\ref{Eq1.3.4}), and the union bound, we have
\begin{eqnarray}
      \mathbb{P}(LIS(\sigma|_{\mathcal{R}_{s}})\geq 24 e^2 L\beta_n^{-1\slash 2})&\leq& CL\exp(-c\beta_n^{-1})+\exp(-c L \beta_n^{-1\slash 2})\nonumber\\
   &\leq& CL \exp(-c \beta_n^{-1\slash 2}).    
\end{eqnarray}
Note that $LIS(\sigma|_{\mathcal{R}_{s}})\leq |\mathcal{I}_{n,s} \cap \mathbb{N}^{*}|\leq 2L\beta_n^{-1}+1\leq 3L\beta_n^{-1}$. Hence
\begin{eqnarray}\label{Eq1.3.6}
    \mathbb{E}[LIS(\sigma|_{\mathcal{R}_{s}})]&\leq& (3L\beta_n^{-1})(CL \exp(-c \beta_n^{-1\slash 2}))+24e^2 L \beta_n^{-1\slash 2}\nonumber\\
    &\leq& CL\beta_n^{-1\slash 2}+CL^2\exp(-c\beta_n^{-1\slash 2}).
\end{eqnarray}

In the following, we show (\ref{Eq1.3.5}) for $s=1$. Let $\sigma$ be drawn from $\mathbb{P}_{n,\beta_n}$, and let $\bar{\sigma}\in S_n$ be such that $\bar{\sigma}(i)=n+1-\sigma(n+1-i)$ for every $i\in [n]$. Arguing as in (\ref{Eq1.3.7}), we obtain that the distribution of $\bar{\sigma}$ is given by $\mathbb{P}_{n,\beta_n}$. As
\begin{equation*}
    (\lfloor n\beta_n\slash L\rfloor-1)L\beta_n^{-1}\leq n-L\beta_n^{-1}<n+1-L\beta_n^{-1},
\end{equation*}
we have
\begin{eqnarray*}
&&[n+1-L\beta_n^{-1},n]\times [n+1-L\beta_n^{-1},n] \nonumber\\
&\subseteq& ((\lfloor n\beta_n\slash L\rfloor -1)L\beta_n^{-1},n]\times ((\lfloor n\beta_n\slash L\rfloor -1)L\beta_n^{-1},n]=\mathcal{R}_{\lfloor n\beta_n \slash L \rfloor}.
\end{eqnarray*}
Hence
\begin{equation}\label{Eq1.3.8}
    LIS(\sigma|_{\mathcal{R}_1})= LIS(\bar{\sigma}|_{[n+1-L\beta_n^{-1},n]^2})\leq LIS(\bar{\sigma}|_{\mathcal{R}_{\lfloor n\beta_n \slash L \rfloor}}).
\end{equation}
By (\ref{Eq1.3.6}) (with $s=\lfloor n\beta_n\slash L\rfloor$) and (\ref{Eq1.3.8}), we have
\begin{equation}
    \mathbb{E}[LIS(\sigma|_{\mathcal{R}_1})]\leq \mathbb{E}[LIS(\bar{\sigma}|_{\mathcal{R}_{\lfloor n\beta_n \slash L \rfloor}})]\leq CL\beta_n^{-1\slash 2}+CL^2\exp(-c\beta_n^{-1\slash 2}).
\end{equation}

\end{proof}

The following proposition gives a more precise bound on $LIS(\sigma|_{\mathcal{R}_s})$ for $\sigma$ drawn from $\mathbb{P}_{n,\beta_n}$ and $s\in[2,\lfloor n\beta_n\slash L\rfloor-1]\cap\mathbb{N}$ that satisfies certain conditions.

\begin{proposition}\label{P4.3}
Let $C_1$ be the constant that appears in Proposition \ref{P2.3} (with $\delta_0=1\slash 4$ and $K=2L$; note that $C_1$ only depends on $L$). Let 
\begin{equation}\label{Eq1.11.1}
    r_s:=\frac{1}{2} \min \{(s-1)L, (\lfloor n\beta_n\slash L\rfloor-s)L,\log(1+\beta_n^{-1})\}
\end{equation}
for any $s\in [2,\lfloor n\beta_n\slash L\rfloor-1]\cap\mathbb{N}$. There exist positive constants $C_L,c_L,C_L',c_L'$ that only depend on $L$ and positive absolute constants $C,C'$ with $C'\geq 1$, such that the following holds. 

Assume that $n\beta_n\geq 4L$ and $\beta_n^{-1}\geq C'L^{10} e^{6L}$, and let $\sigma$ be drawn from $\mathbb{P}_{n,\beta_n}$. Let $\Psi_s:=(1-L^{-1}-C_L r_s^{-1\slash 10})_{+}$ for any $s\in [2,\lfloor n\beta_n\slash L\rfloor-1]\cap\mathbb{N}$. Then for any $s\in [2,\lfloor n\beta_n\slash L\rfloor-1]\cap\mathbb{N}$ such that $r_s\geq C_1$, we have
\begin{eqnarray}
&& \mathbb{E}[|  LIS(\sigma|_{\mathcal{R}_s})-\sqrt{2}L\beta_n^{-1\slash 2}|] \nonumber\\
&\leq & C_L'\beta_n^{-1}\exp(-c_L' \beta_n^{-1\slash 4} \Psi_s^{1\slash 4})+C_L'+C L^{1\slash 2}e^{-L}\beta_n^{-1\slash 2}\nonumber\\
&& +\sqrt{2}L\beta_n^{-1\slash 2}\max\{e^{4L^{-1}} (1+C_L r_s^{-1\slash 10})^{1\slash 2}(1+\max\{c_L\beta_n^{-1}\Psi_s,1\}^{-1\slash 6})-1,\nonumber\\
&&\quad\quad\quad\quad\quad\quad\quad\quad 1-e^{-6L^{-1}}\Psi_s^{1\slash 2} (1-\max\{c_L\beta_n^{-1}\Psi_s,1\}^{-1\slash 6})\}.
\end{eqnarray}

\end{proposition}

\begin{proof}

Let $C_0,c_0,C_1,C_2$ be the constants that appear in Proposition \ref{P2.3} (with $\delta_0=1\slash 4$ and $K=2L$). Note that these constants only depend on $L$. Throughout the proof, we fix an arbitrary $s\in [2,\lfloor n\beta_n\slash L\rfloor-1]\cap\mathbb{N}$ such that $r_s\geq C_1$. We also assume that $n\beta_n\geq 4L$ and $\beta_n^{-1}\geq L^{10}$. 

In the following, we fix any $T,K_0\in \mathbb{N}^{*}$ such that $\min\{T,K_0\}\geq L^2$, any refined path $\Gamma\in \Pi^{T,T,K_0}$, and any $l\in [2T-1]$. We assume that 
\begin{equation}\label{Eq1.6.1}
     L\beta_n^{-1} \geq \max\{8 K_0 T,K_0^2 T^3\}.
\end{equation}

We let
\begin{eqnarray*}
   && Q_{\Gamma,l}:=(x_{l-1}(\Gamma),x_l(\Gamma)]\times ( y_{l-1}(\Gamma), y_l(\Gamma)],\\
   &&Q_{\Gamma,l}':=[ a_{l-1}(\Gamma),   c_l(\Gamma)]\times [ b_{l-1}(\Gamma), d_l(\Gamma)].
\end{eqnarray*}

We also let
\begin{eqnarray*}
 \tilde{Q}_{\Gamma,l} &:=&  ((s-1)L\beta_n^{-1},(s-1)L\beta_n^{-1})+L\beta_n^{-1}Q_{\Gamma,l}\\
   &=& ((s-1)L\beta_n^{-1}+L\beta_n^{-1}x_{l-1}(\Gamma),(s-1)L\beta_n^{-1}+L\beta_n^{-1}x_{l}(\Gamma)]\\
   &&\times ((s-1)L\beta_n^{-1}+L\beta_n^{-1}y_{l-1}(\Gamma),(s-1)L\beta_n^{-1}+L\beta_n^{-1}y_{l}(\Gamma)],
\end{eqnarray*}
\begin{equation*}
    \tilde{Q}_{\Gamma,l}':=((s-1)L\beta_n^{-1},(s-1)L\beta_n^{-1})+L\beta_n^{-1}Q_{\Gamma,l}'.
\end{equation*}

\paragraph{Step 1}

We start by bounding $LIS(\sigma|_{\tilde{Q}_{\Gamma,l}})$. If
\begin{equation*}
    x_{l-1}(\Gamma)=x_l(\Gamma) \text{ or } y_{l-1}(\Gamma)=y_l(\Gamma),
\end{equation*}
then $Q_{\Gamma,l}=\emptyset$ and $LIS(\tau|_{\tilde{Q}_{\Gamma,l}})=0$ for any $\tau\in S_n$. In the following, we assume that $x_{l-1}(\Gamma)<x_l(\Gamma)$ and $y_{l-1}(\Gamma)<y_l(\Gamma)$. Note that 
\begin{equation}\label{Eq1.5.5}
    (2 K_0 T)^{-1}\leq x_l(\Gamma)-x_{l-1}(\Gamma)\leq T^{-1},\quad  (2 K_0 T)^{-1}\leq y_l(\Gamma)-y_{l-1}(\Gamma)\leq T^{-1},
\end{equation}
which implies
\begin{equation}\label{Eq1.8.5}
   \min\{L\beta_n^{-1}(x_l(\Gamma)-x_{l-1}(\Gamma)),L\beta_n^{-1}(y_l(\Gamma)-y_{l-1}(\Gamma))\}\geq \frac{L\beta_n^{-1}}{2K_0 T}\geq  4.
\end{equation}
In the following, we assume that
\begin{eqnarray}\label{Eq1.5.6}
 && ((s-1)L\beta_n^{-1}+L\beta_n^{-1} x_{l-1}(\Gamma), (s-1)L\beta_n^{-1}+L\beta_n^{-1} x_l(\Gamma)]\cap \mathbb{N}^{*}\nonumber\\
 &&=\{s_1,s_1+1,\cdots,s_2\},\nonumber\\
 && ((s-1)L\beta_n^{-1}+L\beta_n^{-1} y_{l-1}(\Gamma),(s-1)L\beta_n^{-1}+L\beta_n^{-1} y_l(\Gamma)]\cap\mathbb{N}^{*}\nonumber\\
 &&=\{s_1',s_1'+1,\cdots,s_2'\}.
\end{eqnarray}

We sample $\sigma_0$ from $\mathbb{P}_{n,\beta_n}$, and then run one step of the hit and run algorithm for the $L^1$ model to obtain $\sigma$ following the procedure described at the beginning of the proof of Proposition \ref{P4.1} (with $\{b_i\}_{i=1}^n$, $\{N_i\}_{i=1}^n$, and $\{Y_i\}_{i=1}^n$ defined as there). Let $\mathcal{C}_s$ be the event that $N_i\geq e^{-2}\beta_n^{-1}\slash 4$ for every $i\in \mathcal{I}_{n,s}\cap\mathbb{N}^{*}$. We recall from (\ref{Eq1.1.12}) that 
\begin{equation}\label{Eq1.5.2}
    \mathbb{P}(\mathcal{C}_s^c)\leq CL\exp(-c\beta_n^{-1}). 
\end{equation}

We let
\begin{eqnarray}\label{Eq2.6.1}
 && \mathcal{S}_{1,l}:=\{i\in \{s_1,\cdots,s_2\}\backslash \{Y_{s_2'+1},\cdots,Y_n\}:b_i> s_2'\}, \nonumber\\
 && \mathcal{S}_{2,l}:=\{i\in \{s_1,\cdots,s_2\}\backslash \{Y_{s_2'+1},\cdots,Y_n\}:s_1'\leq b_i\leq s_2'\},\nonumber\\
 && \mathcal{S}_l':=\{i\in \{s_1,\cdots,s_2\}: s_1'\leq   b_i \leq s_2' \},\quad W_l:=|\mathcal{S}_l'|.
\end{eqnarray}
Note that $\mathcal{S}_{2,l}\subseteq \mathcal{S}_l'$. We also let
\begin{eqnarray}\label{Eq1.4.8}
 && D_l:=|\{i\in [n]: (i,\sigma(i))\in  \tilde{Q}_{\Gamma,l}\}|, \nonumber\\
 &&  D_l':=|\{i\in [n]: (i,\sigma(i))\in  \tilde{Q}_{\Gamma,l}, i\in \mathcal{S}_{2,l}\}|.
\end{eqnarray}
Let
\begin{eqnarray}\label{Eq1.4.9}
  L_{1,l}:=LIS(\sigma|_{\mathcal{S}_{1,l}\times ((s-1)L\beta_n^{-1}+L\beta_n^{-1} y_{l-1}(\Gamma),(s-1)L\beta_n^{-1}+L\beta_n^{-1} y_l(\Gamma)]}), \nonumber\\
  L_{2,l}:=LIS(\sigma|_{\mathcal{S}_{2,l}\times ((s-1)L\beta_n^{-1}+L\beta_n^{-1} y_{l-1}(\Gamma),(s-1)L\beta_n^{-1}+L\beta_n^{-1} y_l(\Gamma)]}).
\end{eqnarray}
Following the argument below (\ref{Eq2}), we can deduce that
\begin{equation}\label{Eq2.12.1}
    L_{1,l}\leq LIS(\sigma|_{\tilde{Q}_{\Gamma,l}}) \leq L_{1,l}+L_{2,l}.
\end{equation}

We bound $W_l$ as follows. Note that
\begin{equation*}
    W_l=\sum_{i=s_1}^{s_2} \mathbbm{1}_{s_1'\leq b_i\leq s_2'}.
\end{equation*}
For any $i\in\{s_1,\cdots,s_2\}$,
\begin{eqnarray*}
&& \mathbb{P}(s_1'\leq b_i\leq s_2'|\sigma_0)=e^{-2\beta_n(s_1'-\max\{i,\sigma_0(i)\})_{+}}-e^{-2\beta_n(s_2'-\max\{i,\sigma_0(i)\})_{+}}\\
&\leq& 1-e^{-2\beta_n((s_2'-\max\{i,\sigma_0(i)\})_{+}-(s_1'-\max\{i,\sigma_0(i)\})_{+})}\leq 1-e^{-2\beta_n(s_2'-s_1')}\\
&\leq& 2\beta_n(s_2'-s_1')\leq 2L(y_l(\Gamma)-y_{l-1}(\Gamma)).
\end{eqnarray*}
For any $i\in\{s_1,\cdots,s_2\}$, let $Z_i:=\mathbbm{1}_{s_1'\leq b_i\leq s_2'}$. Conditional on $\sigma_0$, $Z_{s_1},\cdots,Z_{s_2}$ are mutually independent, and for every $i\in\{s_1,\cdots,s_2\}$, $Z_i$ follows the Bernoulli distribution with parameter $\mathbb{P}(s_1'\leq b_i\leq s_2'|\sigma_0)$. Hence by Hoeffding's inequality, for any $t\geq 0$, we have
\begin{equation*}
    \mathbb{P}(W_l\geq (s_2-s_1+1)(2L (y_l(\Gamma)-y_{l-1}(\Gamma))+t)|\sigma_0)\leq e^{-2(s_2-s_1+1)t^2}.
\end{equation*}
Taking $t=2L(y_l(\Gamma)-y_{l-1}(\Gamma))$, we obtain that 
\begin{eqnarray}
  &&  \mathbb{P}(W_l\geq 4L(s_2-s_1+1)(y_l(\Gamma)-y_{l-1}(\Gamma)))\nonumber\\
  &=& \mathbb{E}[\mathbb{P}(W_l\geq 4L(s_2-s_1+1)(y_l(\Gamma)-y_{l-1}(\Gamma))|\sigma_0)] \nonumber\\
  &\leq& e^{-8L^2(s_2-s_1+1)(y_l(\Gamma)-y_{l-1}(\Gamma))^2}.
\end{eqnarray}
Let $\mathcal{E}_l$ be the event that $W_l\leq 4L(s_2-s_1+1)(y_l(\Gamma)-y_{l-1}(\Gamma))$. We have
\begin{equation}\label{Eq1.5.3}
    \mathbb{P}(\mathcal{E}_l^c)\leq e^{-8L^2(s_2-s_1+1)(y_l(\Gamma)-y_{l-1}(\Gamma))^2}.
\end{equation}

In the following, we bound $D_l'$, $L_{2,l}$, $D_l$, $L_{1,l}$ (as defined in (\ref{Eq1.4.8}) and (\ref{Eq1.4.9})) in \textbf{Sub-steps 1.1-1.4}, respectively.

\subparagraph{Sub-step 1.1}

In this sub-step, we bound $D_l'$. Let $\mathcal{B}_l$ be the $\sigma$-algebra generated by $\sigma_0$, $\{b_i\}_{i=1}^n$, and $\{Y_i\}_{i=s_2'+1}^n$. Following the argument between (\ref{Eq4}) and (\ref{Eq9}), we obtain that for any $t\geq 0$, \begin{equation}\label{Eq1.5.1}
    \mathbb{P}\Big(D_l'\geq \sum_{i=s_1'}^{s_2'}\frac{W_l}{N_i}+(s_2'-s_1'+1)t\Big|\mathcal{B}_l\Big)\leq e^{-2(s_2'-s_1'+1)t^2}.
\end{equation}

Let $\mathcal{D}_l$ be the event that 
\begin{equation}\label{Eq2.7.1}
    D_l'\geq 32e^2L\beta_n(s_2-s_1+1)(s_2'-s_1'+1)(y_l(\Gamma)-y_{l-1}(\Gamma)).
\end{equation}
Taking $t=L\beta_n(s_2-s_1+1)(y_l(\Gamma)-y_{l-1}(\Gamma))$ in (\ref{Eq1.5.1}) and noting the definitions of $\mathcal{C}_s$ and $\mathcal{E}_l$, we obtain that
\begin{equation*}
    \mathbb{P}(\mathcal{D}_l\cap\mathcal{C}_s\cap\mathcal{E}_l|\mathcal{B}_l)\leq e^{-2L^2\beta_n^2(s_2-s_1+1)^2(y_l(\Gamma)-y_{l-1}(\Gamma))^2(s_2'-s_1'+1)}. 
\end{equation*}
Hence
\begin{eqnarray}\label{Eq1.5.4}
&& \mathbb{P}(\mathcal{D}_l\cap\mathcal{C}_s\cap\mathcal{E}_l)=\mathbb{E}[\mathbb{P}(\mathcal{D}_l\cap\mathcal{C}_s\cap\mathcal{E}_l|\mathcal{B}_l)]\nonumber\\
&\leq& e^{-2L^2\beta_n^2(s_2-s_1+1)^2(y_l(\Gamma)-y_{l-1}(\Gamma))^2(s_2'-s_1'+1)}.
\end{eqnarray}

By (\ref{Eq1.5.2}), (\ref{Eq1.5.3}), (\ref{Eq1.5.4}), and the union bound, we have
\begin{eqnarray*}
 \mathbb{P}(\mathcal{D}_l)&\leq& e^{-2L^2\beta_n^2(s_2-s_1+1)^2(y_l(\Gamma)-y_{l-1}(\Gamma))^2(s_2'-s_1'+1)}+CL\exp(-c\beta_n^{-1})\nonumber\\
 &&+e^{-8L^2(s_2-s_1+1)(y_l(\Gamma)-y_{l-1}(\Gamma))^2}.
\end{eqnarray*}
By (\ref{Eq1.6.1}), (\ref{Eq1.5.5}), and (\ref{Eq1.5.6}), we have
\begin{equation}\label{Eq1.8.1}
    s_2-s_1\geq L\beta_n^{-1}(x_l(\Gamma)-x_{l-1}(\Gamma))-2\geq  \frac{L\beta_n^{-1}}{2K_0T}-2\geq \frac{L\beta_n^{-1}}{4K_0T},
\end{equation}
\begin{equation}\label{Eq1.8.2}
    s_2'-s_1'\geq L\beta_n^{-1}(y_l(\Gamma)-y_{l-1}(\Gamma))-2\geq  \frac{L\beta_n^{-1}}{2K_0T}-2\geq \frac{L\beta_n^{-1}}{4K_0T}. 
\end{equation}
Hence
\begin{equation}\label{Eq2.11.1}
    \mathbb{P}(\mathcal{D}_l)\leq CL\exp(-c\beta_n^{-1})+2\exp(-L^3 \beta_n^{-1}\slash (128K_0^5 T^5)).
\end{equation}

\subparagraph{Sub-step 1.2}

In this sub-step, we bound $L_{2,l}$. For any $q\in\mathbb{N}^{*}$, we define 
\begin{equation}\label{Eq1.7.1}
    \Lambda_{l,q}:=\sum_{\substack{i_1<\cdots<i_q,j_1<\cdots<j_q\\ i_1,\cdots,i_q\in \{s_1,\cdots,s_2\}\\j_1,\cdots,j_q\in \{s_1',\cdots,s_2'\}}} \mathbbm{1}_{\sigma(i_1)=j_1,\cdots,\sigma(i_q)=j_q} \mathbbm{1}_{i_1,\cdots,i_q\in\mathcal{S}_{2,l}}.
\end{equation}
For any $k\in [n]$, we let $\mathcal{F}_k$ be the $\sigma$-algebra generated by $\sigma_0$, $\{b_i\}_{i=1}^n$, and $\{Y_i\}_{i=k+1}^n$. For any $i_1,\cdots,i_q\in\{s_1,\cdots,s_2\}$ and $j_1,\cdots,j_q\in \{s_1',\cdots,s_2'\}$ such that $i_1<\cdots<i_q$ and $j_1<\cdots<j_q$, we have
\begin{eqnarray*}
 &&\mathbb{E}[\mathbbm{1}_{\sigma(i_1)=j_1,\cdots,\sigma(i_q)=j_q}\mathbbm{1}_{i_1,\cdots,i_q\in\mathcal{S}_{2,l}}|\mathcal{B}_l]\nonumber\\
 &=& \mathbb{E}[\mathbbm{1}_{\sigma(i_1)=j_1,\cdots,\sigma(i_q)=j_q}|\mathcal{B}_l]\mathbbm{1}_{i_1,\cdots,i_q\in\mathcal{S}_{2,l}} \nonumber\\
 &=& \mathbb{E}[\mathbb{E}[\mathbbm{1}_{\sigma(i_1)=j_1}|\mathcal{F}_{j_1}]\mathbbm{1}_{\sigma(i_2)=j_2,\cdots,\sigma(i_q)=j_q}|\mathcal{B}_l]\mathbbm{1}_{i_1,\cdots,i_q\in\mathcal{S}_{2,l}}\\
  &\leq& \frac{\mathbbm{1}_{i_1,\cdots,i_q\in\mathcal{S}_{2,l}}}{N_{j_1}}\mathbb{E}[\mathbbm{1}_{\sigma(i_2)=j_2,\cdots,\sigma(i_q)=j_q}|\mathcal{B}_l]\leq \cdots\leq \frac{\mathbbm{1}_{i_1,\cdots,i_q\in\mathcal{S}_{2,l}}}{N_{j_1}N_{j_2}\cdots N_{j_q}}. 
\end{eqnarray*}
Recalling the definition of $\mathcal{C}_s$, we obtain that
\begin{equation}\label{Eq1.7.2}
    \mathbb{E}[\mathbbm{1}_{\sigma(i_1)=j_1,\cdots,\sigma(i_q)=j_q} \mathbbm{1}_{i_1,\cdots,i_q\in\mathcal{S}_{2,l}} |\mathcal{B}_l]\mathbbm{1}_{\mathcal{C}_s\cap\mathcal{E}_l}\leq (4e^2\beta_n)^q \mathbbm{1}_{\mathcal{C}_s\cap\mathcal{E}_l}\mathbbm{1}_{i_1,\cdots,i_q\in\mathcal{S}_{2,l}}.
\end{equation}
By (\ref{Eq1.7.1}), (\ref{Eq1.7.2}), and Lemma \ref{Lemma3.1}, recalling the definition of $\mathcal{E}_l$, we have
\begin{eqnarray*}
&& \mathbb{E}[\Lambda_{l,q}|\mathcal{B}_l] \mathbbm{1}_{\mathcal{C}_s\cap\mathcal{E}_l}\leq (4e^2\beta_n)^q\binom{|\mathcal{S}_{2,l}|}{q}\binom{s_2'-s_1'+1}{q} \mathbbm{1}_{\mathcal{C}_s\cap\mathcal{E}_l} \nonumber\\
&\leq& \Big(\frac{4e^4 \beta_n |\mathcal{S}_{2,l}|(s_2'-s_1'+1)}{q^2}\Big)^q\mathbbm{1}_{\mathcal{C}_s\cap\mathcal{E}_l}\leq \Big(\frac{4e^4 \beta_n W_l(s_2'-s_1'+1)}{q^2}\Big)^q\mathbbm{1}_{\mathcal{C}_s\cap\mathcal{E}_l}\nonumber\\
&\leq& \Big(\frac{16e^4 L\beta_n (s_2-s_1+1)(s_2'-s_1'+1)(y_l(\Gamma)-y_{l-1}(\Gamma))}{q^2}\Big)^q.
\end{eqnarray*}
Hence
\begin{eqnarray}\label{Eq1.7.4}
&& \mathbb{P}(\{\Lambda_{l,q}\geq 1\}\cap\mathcal{C}_s\cap\mathcal{E}_l) =\mathbb{E}[\mathbb{E}[\mathbbm{1}_{\Lambda_{l,q}\geq 1}|\mathcal{B}_l]\mathbbm{1}_{\mathcal{C}_s\cap\mathcal{E}_l}]\leq \mathbb{E}[\mathbb{E}[\Lambda_{l,q}|\mathcal{B}_l]\mathbbm{1}_{\mathcal{C}_s\cap\mathcal{E}_l}] \nonumber\\
&&\leq  \Big(\frac{16e^4 L\beta_n (s_2-s_1+1)(s_2'-s_1'+1)(y_l(\Gamma)-y_{l-1}(\Gamma))}{q^2}\Big)^q.
\end{eqnarray}
Let
\begin{equation}
    q_0:= 8e^2 L^{1\slash 2}\beta_n^{1\slash 2} (s_2-s_1+1)^{1\slash 2}(s_2'-s_1'+1)^{1\slash 2}(y_l(\Gamma)-y_{l-1}(\Gamma))^{1\slash 2}.
\end{equation}
Taking $q=\lceil q_0 \rceil$ in (\ref{Eq1.7.4}), we obtain that 
\begin{equation*}
    \mathbb{P}(\{\Lambda_{l,\lceil q_0\rceil}\geq 1\}\cap\mathcal{C}_s\cap\mathcal{E}_l)\leq 2^{-q_0},
\end{equation*}
which leads to
\begin{equation}\label{Eq1.9.1}
    \mathbb{P}(\{L_{2,l}\geq q_0+1\}\cap\mathcal{C}_s\cap\mathcal{E}_l)\leq 2^{-q_0}.
\end{equation}

By (\ref{Eq1.6.1}) and (\ref{Eq1.5.5}), we have
\begin{eqnarray*}
   && L^{3\slash 2}\beta_n^{-1\slash 2} T^{-1\slash 2} \sqrt{(x_l(\Gamma)-x_{l-1}(\Gamma))(y_l(\Gamma)-y_{l-1}(\Gamma))}\nonumber\\
   &&\geq\frac{1}{2}L^{3\slash 2}\beta_n^{-1\slash 2} T^{-3\slash 2}K_0^{-1}\geq \frac{1}{2}L\geq 1.
\end{eqnarray*}

Hence by (\ref{Eq1.5.5}), (\ref{Eq1.8.5}), (\ref{Eq1.8.1}), (\ref{Eq1.8.2}), and the AM-GM inequality, we have
\begin{equation}\label{Eq1.8.6}
   q_0 \geq 8e^2 L^{1\slash 2}\beta_n^{1\slash 2} \cdot \frac{L \beta_n^{-1}}{4K_0 T}\cdot (2K_0T)^{-1\slash 2}\geq L^{3\slash 2}K_0^{-3\slash 2}T^{-3\slash 2}\beta_n^{-1\slash 2},
\end{equation}
\begin{eqnarray}\label{Eq1.8.7}
    q_0+1&\leq& 20e^2 L^{3\slash 2}\beta_n^{-1\slash 2} T^{-1\slash 2} \sqrt{(x_l(\Gamma)-x_{l-1}(\Gamma))(y_l(\Gamma)-y_{l-1}(\Gamma))}\nonumber\\
    &\leq& 20e^2 L^{3\slash 2}\beta_n^{-1\slash 2}  T^{-1\slash 2}(x_l(\Gamma)-x_{l-1}(\Gamma)+y_l(\Gamma)-y_{l-1}(\Gamma)).
\end{eqnarray}
Let $\mathscr{E}_l$ be the event that
\begin{equation}\label{Eq2.11.4}
    L_{2,l}\leq 20e^2 L^{3\slash 2}\beta_n^{-1\slash 2}  T^{-1\slash 2}(x_l(\Gamma)-x_{l-1}(\Gamma)+y_l(\Gamma)-y_{l-1}(\Gamma)).
\end{equation}
By (\ref{Eq1.9.1})-(\ref{Eq1.8.7}), we have
\begin{equation}\label{Eq1.9.2}
    \mathbb{P}(\mathscr{E}_l^c\cap\mathcal{C}_s\cap\mathcal{E}_l)\leq \exp(-cL^{3\slash 2}\beta_n^{-1\slash 2}\slash (K_0^{3\slash 2} T^{3\slash 2} )).
\end{equation}
By (\ref{Eq1.5.2}), (\ref{Eq1.5.3}), (\ref{Eq1.9.2}), and the union bound, we have
\begin{eqnarray}
  \mathbb{P}(\mathscr{E}_l^c)&\leq& \exp(-cL^{3\slash 2}\beta_n^{-1\slash 2}\slash (K_0^{3\slash 2} T^{3\slash 2} ))+CL\exp(-c\beta_n^{-1})\nonumber\\
  &&+\exp(-8L^2(s_2-s_1+1)(y_l(\Gamma)-y_{l-1}(\Gamma))^2).
\end{eqnarray}
Noting (\ref{Eq1.5.5}) and (\ref{Eq1.8.1}), we obtain that
\begin{equation}\label{Eq2.12.2}
    \mathbb{P}(\mathscr{E}_l^c)\leq CL\exp(-c\beta_n^{-1\slash 2}\slash (K_0^3 T^3)).
\end{equation}

\subparagraph{Sub-step 1.3}

In this sub-step, we bound $D_l$. We let 
\begin{equation}
    t_s:=\lceil (s-1) L\beta_n^{-1}\rceil.
\end{equation}
Note that $t_s\in [n]$. Recall the definition of $r_s$ from (\ref{Eq1.11.1}). As
\begin{equation*}
    \min\{(s-1)L\beta_n^{-1},(\lfloor n\beta_n\slash L\rfloor-s+1)L\beta_n^{-1}\}\geq L\beta_n^{-1}\geq 2,
\end{equation*}
we have
\begin{equation*}
    t_s-1  \geq (s-1)L\beta_n^{-1}-1\geq \frac{1}{2}(s-1)L\beta_n^{-1}\geq r_s\beta_n^{-1},
\end{equation*}
\begin{equation*}
    n-t_s\geq (\lfloor n\beta_n\slash L\rfloor-s+1)L\beta_n^{-1}-1\geq \frac{1}{2}(\lfloor n\beta_n\slash L\rfloor-s+1)L\beta_n^{-1}\geq r_s\beta_n^{-1}.
\end{equation*}
Hence
\begin{equation}\label{Eq1.11.2}
    r_s\beta_n^{-1}+1\leq t_s\leq n-r_s \beta_n^{-1}.
\end{equation}
Take $\beta=\beta_n,\delta_0=1\slash 4, K=2L, r=r_s, t_0=t_s$ in Proposition \ref{P2.3}. As $r_s\geq C_1$ and $r_s\leq \log(1+\beta_n^{-1})\leq \log(1+\beta_n^{-1})^8$, noting (\ref{Eq1.11.2}), we obtain that
\begin{equation}\label{Eq2.5.3}
    \mathbb{P}\Big(\sup_{f\in\mathbb{B}_{2L}}\Big|\int f d\mu_{n,t_s}-\int fd\mu\Big|>C_2(\log{r}_s)^{1\slash 4}r_s^{-1\slash 8}\Big)\leq C_0\exp(-c_0 \beta_n^{-3\slash 4}),
\end{equation}
where we recall from Definition \ref{Def2.1} that 
\begin{equation*}
    \mu_{n,t_s}=\beta_n\sum_{i=1}^n\delta_{(\beta_n(i-t_s),\beta_n(\sigma(i)-t_s))}, \quad d\mu=\frac{1}{2}e^{-|x-y|}dxdy.
\end{equation*}

Below we assume that the event 
\begin{equation}\label{Eq1.12.1}
    \Big\{\sup_{f\in\mathbb{B}_{2L}}\Big|\int f d\mu_{n,t_s}-\int fd\mu\Big|\leq C_2(\log{r}_s)^{1\slash 4}r_s^{-1\slash 8}\Big\}
\end{equation}
holds. For any $\mathbf{x}\in \mathbb{R}^2$, we let
\begin{equation*}
    g(\mathbf{x})=\mathbbm{1}_{\tilde{Q}_{\Gamma,l}} ((t_s,t_s)+\beta_n^{-1}\mathbf{x}).
\end{equation*}
For any $\delta\in (0,1)$, we let
\begin{eqnarray*}
    \mathscr{R}_{\Gamma,l;\delta}&:=& ((s-1)L\beta_n^{-1}+L\beta_n^{-1}x_{l-1}(\Gamma)-\delta\beta_n^{-1},(s-1)L\beta_n^{-1}+L\beta_n^{-1}x_{l}(\Gamma)+\delta\beta_n^{-1}]\\
   &&\times ((s-1)L\beta_n^{-1}+L\beta_n^{-1}y_{l-1}(\Gamma)-\delta\beta_n^{-1},(s-1)L\beta_n^{-1}+L\beta_n^{-1}y_{l}(\Gamma)+\delta\beta_n^{-1}],
\end{eqnarray*}
\begin{eqnarray*}
  \mathscr{R}_{\Gamma,l;\delta}' &:=& ((s-1)L\beta_n^{-1}+L\beta_n^{-1}x_{l-1}(\Gamma)+\delta\beta_n^{-1},(s-1)L\beta_n^{-1}+L\beta_n^{-1}x_{l}(\Gamma)-\delta\beta_n^{-1}]\\
   &&\times ((s-1)L\beta_n^{-1}+L\beta_n^{-1}y_{l-1}(\Gamma)+\delta\beta_n^{-1},(s-1)L\beta_n^{-1}+L\beta_n^{-1}y_{l}(\Gamma)-\delta\beta_n^{-1}].
\end{eqnarray*}
For any $\delta\in (0,1)$ and $\mathbf{x}\in \mathbb{R}^2$, we let
\begin{equation*}
    g_{1,\delta}(\mathbf{x})=\min\{1,\delta^{-1}\beta_n\mathbbm{1}_{\tilde{Q}_{\Gamma,l}}((t_s,t_s)+\beta_n^{-1}\mathbf{x})d((t_s,t_s)+\beta_n^{-1}\mathbf{x},\partial \tilde{Q}_{\Gamma,l})\},
\end{equation*}
\begin{equation*}
    g_{2,\delta}(\mathbf{x})=\min\{1,\delta^{-1}\beta_n \mathbbm{1}_{\mathscr{R}_{\Gamma,l;\delta}}((t_s,t_s)+\beta_n^{-1}\mathbf{x})d((t_s,t_s)+\beta_n^{-1}\mathbf{x},\partial \mathscr{R}_{\Gamma,l;\delta})\},
\end{equation*}
where for any $\mathbf{x}\in\mathbb{R}^2$ and any set $A\subseteq \mathbb{R}^2$, $d(\mathbf{x},A):=\inf_{\mathbf{z}\in A}\|\mathbf{x}-\mathbf{z}\|_2$. In the following, we consider any $\delta\in (0,1)$. It can be checked that $\|g_{1,\delta}\|_{\infty}\leq 1$, $\|g_{2,\delta}\|_{\infty}\leq 1$, $\|g_{1,\delta}\|_{Lip}\leq \delta^{-1}$, and $\|g_{2,\delta}\|_{Lip}\leq \delta^{-1}$. Note that 
\begin{eqnarray*}
&& \supp(g_{1,\delta}), \supp(g_{2,\delta}) \subseteq
    \beta_n\overline{\mathscr{R}_{\Gamma,l;\delta}}-\beta_n (t_s,t_s)\nonumber\\
    &=&[(s-1+x_{l-1}(\Gamma))L-\delta-\beta_n t_s,(s-1+x_{l}(\Gamma))L+\delta-\beta_n t_s]\\
   &&\times [(s-1+y_{l-1}(\Gamma))L-\delta-\beta_n t_s,(s-1+y_{l}(\Gamma))L+\delta-\beta_n t_s].
\end{eqnarray*}
As 
\begin{eqnarray*}
   && (s-1+\min\{x_{l-1}(\Gamma),y_{l-1}(\Gamma)\})L-\delta-\beta_n t_s\nonumber\\
   &\geq& \beta_n((s-1)L\beta_n^{-1}-\lceil (s-1) L\beta_n^{-1}\rceil)-\delta \geq -\beta_n-\delta\geq -2\geq -L,
\end{eqnarray*}
\begin{eqnarray*}
 && (s-1+\max\{x_{l}(\Gamma),y_l(\Gamma)\})L+\delta-\beta_n t_s\nonumber\\
 &\leq& \beta_n((s-1)L\beta_n^{-1}-\lceil (s-1) L\beta_n^{-1}\rceil)+L+\delta\leq L+\delta\leq 2L,
\end{eqnarray*}
we have $\supp(g_{1,\delta}), \supp( g_{2,\delta})\subseteq [-2L,2L]^2$. Hence $\delta g_{1,\delta},\delta g_{2,\delta}\in \mathbb{B}_{2L}$ (recall Definition \ref{Def2.3}). By (\ref{Eq1.12.1}), as $r_s\geq \min\{L,\log(1+\beta_n^{-1})\}\slash 2\geq 1 $, we have
\begin{equation}\label{Eq1.12.2}
    \Big|\int g_{1,\delta} d\mu_{n,t_s}-\int g_{1,\delta}d\mu\Big|\leq C_2\delta^{-1}(\log{r}_s)^{1\slash 4}r_s^{-1\slash 8}\leq C_3 \delta^{-1} r_s^{-1\slash 10}, 
\end{equation}
\begin{equation}
    \Big|\int g_{2,\delta} d\mu_{n,t_s}-\int g_{2,\delta}d\mu\Big|\leq C_2\delta^{-1} (\log{r}_s)^{1\slash 4}r_s^{-1\slash 8}\leq C_3 \delta^{-1} r_s^{-1\slash 10},
\end{equation}
where $C_3$ is a positive constant that only depends on $L$. It can be checked that
\begin{equation}
    g_{1,\delta}(\mathbf{x})\leq g(\mathbf{x})\leq g_{2,\delta}(\mathbf{x}) \text{ for any } \mathbf{x}\in \mathbb{R}^2,
\end{equation}
\begin{equation}\label{Eq1.12.3}
    \int g d\mu_{n,t_s}=\beta_n\sum_{i=1}^n \mathbbm{1}_{\tilde{Q}_{\Gamma,l}}((i,\sigma(i)))=\beta_n|S(\sigma)\cap \tilde{Q}_{\Gamma,l}|.
\end{equation}
By (\ref{Eq1.12.2})-(\ref{Eq1.12.3}), 
\begin{equation}\label{Eq2.4.1}
  D_l=  |S(\sigma)\cap \tilde{Q}_{\Gamma,l}|\geq \beta_n^{-1}\int g_{1,\delta}d\mu-C_3\beta_n^{-1}\delta^{-1}r_s^{-1\slash 10},
\end{equation}
\begin{equation}\label{Eq2.4.2}
  D_l=   |S(\sigma)\cap \tilde{Q}_{\Gamma,l}|\leq \beta_n^{-1}\int g_{2,\delta}d\mu+C_3\beta_n^{-1}\delta^{-1}r_s^{-1\slash 10}.
\end{equation}
For any $\mathbf{x}=(x_1,x_2)\in \mathbb{R}^2$,
\begin{equation}\label{Eq1.12.5}
    g_{1,\delta}(\mathbf{x})\geq \mathbbm{1}_{\mathscr{R}_{\Gamma,l;\delta}'}((t_s,t_s)+\beta_n^{-1}\mathbf{x}), \quad g_{2,\delta}(\mathbf{x})\leq \mathbbm{1}_{\mathscr{R}_{\Gamma,l;\delta}}((t_s,t_s)+\beta_n^{-1}\mathbf{x}).
\end{equation}
For any $\mathbf{x}=(x_1,x_2)\in \mathbb{R}^2$ such that $(t_s,t_s)+\beta_n^{-1}\mathbf{x}\in \mathscr{R}_{\Gamma,l;\delta}$, we have 
\begin{equation}\label{Eq1.12.4}
  (y_{l-1}(\Gamma)-x_l(\Gamma))L-2\delta\leq  x_2-x_1\leq (y_l(\Gamma)-x_{l-1}(\Gamma))L+2\delta,
\end{equation}
hence by (\ref{Eq1.5.5}),
\begin{eqnarray}\label{Eq1.12.6}
   && |x_2-x_1|\nonumber\\
   &\leq& |y_{l-1}(\Gamma)-x_{l-1}(\Gamma)|L+\max\{|y_{l}(\Gamma)-y_{l-1}(\Gamma)|,|x_l(\Gamma)-x_{l-1}(\Gamma)|\}L+2\delta\nonumber\\
    &\leq& |y_{l-1}(\Gamma)-x_{l-1}(\Gamma)|L+L T^{-1}+2\delta;
\end{eqnarray}
moreover, by (\ref{Eq1.12.4}), we have
\begin{eqnarray*}
  &&  x_2-x_1-(y_l(\Gamma)-y_{l-1}(\Gamma))L-2\delta\\
  &\leq& (y_{l-1}(\Gamma)-x_{l-1}(\Gamma))L\leq x_2-x_1+(x_l(\Gamma)-x_{l-1}(\Gamma))L+2\delta,
\end{eqnarray*}
hence by (\ref{Eq1.5.5}),
\begin{eqnarray}\label{Eq1.12.7}
 && |x_2-x_1|\nonumber\\
 &\geq& |y_{l-1}(\Gamma)-x_{l-1}(\Gamma)|L-\max\{|y_l(\Gamma)-y_{l-1}(\Gamma)|,|x_l(\Gamma)-x_{l-1}(\Gamma)|\}L-2\delta\nonumber\\
 &\geq& |y_{l-1}(\Gamma)-x_{l-1}(\Gamma)|L-LT^{-1}-2\delta.
\end{eqnarray}
By (\ref{Eq1.12.5}), (\ref{Eq1.12.6}), and (\ref{Eq1.12.7}), we have
\begin{eqnarray}\label{Eq2.4.3}
    \int g_{1,\delta}d\mu&\geq& \frac{1}{2}\int\mathbbm{1}_{\mathscr{R}_{\Gamma,l;\delta}'}((t_s,t_s)+\beta_n^{-1}\mathbf{x})e^{-|x_2-x_1|}dx_1dx_2\nonumber\\
    &\geq& \frac{1}{2}\beta_n^2 e^{-|y_{l-1}(\Gamma)-x_{l-1}(\Gamma)|L-LT^{-1}-2\delta}|\mathscr{R}_{\Gamma,l;\delta}'|\nonumber\\
    &\geq& \frac{1}{2} e^{-|y_{l-1}(\Gamma)-x_{l-1}(\Gamma)|L-LT^{-1}-2\delta}(L(x_l(\Gamma)-x_{l-1}(\Gamma))-2\delta)_{+}\nonumber\\
    &&\times (L(y_l(\Gamma)-y_{l-1}(\Gamma))-2\delta)_{+},
\end{eqnarray}
\begin{eqnarray}\label{Eq2.4.4}
 \int g_{2,\delta}d\mu&\leq& \frac{1}{2} \int \mathbbm{1}_{\mathscr{R}_{\Gamma,l;\delta}}((t_s,t_s)+\beta_n^{-1}\mathbf{x})e^{-|x_2-x_1|}dx_1dx_2\nonumber\\
 &\leq& \frac{1}{2}\beta_n^2 e^{-|y_{l-1}(\Gamma)-x_{l-1}(\Gamma)|L+LT^{-1}+2\delta}|\mathscr{R}_{\Gamma,l;\delta}|\nonumber\\
 &\leq& \frac{1}{2} e^{-|y_{l-1}(\Gamma)-x_{l-1}(\Gamma)|L+LT^{-1}+2\delta}(L(x_l(\Gamma)-x_{l-1}(\Gamma))+2\delta)\nonumber\\
 &&\times(L(y_l(\Gamma)-y_{l-1}(\Gamma))+2\delta).
\end{eqnarray}
Below we take $\delta=1\slash (4K_0T)$. By (\ref{Eq1.5.5}), we have
\begin{equation}\label{Eq2.4.5}
    \min\{x_l(\Gamma)-x_{l-1}(\Gamma),y_l(\Gamma)-y_{l-1}(\Gamma)\}\geq \frac{1}{2K_0T}=2\delta.
\end{equation}
As $\min\{T,K_0\}\geq L^2$, we have $\delta\leq 1\slash (4L^4)$. Hence by (\ref{Eq2.4.1}), (\ref{Eq2.4.2}), and (\ref{Eq2.4.3})-(\ref{Eq2.4.5}), we have
\begin{eqnarray*}
 D_l
  &\geq&  -4C_3 K_0T \beta_n^{-1}r_s^{-1\slash 10} +\frac{1}{2}L^2\beta_n^{-1} e^{-2L^{-1}}(1-L^{-1})^2 e^{-|y_{l-1}(\Gamma)-x_{l-1}(\Gamma)|L}\nonumber\\
  &&\quad\quad\quad\quad\quad\quad\quad\quad\quad\quad\times(x_l(\Gamma)-x_{l-1}(\Gamma))(y_l(\Gamma)-y_{l-1}(\Gamma)),
\end{eqnarray*}
\begin{eqnarray*}
 D_l
  &\leq&  4C_3 K_0 T\beta_n^{-1}r_s^{-1\slash 10} +\frac{1}{2}L^2\beta_n^{-1} e^{2L^{-1}}(1+L^{-1})^2 e^{-|y_{l-1}(\Gamma)-x_{l-1}(\Gamma)|L}\nonumber\\
  &&\quad\quad\quad\quad\quad\quad\quad\quad\quad\times(x_l(\Gamma)-x_{l-1}(\Gamma))(y_l(\Gamma)-y_{l-1}(\Gamma)).
\end{eqnarray*}
As $L\geq 4$, we have $1-L^{-1}\geq e^{-2L^{-1}}$ and $1+L^{-1}\leq e^{L^{-1}}$. Hence 
\begin{eqnarray}\label{Eq2.5.1}
 D_l
  &\geq&  -4C_3K_0T\beta_n^{-1}r_s^{-1\slash 10} +\frac{1}{2}L^2\beta_n^{-1} e^{-6L^{-1}} e^{-|y_{l-1}(\Gamma)-x_{l-1}(\Gamma)|L}\nonumber\\
  &&\quad\quad\quad\quad\quad\quad\quad\quad\quad\times(x_l(\Gamma)-x_{l-1}(\Gamma))(y_l(\Gamma)-y_{l-1}(\Gamma)),
\end{eqnarray}
\begin{eqnarray}\label{Eq2.5.2}
 D_l
  &\leq&  4C_3K_0T\beta_n^{-1}r_s^{-1\slash 10} +\frac{1}{2}L^2\beta_n^{-1} e^{6L^{-1}} e^{-|y_{l-1}(\Gamma)-x_{l-1}(\Gamma)|L}\nonumber\\
  &&\quad\quad\quad\quad\quad\quad\quad\quad\quad\times(x_l(\Gamma)-x_{l-1}(\Gamma))(y_l(\Gamma)-y_{l-1}(\Gamma)).
\end{eqnarray}

Let $\mathcal{H}_l$ be the event that (\ref{Eq2.5.1}) and (\ref{Eq2.5.2}) hold. By (\ref{Eq2.5.3}) and the above discussion, we have
\begin{equation}\label{Eq2.11.2}
    \mathbb{P}(\mathcal{H}_l^c)\leq C_0\exp(-c_0\beta_n^{-3\slash 4}).
\end{equation}

\subparagraph{Sub-step 1.4}

In this sub-step, we bound $L_{1,l}$. Recall the definition of $\mathcal{S}_{1,l}$ in (\ref{Eq2.6.1}). We let
\begin{equation}
    R:=|\{i\in [n]: (i,\sigma(i))\in \mathcal{S}_{1,l}\times \{s_1',s_1'+1,\cdots,s_2'\}\}|.
\end{equation}
We also let $I_1,\cdots,I_n\in \{0\}\cup [n]$ and $J_1,\cdots,J_n\in \{0\}\cup [n]$ be such that
\begin{equation*}
I_{R+1}=\cdots=I_n=0, \quad J_{R+1}=\cdots=J_n=0,
\end{equation*}
\begin{equation*}
1\leq I_1<\cdots<I_R, \quad 1\leq J_1<\cdots<J_R,
\end{equation*}
\begin{equation*}
 \{I_1,\cdots,I_R\}=\{i\in [n]: (i,\sigma(i))\in \mathcal{S}_{1,l}\times \{s_1',s_1'+1,\cdots,s_2'\}\},
\end{equation*}
\begin{equation*}
 \{J_1,\cdots,J_R\}=\{i\in [n]: (\sigma^{-1}(i),i)\in \mathcal{S}_{1,l}\times \{s_1',s_1'+1,\cdots,s_2'\}\}.
\end{equation*} 
Following the argument between (\ref{Eq3.9}) and (\ref{E4.3}), we obtain that $R=D_l-D_l'$. 

Throughout the rest of the proof, we let $S_0$ be the set that consists solely of the empty mapping $\tau_0:\emptyset\rightarrow\emptyset$, and let $LIS(\tau_0):=0$. If $R\geq 1$, we let $\tau\in S_R$ be such that $\sigma(I_s)=J_{\tau(s)}$ for every $s\in [R]$. If $R=0$, we let $\tau$ be the empty mapping. Let $\mathcal{B}_l'$ be the $\sigma$-algebra generated by $\sigma_0$, $\{b_i\}_{i=1}^n$, $\{Y_i\}_{i=s_2'+1}^n$, $R$, $\{I_i\}_{i=1}^n$, and $\{J_i\}_{i=1}^n$. Following the argument in \textbf{Step 4} of Section \ref{Sect.3.1.2}, we can deduce that for any $\delta_0\in (0,1\slash 3)$, 
\begin{equation}
    \mathbb{P}(|LIS(\tau)-2\sqrt{R}|>R^{1\slash 2-\delta_0}|\mathcal{B}_l')\leq C_{\delta_0}\exp(-R^{(1-3\delta_0)\slash 2}),
\end{equation}
where $C_{\delta_0}$ is a positive constant that only depends on $\delta_0$. Taking $\delta_0=1\slash 6$ and noting that $L_{1,l}=LIS(\tau)$, we obtain that
\begin{equation}\label{Eq2.10.3}
    \mathbb{P}(|L_{1,l}-2\sqrt{R}|>R^{1\slash 3}|\mathcal{B}_l')\leq C \exp(-R^{1\slash 4}).
\end{equation}

By (\ref{Eq2.7.1}), (\ref{Eq2.5.1}), and (\ref{Eq2.5.2}), when the event $\mathcal{D}_l^c\cap\mathcal{H}_l$ holds, we have
\begin{eqnarray}\label{Eq2.8.1}
    R&\leq& 4C_3K_0 T\beta_n^{-1}r_s^{-1\slash 10} +\frac{1}{2}L^2\beta_n^{-1} e^{6L^{-1}} e^{-|y_{l-1}(\Gamma)-x_{l-1}(\Gamma)|L}\nonumber\\
  &&\quad\quad\quad\quad\quad\quad\quad\times(x_l(\Gamma)-x_{l-1}(\Gamma))(y_l(\Gamma)-y_{l-1}(\Gamma)),
\end{eqnarray}
\begin{eqnarray}
 R &\geq & -32e^2L\beta_n(s_2-s_1+1)(s_2'-s_1'+1)(y_l(\Gamma)-y_{l-1}(\Gamma))-4C_3K_0T\beta_n^{-1}r_s^{-1\slash 10}\nonumber\\
 &&+\frac{1}{2}L^2\beta_n^{-1} e^{-6L^{-1}} e^{-|y_{l-1}(\Gamma)-x_{l-1}(\Gamma)|L}(x_l(\Gamma)-x_{l-1}(\Gamma))(y_l(\Gamma)-y_{l-1}(\Gamma)).\nonumber\\
 &&
\end{eqnarray}
By (\ref{Eq1.8.5}) and (\ref{Eq1.5.6}), we have
\begin{equation*}
    s_2-s_1+1 \leq L\beta_n^{-1}(x_l(\Gamma)-x_{l-1}(\Gamma))+1\leq 2L\beta_n^{-1}(x_l(\Gamma)-x_{l-1}(\Gamma)),
\end{equation*}
\begin{equation*}
    s_2'-s_1'+1\leq L\beta_n^{-1}(y_l(\Gamma)-y_{l-1}(\Gamma))+1\leq 2L\beta_n^{-1}(y_l(\Gamma)-y_{l-1}(\Gamma)),
\end{equation*}
which by (\ref{Eq1.5.5}) lead to
\begin{eqnarray}
  &&  32e^2L\beta_n(s_2-s_1+1)(s_2'-s_1'+1)(y_l(\Gamma)-y_{l-1}(\Gamma))\nonumber\\
  &\leq& 1000\beta_n^{-1}L^3 T^{-1}(x_l(\Gamma)-x_{l-1}(\Gamma))(y_l(\Gamma)-y_{l-1}(\Gamma))\nonumber\\
  &\leq& \frac{1}{2}L^2\beta_n^{-1} e^{-6L^{-1}} e^{-|y_{l-1}(\Gamma)-x_{l-1}(\Gamma)|L}(x_l(\Gamma)-x_{l-1}(\Gamma))(y_l(\Gamma)-y_{l-1}(\Gamma))  \nonumber\\
  && \times 2000 L e^{2L} T^{-1}.
\end{eqnarray}
Moreover, by (\ref{Eq1.5.5}),
\begin{eqnarray}\label{Eq2.8.2}
&& 4C_3K_0T\beta_n^{-1}r_s^{-1\slash 10}\nonumber\\
&\leq&  \frac{1}{2}L^2\beta_n^{-1} e^{-6L^{-1}} e^{-|y_{l-1}(\Gamma)-x_{l-1}(\Gamma)|L}(x_l(\Gamma)-x_{l-1}(\Gamma))(y_l(\Gamma)-y_{l-1}(\Gamma))\nonumber\\
&& \times 8C_3K_0TL^{-2}e^{2L} (2K_0T)^2 r_s^{-1\slash 10}\nonumber\\
&\leq& \frac{1}{2}L^2\beta_n^{-1} e^{-6L^{-1}} e^{-|y_{l-1}(\Gamma)-x_{l-1}(\Gamma)|L}(x_l(\Gamma)-x_{l-1}(\Gamma))(y_l(\Gamma)-y_{l-1}(\Gamma))\nonumber\\
&&\times C_4 K_0^3 T^3 r_s^{-1\slash 10},
\end{eqnarray}
where $C_4$ is a positive constant that only depends on $L$. 

By (\ref{Eq2.8.1})-(\ref{Eq2.8.2}), when the event $\mathcal{D}_l^c\cap\mathcal{H}_l$ holds, we have
\begin{eqnarray}\label{Eq2.10.1}
 R&\leq& \frac{1}{2}L^2\beta_n^{-1}  e^{-|y_{l-1}(\Gamma)-x_{l-1}(\Gamma)|L}(x_l(\Gamma)-x_{l-1}(\Gamma))(y_l(\Gamma)-y_{l-1}(\Gamma))\nonumber\\
 && \times e^{6L^{-1}} (1+C_4K_0^3T^3 r_s^{-1\slash 10}),
\end{eqnarray}
\begin{eqnarray}\label{Eq2.9.1}
 R &\geq& \frac{1}{2}L^2\beta_n^{-1}  e^{-|y_{l-1}(\Gamma)-x_{l-1}(\Gamma)|L}(x_l(\Gamma)-x_{l-1}(\Gamma))(y_l(\Gamma)-y_{l-1}(\Gamma))\nonumber\\
 && \times e^{-6L^{-1}} (1-C_4K_0^3T^3 r_s^{-1\slash 10}-2000Le^{2L} T^{-1})_{+}.
\end{eqnarray}
Note that (\ref{Eq1.5.5}) and (\ref{Eq2.9.1}) imply that
\begin{equation}\label{Eq2.10.2}
    R\geq \frac{1}{8}\beta_n^{-1}L^2e^{-2L} K_0^{-2} T^{-2}(1-C_4K_0^3T^3 r_s^{-1\slash 10}-2000Le^{2L} T^{-1})_{+}.
\end{equation}
We let
\begin{equation}\label{Eq2.14.1}
    \Phi_1:=e^{6L^{-1}} (1+C_4K_0^3T^3 r_s^{-1\slash 10}),
\end{equation}
\begin{equation}
    \Phi_2:=e^{-6L^{-1}} (1-C_4K_0^3T^3 r_s^{-1\slash 10}-2000Le^{2L} T^{-1})_{+},
\end{equation}
\begin{equation}
 \Phi_3:=\max\Big\{\frac{1}{8}\beta_n^{-1}L^2e^{-2L} K_0^{-2} T^{-2}(1-C_4K_0^3T^3 r_s^{-1\slash 10}-2000Le^{2L} T^{-1})_{+},1\Big\},
\end{equation}
\begin{equation}\label{Eq2.14.2}
 \Phi_4:=\frac{1}{8}\beta_n^{-1}L^2e^{-2L} K_0^{-2} T^{-2}(1-C_4K_0^3T^3 r_s^{-1\slash 10}-2000Le^{2L} T^{-1})_{+}.
\end{equation}
By (\ref{Eq2.10.1})-(\ref{Eq2.10.2}), when the event $\{|L_{1,l}-2\sqrt{R}|\leq R^{1\slash 3}\}\cap \mathcal{D}_l^c\cap\mathcal{H}_l$ holds, 
\begin{eqnarray}
&& L_{1,l}\leq 2\sqrt{R}+2R^{1\slash 3}= 2\sqrt{R}(1+\max\{R,1\}^{-1\slash 6})\nonumber\\
&\leq& \sqrt{2}L\beta_n^{-1\slash 2}e^{-|y_{l-1}(\Gamma)-x_{l-1}(\Gamma)|L\slash 2}\sqrt{(x_l(\Gamma)-x_{l-1}(\Gamma))(y_l(\Gamma)-y_{l-1}(\Gamma))}\nonumber\\
&& \times \Phi_1^{1\slash 2}(1+\Phi_3^{-1\slash 6}),
\end{eqnarray}
\begin{eqnarray}
 && L_{1,l}\geq 2\sqrt{R}-2R^{1\slash 3}= 2\sqrt{R}(1-\max\{R,1\}^{-1\slash 6})\nonumber\\
 &\geq& \sqrt{2}L\beta_n^{-1\slash 2}e^{-|y_{l-1}(\Gamma)-x_{l-1}(\Gamma)|L\slash 2}\sqrt{(x_l(\Gamma)-x_{l-1}(\Gamma))(y_l(\Gamma)-y_{l-1}(\Gamma))}\nonumber\\
&& \times \Phi_2^{1\slash 2}(1-\Phi_3^{-1\slash 6}).
\end{eqnarray}
Let $\mathscr{E}_l'$ be the event that \begin{eqnarray}\label{Eq2.11.5}
&&\frac{L_{1,l}}{\sqrt{2}L\beta_n^{-1\slash 2}e^{-|y_{l-1}(\Gamma)-x_{l-1}(\Gamma)|L\slash 2}\sqrt{(x_l(\Gamma)-x_{l-1}(\Gamma))(y_l(\Gamma)-y_{l-1}(\Gamma))}}\nonumber\\
&&\in[\Phi_2^{1\slash 2}(1-\Phi_3^{-1\slash 6}),\Phi_1^{1\slash 2}(1+\Phi_3^{-1\slash 6})].
\end{eqnarray}
We have $\{|L_{1,l}-2\sqrt{R}|\leq R^{1\slash 3}\}\cap \mathcal{D}_l^c\cap\mathcal{H}_l\subseteq \mathscr{E}_l'\cap \mathcal{D}_l^c\cap\mathcal{H}_l$, which by (\ref{Eq2.10.2}) leads to 
\begin{eqnarray}
  (\mathscr{E}_l')^c\cap \mathcal{D}_l^c\cap\mathcal{H}_l &\subseteq&  \{|L_{1,l}-2\sqrt{R}|> R^{1\slash 3}\}\cap \mathcal{D}_l^c\cap\mathcal{H}_l \nonumber\\
  &\subseteq& \{|L_{1,l}-2\sqrt{R}|> R^{1\slash 3}\}\cap\{R\geq \Phi_4\}.
\end{eqnarray}
Hence by (\ref{Eq2.10.3}),
\begin{eqnarray}\label{Eq2.11.3}
\mathbb{P}((\mathscr{E}_l')^c\cap \mathcal{D}_l^c\cap\mathcal{H}_l)&\leq& \mathbb{P}(\{|L_{1,l}-2\sqrt{R}|> R^{1\slash 3}\}\cap\{R\geq \Phi_4\}) \nonumber\\
&=& \mathbb{E}[\mathbb{P}(|L_{1,l}-2\sqrt{R}|>R^{1\slash 3}|\mathcal{B}_l') \mathbbm{1}_{R\geq \Phi_4} ]\nonumber\\
&\leq& C \mathbb{E}[\exp(-R^{1\slash 4})\mathbbm{1}_{R\geq \Phi_4}] \leq C\exp(-\Phi_4^{1\slash 4}).
\end{eqnarray}
By (\ref{Eq2.11.1}), (\ref{Eq2.11.2}), (\ref{Eq2.11.3}), and the union bound, we have
\begin{eqnarray}\label{Eq2.12.3}
&& \mathbb{P}((\mathscr{E}_l')^c)\leq \mathbb{P}((\mathscr{E}_l')^c\cap \mathcal{D}_l^c\cap\mathcal{H}_l)+\mathbb{P}(\mathcal{D}_l)+\mathbb{P}(\mathcal{H}_l^c) \nonumber\\
&\leq& C\exp(-\Phi_4^{1\slash 4})+CL\exp(-c\beta_n^{-1}\slash (K_0^5 T^5))+C_0\exp(-c_0\beta_n^{-3\slash 4}). \nonumber\\
&&
\end{eqnarray}

\medskip

Let $\mathscr{C}_{\Gamma,l}$ be the event that 
\begin{eqnarray}\label{Eq2.15.1}
 && \sqrt{2}L\beta_n^{-1\slash 2}e^{-|y_{l-1}(\Gamma)-x_{l-1}(\Gamma)|L\slash 2}\sqrt{(x_l(\Gamma)-x_{l-1}(\Gamma))(y_l(\Gamma)-y_{l-1}(\Gamma))}\nonumber\\
 &&\times \Phi_2^{1\slash 2}(1-\Phi_3^{-1\slash 6}) \nonumber\\
 &\leq& LIS(\sigma|_{\tilde{Q}_{\Gamma,l}})\nonumber\\
 &\leq& 200 L^{3\slash 2}\beta_n^{-1\slash 2}  T^{-1\slash 2}(x_l(\Gamma)-x_{l-1}(\Gamma)+y_l(\Gamma)-y_{l-1}(\Gamma))\nonumber\\
 &&+ \sqrt{2}L\beta_n^{-1\slash 2}e^{-|y_{l-1}(\Gamma)-x_{l-1}(\Gamma)|L\slash 2}\sqrt{(x_l(\Gamma)-x_{l-1}(\Gamma))(y_l(\Gamma)-y_{l-1}(\Gamma))}\nonumber\\
 &&\quad\times \Phi_1^{1\slash 2}(1+\Phi_3^{-1\slash 6}).
\end{eqnarray}
By (\ref{Eq2.12.1}), (\ref{Eq2.11.4}), and (\ref{Eq2.11.5}), we have $\mathscr{E}_l\cap \mathscr{E}_l'\subseteq \mathscr{C}_{\Gamma,l}$. Hence by (\ref{Eq2.12.2}), (\ref{Eq2.12.3}), and the union bound, we have
\begin{eqnarray}
    &&\mathbb{P}((\mathscr{C}_{\Gamma,l})^c)\leq \mathbb{P}(\mathscr{E}_l^c)+\mathbb{P}((\mathscr{E}_l')^c)\nonumber\\
    &\leq& C\exp(-\Phi_4^{1\slash 4})+ CL\exp(-c\beta_n^{-1\slash 2}\slash (K_0^5 T^5))+C_0\exp(-c_0\beta_n^{-3\slash 4}).\nonumber\\
    &&
\end{eqnarray}

\paragraph{Step 2}

Throughout the rest of the proof, we take $T=\lceil 2000 L^2 e^{2L} \rceil$ and $K_0=2L^2+1$. Note that $\min\{T,K_0\}\geq L^2$ and $\max\{8K_0T,K_0^2T^3\}\leq C' L^{10}e^{6L}$, where $C'\geq 1$ is an absolute constant. We also assume that $\beta_n^{-1}\geq C' L^{10} e^{6L}$. Note that this implies (\ref{Eq1.6.1}) and $\beta_n^{-1}\geq L^{10}$. We denote by $C_L',c_L'$ positive constants that only depend on $L$. The values of these constants may change from line to line.

Recalling (\ref{Eq2.14.1})-(\ref{Eq2.14.2}), we have
\begin{equation}
    \Phi_1\leq e^{6L^{-1}}(1+C_L r_s^{-1\slash 10}), \quad \Phi_2\geq e^{-6L^{-1}}(1-L^{-1}-C_L r_s^{-1\slash 10})_{+},
\end{equation}
\begin{equation}\label{Eq2.15.2}
    \Phi_3\geq \max\{c_L\beta_n^{-1}(1-L^{-1}-C_L r_s^{-1\slash 10})_{+},1\}, \quad \Phi_4\geq c_L\beta_n^{-1}(1-L^{-1}-C_L r_s^{-1\slash 10})_{+},
\end{equation}
where $C_L,c_L$ are positive constants that only depend on $L$. In the following, we denote
\begin{equation}
    \Psi_s:=(1-L^{-1}-C_L r_s^{-1\slash 10})_{+}.
\end{equation}

For any $\Gamma\in \Pi^{T,T,K_0}$ and any $l\in [2T-1]$, we let $\mathscr{D}_{\Gamma,l}$ be the event that 
\begin{eqnarray}\label{Eq2.18.1}
 && \sqrt{2}L\beta_n^{-1\slash 2}e^{-|y_{l-1}(\Gamma)-x_{l-1}(\Gamma)|L\slash 2}\sqrt{(x_l(\Gamma)-x_{l-1}(\Gamma))(y_l(\Gamma)-y_{l-1}(\Gamma))}\nonumber\\
 &&\times e^{-3L^{-1}}\Psi_s^{1\slash 2} (1-\max\{c_L\beta_n^{-1}\Psi_s,1\}^{-1\slash 6}) \nonumber\\
 &\leq& LIS(\sigma|_{\tilde{Q}_{\Gamma,l}})\nonumber\\
 &\leq& 5 L^{1\slash 2} e^{-L} \beta_n^{-1\slash 2}  (x_l(\Gamma)-x_{l-1}(\Gamma)+y_l(\Gamma)-y_{l-1}(\Gamma))+1\nonumber\\
 &&+ \sqrt{2}L\beta_n^{-1\slash 2}e^{-|y_{l-1}(\Gamma)-x_{l-1}(\Gamma)|L\slash 2}\sqrt{(x_l(\Gamma)-x_{l-1}(\Gamma))(y_l(\Gamma)-y_{l-1}(\Gamma))}\nonumber\\
 &&\quad\times e^{3L^{-1}}(1+C_L r_s^{-1\slash 10})^{1\slash 2}(1+\max\{c_L\beta_n^{-1}\Psi_s,1\}^{-1\slash 6}).
\end{eqnarray}
By (\ref{Eq2.15.1})-(\ref{Eq2.15.2}), we have
\begin{eqnarray}\label{Eq2.16.1}
    \mathbb{P}((\mathscr{D}_{\Gamma,l})^c) &\leq& C\exp(-\Phi_4^{1\slash 4})+ CL\exp(-c\beta_n^{-1\slash 2}\slash (K_0^5 T^5))+C_0\exp(-c_0\beta_n^{-3\slash 4})\nonumber\\
    &\leq& C_L'\exp(-c_L' \beta_n^{-1\slash 4} \Psi_s^{1\slash 4}).
\end{eqnarray}

For any $\Gamma\in \Pi^{T,T,K_0}$ and any $l\in [2T-1]$, we let $\mathscr{D}_{\Gamma,l}'$ be the event that 
\begin{eqnarray}\label{Eq2.19.1}
 && \sqrt{2}L\beta_n^{-1\slash 2}e^{-|b_{l-1}(\Gamma)-a_{l-1}(\Gamma)|L\slash 2}\sqrt{(c_l(\Gamma)-a_{l-1}(\Gamma))(d_l(\Gamma)-b_{l-1}(\Gamma))}\nonumber\\
 &&\times e^{-3L^{-1}}\Psi_s^{1\slash 2} (1-\max\{c_L\beta_n^{-1}\Psi_s,1\}^{-1\slash 6}) \nonumber\\
 &\leq& LIS(\sigma|_{\tilde{Q}_{\Gamma,l}'})\nonumber\\
 &\leq& 5 L^{1\slash 2} e^{-L} \beta_n^{-1\slash 2}  (c_l(\Gamma)-a_{l-1}(\Gamma)+d_l(\Gamma)-b_{l-1}(\Gamma))+1\nonumber\\
 &&+ \sqrt{2}L\beta_n^{-1\slash 2}e^{-|b_{l-1}(\Gamma)-a_{l-1}(\Gamma)|L\slash 2}\sqrt{(c_l(\Gamma)-a_{l-1}(\Gamma))(d_l(\Gamma)-b_{l-1}(\Gamma))}\nonumber\\
 &&\quad\times e^{3L^{-1}}(1+C_L r_s^{-1\slash 10})^{1\slash 2}(1+\max\{c_L\beta_n^{-1}\Psi_s,1\}^{-1\slash 6}).
\end{eqnarray}
Similarly, we have
\begin{equation}\label{Eq2.16.2}
    \mathbb{P}((\mathscr{D}_{\Gamma,l}')^c)\leq C_L'\exp(-c_L' \beta_n^{-1\slash 4} \Psi_s^{1\slash 4}).
\end{equation}

Now we let 
\begin{equation}
    \mathscr{A}:=\bigcap_{\Gamma\in \Pi^{T,T,K_0}}\bigcap_{l=1}^{2T-1}(\mathscr{D}_{\Gamma,l}\cap \mathscr{D}_{\Gamma,l}').
\end{equation}
By (\ref{Eq2.16.1}), (\ref{Eq2.16.2}), and the union bound, we have
\begin{equation}\label{Eq2.21.3}
    \mathbb{P}(\mathscr{A}^c)\leq C_L'\exp(-c_L' \beta_n^{-1\slash 4} \Psi_s^{1\slash 4}).
\end{equation}

\paragraph{Step 3}

Let $\Gamma_0\in \Pi^{T,T,K_0}$ be 
\begin{equation*}
    (1,1), \frac{K_0+1}{2}, (2,1), \frac{K_0+1}{2}, (2,2), \frac{K_0+1}{2}, \cdots, (T,T-1), \frac{K_0+1}{2}, (T,T). 
\end{equation*}
We have $(x_0(\Gamma_0),y_0(\Gamma_0))=(0,0)$, $(x_{2T-1}(\Gamma_0),y_{2T-1}(\Gamma_0))=(1,1)$. For any $l\in [2T-2]$,
\begin{equation*}
    (x_l(\Gamma_0),y_l(\Gamma_0))=\Big(\frac{l+1}{2T},\frac{l}{2T}\Big).
\end{equation*}
By Lemma \ref{L1}, we have 
\begin{equation}\label{Eq2.17.1}
    LIS(\sigma|_{\mathcal{R}_s})\geq \sum_{l=1}^{2T-1} LIS(\sigma|_{\tilde{Q}_{\Gamma_0,l}}). 
\end{equation}
When the event $\mathscr{A}$ holds, by (\ref{Eq2.18.1}) and (\ref{Eq2.17.1}), we have 
\begin{eqnarray}\label{Eq2.21.1}
 LIS(\sigma|_{\mathcal{R}_s})&\geq& \sqrt{2} L \beta_n^{-1\slash 2}  \cdot \frac{2T-3}{2T} \cdot e^{-4L^{-1}}\Psi_s^{1\slash 2} (1-\max\{c_L\beta_n^{-1}\Psi_s,1\}^{-1\slash 6}) \nonumber\\
 &\geq& \sqrt{2} L \beta_n^{-1\slash 2} e^{-6L^{-1}}\Psi_s^{1\slash 2} (1-\max\{c_L\beta_n^{-1}\Psi_s,1\}^{-1\slash 6}),
\end{eqnarray}
where we use the fact that $1-3\slash (2T)\geq 1-L^{-1}\geq e^{-2L^{-1}}$. 

Below we consider any $\Gamma\in \Pi^{T,T,K_0}$. When the event $\mathscr{A}$ holds, by (\ref{Eq2.19.1}), we have 
\begin{eqnarray}\label{Eq2.19.3}
&& \sum_{l=1}^{2T-1}  LIS(\sigma|_{\tilde{Q}_{\Gamma,l}'})\nonumber\\
&\leq&  5L^{1\slash 2} e^{-L}\beta_n^{-1\slash 2}\sum_{l=1}^{2T-1}(c_l(\Gamma)-a_{l-1}(\Gamma)+d_l(\Gamma)-b_{l-1}(\Gamma))+2T-1\nonumber\\
&&+\sqrt{2}L\beta_n^{-1\slash 2}e^{3L^{-1}}(1+C_L r_s^{-1\slash 10})^{1\slash 2}(1+\max\{c_L\beta_n^{-1}\Psi_s,1\}^{-1\slash 6})\nonumber\\
&& \quad\times \sum_{l=1}^{2T-1}\sqrt{(c_l(\Gamma)-a_{l-1}(\Gamma))(d_l(\Gamma)-b_{l-1}(\Gamma))}.
\end{eqnarray}
Note that for any $l\in [2T-1]$,
\begin{eqnarray*}
  &&  |c_l(\Gamma)-x_l(\Gamma)|\leq (2K_0T)^{-1},\quad |a_{l-1}(\Gamma)-x_{l-1}(\Gamma)|\leq (2K_0T)^{-1},\nonumber\\
  &&|d_l(\Gamma)-y_l(\Gamma)|\leq (2K_0T)^{-1},\quad |b_{l-1}(\Gamma)-y_{l-1}(\Gamma)|\leq (2K_0T)^{-1}.
\end{eqnarray*}
Hence by the AM-GM inequality, we have
\begin{eqnarray}\label{Eq2.19.4}
 && \sum_{l=1}^{2T-1}\sqrt{(c_l(\Gamma)-a_{l-1}(\Gamma))(d_l(\Gamma)-b_{l-1}(\Gamma))} \nonumber\\
 &\leq& \frac{1}{2}\sum_{l=1}^{2T-1}(c_l(\Gamma)-a_{l-1}(\Gamma)+d_l(\Gamma)-b_{l-1}(\Gamma))\nonumber\\
 &\leq& \frac{1}{2}\sum_{l=1}^{2T-1}(x_l(\Gamma)-x_{l-1}(\Gamma)+y_l(\Gamma)-y_{l-1}(\Gamma))+\frac{2T-1}{K_0T}\nonumber\\
 &\leq& 1+\frac{2}{K_0}\leq 1+L^{-1}.
\end{eqnarray}
By (\ref{Eq2.19.3}) and (\ref{Eq2.19.4}), when the event $\mathscr{A}$ holds, we have
\begin{eqnarray}\label{Eq2.19.5}
 &&\sum_{l=1}^{2T-1}  LIS(\sigma|_{\tilde{Q}_{\Gamma,l}'})\nonumber\\
 &\leq& \sqrt{2}L\beta_n^{-1\slash 2}e^{4L^{-1}}(1+C_L r_s^{-1\slash 10})^{1\slash 2}(1+\max\{c_L\beta_n^{-1}\Psi_s,1\}^{-1\slash 6})\nonumber\\
 &&+20L^{1\slash 2}e^{-L}\beta_n^{-1\slash 2}+5000L^2 e^{2L}.
\end{eqnarray}

By Lemma \ref{L1} and (\ref{Eq2.19.5}), when the event $\mathscr{A}$ holds, we have
\begin{eqnarray}\label{Eq2.21.2}
   &&  LIS(\sigma|_{\mathcal{R}_s})\leq \max_{\Gamma\in \Pi^{T,T,K_0}}\Big\{\sum_{l=1}^{2T-1}LIS(\sigma|_{\tilde{Q}_{\Gamma,l}'})\Big\} \nonumber\\
    &\leq & \sqrt{2}L\beta_n^{-1\slash 2}e^{4L^{-1}}(1+C_L r_s^{-1\slash 10})^{1\slash 2}(1+\max\{c_L\beta_n^{-1}\Psi_s,1\}^{-1\slash 6})\nonumber\\
 &&+20L^{1\slash 2}e^{-L}\beta_n^{-1\slash 2}+5000L^2 e^{2L}.
\end{eqnarray}

By (\ref{Eq2.21.1}) and (\ref{Eq2.21.2}), when the event $\mathscr{A}$ holds, we have
\begin{eqnarray}\label{Eq2.21.4}
&&|  LIS(\sigma|_{\mathcal{R}_s})-\sqrt{2}L\beta_n^{-1\slash 2}|\leq 20L^{1\slash 2}e^{-L}\beta_n^{-1\slash 2}+5000L^2 e^{2L}\nonumber\\
&& +\sqrt{2}L\beta_n^{-1\slash 2}\max\Big\{e^{4L^{-1}} (1+C_L r_s^{-1\slash 10})^{1\slash 2}(1+\max\{c_L\beta_n^{-1}\Psi_s,1\}^{-1\slash 6})-1,\nonumber\\
&&\quad\quad\quad\quad\quad\quad\quad\quad\quad 1-e^{-6L^{-1}}\Psi_s^{1\slash 2} (1-\max\{c_L\beta_n^{-1}\Psi_s,1\}^{-1\slash 6})\Big\}.
\end{eqnarray}
Note that $LIS(\sigma|_{\mathcal{R}_s})\leq |\mathcal{I}_{n,s}\cap\mathbb{N}^{*}|\leq 2L\beta_n^{-1}+1\leq 3L\beta_n^{-1} $. Hence by (\ref{Eq2.21.3}) and (\ref{Eq2.21.4}), we have
\begin{eqnarray}
&& \mathbb{E}[|  LIS(\sigma|_{\mathcal{R}_s})-\sqrt{2}L\beta_n^{-1\slash 2}|] \nonumber\\
&\leq& (3L\beta_n^{-1})(C_L'\exp(-c_L' \beta_n^{-1\slash 4} \Psi_s^{1\slash 4}))+20L^{1\slash 2}e^{-L}\beta_n^{-1\slash 2}+5000L^2 e^{2L}\nonumber\\
&&+\sqrt{2}L\beta_n^{-1\slash 2}\max\{e^{4L^{-1}} (1+C_L r_s^{-1\slash 10})^{1\slash 2}(1+\max\{c_L\beta_n^{-1}\Psi_s,1\}^{-1\slash 6})-1,\nonumber\\
&&\quad\quad\quad\quad\quad\quad\quad\quad\quad 1-e^{-6L^{-1}}\Psi_s^{1\slash 2} (1-\max\{c_L\beta_n^{-1}\Psi_s,1\}^{-1\slash 6})\}\nonumber\\
&\leq & C_L'\beta_n^{-1}\exp(-c_L' \beta_n^{-1\slash 4} \Psi_s^{1\slash 4})+C_L'+C L^{1\slash 2}e^{-L}\beta_n^{-1\slash 2}\nonumber\\
&& +\sqrt{2}L\beta_n^{-1\slash 2}\max\{e^{4L^{-1}} (1+C_L r_s^{-1\slash 10})^{1\slash 2}(1+\max\{c_L\beta_n^{-1}\Psi_s,1\}^{-1\slash 6})-1,\nonumber\\
&&\quad\quad\quad\quad\quad\quad\quad\quad\quad 1-e^{-6L^{-1}}\Psi_s^{1\slash 2} (1-\max\{c_L\beta_n^{-1}\Psi_s,1\}^{-1\slash 6})\}.
\end{eqnarray}

\end{proof}

\subsection{Proof of Theorem \ref{limit_l1_2}}\label{Sect.4.2}

In this subsection, we finish the proof of Theorem \ref{limit_l1_2} based on Propositions \ref{P4.1}-\ref{P4.3}.

\begin{proof}[Proof of Theorem \ref{limit_l1_2}]

Throughout the proof, we fix an arbitrary sequence of positive numbers $(\beta_n)_{n=1}^{\infty}$ such that $\lim_{n\rightarrow\infty}\beta_n=0$ and $\lim_{n\rightarrow\infty} n\beta_n=\infty$. For each $n\in \mathbb{N}^{*}$, we let $\gamma_n:=\sqrt{n\beta_n}$. Note that
\begin{equation}\label{Eq3.5.2}
 \lim_{n\rightarrow\infty}\gamma_n=\infty, \quad \lim_{n\rightarrow}\frac{\gamma_n}{n\beta_n}=0.
\end{equation}
We fix any $L\in \mathbb{N}^{*}$ such that $L\geq 4$ ($L$ is independent of $n$). 

Let $C_1,C_L,c_L,C_L',c_L',C'$ and $r_s,\Psi_s$ be defined as in Proposition \ref{P4.3}. In the following, we assume that $n\in\mathbb{N}^{*}$ is sufficiently large, so that
\begin{eqnarray}\label{Eq3.1.1}
  &&  n\beta_n\geq 20L, \quad \beta_n^{-1}\geq C'L^{10} e^{6L},\quad   \gamma_n\in [2,n\beta_n\slash (4L)],  \nonumber\\
  && \min\{(\gamma_n-1)L,\log(1+\beta_n^{-1})\}\geq 2\max\{(C_L L)^{10},C_1\}.
\end{eqnarray}

Let $\mathcal{S}_1:=[\gamma_n,n\beta_n\slash L-\gamma_n]\cap\mathbb{N}$. As
\begin{equation*}
    \gamma_n\geq 2, \quad n\beta_n\slash L-\gamma_n\leq n\beta_n\slash L -2\leq \lfloor n\beta_n\slash L\rfloor -1,
\end{equation*}
we have $\mathcal{S}_1\subseteq [2, \lfloor n\beta_n\slash L\rfloor -1]\cap\mathbb{N}$. Let $\mathcal{S}_2:=[\lfloor n\beta_n\slash L \rfloor]\backslash \mathcal{S}_1$. Note that 
\begin{equation}\label{Eq3.4.1}
    |\mathcal{S}_1|\leq n\beta_n\slash L, \quad |\mathcal{S}_1|\geq n\beta_n\slash L-2\gamma_n-1\geq n\beta_n\slash L-3\gamma_n,
\end{equation}
\begin{equation}\label{Eq3.3.1}
    |\mathcal{S}_2|\leq n\beta_n\slash L-|\mathcal{S}_1|\leq 3\gamma_n.
\end{equation}

By (\ref{Eq3.1.1}), for any $s\in  \mathcal{S}_1$, we have
\begin{equation*}
    r_s \geq \frac{1}{2}\min\{(\gamma_n-1)L,\log(1+\beta_n^{-1})\} \geq \max\{(C_L L)^{10},C_1\}, 
\end{equation*}
hence $\Psi_s\geq 1-2L^{-1}\geq 1\slash 2$. By Proposition \ref{P4.3}, for any $s\in  \mathcal{S}_1$, we have 
\begin{eqnarray}\label{Eq3.3.3}
&&  \mathbb{E}[|  LIS(\sigma|_{\mathcal{R}_s})-\sqrt{2}L\beta_n^{-1\slash 2}|] \nonumber\\
&\leq& C_L'\beta_n^{-1} \exp(-c_L'\beta_n^{-1\slash 4} \slash 2)+C_L'+C L^{1\slash 2}e^{-L}\beta_n^{-1\slash 2}\nonumber\\
&& +\sqrt{2} L\beta_n^{-1\slash 2} \max\{e^{4L^{-1}} (1+L^{-1})^{1\slash 2} (1+\max\{c_L\beta_n^{-1}\slash 2,1\}^{-1\slash 6})-1, \nonumber\\
&& \quad\quad\quad\quad 1-e^{-6L^{-1}}(1-2L^{-1})^{1\slash 2} (1-\max\{c_L\beta_n^{-1}\slash 2,1\}^{-1\slash 6})\}.
\end{eqnarray}

By (\ref{Eq3.2.1}) and (\ref{Eq3.2.2}), we have
\begin{eqnarray}\label{Eq3.3.4}
  && \mathbb{E}[|LIS(\sigma)-n\sqrt{2\beta_n}|] \nonumber\\
  &\leq&  \sum_{s\in \mathcal{S}_1}\mathbb{E}[|LIS(\sigma|_{\mathcal{R}_s})-\sqrt{2}L\beta_n^{-1\slash 2}|]+|n\sqrt{2\beta_n}-\sqrt{2}L\beta_n^{-1\slash 2}|\mathcal{S}_1||\nonumber\\
  &&+\sum_{s\in \mathcal{S}_2}\mathbb{E}[LIS(\sigma|_{\mathcal{R}_s})]+\sum_{s=2}^{\lfloor n\beta_n\slash L\rfloor}\mathbb{E}[LIS(\sigma|_{\mathcal{R}_s'})]+\sum_{s=2}^{\lfloor n\beta_n\slash L\rfloor}\mathbb{E}[LIS(\sigma|_{\mathcal{R}_s''})].\nonumber\\
  &&
\end{eqnarray}
By (\ref{Eq3.4.1}) and (\ref{Eq3.3.3}), 
\begin{eqnarray}
&&  \sum_{s\in\mathcal{S}_1}\mathbb{E}[|  LIS(\sigma|_{\mathcal{R}_s})-\sqrt{2}L\beta_n^{-1\slash 2}|] \nonumber\\
&\leq& C_L'n \exp(-c_L'\beta_n^{-1\slash 4} \slash 2)+C_L'n\beta_n+C L^{-1\slash 2}e^{-L}n\sqrt{\beta_n}\nonumber\\
&& +n\sqrt{2\beta_n}  \max\{e^{4L^{-1}} (1+L^{-1})^{1\slash 2} (1+\max\{c_L\beta_n^{-1}\slash 2,1\}^{-1\slash 6})-1, \nonumber\\
&& \quad\quad\quad 1-e^{-6L^{-1}}(1-2L^{-1})^{1\slash 2} (1-\max\{c_L\beta_n^{-1}\slash 2,1\}^{-1\slash 6})\}.
\end{eqnarray}
By (\ref{Eq3.4.1}),
\begin{equation}
   0\leq  n\sqrt{2\beta_n}-\sqrt{2}L\beta_n^{-1\slash 2}|\mathcal{S}_1|\leq C L\gamma_n\beta_n^{-1\slash 2}.
\end{equation}
By Proposition \ref{P4.2}, (\ref{Eq3.1.1}), and (\ref{Eq3.3.1}), 
\begin{eqnarray}
    \sum_{s\in \mathcal{S}_2} \mathbb{E}[LIS(\sigma|_{\mathcal{R}_s})]&\leq& CL\beta_n^{-1\slash 2}|\mathcal{S}_2|+CL^2\exp(-c\beta_n^{-1\slash 2})|\mathcal{S}_2|\nonumber\\
    &\leq& CL\gamma_n \beta_n^{-1\slash 2}+CL^2\gamma_n\exp(-c\beta_n^{-1\slash 2})\nonumber\\
    &\leq& CL\gamma_n \beta_n^{-1\slash 2}+CL^2n\beta_n\exp(-c\beta_n^{-1\slash 2}).
\end{eqnarray}
By Proposition \ref{P4.1},
\begin{eqnarray}
    \sum_{s=2}^{\lfloor n\beta_n\slash L\rfloor}\mathbb{E}[LIS(\sigma|_{\mathcal{R}_s'})]&\leq& (n\beta_n\slash L)(CL^{1\slash 2}\beta_n^{-1\slash 2}+CL^2\exp(-c\beta_n^{-1\slash 2}))\nonumber\\
    &\leq& CL^{-1\slash 2}n\sqrt{\beta_n}+CLn\beta_n\exp(-c\beta_n^{-1\slash 2}),
\end{eqnarray}
\begin{eqnarray}\label{Eq3.3.5}
    \sum_{s=2}^{\lfloor n\beta_n\slash L\rfloor}\mathbb{E}[LIS(\sigma|_{\mathcal{R}_s''})]&\leq& (n\beta_n\slash L)(CL^{1\slash 2}\beta_n^{-1\slash 2}+CL^2\exp(-c\beta_n^{-1\slash 2}))\nonumber\\
    &\leq& CL^{-1\slash 2}n\sqrt{\beta_n}+CLn\beta_n\exp(-c\beta_n^{-1\slash 2}).
\end{eqnarray}
By (\ref{Eq3.3.4})-(\ref{Eq3.3.5}), we have
\begin{eqnarray}
&& \frac{\mathbb{E}[|LIS(\sigma)-n\sqrt{2\beta_n}|]}{n\sqrt{\beta_n}}\nonumber\\
&\leq& \frac{CL\gamma_n}{n\beta_n}+CL^2\beta_n^{1\slash 2}\exp(-c\beta_n^{-1\slash 2})+CL^{-1\slash 2}\nonumber\\
&& +C_L'\beta_n^{-1\slash 2}\exp(-c_L'\beta_n^{-1\slash 4}\slash 2) +C_L'\beta_n^{1\slash 2}\nonumber\\
&& +\sqrt{2} \max\{e^{4L^{-1}} (1+L^{-1})^{1\slash 2} (1+\max\{c_L\beta_n^{-1}\slash 2,1\}^{-1\slash 6})-1, \nonumber\\
&& \quad\quad 1-e^{-6L^{-1}}(1-2L^{-1})^{1\slash 2} (1-\max\{c_L\beta_n^{-1}\slash 2,1\}^{-1\slash 6})\}.
\end{eqnarray}
Hence by (\ref{Eq3.5.2}),
\begin{eqnarray}
 && \limsup_{n\rightarrow\infty}\Big\{\frac{\mathbb{E}[|LIS(\sigma)-n\sqrt{2\beta_n}|]}{n\sqrt{\beta_n}}\Big\} \nonumber\\
 &\leq& CL^{-1\slash 2}+\sqrt{2}\max\{e^{4L^{-1}}(1+L^{-1})^{1\slash 2}-1,1-e^{-6L^{-1}}(1-2L^{-1})^{1\slash 2}\}.  \nonumber\\
 &&
\end{eqnarray}
Taking $L\rightarrow\infty$, we obtain that
\begin{equation}
    \limsup_{n\rightarrow\infty}\Big\{\frac{\mathbb{E}[|LIS(\sigma)-n\sqrt{2\beta_n}|]}{n\sqrt{\beta_n}}\Big\}\leq 0.
\end{equation}
Hence 
\begin{equation}
    \lim_{n\rightarrow\infty} \  \mathbb{E}\Big[\Big|\frac{LIS(\sigma)}{n\sqrt{\beta}_n}-\sqrt{2}\Big|\Big]=0, \text{ i.e., }  \frac{LIS(\sigma)}{n\sqrt{\beta}_n}\xrightarrow[]{L^1} \sqrt{2}.
\end{equation}

\end{proof}

\section{Proof of Theorem \ref{limit_l2_2}}\label{Sect.5}

In this section, we give the proof of Theorem \ref{limit_l2_2}. We first establish three preliminary propositions in Section \ref{Sect.5.1}. Based on these propositions, we finish the proof of Theorem \ref{limit_l2_2} in Section \ref{Sect.5.2}.

\subsection{Three preliminary propositions}\label{Sect.5.1}

In this subsection, we establish three preliminary propositions. These propositions will be used in the proof of Theorem \ref{limit_l2_2}.

Throughout this subsection, we fix an arbitrary sequence of positive numbers $(\beta_n)_{n=1}^{\infty}$ such that $\lim_{n\rightarrow\infty}\beta_n=0$ and $\lim_{n\rightarrow\infty} n^2\beta_n=\infty$. We denote by $C_1$ the constant $C$ that appears in Proposition \ref{P2.2.2} (with $C_0=1$). Without loss of generality, we assume that $C_1\geq 1$. We let $L_0=8C_1$, and fix any $L\geq 4$ such that $L\slash L_0 \in  \mathbb{N}^{*}$. 

Below we consider any $n\in\mathbb{N}^{*}$ such that $n\beta_n^{1\slash 2}\geq 4L$ and $\beta_n\leq 1\slash 100$. For any $s\in [\lfloor n\beta_n^{1\slash 2}\slash L\rfloor-1]$, we let
\begin{equation}
    \mathcal{I}_{n,s}:=((s-1)L\beta_n^{-1\slash 2},sL\beta_n^{-1\slash 2}].
\end{equation}
We also let
\begin{equation}
   \mathcal{I}_{n,\lfloor n\beta_n^{1\slash 2}\slash L\rfloor}:=((\lfloor n\beta_n^{1\slash 2}\slash L\rfloor -1)L\beta_n^{-1\slash 2},n].
\end{equation}
For any $s\in [\lfloor n\beta_n^{1\slash 2}\slash L\rfloor]$, we let
\begin{equation}
    \mathcal{R}_s:=\mathcal{I}_{n,s}\times \mathcal{I}_{n,s}.
\end{equation}
For any $s\in [\lfloor n\beta_n^{1\slash 2}\slash L\rfloor-1]$, we let
\begin{equation}
   \mathcal{R}_s':=(s L\beta_n^{-1\slash 2},n]\times \mathcal{I}_{n,s},   \quad  \mathcal{R}_s'':=\mathcal{I}_{n,s}\times (s L\beta_n^{-1\slash 2},n].
\end{equation}
Note that 
\begin{equation*}
    \Big(\bigcup_{s=1}^{\lfloor n\beta_n^{1\slash 2}\slash L\rfloor}\mathcal{R}_s\Big) \bigcup \Big( \bigcup_{s=1}^{\lfloor n\beta_n^{1\slash 2}\slash L\rfloor-1}\mathcal{R}_s' \Big)\bigcup \Big( \bigcup_{s=1}^{\lfloor n\beta_n^{1\slash 2}\slash L\rfloor-1}\mathcal{R}_s'' \Big)=(0,n]^2.
\end{equation*}
Hence for any $\sigma\in S_n$, we have
\begin{equation}\label{Eq48.1.2}
    LIS(\sigma)\leq \sum_{s=1}^{\lfloor n\beta_n^{1\slash 2}\slash L\rfloor}LIS(\sigma|_{\mathcal{R}_s})+\sum_{s=1}^{\lfloor n\beta_n^{1\slash 2}\slash L\rfloor-1}LIS(\sigma|_{\mathcal{R}_s'})+\sum_{s=1}^{\lfloor n\beta_n^{1\slash 2}\slash L\rfloor-1}LIS(\sigma|_{\mathcal{R}_s''}),
\end{equation}
\begin{equation}\label{Eq48.1.4}
    LIS(\sigma)\geq \sum_{s=1}^{\lfloor n\beta_n^{1\slash 2}\slash L\rfloor}LIS(\sigma|_{\mathcal{R}_s}).
\end{equation}

The following proposition bounds $LIS(\sigma|_{\mathcal{R}_s'})$ and $LIS(\sigma|_{\mathcal{R}_s''})$ for $\sigma$ drawn from $\tilde{\mathbb{P}}_{n,\beta_n}$ and any $s\in [\lfloor n\beta_n^{1\slash 2}\slash L\rfloor-1]$.  

\begin{proposition}\label{P5.1}
Assume that $n\beta_n^{1\slash 2}\geq 4L$ and $\beta_n\leq 1\slash 100$, and let $\sigma$ be drawn from $\tilde{\mathbb{P}}_{n,\beta_n}$. Then there exist positive absolute constants $C,c$, such that for any $s\in [\lfloor n\beta_n^{1\slash 2}\slash L\rfloor-1]$, we have
\begin{equation}\label{Eq7.1.1}
    \mathbb{E}[LIS(\sigma|_{\mathcal{R}_s'})]\leq CL^{1\slash 2}\beta_n^{-1\slash 4}+CL^2\exp(-c\beta_n^{-1\slash 4}),
\end{equation}
\begin{equation}\label{Eq8.1.6}
    \mathbb{E}[LIS(\sigma|_{\mathcal{R}_s''})]\leq CL^{1\slash 2}\beta_n^{-1\slash 4}+CL^2\exp(-c\beta_n^{-1\slash 4}).
\end{equation}
\end{proposition}
\begin{proof}

Throughout the proof, we fix an arbitrary $s\in [\lfloor n\beta_n^{1\slash 2}\slash L\rfloor-1]$. 

We start by showing (\ref{Eq7.1.1}). We let
\begin{equation}
    \mathcal{J}_{s,1}:=\bigcup_{t=1}^{L\slash L_0}((s-1)L\beta_n^{-1\slash 2}+(t-1)L_0\beta_n^{-1\slash 2}, (s-1)L\beta_n^{-1\slash 2}+(t-1\slash 2)L_0\beta_n^{-1\slash 2}],
\end{equation}
\begin{equation}
    \mathcal{J}_{s,2}:=\bigcup_{t=1}^{L\slash L_0}((s-1)L\beta_n^{-1\slash 2}+(t-1\slash 2)L_0\beta_n^{-1\slash 2}, (s-1)L\beta_n^{-1\slash 2}+ t L_0\beta_n^{-1\slash 2}].
\end{equation}
We also let
\begin{equation}
    \mathcal{T}_{s,1}:=(sL\beta_n^{-1\slash 2},n]\times \mathcal{J}_{s,1},   \quad \mathcal{T}_{s,2}:=(sL\beta_n^{-1\slash 2},n]\times \mathcal{J}_{s,2}.
\end{equation}
Note that $\mathcal{R}_s'=\mathcal{T}_{s,1}\cup\mathcal{T}_{s,2}$. Hence
\begin{equation}\label{Eq15.1.1}
    LIS(\sigma|_{\mathcal{R}_s'})\leq LIS(\sigma|_{\mathcal{T}_{s,1}})+LIS(\sigma|_{\mathcal{T}_{s,2}}).
\end{equation}

For any $t\in [L\slash L_0]$, we let 
\begin{equation}
    \mathcal{X}_{t,1}:=[(s-1)L\beta_n^{-1\slash 2}+(t-1)L_0\beta_n^{-1\slash 2},n]\cap\mathbb{N}^{*}, 
\end{equation}
\begin{equation}
    \mathcal{Y}_{t,1}:=((s-1)L\beta_n^{-1\slash 2}+(t-1)L_0\beta_n^{-1\slash 2},(s-1)L\beta_n^{-1\slash 2}+ t L_0\beta_n^{-1\slash 2}]\cap\mathbb{N}^{*};
\end{equation}
\begin{equation}
    \mathcal{X}_{t,2}:=[(s-1)L\beta_n^{-1\slash 2}+(t-1\slash 2)L_0\beta_n^{-1\slash 2},n]\cap\mathbb{N}^{*}, 
\end{equation}
\begin{equation}
    \mathcal{Y}_{t,2}:=((s-1)L\beta_n^{-1\slash 2}+(t-1\slash 2)L_0\beta_n^{-1\slash 2},(s-1)L\beta_n^{-1\slash 2}+ (t+1\slash 2) L_0\beta_n^{-1\slash 2}]\cap\mathbb{N}^{*}.
\end{equation}

In the following, we bound $\mathbb{E}[LIS(\sigma|_{\mathcal{T}_{s,1}})]$ and $\mathbb{E}[LIS(\sigma|_{\mathcal{T}_{s,2}})]$ in \textbf{Steps 1-2}.

\paragraph{Step 1}

In this step, we bound $\mathbb{E}[LIS(\sigma|_{\mathcal{T}_{s,1}})]$. For every $t\in [L\slash L_0]$, we let $\alpha_t:=(s-1)L\beta_n^{-1\slash 2}+(t-1)L_0\beta_n^{-1\slash 2}-1$.

We sample $\sigma_0$ from $\tilde{\mathbb{P}}_{n,\beta_n}$. Then sequentially for $t=1,2,\cdots,L\slash L_0$, we run the resampling algorithm for the $L^2$ model (as described at the end of Section \ref{Sect.1.5}) with inputs $\sigma_{t-1},\mathcal{X}_{t,1}, \mathcal{Y}_{t,1}, \alpha_t$ to obtain $\sigma_t$. We let $\sigma=\sigma_{L \slash  L_0}$. By Lemma \ref{L2.2}, the distribution of $\sigma$ is given by $\tilde{\mathbb{P}}_{n,\beta_n}$. 

For any $t\in [L\slash L_0]$, let $M_t\in \mathbb{N}$, $I_{t,1},\cdots,I_{t,M_t}\in \mathcal{X}_{t,1}$, $J_{t,1},\cdots,J_{t,M_t}\in \mathcal{Y}_{t,1}$ be such that $I_{t,1}<\cdots<I_{t,M_t}$, $J_{t,1}<\cdots<J_{t,M_t}$, 
\begin{equation}
     \{i\in \mathcal{X}_{t,1}:\sigma_0(i)\in \mathcal{Y}_{t,1}\}=\{I_{t,1},\cdots,I_{t,M_t}\},
\end{equation}
\begin{equation}
    \{j\in \mathcal{Y}_{t,1}:\sigma_0^{-1}(j)\in \mathcal{X}_{t,1}\}=\{J_{t,1},\cdots,J_{t,M_t}\}.
\end{equation}
For any $t\in [L\slash L_0]$ and any $m\in [n]\backslash [M_t]$, we let $I_{t,m}=0$ and $J_{t,m}=0$. According to the resampling algorithm for the $L^2$ model, $\sigma$ can be generated as follows. Sequentially for $t=1,2,\cdots,L\slash L_0$, we do the following:
\begin{itemize}
    \item For each $m\in [M_t]$, we independently sample $u_{t,m}$ from the uniform distribution on $[0,e^{2\beta_n (I_{t,m}-\alpha_t) \sigma_0(I_{t,m})}]$, and let
    \begin{equation*}
        b_{t,m}=\log(u_{t,m})\slash (2\beta_n (I_{t,m}-\alpha_t) ).
    \end{equation*}
    For each $m\in [n]\backslash [M_t]$, we let $b_{t,m}=0$.
    \item For each $m\in [M_t]$, let 
    \begin{equation}
        N_{t,m}=|\{m'\in  [M_t]: b_{t,m'} \leq J_{t,m} \}|-m+1.
    \end{equation}
   Now look at the $N_{t,1}$ integers $m'\in [M_t]$ with $b_{t,m'}\leq J_{t,1}$, and pick $Y_{t,1}$ uniformly from these integers; then look at the $N_{t,2}$ remaining integers $m'\in [M_t]$ with $b_{t,m'}\leq J_{t,2}$ (with $Y_{t,1}$ deleted from the list), and pick $Y_{t,2}$ uniformly from these integers; and so on. In this way we obtain $\{Y_{t,m}\}_{m\in [M_t]}$. For each $m\in [n]\backslash [M_t]$, we let $N_{t,m}=0$ and $Y_{t,m}=0$.
\end{itemize}
We let $\sigma\in S_n$ be the unique permutation that satisfies the following conditions:
\begin{itemize}
    \item For any $t \in [L\slash L_0]$ and any $m\in [M_t]$, $\sigma(I_{t,Y_{t,m}})=J_{t,m}$.
    \item For any $i\in [n]\backslash \big(\bigcup_{t=1}^{L\slash L_0} \{I_{t,1},\cdots,I_{t,M_t}\} \big)$, $\sigma(i)=\sigma_0(i)$.
\end{itemize}

For any $t\in [L\slash L_0]$, let
\begin{equation}\label{Eq5.1.5}
    z_{t}:=(s-1)L\beta_n^{-1\slash 2}+(t-1\slash 2)L_0\beta_n^{-1\slash 2},\quad z_{t}':=(s-1)L\beta_n^{-1\slash 2}+tL_0\beta_n^{-1\slash 2}.
\end{equation}
Recall Definition \ref{Def2.2}. As $\lceil z_{t}\rceil-1,\lfloor z_{t}'\rfloor\in [n]$, by (\ref{Eq5.1.1}) and Proposition \ref{P2.2.2},
\begin{equation}\label{Eq5.1.2}
    \mathbb{P}(|\mathcal{D}_{\lceil z_{t}\rceil-1}(\sigma_0)|\geq C_1\beta_n^{-1\slash 2}) \leq C\exp(-c\beta_n^{-1\slash 2}),
\end{equation}
\begin{equation}\label{Eq5.1.3}
    \mathbb{P}(|\mathcal{D}'_{\lfloor z_{t}'\rfloor}(\sigma_0)|\geq C_1\beta_n^{-1\slash 2}) \leq C\exp(-c\beta_n^{-1\slash 2}).
\end{equation}
Let $\mathcal{Z}$ be the event that for any $t\in [L\slash L_0]$, $|\mathcal{D}_{\lceil z_{t}\rceil-1}(\sigma_0)|\leq C_1\beta_n^{-1\slash 2}$ and $|\mathcal{D}'_{\lfloor z_{t}'\rfloor}(\sigma_0)|\leq C_1\beta_n^{-1\slash 2}$. By (\ref{Eq5.1.2}), (\ref{Eq5.1.3}), and the union bound, we have 
\begin{equation}\label{Eq7.1.8}
    \mathbb{P}(\mathcal{Z}^c)\leq CL\exp(-c\beta_n^{-1\slash 2}).
\end{equation}

Now for any $t\in [L\slash L_0]$, when the event $\mathcal{Z}$ holds, as $L_0=8C_1$, $L_0\beta_n^{-1\slash 2}\geq 80$, and $\lfloor z_{t}'\rfloor-\lceil z_{t}\rceil \geq L_0\beta_n^{-1\slash 2}\slash 2-2$, we have
\begin{eqnarray}\label{Eq5.1.4}
   && |S(\sigma_0)\cap [z_{t},z'_{t}]^2|\geq |S(\sigma_0)\cap [\lceil z_{t}\rceil ,\lfloor z'_{t}\rfloor ]^2| \nonumber\\ 
   &\geq& |[\lceil z_{t}\rceil ,\lfloor z'_{t}\rfloor]\cap\mathbb{N}^{*}|-|\mathcal{D}_{\lceil z_{t}\rceil-1}(\sigma_0)|-|\mathcal{D}'_{\lfloor z_{t}'\rfloor}(\sigma_0)|\nonumber\\
   &\geq&  \lfloor z'_{t}\rfloor-\lceil z_{t}\rceil+1-2C_1\beta_n^{-1\slash 2}\geq \frac{1}{2}L_0\beta_n^{-1\slash 2}-1-\frac{1}{4}L_0\beta_n^{-1\slash 2}\nonumber\\
   &=& \frac{1}{4}L_0\beta_n^{-1\slash 2}-1 \geq \frac{1}{8}L_0\beta_n^{-1\slash 2}.
\end{eqnarray}

For any $t\in [L\slash L_0]$, we let $\mathscr{M}_t$ be the set of $m\in [M_t]$ that satisfies 
\begin{equation}\label{Eq5.1.6}
    J_{t,m}\in ((s-1)L\beta_n^{-1\slash 2}+(t-1)L_0\beta_n^{-1\slash 2}, (s-1)L\beta_n^{-1\slash 2}+(t-1\slash 2)L_0\beta_n^{-1\slash 2}].
\end{equation}
Below we consider any $t\in [L\slash L_0]$ and $m\in [n]$. If $m\in \mathscr{M}_t$, for any $m'\in [M_t]$ such that $\sigma_0(I_{t,m'})< J_{t,m}$ (note that there are $m-1$ such $m'$), we have that $b_{t,m'}\leq \sigma_0(I_{t,m'})< J_{t,m}$, hence
\begin{equation}\label{Eq5.1.7}
    N_{t,m}=\sum_{\substack{m'\in [M_t]: \\ \sigma_0(I_{t,m'})\geq J_{t,m}}} \mathbbm{1}_{b_{t,m'}\leq J_{t,m}}.
\end{equation}
For any $i\in [n]$ such that $(i,\sigma_0(i))\in [z_{t}, z'_{t}]^2$, we have $(i,\sigma_0(i))\in \mathcal{X}_{t,1}\times \mathcal{Y}_{t,1}$. Hence there exists some $m'\in [M_t]$, such that $i=I_{t,m'}$. Let
\begin{equation}
    \mathcal{M}_t:=\{m'\in [M_t]: (I_{t,m'},\sigma_0(I_{t,m'}))\in [z_{t},z'_{t}]^2\}.
\end{equation}
By (\ref{Eq5.1.4}), when the event $\mathcal{Z}$ holds, we have
\begin{equation}\label{Eq5.1.11}
    |\mathcal{M}_t| \geq |S(\sigma_0)\cap [z_{t},z'_{t}]^2|\geq \frac{1}{8}L_0\beta_n^{-1\slash 2}\geq \beta_n^{-1\slash 2}.
\end{equation}
If $m \in \mathscr{M}_t$, for any $m'\in \mathcal{M}_t$, we have $\sigma_0(I_{t,m'})\geq z_{t}\geq J_{t,m}$ (note (\ref{Eq5.1.6})). Hence by (\ref{Eq5.1.7}),
\begin{equation}\label{Eq5.1.8}
    N_{t,m} \geq \sum_{m'\in \mathcal{M}_t} \mathbbm{1}_{b_{t,m'}\leq J_{t,m}}.
\end{equation}
Now note that if $m\in \mathscr{M}_t$, conditional on $\sigma_0$, $\{\mathbbm{1}_{b_{t,m'}\leq J_{t,m}}\}_{m'\in \mathcal{M}_t}$ are mutually independent, and for each $m'\in \mathcal{M}_t$, $\mathbbm{1}_{b_{t,m'}\leq J_{t,m}}$ follows the Bernoulli distribution with 
\begin{eqnarray}\label{Eq5.1.9}
   \mathbb{P}(\mathbbm{1}_{b_{t,m'}\leq J_{t,m}}=1|\sigma_0)&=&\mathbb{P}(b_{t,m'}\leq J_{t,m}|\sigma_0)=\mathbb{P}(u_{t,m'}\leq e^{2\beta_n (I_{t,m'}-\alpha_t)J_{t,m}}|\sigma_0) \nonumber\\
   &=& e^{-2\beta_n (I_{t,m'}-\alpha_t)(\sigma_0(I_{t,m'})-J_{t,m})}.
\end{eqnarray}
If $m \in \mathscr{M}_t$, for any $m'\in \mathcal{M}_t$, as $(I_{t,m'},\sigma_0(I_{t,m'}))\in [z_{t},z_{t}']^2$, by (\ref{Eq5.1.6}), we have 
\begin{equation*}
    I_{t,m'}-\alpha_t\begin{cases}
     \geq z_{t}-\alpha_t\geq 1\\
     \leq z'_{t}-\alpha_t= L_0\beta_n^{-1\slash 2}+1\leq 2L_0\beta_n^{-1\slash 2}
    \end{cases},
\end{equation*}
\begin{equation*}
   0\leq  \sigma_0(I_{t,m'})-J_{t,m} \leq z'_{t}-J_{t,m}\leq L_0\beta_n^{-1\slash 2},
\end{equation*}
hence by (\ref{Eq5.1.9}), we have
\begin{equation}\label{Eq5.1.10}
    \mathbb{P}(\mathbbm{1}_{b_{t,m'}\leq J_{t,m}}=1|\sigma_0)\geq e^{-4L_0^2}.
\end{equation}
By (\ref{Eq5.1.8}), (\ref{Eq5.1.10}), and Hoeffding's inequality, for any $x\in [0,e^{-4L_0^2}]$, we have
\begin{equation*}
    \mathbb{P}(N_{t,m}\leq (e^{-4L_0^2}-x)|\mathcal{M}_t||\sigma_0)\mathbbm{1}_{m\in \mathscr{M}_t}\leq e^{-2|\mathcal{M}_t|x^2}\mathbbm{1}_{m\in \mathscr{M}_t},
\end{equation*}
which by (\ref{Eq5.1.11}) leads to 
\begin{equation*}
    \mathbb{P}(N_{t,m}\leq (e^{-4L_0^2}-x)\beta_n^{-1\slash 2}|\sigma_0) \mathbbm{1}_{m\in \mathscr{M}_t} \mathbbm{1}_{\mathcal{Z}}  \leq e^{-2\beta_n^{-1\slash 2}x^2}\mathbbm{1}_{m\in \mathscr{M}_t}.
\end{equation*}
Taking $x=e^{-4L_0^2}\slash 2$, we have
\begin{eqnarray}\label{Eq5.1.15}
  && \mathbb{P}(\{N_{t,m}\leq e^{-4L_0^2}\beta_n^{-1\slash 2} \slash 2\}\cap\{m\in \mathscr{M}_t\}\cap\mathcal{Z}|\sigma_0)\nonumber\\
   &=& \mathbb{P}(N_{t,m}\leq e^{-4L_0^2}\beta_n^{-1\slash 2} \slash 2|\sigma_0) \mathbbm{1}_{m\in \mathscr{M}_t} \mathbbm{1}_{\mathcal{Z}} \leq e^{-c\beta_n^{-1\slash 2}}\mathbbm{1}_{m\in \mathscr{M}_t}. 
\end{eqnarray}

For any $t\in [L\slash L_0]$, let $\mathcal{C}_t$ be the event that $N_{t,m}\geq e^{-4L_0^2}\beta_n^{-1\slash 2} \slash 2$ for any $m\in \mathscr{M}_t$. By (\ref{Eq5.1.15}) and the union bound,
\begin{eqnarray}
&&\mathbb{P}(\mathcal{C}_t^c\cap\mathcal{Z}|\sigma_0)\leq 
\mathbb{P}\Big(\bigcup_{m=1}^n\big(\{N_{t,m}\leq e^{-4L_0^2}\beta_n^{-1\slash 2} \slash 2\}\cap\{m\in \mathscr{M}_t\}\cap\mathcal{Z}\big)\Big|\sigma_0\Big)\nonumber\\
&\leq& \sum_{m=1}^n \mathbb{P}(\{N_{t,m}\leq e^{-4L_0^2}\beta_n^{-1\slash 2} \slash 2\}\cap\{m\in \mathscr{M}_t\}\cap\mathcal{Z}|\sigma_0)\nonumber\\ &\leq& e^{-c\beta_n^{-1\slash 2}}\sum_{m=1}^n \mathbbm{1}_{m\in \mathscr{M}_t}=|\mathscr{M}_t| e^{-c\beta_n^{-1\slash 2}}
\leq |\mathcal{Y}_{t,1}|e^{-c\beta_n^{-1\slash 2}}\nonumber\\
&\leq& (L_0\beta_n^{-1\slash 2}+1)e^{-c\beta_n^{-1\slash 2}}\leq C\beta_n^{-1\slash 2}e^{-c\beta_n^{-1\slash 2}}\leq C\exp(-c\beta_n^{-1\slash 2}).
\end{eqnarray}
Hence 
\begin{equation}\label{Eq7.1.6}
    \mathbb{P}(\mathcal{C}_t^c\cap\mathcal{Z})=\mathbb{E}[\mathbb{P}(\mathcal{C}_t^c\cap\mathcal{Z}|\sigma_0)]\leq C\exp(-c\beta_n^{-1\slash 2}).
\end{equation}
Let $\mathcal{C}:=\bigcap_{t=1}^{L\slash L_0}\mathcal{C}_t$. By (\ref{Eq7.1.8}), (\ref{Eq7.1.6}), and the union bound, 
\begin{equation}\label{Eq10.1.3}
    \mathbb{P}(\mathcal{C}^c)\leq CL\exp(-c\beta_n^{-1\slash 2}).
\end{equation}

Let 
\begin{equation}\label{Eq14.1.1}
    W_s:=\{i\in [n]: (i,\sigma_0(i))\in \mathcal{R}_s'\}.
\end{equation}
For any $i\in [n]$ such that $(i,\sigma_0(i))\in \mathcal{R}_s'$, we have 
\begin{equation*}
    i> sL\beta_n^{-1\slash 2},  \text{ hence } i\geq \lfloor sL\beta_n^{-1\slash 2}\rfloor+1,
\end{equation*}
\begin{equation*}
    \sigma_0(i)\leq sL\beta_n^{-1\slash 2},  \text{ hence } \sigma_0(i)\leq \lfloor sL\beta_n^{-1\slash 2} \rfloor.
\end{equation*}
Hence noting that $\lfloor sL\beta_n^{-1\slash 2} \rfloor\in [n]$, we have
\begin{equation}\label{Eq7.1.2}
    |W_s|\leq |\mathcal{D}'_{\lfloor sL\beta_n^{-1\slash 2}\rfloor}(\sigma_0)|=|\mathcal{D}_{\lfloor sL\beta_n^{-1\slash 2}\rfloor}(\sigma_0)|.
\end{equation}
Let $\mathcal{W}_s$ be the event that $|W_s|\leq C_1\beta_n^{-1\slash 2}$. By (\ref{Eq7.1.2}) and Proposition \ref{P2.2.2},
\begin{equation}\label{Eq10.1.4}
    \mathbb{P}(\mathcal{W}_s^c)\leq \mathbb{P}(|\mathcal{D}_{\lfloor sL\beta_n^{-1\slash 2}\rfloor}(\sigma_0)|\geq C_1\beta_n^{-1\slash 2})\leq C\exp(-c\beta_n^{-1\slash 2}). 
\end{equation}

For any $q\in \mathbb{N}^{*}$, we let
\begin{equation}\label{Eq8.1.7}
    \Lambda_{s,q,1}:=\sum_{\substack{ i_1<\cdots<i_q, j_1<\cdots<j_q\\i_1,\cdots,i_q\in (sL\beta_n^{-1\slash 2},n]\cap\mathbb{N}^{*}\\j_1,\cdots,j_q\in \mathcal{J}_{s,1}\cap\mathbb{N}^{*}}} \mathbbm{1}_{\sigma(i_1)=j_1,\cdots,\sigma(i_q)=j_q} \mathbbm{1}_{i_1,\cdots,i_q\in W_s}.
\end{equation}
In the following, we bound $\Lambda_{s,q,1}$ for any $q\in \mathbb{N}^{*}$.

Consider any $t\in [L\slash L_0]$. For any $m\in [M_t]$, we let $\mathscr{N}_{J_{t,m}}=N_{t,m}$; for any $j\in \mathcal{Y}_{t,1}\backslash \{J_{t,1},\cdots,J_{t,M_t}\}$, we let $\mathscr{N}_{j}=n$. For any $m\in [M_t]$, we let $\mathscr{Y}_{J_{t,m}}=I_{t,Y_{t,m}}$; for any $j\in \mathcal{Y}_{t,1}\backslash \{J_{t,1},\cdots,J_{t,M_t}\}$, we let $\mathscr{Y}_{j}=0$. 

We let $\mathcal{B}$ be the $\sigma$-algebra generated by $\sigma_0$ and $\{b_{t,m}\}_{t\in [L\slash L_0], m\in [n]}$. For any $j\in ((s-1)L\beta_n^{-1\slash 2},sL\beta_n^{-1\slash 2}]\cap\mathbb{N}^{*}$, we let $\mathcal{F}_{j}$ be the $\sigma$-algebra generated by $\sigma_0$, $\{b_{t,m}\}_{t\in [L\slash L_0],m\in [n]}$, and $\{\mathscr{Y}_{l}\}_{l\in [j-1]\cap ((s-1)L\beta_n^{-1\slash 2},sL\beta_n^{-1\slash 2}]\cap\mathbb{N}^{*}}$.

We assume that the event $\mathcal{C}$ holds. For any $t\in [L\slash L_0]$ and any $m\in \mathscr{M}_t$, we have $\mathscr{N}_{J_{t,m}}=N_{t,m}\geq e^{-4L_0^2}\beta_n^{-1\slash 2}\slash 2$. Hence for any $t\in [L\slash L_0]$ and any $j$ from 
\begin{eqnarray*}
    &&\{J_{t,m}:m\in [M_t]\}\nonumber\\
    &&\cap ((s-1)L\beta_n^{-1\slash 2}+(t-1)L_0\beta_n^{-1\slash 2}, (s-1)L\beta_n^{-1\slash 2}+(t-1\slash 2)L_0\beta_n^{-1\slash 2}]\cap\mathbb{N}^{*},
\end{eqnarray*}
we have $\mathscr{N}_j\geq e^{-4L_0^2}\beta_n^{-1\slash 2}\slash 2$. Now for any $t\in [L\slash L_0]$ and any $j$ from
\begin{eqnarray*}
 &&  \{J_{t,m}:m\in [M_t]\}^c\nonumber\\
 &&\cap
   ((s-1)L\beta_n^{-1\slash 2}+(t-1)L_0\beta_n^{-1\slash 2}, (s-1)L\beta_n^{-1\slash 2}+(t-1\slash 2)L_0\beta_n^{-1\slash 2}]\cap\mathbb{N}^{*},
\end{eqnarray*}
we have $\mathscr{N}_j=n\geq e^{-4L_0^2}\beta_n^{-1\slash 2}\slash 2$ (note that $n\beta_n^{1\slash 2}\geq 4L\geq 4$). Hence for any $j$ from the set
\begin{eqnarray*}
   &&\bigcup_{t=1}^{L\slash L_0} (((s-1)L\beta_n^{-1\slash 2}+(t-1)L_0\beta_n^{-1\slash 2}, (s-1)L\beta_n^{-1\slash 2}+(t-1\slash 2)L_0\beta_n^{-1\slash 2}]\cap\mathbb{N}^{*})\nonumber\\
    &&= \mathcal{J}_{s,1}\cap\mathbb{N}^{*},
\end{eqnarray*}
we have 
\begin{equation}\label{Eq8.1.2}
    \mathscr{N}_j\geq e^{-4L_0^2}\beta_n^{-1\slash 2}\slash 2.
\end{equation}

Consider any $i_1,\cdots,i_q\in (sL\beta_n^{-1\slash 2},n]\cap\mathbb{N}^{*}$ and $j_1,\cdots,j_q\in \mathcal{J}_{s,1}\cap\mathbb{N}^{*}$ such that $i_1<\cdots<i_q$ and $j_1<\cdots<j_q$. Note that 
\begin{equation}\label{Eq8.1.10}
    ((sL\beta_n^{-1\slash 2},n]\cap\mathbb{N}^{*})\times (\mathcal{J}_{s,1}\cap\mathbb{N}^{*})\subseteq \bigcup_{t=1}^{L\slash L_0} \mathcal{X}_{t,1}\times \mathcal{Y}_{t,1}. 
\end{equation}
Hence for any $l\in [q]$, if $\sigma(i_l)=j_l$, then there exists some $t\in [L \slash L_0]$, such that $(i_l,\sigma(i_l))=(i_l,j_l)\in \mathcal{X}_{t,1}\times \mathcal{Y}_{t,1}$; this implies $\mathscr{Y}_{j_l}=i_l$. Hence we have
\begin{eqnarray}\label{Eq8.1.3}
  &&  \mathbb{E}[\mathbbm{1}_{\sigma(i_1)=j_1,\cdots,\sigma(i_q)=j_q} \mathbbm{1}_{i_1,\cdots,i_q\in W_s}|\mathcal{B}]\leq \mathbb{E}[\mathbbm{1}_{\mathscr{Y}_{j_1}=i_1,\cdots,\mathscr{Y}_{j_q}=i_q}\mathbbm{1}_{i_1,\cdots,i_q\in W_s}|\mathcal{B}]\nonumber\\
  &=& \mathbbm{1}_{i_1,\cdots,i_q\in W_s}\mathbb{E}[\mathbbm{1}_{\mathscr{Y}_{j_1}=i_1,\cdots,\mathscr{Y}_{j_q}=i_q}|\mathcal{B}]\nonumber\\
  &=&\mathbbm{1}_{i_1,\cdots,i_q\in W_s}\mathbb{E}[\mathbb{E}[\mathbbm{1}_{\mathscr{Y}_{j_q}=i_q}|\mathcal{F}_{j_q}]\mathbbm{1}_{\mathscr{Y}_{j_1}=i_1,\cdots,\mathscr{Y}_{j_{q-1}}=i_{q-1}}|\mathcal{B}]\nonumber\\
  &\leq& \frac{\mathbbm{1}_{i_1,\cdots,i_q\in W_s}}{\mathscr{N}_{j_q}}\mathbb{E}[\mathbbm{1}_{\mathscr{Y}_{j_1}=i_1,\cdots,\mathscr{Y}_{j_{q-1}}=i_{q-1}}|\mathcal{B}]\leq \cdots\leq\frac{\mathbbm{1}_{i_1,\cdots,i_q\in W_s}}{\mathscr{N}_{j_1}\mathscr{N}_{j_2}\cdots \mathscr{N}_{j_q}}. 
\end{eqnarray}
By (\ref{Eq8.1.2}) and (\ref{Eq8.1.3}), we have
\begin{eqnarray}\label{Eq8.1.8}
  &&  \mathbb{E}[\mathbbm{1}_{\sigma(i_1)=j_1,\cdots,\sigma(i_q)=j_q} \mathbbm{1}_{i_1,\cdots,i_q\in W_s}|\mathcal{B}]\mathbbm{1}_{\mathcal{C}\cap\mathcal{W}_s}\nonumber\\
  &\leq& \frac{\mathbbm{1}_{i_1,\cdots,i_q\in W_s}\mathbbm{1}_{\mathcal{C}\cap\mathcal{W}_s}}{\mathscr{N}_{j_1}\mathscr{N}_{j_2}\cdots \mathscr{N}_{j_q}}\leq (2e^{4L_0^2}\beta_n^{1\slash 2})^q \mathbbm{1}_{i_1,\cdots,i_q\in W_s}\mathbbm{1}_{\mathcal{W}_s}.
\end{eqnarray}

By (\ref{Eq8.1.7}), (\ref{Eq8.1.8}), and Lemma \ref{Lemma3.1}, we have
\begin{eqnarray*}
   &&  \mathbb{E}[\Lambda_{s,q,1}|\mathcal{B}]\mathbbm{1}_{\mathcal{C}\cap\mathcal{W}_s}\leq (2e^{4L_0^2}\beta_n^{1\slash 2})^q\mathbbm{1}_{\mathcal{W}_s} \sum_{\substack{ i_1<\cdots<i_q, j_1<\cdots<j_q\\i_1,\cdots,i_q\in (sL\beta_n^{-1\slash 2},n]\cap\mathbb{N}^{*}\\j_1,\cdots,j_q\in \mathcal{J}_{s,1}\cap\mathbb{N}^{*}}}\mathbbm{1}_{i_1,\cdots,i_q\in W_s} \nonumber\\
   &\leq& (2e^{4L_0^2}\beta_n^{1\slash 2})^q\binom{|W_s|}{q}\binom{|\mathcal{J}_{s,1}\cap\mathbb{N}^{*}|}{q}\mathbbm{1}_{\mathcal{W}_s}\nonumber\\
   &\leq& \Big(\frac{2e^{2+4L_0^2}\beta_n^{1\slash 2}|W_s||\mathcal{J}_{s,1}\cap\mathbb{N}^{*}|}{q^2}\Big)^q \mathbbm{1}_{\mathcal{W}_s}\leq (CL\beta_n^{-1\slash 2}q^{-2} )^q,
\end{eqnarray*}
where we use the fact that
\begin{equation*}
    |\mathcal{J}_{s,1}\cap\mathbb{N}^{*}|\leq |((s-1)L\beta_n^{-1\slash 2},sL\beta_n^{-1\slash 2}]\cap\mathbb{N}^{*}|\leq L\beta_n^{-1\slash 2}+1\leq 2L\beta_n^{-1\slash 2}
\end{equation*}
in the last line. Hence
\begin{equation}\label{Eq10.1.1}
    \mathbb{E}[\Lambda_{s,q,1}\mathbbm{1}_{\mathcal{C}\cap\mathcal{W}_s}]=\mathbb{E}[\mathbb{E}[\Lambda_{s,q,1}|\mathcal{B}]\mathbbm{1}_{\mathcal{C}\cap\mathcal{W}_s}]\leq (CL\beta_n^{-1\slash 2}q^{-2} )^q.
\end{equation}

Now for any $q\in \mathbb{N}^{*}$, if $LIS(\sigma|_{\mathcal{T}_{s,1}})\geq q$, then there exist
\begin{equation*}
    i_1,\cdots,i_q\in (sL\beta_n^{-1\slash 2},n]\cap\mathbb{N}^{*}, \quad j_1,\cdots,j_q\in \mathcal{J}_{s,1}\cap\mathbb{N}^{*},
\end{equation*}
such that $i_1<\cdots<i_q$, $j_1<\cdots<j_q$, and $\sigma(i_l)=j_l$ for every $l \in [q]$. For any $l\in [q]$, by (\ref{Eq8.1.10}), we have $(i_l,\sigma(i_l))=(i_l,j_l)\in \mathcal{X}_{t,1} \times \mathcal{Y}_{t,1} $ for some $t\in [L\slash L_0]$, hence $(i_l,\sigma_0(i_l))\in \mathcal{X}_{t,1} \times \mathcal{Y}_{t,1}$ and $\sigma_0(i_l)\in \mathcal{Y}_{t,1} \subseteq \mathcal{I}_{n,s}\cap\mathbb{N}^{*}$; as $i_l\in (sL\beta_n^{-1\slash 2},n]\cap\mathbb{N}^{*}$, we have $(i_l,\sigma_0(i_l))\in \mathcal{R}_s'$, hence $i_l\in W_s$. Hence $\Lambda_{s,q,1}\geq 1$. We conclude that for any $q\in \mathbb{N}^{*}$,
\begin{equation}\label{Eq10.1.2}
    \{LIS(\sigma|_{\mathcal{T}_{s,1}})\geq q\}\subseteq \{\Lambda_{s,q,1}\geq  1\}.
\end{equation}

By (\ref{Eq10.1.1}) and (\ref{Eq10.1.2}), for any $q\in\mathbb{N}^{*}$, we have
\begin{eqnarray}
&&\mathbb{P}(\{LIS(\sigma|_{\mathcal{T}_{s,1}})\geq q\}\cap\mathcal{C}\cap\mathcal{W}_s)\leq \mathbb{P}(\{\Lambda_{s,q,1}\geq  1\}\cap \mathcal{C}\cap\mathcal{W}_s)\nonumber\\
&=& \mathbb{E}[\mathbbm{1}_{\Lambda_{s,q,1}\geq  1}\mathbbm{1}_{\mathcal{C}\cap\mathcal{W}_s}]\leq \mathbb{E}[\Lambda_{s,q,1}\mathbbm{1}_{\mathcal{C}\cap\mathcal{W}_s}] \leq (C_0L\beta_n^{-1\slash 2}q^{-2} )^q,
\end{eqnarray}
where $C_0\geq 1$ is a positive absolute constant. Taking $q=\lceil \sqrt{2C_0} L^{1\slash 2}\beta_n^{-1\slash 4} \rceil$, we obtain that
\begin{eqnarray}\label{Eq10.1.6}
 && \mathbb{P}(\{LIS(\sigma|_{\mathcal{T}_{s,1}})\geq 2\sqrt{2C_0}L^{1\slash 2}\beta_n^{-1\slash 4}\}\cap\mathcal{C}\cap\mathcal{W}_s) \nonumber\\
 &\leq& \mathbb{P}(\{LIS(\sigma|_{\mathcal{T}_{s,1}})\geq \lceil \sqrt{2C_0}L^{1\slash 2}\beta_n^{-1\slash 4} \rceil \}\cap\mathcal{C}\cap\mathcal{W}_s)\nonumber\\
 &\leq& 2^{-\lceil \sqrt{2C_0} L^{1\slash 2}\beta_n^{-1\slash 4} \rceil}\leq \exp(-cL^{1\slash 2}\beta_n^{-1\slash 4}).
\end{eqnarray}
By (\ref{Eq10.1.3}), (\ref{Eq10.1.4}), (\ref{Eq10.1.6}), and the union bound, we have
\begin{eqnarray}
 && \mathbb{P}(LIS(\sigma|_{\mathcal{T}_{s,1}})\geq 2\sqrt{2C_0}L^{1\slash 2}\beta_n^{-1\slash 4}) \nonumber\\
 &\leq& \exp(-cL^{1\slash 2}\beta_n^{-1\slash 4})+CL\exp(-c\beta_n^{-1\slash 2}) \leq CL\exp(-c\beta_n^{-1\slash 4}).
\end{eqnarray}
Note that $LIS(\sigma|_{\mathcal{T}_{s,1}})\leq |\mathcal{I}_{n,s}\cap\mathbb{N}^{*}|\leq L\beta_n^{-1\slash 2}+1\leq 2L\beta_n^{-1\slash 2}$. Hence
\begin{eqnarray}\label{Eq15.1.2}
  \mathbb{E}[LIS(\sigma|_{\mathcal{T}_{s,1}})]&\leq& (2L\beta_n^{-1\slash 2})(CL\exp(-c\beta_n^{-1\slash 4}))+2\sqrt{2C_0}L^{1\slash 2}\beta_n^{-1\slash 4}\nonumber\\
  &\leq& CL^{1\slash 2}\beta_n^{-1\slash 4}+CL^2\exp(-c\beta_n^{-1\slash 4}).
\end{eqnarray}

\paragraph{Step 2}

In this step, we bound $\mathbb{E}[LIS(\sigma|_{\mathcal{T}_{s,2}})]$. For every $t\in [L\slash L_0]$, we let $\tilde{\alpha}_t:=(s-1)L\beta_n^{-1\slash 2}+(t-1\slash 2)L_0\beta_n^{-1\slash 2}-1$. 

We sample $\sigma_0$ from $\tilde{\mathbb{P}}_{n,\beta_n}$. Then sequentially for $t=1,2,\cdots,L\slash L_0$, we run the resampling algorithm for the $L^2$ model (as described at the end of Section \ref{Sect.1.5}) with inputs $\sigma_{t-1},\mathcal{X}_{t,2},\mathcal{Y}_{t,2},\tilde{\alpha}_t$ to obtain $\sigma_t$. We let $\sigma=\sigma_{L\slash L_0}$. By Lemma \ref{L2.2}, the distribution of $\sigma$ is given by $\tilde{\mathbb{P}}_{n,\beta_n}$. 

For any $t\in [L\slash L_0]$, let $\tilde{M}_t\in \mathbb{N}$, $\tilde{I}_{t,1},\cdots,\tilde{I}_{t,\tilde{M}_t}\in \mathcal{X}_{t,2}$, $\tilde{J}_{t,1}, \cdots, \tilde{J}_{t,\tilde{M}_t}\in \mathcal{Y}_{t,2}$ be such that $\tilde{I}_{t,1}<\cdots<\tilde{I}_{t,\tilde{M}_t}$, $\tilde{J}_{t,1}<\cdots<\tilde{J}_{t,\tilde{M}_t}$,
\begin{equation}
     \{i\in \mathcal{X}_{t,2}:\sigma_0(i)\in \mathcal{Y}_{t,2}\}=\{\tilde{I}_{t,1},\cdots,\tilde{I}_{t,\tilde{M}_t}\},
\end{equation}
\begin{equation}
    \{j\in \mathcal{Y}_{t,2}:\sigma_0^{-1}(j)\in \mathcal{X}_{t,2}\}=\{\tilde{J}_{t,1},\cdots,\tilde{J}_{t,\tilde{M}_t}\}.
\end{equation}
For any $t\in [L\slash L_0]$ and any $m\in [n]\backslash [\tilde{M}_t]$, we let $\tilde{I}_{t,m}=0$ and $\tilde{J}_{t,m}=0$. According to the resampling algorithm for the $L^2$ model, $\sigma$ can be generated as follows. Sequentially for $t=1,2,\cdots,L\slash L_0$, we do the following:
\begin{itemize}
    \item For each $m\in [\tilde{M}_t]$, we independently sample $\tilde{u}_{t,m}$ from the uniform distribution on $[0,e^{2\beta_n (\tilde{I}_{t,m}-\tilde{\alpha}_t) \sigma_0(\tilde{I}_{t,m})}]$, and let
    \begin{equation*}
        \tilde{b}_{t,m}=\log(\tilde{u}_{t,m})\slash (2\beta_n (\tilde{I}_{t,m}-\tilde{\alpha}_t) ).
    \end{equation*}
    For each $m\in [n]\backslash [\tilde{M}_t]$, we let $\tilde{b}_{t,m}=0$.
    \item For each $m\in [\tilde{M}_t]$, let 
    \begin{equation}
        \tilde{N}_{t,m}=|\{m'\in  [\tilde{M}_t]: \tilde{b}_{t,m'} \leq \tilde{J}_{t,m} \}|-m+1.
    \end{equation}
   Now look at the $\tilde{N}_{t,1}$ integers $m'\in [\tilde{M}_t]$ with $\tilde{b}_{t,m'}\leq \tilde{J}_{t,1}$, and pick $\tilde{Y}_{t,1}$ uniformly from these integers; then look at the $\tilde{N}_{t,2}$ remaining integers $m'\in [\tilde{M}_t]$ with $\tilde{b}_{t,m'}\leq \tilde{J}_{t,2}$ (with $\tilde{Y}_{t,1}$ deleted from the list), and pick $\tilde{Y}_{t,2}$ uniformly from these integers; and so on. In this way we obtain $\{\tilde{Y}_{t,m}\}_{m\in [\tilde{M}_t]}$. For each $m\in [n]\backslash [\tilde{M}_t]$, we let $\tilde{N}_{t,m}=0$ and $\tilde{Y}_{t,m}=0$.
\end{itemize}
We let $\sigma\in S_n$ be the unique permutation that satisfies the following conditions:
\begin{itemize}
    \item For any $t \in [L\slash L_0]$ and any $m\in [\tilde{M}_t]$, $\sigma(\tilde{I}_{t,\tilde{Y}_{t,m}})=\tilde{J}_{t,m}$.
    \item For any $i\in [n]\backslash \big(\bigcup_{t=1}^{L\slash L_0} \{\tilde{I}_{t,1},\cdots,\tilde{I}_{t,\tilde{M}_t}\} \big)$, $\sigma(i)=\sigma_0(i)$.
\end{itemize}

For any $t\in [L\slash L_0]$, let
\begin{equation}
    \tilde{z}_{t}:=(s-1)L\beta_n^{-1\slash 2}+t L_0\beta_n^{-1\slash 2},\quad \tilde{z}_{t}':=(s-1)L\beta_n^{-1\slash 2}+(t+1\slash 2)L_0\beta_n^{-1\slash 2}.
\end{equation}
Recall Definition \ref{Def2.2}. As $\lceil \tilde{z}_{t}\rceil-1,\lfloor \tilde{z}_{t}'\rfloor\in [n]$, by (\ref{Eq5.1.1}) and Proposition \ref{P2.2.2},
\begin{equation}\label{Eq11.1.1}
    \mathbb{P}(|\mathcal{D}_{\lceil \tilde{z}_{t}\rceil-1}(\sigma_0)|\geq C_1\beta_n^{-1\slash 2}) \leq C\exp(-c\beta_n^{-1\slash 2}),
\end{equation}
\begin{equation}\label{Eq11.1.2}
    \mathbb{P}(|\mathcal{D}'_{\lfloor \tilde{z}_{t}'\rfloor}(\sigma_0)|\geq C_1\beta_n^{-1\slash 2}) \leq C\exp(-c\beta_n^{-1\slash 2}).
\end{equation}
Let $\tilde{\mathcal{Z}}$ be the event that for any $t\in [L\slash L_0]$, $|\mathcal{D}_{\lceil \tilde{z}_{t}\rceil-1}(\sigma_0)|\leq C_1\beta_n^{-1\slash 2}$ and $|\mathcal{D}'_{\lfloor \tilde{z}_{t}'\rfloor}(\sigma_0)|\leq C_1\beta_n^{-1\slash 2}$. By (\ref{Eq11.1.1}), (\ref{Eq11.1.2}), and the union bound, we have 
\begin{equation}\label{Eq11.1.3}
    \mathbb{P}(\tilde{\mathcal{Z}}^c)\leq CL\exp(-c\beta_n^{-1\slash 2}).
\end{equation}
Now for any $t\in [L\slash L_0]$, when the event $\tilde{\mathcal{Z}}$ holds, as $L_0=8C_1$, $L_0\beta_n^{-1\slash 2}\geq 80$, and $\lfloor \tilde{z}_{t}'\rfloor-\lceil \tilde{z}_{t}\rceil \geq L_0\beta_n^{-1\slash 2}\slash 2-2$, we have
\begin{eqnarray}\label{Eq11.1.4}
   && |S(\sigma_0)\cap [\tilde{z}_{t},\tilde{z}'_{t}]^2|\geq |S(\sigma_0)\cap [\lceil \tilde{z}_{t}\rceil ,\lfloor \tilde{z}'_{t}\rfloor ]^2| \nonumber\\ 
   &\geq& |[\lceil \tilde{z}_{t}\rceil ,\lfloor \tilde{z}'_{t}\rfloor]\cap\mathbb{N}^{*}|-|\mathcal{D}_{\lceil \tilde{z}_{t}\rceil-1}(\sigma_0)|-|\mathcal{D}'_{\lfloor \tilde{z}_{t}'\rfloor}(\sigma_0)|\nonumber\\
   &\geq&  \lfloor \tilde{z}'_{t}\rfloor-\lceil \tilde{z}_{t}\rceil+1-2C_1\beta_n^{-1\slash 2}\geq \frac{1}{2}L_0\beta_n^{-1\slash 2}-1-\frac{1}{4}L_0\beta_n^{-1\slash 2}\nonumber\\
   &=& \frac{1}{4}L_0\beta_n^{-1\slash 2}-1 \geq \frac{1}{8}L_0\beta_n^{-1\slash 2}.
\end{eqnarray}

For any $t\in [L\slash L_0]$, we let $\tilde{\mathscr{M}}_t$ be the set of $m\in [\tilde{M}_t]$ that satisfies 
\begin{equation}\label{Eq11.1.5}
    \tilde{J}_{t,m}\in ((s-1)L\beta_n^{-1\slash 2}+(t-1\slash 2)L_0\beta_n^{-1\slash 2}, (s-1)L\beta_n^{-1\slash 2}+t L_0\beta_n^{-1\slash 2}].
\end{equation}
Below we consider any $t\in [L\slash L_0]$ and $m\in [n]$. If $m\in \tilde{\mathscr{M}}_t$, for any $m'\in [\tilde{M}_t]$ such that $\sigma_0(\tilde{I}_{t,m'})< \tilde{J}_{t,m}$ (note that there are $m-1$ such $m'$), we have that $\tilde{b}_{t,m'}\leq \sigma_0(\tilde{I}_{t,m'})< \tilde{J}_{t,m}$, hence
\begin{equation}\label{Eq11.1.6}
    \tilde{N}_{t,m}=\sum_{\substack{m'\in [\tilde{M}_t]: \\ \sigma_0(\tilde{I}_{t,m'})\geq \tilde{J}_{t,m}}} \mathbbm{1}_{\tilde{b}_{t,m'}\leq \tilde{J}_{t,m}}.
\end{equation}
For any $i\in [n]$ such that $(i,\sigma_0(i))\in [\tilde{z}_{t}, \tilde{z}'_{t}]^2$, we have $(i,\sigma_0(i))\in \mathcal{X}_{t,2}\times \mathcal{Y}_{t,2}$. Hence there exists some $m'\in [\tilde{M}_t]$, such that $i=\tilde{I}_{t,m'}$. Let
\begin{equation}
    \tilde{\mathcal{M}}_t:=\{m'\in [\tilde{M}_t]: (\tilde{I}_{t,m'},\sigma_0(\tilde{I}_{t,m'}))\in [\tilde{z}_{t},\tilde{z}'_{t}]^2\}.
\end{equation}
By (\ref{Eq11.1.4}), when the event $\tilde{\mathcal{Z}}$ holds, we have
\begin{equation}\label{Eq11.1.7}
    |\tilde{\mathcal{M}}_t| \geq |S(\sigma_0)\cap [\tilde{z}_{t},\tilde{z}'_{t}]^2|\geq \frac{1}{8}L_0\beta_n^{-1\slash 2}\geq \beta_n^{-1\slash 2}.
\end{equation}
If $m \in \tilde{\mathscr{M}}_t$, for any $m'\in \tilde{\mathcal{M}}_t$, we have $\sigma_0(\tilde{I}_{t,m'})\geq \tilde{z}_{t}\geq \tilde{J}_{t,m}$ (note (\ref{Eq11.1.5})). Hence by (\ref{Eq11.1.6}),
\begin{equation}\label{Eq11.1.8}
    \tilde{N}_{t,m} \geq \sum_{m'\in \tilde{\mathcal{M}}_t} \mathbbm{1}_{\tilde{b}_{t,m'}\leq \tilde{J}_{t,m}}.
\end{equation}
Now note that if $m\in \tilde{\mathscr{M}}_t$, conditional on $\sigma_0$, $\{\mathbbm{1}_{\tilde{b}_{t,m'}\leq \tilde{J}_{t,m}}\}_{m'\in \tilde{\mathcal{M}}_t}$ are mutually independent, and for each $m'\in \tilde{\mathcal{M}}_t$, $\mathbbm{1}_{\tilde{b}_{t,m'}\leq \tilde{J}_{t,m}}$ follows the Bernoulli distribution with 
\begin{eqnarray}\label{Eq11.1.9}
   \mathbb{P}(\mathbbm{1}_{\tilde{b}_{t,m'}\leq \tilde{J}_{t,m}}=1|\sigma_0)&=&\mathbb{P}(\tilde{b}_{t,m'}\leq \tilde{J}_{t,m}|\sigma_0)=\mathbb{P}(\tilde{u}_{t,m'}\leq e^{2\beta_n (\tilde{I}_{t,m'}-\tilde{\alpha}_t)\tilde{J}_{t,m}}|\sigma_0) \nonumber\\
   &=& e^{-2\beta_n (\tilde{I}_{t,m'}-\tilde{\alpha}_t)(\sigma_0(\tilde{I}_{t,m'})-\tilde{J}_{t,m})}.
\end{eqnarray}
If $m \in \tilde{\mathscr{M}}_t$, for any $m'\in \tilde{\mathcal{M}}_t$, as $(\tilde{I}_{t,m'},\sigma_0(\tilde{I}_{t,m'}))\in [\tilde{z}_{t},\tilde{z}_{t}']^2$, by (\ref{Eq11.1.5}), we have 
\begin{equation*}
    \tilde{I}_{t,m'}-\tilde{\alpha}_t\begin{cases}
     \geq \tilde{z}_{t}-\tilde{\alpha}_t\geq 1\\
     \leq \tilde{z}'_{t}-\tilde{\alpha}_t= L_0\beta_n^{-1\slash 2}+1\leq 2L_0\beta_n^{-1\slash 2}
    \end{cases},
\end{equation*}
\begin{equation*}
   0\leq  \sigma_0(\tilde{I}_{t,m'})-\tilde{J}_{t,m} \leq \tilde{z}'_{t}-\tilde{J}_{t,m}\leq L_0\beta_n^{-1\slash 2},
\end{equation*}
hence by (\ref{Eq11.1.9}), we have
\begin{equation}\label{Eq11.1.10}
    \mathbb{P}(\mathbbm{1}_{\tilde{b}_{t,m'}\leq \tilde{J}_{t,m}}=1|\sigma_0)\geq e^{-4L_0^2}.
\end{equation}
By (\ref{Eq11.1.8}), (\ref{Eq11.1.10}), and Hoeffding's inequality, for any $x\in [0,e^{-4L_0^2}]$, we have
\begin{equation*}
    \mathbb{P}(\tilde{N}_{t,m}\leq (e^{-4L_0^2}-x)|\tilde{\mathcal{M}}_t||\sigma_0)\mathbbm{1}_{m\in \tilde{\mathscr{M}}_t}\leq e^{-2|\tilde{\mathcal{M}}_t|x^2}\mathbbm{1}_{m\in \tilde{\mathscr{M}}_t},
\end{equation*}
which by (\ref{Eq11.1.7}) leads to 
\begin{equation*}
    \mathbb{P}(\tilde{N}_{t,m}\leq (e^{-4L_0^2}-x)\beta_n^{-1\slash 2}|\sigma_0) \mathbbm{1}_{m\in \tilde{\mathscr{M}}_t} \mathbbm{1}_{\tilde{\mathcal{Z}}}  \leq e^{-2\beta_n^{-1\slash 2}x^2}\mathbbm{1}_{m\in \tilde{\mathscr{M}}_t}.
\end{equation*}
Taking $x=e^{-4L_0^2}\slash 2$, we have
\begin{eqnarray}\label{Eq11.1.11}
  && \mathbb{P}(\{\tilde{N}_{t,m}\leq e^{-4L_0^2}\beta_n^{-1\slash 2} \slash 2\}\cap\{m\in \tilde{\mathscr{M}}_t\}\cap\tilde{\mathcal{Z}}|\sigma_0)\nonumber\\
   &=& \mathbb{P}(\tilde{N}_{t,m}\leq e^{-4L_0^2}\beta_n^{-1\slash 2} \slash 2|\sigma_0) \mathbbm{1}_{m\in \tilde{\mathscr{M}}_t} \mathbbm{1}_{\tilde{\mathcal{Z}}} \leq e^{-c\beta_n^{-1\slash 2}}\mathbbm{1}_{m\in \tilde{\mathscr{M}}_t}. 
\end{eqnarray}

For any $t\in [L\slash L_0]$, let $\tilde{\mathcal{C}}_t$ be the event that $\tilde{N}_{t,m}\geq e^{-4L_0^2}\beta_n^{-1\slash 2} \slash 2$ for any $m\in \tilde{\mathscr{M}}_t$. By (\ref{Eq11.1.11}) and the union bound,
\begin{eqnarray}
&&\mathbb{P}(\tilde{\mathcal{C}}_t^c\cap\tilde{\mathcal{Z}}|\sigma_0)\leq 
\mathbb{P}\Big(\bigcup_{m=1}^n\big(\{\tilde{N}_{t,m}\leq e^{-4L_0^2}\beta_n^{-1\slash 2} \slash 2\}\cap\{m\in \tilde{\mathscr{M}}_t\}\cap\tilde{\mathcal{Z}}\big)\Big|\sigma_0\Big)\nonumber\\
&\leq& \sum_{m=1}^n \mathbb{P}(\{\tilde{N}_{t,m}\leq e^{-4L_0^2}\beta_n^{-1\slash 2} \slash 2\}\cap\{m\in \tilde{\mathscr{M}}_t\}\cap\tilde{\mathcal{Z}}|\sigma_0)\nonumber\\ &\leq& e^{-c\beta_n^{-1\slash 2}}\sum_{m=1}^n \mathbbm{1}_{m\in \tilde{\mathscr{M}}_t}=|\tilde{\mathscr{M}}_t| e^{-c\beta_n^{-1\slash 2}}
\leq |\mathcal{Y}_{t,2}|e^{-c\beta_n^{-1\slash 2}}\nonumber\\
&\leq& (L_0\beta_n^{-1\slash 2}+1)e^{-c\beta_n^{-1\slash 2}}\leq C\beta_n^{-1\slash 2}e^{-c\beta_n^{-1\slash 2}}\leq C\exp(-c\beta_n^{-1\slash 2}).
\end{eqnarray}
Hence 
\begin{equation}\label{Eq11.1.12}
    \mathbb{P}(\tilde{\mathcal{C}}_t^c\cap\tilde{\mathcal{Z}})=\mathbb{E}[\mathbb{P}(\tilde{\mathcal{C}}_t^c\cap\tilde{\mathcal{Z}}|\sigma_0)]\leq C\exp(-c\beta_n^{-1\slash 2}).
\end{equation}
Let $\tilde{\mathcal{C}}:=\bigcap_{t=1}^{L\slash L_0}\tilde{\mathcal{C}}_t$. By (\ref{Eq11.1.3}), (\ref{Eq11.1.12}), and the union bound, 
\begin{equation}\label{Eq12.1.2}
    \mathbb{P}(\tilde{\mathcal{C}}^c)\leq CL\exp(-c\beta_n^{-1\slash 2}).
\end{equation}

Let $\tilde{W}_s$ be the set of $i\in [n]$ such that
\begin{equation}
     (i,\sigma_0(i))\in (sL\beta_n^{-1\slash 2},n]\times ((s-1)L\beta_n^{-1\slash 2},sL\beta_n^{-1\slash 2}+L_0\beta_n^{-1\slash 2}\slash 2].
\end{equation}
Note that 
\begin{eqnarray}
   && (sL\beta_n^{-1\slash 2},n]\times ((s-1)L\beta_n^{-1\slash 2},sL\beta_n^{-1\slash 2}+L_0\beta_n^{-1\slash 2}\slash 2]\nonumber\\
   && \subseteq   \mathcal{R}_s'\cup ((0,n]\times (sL\beta_n^{-1\slash 2},sL\beta_n^{-1\slash 2}+L_0\beta_n^{-1\slash 2}\slash 2]).
\end{eqnarray}
Recalling the definition of $W_s$ from (\ref{Eq14.1.1}), we have
\begin{eqnarray}\label{Eq14.1.2}
   && |\tilde{W}_s|\leq |W_s|+|(sL\beta_n^{-1\slash 2},sL\beta_n^{-1\slash 2}+L_0\beta_n^{-1\slash 2}\slash 2]\cap\mathbb{N}^{*}| \nonumber\\
   &\leq& |W_s|+\frac{1}{2}L_0\beta_n^{-1\slash 2}+1\leq |W_s|+5C_1\beta_n^{-1\slash 2}. 
\end{eqnarray}
Let $\tilde{\mathcal{W}}_s$ be the event that $|\tilde{W}_s|\leq 6C_1\beta_n^{-1\slash 2}$. By (\ref{Eq10.1.4}) and (\ref{Eq14.1.2}), we have
\begin{equation}\label{Eq12.1.1}
    \mathbb{P}(\tilde{\mathcal{W}}_s^c)\leq \mathbb{P}(\mathcal{W}_s^c)\leq C\exp(-c\beta_n^{-1\slash 2}).
\end{equation}

For any $q\in \mathbb{N}^{*}$, we let
\begin{equation}\label{Eq12.1.3}
    \Lambda_{s,q,2}:=\sum_{\substack{ i_1<\cdots<i_q, j_1<\cdots<j_q\\i_1,\cdots,i_q\in (sL\beta_n^{-1\slash 2},n]\cap\mathbb{N}^{*}\\j_1,\cdots,j_q\in \mathcal{J}_{s,2}\cap\mathbb{N}^{*}}} \mathbbm{1}_{\sigma(i_1)=j_1,\cdots,\sigma(i_q)=j_q} \mathbbm{1}_{i_1,\cdots,i_q\in \tilde{W}_s}.
\end{equation}
In the following, we bound $\Lambda_{s,q,2}$ for any $q\in \mathbb{N}^{*}$.

Consider any $t\in [L\slash L_0]$. For any $m\in [\tilde{M}_t]$, we let $\tilde{\mathscr{N}}_{\tilde{J}_{t,m}}=\tilde{N}_{t,m}$; for any $j\in \mathcal{Y}_{t,2}\backslash \{\tilde{J}_{t,1},\cdots,\tilde{J}_{t,\tilde{M}_t}\}$, we let $\tilde{\mathscr{N}}_{j}=n$. For any $m\in [\tilde{M}_t]$, we let $\tilde{\mathscr{Y}}_{\tilde{J}_{t,m}}=\tilde{I}_{t,\tilde{Y}_{t,m}}$; for any $j\in \mathcal{Y}_{t,2}\backslash \{\tilde{J}_{t,1},\cdots,\tilde{J}_{t,\tilde{M}_t}\}$, we let $\tilde{\mathscr{Y}}_{j}=0$. 

We let $\tilde{\mathcal{B}}$ be the $\sigma$-algebra generated by $\sigma_0$ and $\{\tilde{b}_{t,m}\}_{t\in [L\slash L_0], m\in [n]}$. For any $j\in ((s-1)L\beta_n^{-1\slash 2}+L_0 \beta_n^{-1\slash 2}\slash 2,sL\beta_n^{-1\slash 2}+L_0  \beta_n^{-1\slash 2}\slash 2]\cap\mathbb{N}^{*}$, we let $\tilde{\mathcal{F}}_{j}$ be the $\sigma$-algebra generated by $\sigma_0$, $\{\tilde{b}_{t,m}\}_{t\in [L\slash L_0],m\in [n]}$, and
\begin{equation*}
 \{\tilde{\mathscr{Y}}_{l}\}_{l\in [j-1]\cap ((s-1)L\beta_n^{-1\slash 2}+L_0  \beta_n^{-1\slash 2}\slash 2,sL\beta_n^{-1\slash 2}+L_0  \beta_n^{-1\slash 2}\slash 2]\cap\mathbb{N}^{*}}.
\end{equation*}

We assume that the event $\tilde{\mathcal{C}}$ holds. For any $t\in [L\slash L_0]$ and any $m\in \tilde{\mathscr{M}}_t$, we have $\tilde{\mathscr{N}}_{\tilde{J}_{t,m}}=\tilde{N}_{t,m}\geq e^{-4L_0^2}\beta_n^{-1\slash 2}\slash 2$. Hence for any $t\in [L\slash L_0]$ and any $j$ from 
\begin{eqnarray*}
  &&  \{\tilde{J}_{t,m}:m\in [\tilde{M}_t]\}\nonumber\\
  &&\cap ((s-1)L\beta_n^{-1\slash 2}+(t-1\slash 2)L_0\beta_n^{-1\slash 2}, (s-1)L\beta_n^{-1\slash 2}+t L_0\beta_n^{-1\slash 2}]\cap\mathbb{N}^{*},
\end{eqnarray*}
we have $\tilde{\mathscr{N}}_j\geq e^{-4L_0^2}\beta_n^{-1\slash 2}\slash 2$. Now for any $t\in [L\slash L_0]$ and any $j$ from
\begin{eqnarray*}
   &&  \{\tilde{J}_{t,m}:m\in [\tilde{M}_t]\}^c\nonumber\\
  &&\cap ((s-1)L\beta_n^{-1\slash 2}+(t-1\slash 2)L_0\beta_n^{-1\slash 2}, (s-1)L\beta_n^{-1\slash 2}+t L_0\beta_n^{-1\slash 2}]\cap\mathbb{N}^{*},
\end{eqnarray*}
we have $\tilde{\mathscr{N}}_j=n\geq e^{-4L_0^2}\beta_n^{-1\slash 2}\slash 2$ (note that $n\beta_n^{1\slash 2}\geq 4L\geq 4$). Hence for any $j$ from the set
\begin{eqnarray*}
   && \bigcup_{t=1}^{L\slash L_0} (((s-1)L\beta_n^{-1\slash 2}+(t-1\slash 2)L_0\beta_n^{-1\slash 2}, (s-1)L\beta_n^{-1\slash 2}+t L_0\beta_n^{-1\slash 2}]\cap\mathbb{N}^{*})\nonumber\\
    && =\mathcal{J}_{s,2}\cap\mathbb{N}^{*},
\end{eqnarray*}
we have 
\begin{equation}\label{Eq12.1.4}
    \tilde{\mathscr{N}}_j\geq e^{-4L_0^2}\beta_n^{-1\slash 2}\slash 2.
\end{equation}

Consider any $i_1,\cdots,i_q\in (sL\beta_n^{-1\slash 2},n]\cap\mathbb{N}^{*}$ and $j_1,\cdots,j_q\in \mathcal{J}_{s,2}\cap\mathbb{N}^{*}$ such that $i_1<\cdots<i_q$ and $j_1<\cdots<j_q$. Note that 
\begin{equation}\label{Eq12.1.5}
    ((sL\beta_n^{-1\slash 2},n]\cap\mathbb{N}^{*})\times (\mathcal{J}_{s,2}\cap\mathbb{N}^{*})\subseteq \bigcup_{t=1}^{L\slash L_0} \mathcal{X}_{t,2}\times \mathcal{Y}_{t,2}. 
\end{equation}
Hence for any $l\in [q]$, if $\sigma(i_l)=j_l$, then there exists some $t\in [L \slash L_0]$, such that $(i_l,\sigma(i_l))=(i_l,j_l)\in \mathcal{X}_{t,2}\times \mathcal{Y}_{t,2}$; this implies $\tilde{\mathscr{Y}}_{j_l}=i_l$. Hence we have
\begin{eqnarray}\label{Eq12.1.6}
  &&  \mathbb{E}[\mathbbm{1}_{\sigma(i_1)=j_1,\cdots,\sigma(i_q)=j_q} \mathbbm{1}_{i_1,\cdots,i_q\in \tilde{W}_s}|\tilde{\mathcal{B}}]\leq \mathbb{E}[\mathbbm{1}_{\tilde{\mathscr{Y}}_{j_1}=i_1,\cdots,\tilde{\mathscr{Y}}_{j_q}=i_q}\mathbbm{1}_{i_1,\cdots,i_q\in \tilde{W}_s}|\tilde{\mathcal{B}}]\nonumber\\
  &=& \mathbbm{1}_{i_1,\cdots,i_q\in \tilde{W}_s}\mathbb{E}[\mathbbm{1}_{\tilde{\mathscr{Y}}_{j_1}=i_1,\cdots,\tilde{\mathscr{Y}}_{j_q}=i_q}|\tilde{\mathcal{B}}]\nonumber\\
  &=&\mathbbm{1}_{i_1,\cdots,i_q\in \tilde{W}_s}\mathbb{E}[\mathbb{E}[\mathbbm{1}_{\tilde{\mathscr{Y}}_{j_q}=i_q}|\tilde{\mathcal{F}}_{j_q}]\mathbbm{1}_{\tilde{\mathscr{Y}}_{j_1}=i_1,\cdots,\tilde{\mathscr{Y}}_{j_{q-1}}=i_{q-1}}|\tilde{\mathcal{B}}]\nonumber\\
  &\leq& \frac{\mathbbm{1}_{i_1,\cdots,i_q\in \tilde{W}_s}}{\tilde{\mathscr{N}}_{j_q}}\mathbb{E}[\mathbbm{1}_{\tilde{\mathscr{Y}}_{j_1}=i_1,\cdots,\tilde{\mathscr{Y}}_{j_{q-1}}=i_{q-1}}|\tilde{\mathcal{B}}]\leq \cdots\leq\frac{\mathbbm{1}_{i_1,\cdots,i_q\in \tilde{W}_s}}{\tilde{\mathscr{N}}_{j_1}\tilde{\mathscr{N}}_{j_2}\cdots \tilde{\mathscr{N}}_{j_q}}. 
\end{eqnarray}
By (\ref{Eq12.1.4}) and (\ref{Eq12.1.6}), we have
\begin{eqnarray}\label{Eq12.1.7}
  &&  \mathbb{E}[\mathbbm{1}_{\sigma(i_1)=j_1,\cdots,\sigma(i_q)=j_q} \mathbbm{1}_{i_1,\cdots,i_q\in \tilde{W}_s}|\tilde{\mathcal{B}}]\mathbbm{1}_{\tilde{\mathcal{C}}\cap \tilde{\mathcal{W}}_s}\nonumber\\
  &\leq& \frac{\mathbbm{1}_{i_1,\cdots,i_q\in \tilde{W}_s}\mathbbm{1}_{\tilde{\mathcal{C}}\cap\tilde{\mathcal{W}}_s}}{\tilde{\mathscr{N}}_{j_1}\tilde{\mathscr{N}}_{j_2}\cdots \tilde{\mathscr{N}}_{j_q}}\leq (2e^{4L_0^2}\beta_n^{1\slash 2})^q \mathbbm{1}_{i_1,\cdots,i_q\in \tilde{W}_s}\mathbbm{1}_{\tilde{\mathcal{W}}_s}.
\end{eqnarray}

By (\ref{Eq12.1.3}), (\ref{Eq12.1.7}), and Lemma \ref{Lemma3.1}, we have
\begin{eqnarray*}
   &&  \mathbb{E}[\Lambda_{s,q,2}|\tilde{\mathcal{B}}]\mathbbm{1}_{\tilde{\mathcal{C}}\cap\tilde{\mathcal{W}}_s}\leq (2e^{4L_0^2}\beta_n^{1\slash 2})^q\mathbbm{1}_{\tilde{\mathcal{W}}_s} \sum_{\substack{ i_1<\cdots<i_q, j_1<\cdots<j_q\\i_1,\cdots,i_q\in (sL\beta_n^{-1\slash 2},n]\cap\mathbb{N}^{*}\\j_1,\cdots,j_q\in \mathcal{J}_{s,2}\cap\mathbb{N}^{*}}}\mathbbm{1}_{i_1,\cdots,i_q\in \tilde{W}_s} \nonumber\\
   &\leq& (2e^{4L_0^2}\beta_n^{1\slash 2})^q\binom{|\tilde{W}_s|}{q}\binom{|\mathcal{J}_{s,2}\cap\mathbb{N}^{*}|}{q}\mathbbm{1}_{\tilde{\mathcal{W}}_s}\nonumber\\
   &\leq& \Big(\frac{2e^{2+4L_0^2}\beta_n^{1\slash 2}|\tilde{W}_s||\mathcal{J}_{s,2}\cap\mathbb{N}^{*}|}{q^2}\Big)^q \mathbbm{1}_{\tilde{\mathcal{W}}_s}\leq (CL\beta_n^{-1\slash 2}q^{-2} )^q,
\end{eqnarray*}
where we use the fact that
\begin{equation*}
    |\mathcal{J}_{s,2}\cap\mathbb{N}^{*}|\leq |((s-1)L\beta_n^{-1\slash 2},sL\beta_n^{-1\slash 2}]\cap\mathbb{N}^{*}|\leq L\beta_n^{-1\slash 2}+1\leq 2L\beta_n^{-1\slash 2}
\end{equation*}
in the last line. Hence
\begin{equation}\label{Eq12.1.8}
    \mathbb{E}[\Lambda_{s,q,2}\mathbbm{1}_{\tilde{\mathcal{C}}\cap\tilde{\mathcal{W}}_s}]=\mathbb{E}[\mathbb{E}[\Lambda_{s,q,2}|\tilde{\mathcal{B}}]\mathbbm{1}_{\tilde{\mathcal{C}}\cap\tilde{\mathcal{W}}_s}]\leq (CL\beta_n^{-1\slash 2}q^{-2} )^q.
\end{equation}

Now for any $q\in \mathbb{N}^{*}$, if $LIS(\sigma|_{\mathcal{T}_{s,2}})\geq q$, then there exist
\begin{equation*}
    i_1,\cdots,i_q\in (sL\beta_n^{-1\slash 2},n]\cap\mathbb{N}^{*}, \quad j_1,\cdots,j_q\in \mathcal{J}_{s,2}\cap\mathbb{N}^{*},
\end{equation*}
such that $i_1<\cdots<i_q$, $j_1<\cdots<j_q$, and $\sigma(i_l)=j_l$ for every $l \in [q]$. For any $l\in [q]$, by (\ref{Eq12.1.5}), we have $(i_l,\sigma(i_l))=(i_l,j_l)\in \mathcal{X}_{t,2} \times \mathcal{Y}_{t,2} $ for some $t\in [L\slash L_0]$, hence $(i_l,\sigma_0(i_l))\in \mathcal{X}_{t,2} \times \mathcal{Y}_{t,2}$ and
\begin{equation*}
    \sigma_0(i_l)\in \mathcal{Y}_{t,2}\subseteq ((s-1)L\beta_n^{-1\slash 2},sL\beta_n^{-1\slash 2}+L_0\beta_n^{-1\slash 2}\slash 2];
\end{equation*}
as $i_l\in (sL\beta_n^{-1\slash 2},n]\cap\mathbb{N}^{*}$, we have
\begin{equation*}
    (i_l,\sigma_0(i_l))\in (sL\beta_n^{-1\slash 2},n]\times ((s-1)L\beta_n^{-1\slash 2},sL\beta_n^{-1\slash 2}+L_0\beta_n^{-1\slash 2}\slash 2],
\end{equation*}
hence $i_l\in \tilde{W}_s$. Hence $\Lambda_{s,q,2}\geq 1$. We conclude that for any $q\in \mathbb{N}^{*}$,
\begin{equation}\label{Eq12.1.9}
    \{LIS(\sigma|_{\mathcal{T}_{s,2}})\geq q\}\subseteq \{\Lambda_{s,q,2}\geq  1\}.
\end{equation}

By (\ref{Eq12.1.8}) and (\ref{Eq12.1.9}), for any $q\in\mathbb{N}^{*}$, we have
\begin{eqnarray}
&&\mathbb{P}(\{LIS(\sigma|_{\mathcal{T}_{s,2}})\geq q\}\cap\tilde{\mathcal{C}}\cap\tilde{\mathcal{W}}_s)\leq \mathbb{P}(\{\Lambda_{s,q,2}\geq  1\}\cap \tilde{\mathcal{C}}\cap\tilde{\mathcal{W}}_s)\nonumber\\
&=& \mathbb{E}[\mathbbm{1}_{\Lambda_{s,q,2}\geq  1}\mathbbm{1}_{\tilde{\mathcal{C}}\cap\tilde{\mathcal{W}}_s}]\leq \mathbb{E}[\Lambda_{s,q,2}\mathbbm{1}_{\tilde{\mathcal{C}}\cap\tilde{\mathcal{W}}_s}] \leq (C_0' L\beta_n^{-1\slash 2}q^{-2} )^q,
\end{eqnarray}
where $C_0' \geq 1$ is a positive absolute constant. Taking $q=\lceil \sqrt{2C_0'} L^{1\slash 2}\beta_n^{-1\slash 4} \rceil$, we obtain that
\begin{eqnarray}\label{Eq12.1.10}
 && \mathbb{P}(\{LIS(\sigma|_{\mathcal{T}_{s,2}})\geq 2\sqrt{2C_0'}L^{1\slash 2}\beta_n^{-1\slash 4}\}\cap\tilde{\mathcal{C}}\cap\tilde{\mathcal{W}}_s) \nonumber\\
 &\leq& \mathbb{P}(\{LIS(\sigma|_{\mathcal{T}_{s,2}})\geq \lceil \sqrt{2C_0'}L^{1\slash 2}\beta_n^{-1\slash 4} \rceil \}\cap\tilde{\mathcal{C}}\cap\tilde{\mathcal{W}}_s)\nonumber\\
 &\leq& 2^{-\lceil \sqrt{2C_0'} L^{1\slash 2}\beta_n^{-1\slash 4} \rceil}\leq \exp(-cL^{1\slash 2}\beta_n^{-1\slash 4}).
\end{eqnarray}
By (\ref{Eq12.1.2}), (\ref{Eq12.1.1}), (\ref{Eq12.1.10}), and the union bound, we have
\begin{eqnarray}
 && \mathbb{P}(LIS(\sigma|_{\mathcal{T}_{s,2}})\geq 2\sqrt{2C_0'}L^{1\slash 2}\beta_n^{-1\slash 4}) \nonumber\\
 &\leq& \exp(-cL^{1\slash 2}\beta_n^{-1\slash 4})+CL\exp(-c\beta_n^{-1\slash 2}) \leq CL\exp(-c\beta_n^{-1\slash 4}).
\end{eqnarray}
Note that $LIS(\sigma|_{\mathcal{T}_{s,2}})\leq |\mathcal{I}_{n,s}\cap\mathbb{N}^{*}|\leq L\beta_n^{-1\slash 2}+1\leq 2L\beta_n^{-1\slash 2}$. Hence
\begin{eqnarray}\label{Eq15.1.3}
  \mathbb{E}[LIS(\sigma|_{\mathcal{T}_{s,2}})]&\leq& (2L\beta_n^{-1\slash 2})(CL\exp(-c\beta_n^{-1\slash 4}))+2\sqrt{2C_0'}L^{1\slash 2}\beta_n^{-1\slash 4}\nonumber\\
  &\leq& CL^{1\slash 2}\beta_n^{-1\slash 4}+CL^2\exp(-c\beta_n^{-1\slash 4}).
\end{eqnarray}

\bigskip

By (\ref{Eq15.1.1}), (\ref{Eq15.1.2}), and (\ref{Eq15.1.3}), we conclude that \begin{eqnarray}
    \mathbb{E}[LIS(\sigma|_{\mathcal{R}_s'})]&\leq& \mathbb{E}[LIS(\sigma|_{\mathcal{T}_{s,1}})]+\mathbb{E}[LIS(\sigma|_{\mathcal{T}_{s,2}})] \nonumber\\
    &\leq& CL^{1\slash 2}\beta_n^{-1\slash 4}+CL^2\exp(-c\beta_n^{-1\slash 4}).
\end{eqnarray}

\bigskip

In the following, we show (\ref{Eq8.1.6}). Let $\sigma$ be drawn from $\tilde{\mathbb{P}}_{n,\beta_n}$. Note that the distribution of $\sigma^{-1}$ is given by $\tilde{\mathbb{P}}_{n,\beta_n}$, and $LIS(\sigma^{-1}|_{\mathcal{R}_s'})=LIS(\sigma|_{\mathcal{R}_s''})$. Hence by (\ref{Eq7.1.1}),
\begin{equation}
    \mathbb{E}[LIS(\sigma|_{\mathcal{R}_s''})]=\mathbb{E}[LIS(\sigma^{-1}|_{\mathcal{R}_s'})]\leq CL^{1\slash 2}\beta_n^{-1\slash 4}+CL^2\exp(-c\beta_n^{-1\slash 4}).
\end{equation}

\end{proof}

The following proposition bounds $LIS(\sigma|_{\mathcal{R}_s})$ for $\sigma$ drawn from $\tilde{\mathbb{P}}_{n,\beta_n}$ and any $s\in [\lfloor n\beta_n^{1\slash 2} \slash L \rfloor ]$.

\begin{proposition}\label{P5.2}
Assume that $n\beta_n^{1\slash 2}\geq 4L$ and $\beta_n\leq 1\slash 100$, and let $\sigma$ be drawn from $\tilde{\mathbb{P}}_{n,\beta_n}$. Then there exist positive absolute constants $C,c$, such that for any $s\in [\lfloor n\beta_n^{1\slash 2}\slash L\rfloor]$, 
\begin{equation}\label{Eq15.2.1}
    \mathbb{E}[LIS(\sigma|_{\mathcal{R}_s})]\leq CL\beta_n^{-1\slash 4}+CL^2\exp(-c\beta_n^{-1\slash 4}).
\end{equation}
\end{proposition}

\begin{proof}

We start by showing (\ref{Eq15.2.1}) for any $s\in [\lfloor n\beta_n^{1\slash 2}\slash L\rfloor-1]$. In the following, we fix an arbitrary $s\in [\lfloor n\beta_n^{1\slash 2}\slash L\rfloor-1]$. 

\paragraph{Step 1}

For any $t\in [ L\slash L_0 ]$, we let
\begin{equation}
    \mathcal{X}_{t}:=((s-1)L\beta_n^{-1\slash 2}+(t-1)L_0\beta_n^{-1\slash 2},n]\cap\mathbb{N}^{*},
\end{equation}
\begin{equation}
    \mathcal{Y}_{t}:=((s-1)L\beta_n^{-1\slash 2}+(t-1)L_0\beta_n^{-1\slash 2},(s-1)L\beta_n^{-1\slash 2}+t L_0\beta_n^{-1\slash 2}]\cap\mathbb{N}^{*}.
\end{equation}
\begin{equation}
    \mathcal{J}_t:=((s-1)L\beta_n^{-1\slash 2}+(t-1)L_0\beta_n^{-1\slash 2},(s-1)L\beta_n^{-1\slash 2}+(t+1) L_0\beta_n^{-1\slash 2}]\cap\mathbb{N}^{*}.
\end{equation}
Note that $(s-1)L\beta_n^{-1\slash 2}+(t+1) L_0\beta_n^{-1\slash 2}\leq (s+1)L\beta_n^{-1\slash 2}\leq n$. We also let $\mathcal{Q}_{t}:=\mathcal{X}_{t}\times \mathcal{Y}_{t}$ and $\alpha_t:=(s-1)L\beta_n^{-1\slash 2}+(t-1)L_0\beta_n^{-1\slash 2}-1$. 

In the following, we fix an arbitrary $t\in [L\slash L_0]$, and bound $LIS(\sigma|_{\mathcal{Q}_{t}})$. 

We first sample $\sigma_0$ from $\tilde{\mathbb{P}}_{n,\beta_n}$, and then run the resampling algorithm for the $L^2$ model (as described at the end of Section \ref{Sect.1.5}) with inputs $\sigma_0,\mathcal{X}_{t},\mathcal{J}_{t},\alpha_t$. By Lemma \ref{L2.2}, the distribution of $\sigma$ is given by $\tilde{\mathbb{P}}_{n,\beta_n}$. 

Let $M_t\in\mathbb{N}$, $I_{t,1},\cdots,I_{t,M_t}\in \mathcal{X}_t$, and $J_{t,1},\cdots,J_{t,M_t}\in \mathcal{J}_t$ be such that $I_{t,1}<\cdots<I_{t,M_t}$, $J_{t,1}<\cdots<J_{t,M_t}$, 
\begin{equation}
     \{i\in \mathcal{X}_{t}:\sigma_0(i)\in \mathcal{J}_{t}\}=\{I_{t,1},\cdots,I_{t,M_t}\},
\end{equation}
\begin{equation}
    \{j\in \mathcal{J}_{t}:\sigma_0^{-1}(j)\in \mathcal{X}_{t}\}=\{J_{t,1},\cdots,J_{t,M_t}\}.
\end{equation}
For any $m\in [n]\backslash [M_t]$, we let $I_{t,m}=0$ and $J_{t,m}=0$. According to the resampling algorithm for the $L^2$ model, $\sigma$ can be generated as follows:
\begin{itemize}
    \item For each $m\in [M_t]$, we independently sample $u_{t,m}$ from the uniform distribution on $[0,e^{2\beta_n (I_{t,m}-\alpha_t) \sigma_0(I_{t,m})}]$, and let
    \begin{equation*}
        b_{t,m}=\log(u_{t,m})\slash (2\beta_n (I_{t,m}-\alpha_t) ).
    \end{equation*}
    For each $m\in [n]\backslash [M_t]$, we let $b_{t,m}=0$.
    \item For each $m\in [M_t]$, let 
    \begin{equation}
        N_{t,m}=|\{m'\in  [M_t]: b_{t,m'} \leq J_{t,m} \}|-m+1.
    \end{equation}
   Now look at the $N_{t,1}$ integers $m'\in [M_t]$ with $b_{t,m'}\leq J_{t,1}$, and pick $Y_{t,1}$ uniformly from these integers; then look at the $N_{t,2}$ remaining integers $m'\in [M_t]$ with $b_{t,m'}\leq J_{t,2}$ (with $Y_{t,1}$ deleted from the list), and pick $Y_{t,2}$ uniformly from these integers; and so on. In this way we obtain $\{Y_{t,m}\}_{m\in [M_t]}$. For each $m\in [n]\backslash [M_t]$, we let $N_{t,m}=0$ and $Y_{t,m}=0$.
\end{itemize}
We let $\sigma\in S_n$ be the unique permutation that satisfies the following conditions:
\begin{itemize}
    \item For any $m\in [M_t]$, $\sigma(I_{t,Y_{t,m}})=J_{t,m}$.
    \item For any $i\in [n]\backslash \{I_{t,1},\cdots,I_{t,M_t}\}$, $\sigma(i)=\sigma_0(i)$.
\end{itemize}

We let
\begin{equation}
    z_t:=(s-1)L\beta_n^{-1\slash 2}+t L_0\beta_n^{-1\slash 2}, \quad z_t':=(s-1)L\beta_n^{-1\slash 2}+(t+1)L_0\beta_n^{-1\slash 2}. 
\end{equation}
Recall Definition \ref{Def2.2}. As $\lceil z_t\rceil-1,\lfloor z_t'\rfloor \in [n]$, by (\ref{Eq5.1.1}) and Proposition \ref{P2.2.2}, 
\begin{equation}\label{Eq16.1.1}
    \mathbb{P}(|\mathcal{D}_{\lceil z_{t}\rceil-1}(\sigma_0)|\geq C_1\beta_n^{-1\slash 2}) \leq C\exp(-c\beta_n^{-1\slash 2}),
\end{equation}
\begin{equation}\label{Eq16.1.2}
    \mathbb{P}(|\mathcal{D}'_{\lfloor z_{t}'\rfloor}(\sigma_0)|\geq C_1\beta_n^{-1\slash 2}) \leq C\exp(-c\beta_n^{-1\slash 2}).
\end{equation}
Let $\mathcal{Z}_t$ be the event that $|\mathcal{D}_{\lceil z_{t}\rceil-1}(\sigma_0)|\leq C_1\beta_n^{-1\slash 2}$ and $|\mathcal{D}'_{\lfloor z_{t}'\rfloor}(\sigma_0)|\leq C_1\beta_n^{-1\slash 2}$. By (\ref{Eq16.1.1}), (\ref{Eq16.1.2}), and the union bound, we have
\begin{equation}\label{Eq16.1.12}
    \mathbb{P}(\mathcal{Z}_t^c)\leq C\exp(-c\beta_n^{-1\slash 2}).
\end{equation}
Note that $L_0=8C_1$, $L_0\beta_n^{-1\slash 2}\geq 80$, and $\lfloor z_t'\rfloor-\lceil z_t\rceil\geq z_t'-z_t-2=L_0\beta_n^{-1\slash 2}-2$. Hence when the event $\mathcal{Z}_t$ holds, we have 
\begin{eqnarray}\label{Eq16.1.3}
&&  |S(\sigma_0)\cap [z_t,z_t']^2|\geq |S(\sigma_0)\cap [\lceil z_t\rceil ,\lfloor z_t' \rfloor]^2|  \nonumber\\
&\geq& |[\lceil z_{t}\rceil ,\lfloor z'_{t}\rfloor]\cap\mathbb{N}^{*}|-|\mathcal{D}_{\lceil z_{t}\rceil-1}(\sigma_0)|-|\mathcal{D}'_{\lfloor z_{t}'\rfloor}(\sigma_0)|\nonumber\\
   &\geq&  \lfloor z'_{t}\rfloor-\lceil z_{t}\rceil+1-2C_1\beta_n^{-1\slash 2}\geq L_0\beta_n^{-1\slash 2}-1-\frac{1}{4}L_0\beta_n^{-1\slash 2}\nonumber\\
   &=& \frac{3}{4}L_0\beta_n^{-1\slash 2}-1 \geq \frac{1}{2}L_0\beta_n^{-1\slash 2}.
\end{eqnarray}

Let $\mathscr{M}_t$ be the set of $m\in [M_t]$ such that $J_{t,m}\in \mathcal{Y}_t$. Below we consider any $m\in [n]$. If $m\in\mathscr{M}_t$, for any $m'\in [M_t]$ such that $\sigma_0(I_{t,m'})<J_{t,m}$ (note that there are $m-1$ such $m'$), we have $b_{t,m'}\leq \sigma_0(I_{t,m'})<J_{t,m}$, hence
\begin{equation}\label{Eq16.1.4}
    N_{t,m}=\sum_{\substack{m'\in [M_t]:\\ \sigma_0(I_{t,m'})\geq J_{t,m}}}  \mathbbm{1}_{b_{t,m'}\leq J_{t,m}}. 
\end{equation}
For any $i\in [n]$ such that $(i,\sigma_0(i))\in [z_t,z_t']^2$, we have $(i,\sigma_0(i))\in\mathcal{X}_t\times \mathcal{J}_t$. Hence there exists some $m'\in [M_t]$, such that $i=I_{t,m'}$. Let
\begin{equation}
    \mathcal{M}_t:=\{m'\in [M_t]: (I_{t,m'},\sigma_0(I_{t,m'}))\in [z_t,z_t']^2\}. 
\end{equation}
By (\ref{Eq16.1.3}), when the event $\mathcal{Z}_t$ holds, we have
\begin{equation}\label{Eq16.1.8}
    |\mathcal{M}_t|\geq |S(\sigma_0)\cap [z_t,z_t']^2|\geq \frac{1}{2}L_0\beta_n^{-1\slash 2}\geq \beta_n^{-1\slash 2}. 
\end{equation}
If $m\in \mathscr{M}_t$, then $J_{t,m}\in \mathcal{Y}_t$, hence for any $m'\in \mathcal{M}_t$, $\sigma_0(I_{t,m'})\geq z_t\geq J_{t,m}$. Hence by (\ref{Eq16.1.4}), we have
\begin{equation}\label{Eq16.1.5}
    N_{t,m}\geq \sum_{m'\in\mathcal{M}_t} \mathbbm{1}_{b_{t,m'}\leq J_{t,m}}.
\end{equation}
Now note that if $m\in\mathscr{M}_t$, conditional on $\sigma_0$, $\{\mathbbm{1}_{b_{t,m'}\leq J_{t,m}}\}_{m'\in \mathcal{M}_t}$ are mutually independent, and for each $m'\in \mathcal{M}_t$, $\mathbbm{1}_{b_{t,m'}\leq J_{t,m}}$ follows the Bernoulli distribution with 
\begin{eqnarray}\label{Eq16.1.6}
   \mathbb{P}(\mathbbm{1}_{b_{t,m'}\leq J_{t,m}}=1|\sigma_0)&=&\mathbb{P}(b_{t,m'}\leq J_{t,m}|\sigma_0)=\mathbb{P}(u_{t,m'}\leq e^{2\beta_n (I_{t,m'}-\alpha_t)J_{t,m}}|\sigma_0) \nonumber\\
   &=& e^{-2\beta_n (I_{t,m'}-\alpha_t)(\sigma_0(I_{t,m'})-J_{t,m})}.
\end{eqnarray}
If $m\in\mathscr{M}_t$, for any $m'\in\mathcal{M}_t$, as $J_{t,m}\in\mathcal{Y}_t$ and $(I_{t,m'},\sigma_0(I_{t,m'}))\in [z_t,z_t']^2$, we have
\begin{equation*}
    I_{t,m'}-\alpha_t\begin{cases}
     \geq z_{t}-\alpha_t\geq 1\\
     \leq z'_{t}-\alpha_t= 2L_0\beta_n^{-1\slash 2}+1\leq 3L_0\beta_n^{-1\slash 2}
    \end{cases},
\end{equation*}
\begin{equation*}
   0\leq  \sigma_0(I_{t,m'})-J_{t,m} \leq z'_{t}-J_{t,m}\leq 2L_0\beta_n^{-1\slash 2},
\end{equation*}
hence by (\ref{Eq16.1.6}), we have
\begin{equation}\label{Eq16.1.7}
    \mathbb{P}(\mathbbm{1}_{b_{t,m'}\leq J_{t,m}}=1|\sigma_0)\geq e^{-12 L_0^2}.
\end{equation}
By (\ref{Eq16.1.5}), (\ref{Eq16.1.7}), and Hoeffding's inequality, for any $x\in [0,e^{-12L_0^2}]$, we have
\begin{equation*}
    \mathbb{P}(N_{t,m}\leq (e^{-12L_0^2}-x)|\mathcal{M}_t||\sigma_0)\mathbbm{1}_{m\in\mathscr{M}_t}\leq e^{-2|\mathcal{M}_t|x^2}\mathbbm{1}_{m\in\mathscr{M}_t}, 
\end{equation*}
which by (\ref{Eq16.1.8}) leads to
\begin{equation*}
    \mathbb{P}(N_{t,m}\leq (e^{-12L_0^2}-x)\beta_n^{-1\slash 2}|\sigma_0) \mathbbm{1}_{m\in \mathscr{M}_t} \mathbbm{1}_{\mathcal{Z}_t}  \leq e^{-2\beta_n^{-1\slash 2}x^2}\mathbbm{1}_{m\in \mathscr{M}_t}.
\end{equation*}
Taking $x=e^{-12L_0^2}\slash 2$, we have
\begin{eqnarray}\label{Eq16.1.9}
  && \mathbb{P}(\{N_{t,m}\leq e^{-12L_0^2}\beta_n^{-1\slash 2} \slash 2\}\cap\{m\in \mathscr{M}_t\}\cap\mathcal{Z}_t|\sigma_0)\nonumber\\
   &=& \mathbb{P}(N_{t,m}\leq e^{-12L_0^2}\beta_n^{-1\slash 2} \slash 2|\sigma_0) \mathbbm{1}_{m\in \mathscr{M}_t} \mathbbm{1}_{\mathcal{Z}_t} \leq e^{-c\beta_n^{-1\slash 2}}\mathbbm{1}_{m\in \mathscr{M}_t}. 
\end{eqnarray}

Let $\mathcal{C}_t$ be the event that $N_{t,m}\geq e^{-12L_0^2}\beta_n^{-1\slash 2} \slash 2$ for any $m\in\mathscr{M}_t$. By (\ref{Eq16.1.9}) and the union bound, \begin{eqnarray}
&& \mathbb{P}(\mathcal{C}_t^c\cap\mathcal{Z}_t|\sigma_0)\leq 
\mathbb{P}\Big(\bigcup_{m=1}^n\big(\{N_{t,m}\leq e^{-12L_0^2}\beta_n^{-1\slash 2} \slash 2\}\cap\{m\in \mathscr{M}_t\}\cap\mathcal{Z}_t\big)\Big|\sigma_0\Big)\nonumber\\
&\leq& \sum_{m=1}^n \mathbb{P}(\{N_{t,m}\leq e^{-12L_0^2}\beta_n^{-1\slash 2} \slash 2\}\cap\{m\in \mathscr{M}_t\}\cap\mathcal{Z}_t|\sigma_0)\nonumber\\ &\leq& e^{-c\beta_n^{-1\slash 2}}\sum_{m=1}^n \mathbbm{1}_{m\in \mathscr{M}_t}=|\mathscr{M}_t| e^{-c\beta_n^{-1\slash 2}}
\leq |\mathcal{Y}_{t}|e^{-c\beta_n^{-1\slash 2}}\nonumber\\
&\leq& (L_0\beta_n^{-1\slash 2}+1)e^{-c\beta_n^{-1\slash 2}}\leq C\beta_n^{-1\slash 2}e^{-c\beta_n^{-1\slash 2}}\leq C\exp(-c\beta_n^{-1\slash 2}).
\end{eqnarray}
Hence 
\begin{equation}\label{Eq16.1.10}
    \mathbb{P}(\mathcal{C}_t^c\cap\mathcal{Z}_t)=\mathbb{E}[\mathbb{P}(\mathcal{C}_t^c\cap\mathcal{Z}_t|\sigma_0)]\leq C\exp(-c\beta_n^{-1\slash 2}).
\end{equation}
By (\ref{Eq16.1.12}), (\ref{Eq16.1.10}), and the union bound, 
\begin{equation}\label{Eq16.1.11}
    \mathbb{P}(\mathcal{C}_t^c)\leq C\exp(-c\beta_n^{-1\slash 2}).
\end{equation}

Let
\begin{equation}\label{Eq16.7.1}
    W_t:=\{i\in [n]: (i,\sigma_0(i))\in \mathcal{X}_t\times \mathcal{J}_t\}.
\end{equation}
Note that
\begin{equation}\label{Eq16.6.7}
    |W_t|\leq |\mathcal{J}_t|\leq 2L_0\beta_n^{-1\slash 2}+1\leq 3L_0\beta_n^{-1\slash 2}. 
\end{equation}

For any $q\in \mathbb{N}^{*}$, we let
\begin{equation}\label{Eq16.6.1}
    \Lambda_{t,q}:=\sum_{\substack{i_1<\cdots<i_q, j_1<\cdots<j_q\\i_1,\cdots,i_q\in\mathcal{X}_t\\j_1,\cdots,j_q\in\mathcal{Y}_t}} \mathbbm{1}_{\sigma(i_1)=j_1,\cdots,\sigma(i_q)=j_q}\mathbbm{1}_{i_1,\cdots,i_q\in W_t}. 
\end{equation}
In the following, we bound $\Lambda_{t,q}$ for any $q\in\mathbb{N}^{*}$. 

For any $m\in [M_t]$, we let $\mathscr{N}_{J_{t,m}}=N_{t,m}$; for any $j\in \mathcal{J}_t \backslash\{J_{t,1},\cdots,J_{t,M_t}\}$, we let $\mathscr{N}_j=n$. Moreover, for any $m\in [M_t]$, we let $\mathscr{Y}_{J_{t,m}}=I_{t,Y_{t,m}}$; for any $j\in \mathcal{J}_t\backslash \{J_{t,1},\cdots,J_{t,M_t}\}$, we let $\mathscr{Y}_j=0$. 

Let $\mathcal{B}$ be the $\sigma$-algebra generated by $\sigma_0$ and $\{b_{t,m}\}_{m\in [n]}$. For any $j\in \mathcal{J}_t$, let $\mathcal{F}_j$ be the $\sigma$-algebra generated by $\sigma_0$, $\{b_{t,m}\}_{m\in [n]}$, and $\{\mathscr{Y}_l\}_{l\in [j-1]\cap\mathcal{J}_t}$.  

We assume that the event $\mathcal{C}_t$ holds. For any $m\in\mathscr{M}_t$, we have
\begin{equation*}
    \mathscr{N}_{J_{t,m}}=N_{t,m}\geq e^{-12L_0^2}\beta_n^{-1\slash 2}\slash 2.
\end{equation*}
Hence for any $j\in \{J_{t,m}:m\in [M_t]\}\cap \mathcal{Y}_t$, $\mathscr{N}_j\geq e^{-12L_0^2}\beta_n^{-1\slash 2}\slash 2 $. Moreover, for any $j\in \{J_{t,m}:m\in [M_t]\}^c\cap \mathcal{Y}_t$, $\mathscr{N}_j=n\geq e^{-12L_0^2}\beta_n^{-1\slash 2}\slash 2$ (note that $n\beta_n^{1\slash 2}\geq 4L\geq 4$). Hence for any $j\in \mathcal{Y}_t$,
\begin{equation}\label{Eq16.3.1}
    \mathscr{N}_j\geq e^{-12L_0^2}\beta_n^{-1\slash 2}\slash 2.
\end{equation}

Consider any $i_1,\cdots,i_q\in \mathcal{X}_t$ and $j_1,\cdots,j_q\in\mathcal{Y}_t$ such that $i_1<\cdots<i_q$ and $j_1<\cdots<j_q$. For any $l\in [q]$, if $\sigma(i_l)=j_l$, then
\begin{equation*}
(i_l,\sigma(i_l))=(i_l,j_l)\in\mathcal{X}_t\times\mathcal{Y}_t\subseteq \mathcal{X}_t\times \mathcal{J}_t,
\end{equation*}
which implies $\mathscr{Y}_{j_l}=i_l$. Hence we have
\begin{eqnarray}\label{Eq16.4.1}
  && \mathbb{E}[\mathbbm{1}_{\sigma(i_1)=j_1,\cdots,\sigma(i_q)=j_q}\mathbbm{1}_{i_1,\cdots,i_q\in W_t}|\mathcal{B}]\leq \mathbb{E}[\mathbbm{1}_{\mathscr{Y}_{j_1}=i_1,\cdots,\mathscr{Y}_{j_q}=i_q}\mathbbm{1}_{i_1,\cdots,i_q\in W_t}|\mathcal{B}]\nonumber\\
  &=& \mathbbm{1}_{i_1,\cdots,i_q\in W_t} \mathbb{E}[\mathbbm{1}_{\mathscr{Y}_{j_1}=i_1,\cdots,\mathscr{Y}_{j_q}=i_q}|\mathcal{B}]\nonumber\\
  &=&  \mathbbm{1}_{i_1,\cdots,i_q\in W_t} \mathbb{E}[\mathbb{E}[\mathbbm{1}_{\mathcal{Y}_{j_q}=i_q}|\mathcal{F}_{j_q}]\mathbbm{1}_{\mathscr{Y}_{j_1}=i_1,\cdots,\mathscr{Y}_{j_{q-1}}=i_{q-1}}|\mathcal{B}]\nonumber\\
  &\leq& \frac{\mathbbm{1}_{i_1,\cdots,i_q\in W_t}}{\mathscr{N}_{j_q}} \mathbb{E}[\mathbbm{1}_{\mathscr{Y}_{j_1}=i_1,\cdots,\mathscr{Y}_{j_{q-1}}=i_{q-1}}|\mathcal{B}]\leq \cdots\leq \frac{\mathbbm{1}_{i_1,\cdots,i_q\in W_t}}{\mathscr{N}_{j_1}\mathscr{N}_{j_2}\cdots \mathscr{N}_{j_q}}.
\end{eqnarray}
By (\ref{Eq16.3.1}) and (\ref{Eq16.4.1}), we have
\begin{eqnarray}\label{Eq16.6.2}
  &&   \mathbb{E}[\mathbbm{1}_{\sigma(i_1)=j_1,\cdots,\sigma(i_q)=j_q}\mathbbm{1}_{i_1,\cdots,i_q\in W_t}|\mathcal{B}] \mathbbm{1}_{\mathcal{C}_t} \nonumber\\
  &\leq& \frac{\mathbbm{1}_{i_1,\cdots,i_q\in W_t}\mathbbm{1}_{\mathcal{C}_t} }{\mathscr{N}_{j_1}\mathscr{N}_{j_2}\cdots \mathscr{N}_{j_q}} \leq (2e^{12L_0^2}\beta_n^{1\slash 2})^q \mathbbm{1}_{i_1,\cdots,i_q\in W_t}.
\end{eqnarray}

By (\ref{Eq16.6.7}), (\ref{Eq16.6.1}), (\ref{Eq16.6.2}), and Lemma \ref{Lemma3.1}, we have
\begin{eqnarray*}
&& \mathbb{E}[\Lambda_{t,q}|\mathcal{B}]\mathbbm{1}_{\mathcal{C}_t} \leq (2e^{12L_0^2}\beta_n^{1\slash 2})^q\sum_{\substack{i_1<\cdots<i_q, j_1<\cdots<j_q\\i_1,\cdots,i_q\in\mathcal{X}_t\\j_1,\cdots,j_q\in\mathcal{Y}_t}} \mathbbm{1}_{i_1,\cdots,i_q\in W_t} \nonumber\\
&\leq& (2e^{12L_0^2}\beta_n^{1\slash 2})^q \binom{|W_t|}{q}\binom{|\mathcal{Y}_t|}{q}\leq \Big(\frac{2e^{2+12L_0^2}\beta_n^{1\slash 2}|W_t||\mathcal{Y}_t|}{q^2}\Big)^q\nonumber\\
&\leq& (C\beta_n^{-1\slash 2} q^{-2})^q,
\end{eqnarray*}
where we use the fact that $|\mathcal{Y}_t|\leq L_0\beta_n^{-1\slash 2}+1\leq 2L_0\beta_n^{-1\slash 2}$ in the last line. Hence
\begin{equation}\label{Eq16.8.1}
    \mathbb{E}[\Lambda_{t,q}\mathbbm{1}_{\mathcal{C}_t}]=\mathbb{E}[\mathbb{E}[\Lambda_{t,q}|\mathcal{B}]\mathbbm{1}_{\mathcal{C}_t}]\leq (C\beta_n^{-1\slash 2} q^{-2})^q.
\end{equation}

Now for any $q\in\mathbb{N}^{*}$, if $LIS(\sigma|_{\mathcal{Q}_t})\geq q$, then there exist $i_1,\cdots,i_q\in\mathcal{X}_t$ and $j_1,\cdots,j_q\in\mathcal{Y}_t$, such that $i_1<\cdots<i_q$, $j_1<\cdots<j_q$, and $\sigma(i_l)=j_l$ for every $l\in [q]$. For any $l\in [q]$, we have $(i_l,\sigma(i_l))=(i_l,j_l)\in\mathcal{X}_t\times\mathcal{Y}_t\subseteq  \mathcal{X}_t\times\mathcal{J}_t$, hence $(i_l,\sigma_0(i_l))\in \mathcal{X}_t\times\mathcal{J}_t$ and $i_l\in W_t$ (recall (\ref{Eq16.7.1})). Hence $\Lambda_{t,q}\geq 1$. We conclude that for any $q\in\mathbb{N}^{*}$,
\begin{equation}\label{Eq16.8.2}
    \{ LIS(\sigma|_{\mathcal{Q}_t}) \geq q \}\subseteq \{\Lambda_{t,q}\geq 1\}.
\end{equation}

By (\ref{Eq16.8.1}) and (\ref{Eq16.8.2}), for any $q\in\mathbb{N}^{*}$, we have
\begin{eqnarray}
 && \mathbb{P}(\{ LIS(\sigma|_{\mathcal{Q}_t}) \geq q \}\cap\mathcal{C}_t)\leq \mathbb{P}(
 \{\Lambda_{t,q}\geq 1\}\cap\mathcal{C}_t) \nonumber\\
 &=& \mathbb{E}[\mathbbm{1}_{\Lambda_{t,q}\geq 1}\mathbbm{1}_{\mathcal{C}_t}]\leq \mathbb{E}[\Lambda_{t,q}\mathbbm{1}_{\mathcal{C}_t}]\leq (C_0 \beta_n^{-1\slash 2} q^{-2})^q,
\end{eqnarray}
where $C_0\geq 1$ is a positive absolute constant. Taking $q=\lceil \sqrt{2C_0}\beta_n^{-1\slash 4}\rceil$, we obtain that
\begin{eqnarray}\label{Eq16.9.1}
 &&\mathbb{P}(\{ LIS(\sigma|_{\mathcal{Q}_t}) \geq 2\sqrt{2C_0}\beta_n^{-1\slash 4} \}\cap\mathcal{C}_t)\nonumber\\
 &\leq& \mathbb{P}(\{ LIS(\sigma|_{\mathcal{Q}_t}) \geq \lceil\sqrt{2C_0}\beta_n^{-1\slash 4}\rceil \}\cap\mathcal{C}_t)\nonumber\\
 &\leq& 2^{-\lceil \sqrt{2C_0}\beta_n^{-1\slash 4}\rceil}\leq \exp(-c\beta_n^{-1\slash 4}).
\end{eqnarray}
By (\ref{Eq16.1.11}), (\ref{Eq16.9.1}), and the union bound, we have
\begin{equation}
 \mathbb{P}(LIS(\sigma|_{\mathcal{Q}_t}) \geq 2\sqrt{2C_0}\beta_n^{-1\slash 4})\leq C\exp(-c\beta_n^{-1\slash 4}).
\end{equation}
Note that $LIS(\sigma|_{\mathcal{Q}_t})\leq |\mathcal{Y}_t|\leq 2L_0\beta_n^{-1\slash 2}$. Hence
\begin{eqnarray}\label{Eq18.1.1}
   \mathbb{E}[LIS(\sigma|_{\mathcal{Q}_t})]&\leq& (2L_0\beta_n^{-1\slash 2})(C\exp(-c\beta_n^{-1\slash 4}))+2\sqrt{2C_0}\beta_n^{-1\slash 4}\nonumber\\
  &\leq& C\beta_n^{-1\slash 4}+C\exp(-c\beta_n^{-1\slash 4})\leq C\beta_n^{-1\slash 4}.
\end{eqnarray}

\paragraph{Step 2}

Now we let 
\begin{equation}
    \mathcal{R}_{s,1}:=\mathcal{R}_s\cap\{(x,y)\in\mathbb{R}^2:x\geq y\}, \quad \mathcal{R}_{s,2}:=\mathcal{R}_s\cap\{(x,y)\in\mathbb{R}^2:x\leq y\}.
\end{equation}
For any $\tau\in S_n$, we let
\begin{eqnarray}
    \mathcal{L}_{s,1}(\tau)&:=&\max\{k\in \{0\} \cup [n]: \text{ there exist }i_1,\cdots,i_k\in [n], \text{ such that }\nonumber\\
    && i_1<\cdots<i_k, \tau(i_1)<\cdots<\tau(i_k), (i_l,\tau(i_l))\in \mathcal{R}_{s,1} \text{ for every }l\in [k] \}, \nonumber\\
    &&
\end{eqnarray}
\begin{eqnarray}
    \mathcal{L}_{s,2}(\tau)&:=&\max\{k\in \{0\} \cup [n]: \text{ there exist }i_1,\cdots,i_k\in [n], \text{ such that }\nonumber\\
    && i_1<\cdots<i_k, \tau(i_1)<\cdots<\tau(i_k), (i_l,\tau(i_l))\in \mathcal{R}_{s,2} \text{ for every }l\in [k] \}. \nonumber\\
    &&
\end{eqnarray}
As $\mathcal{R}_s\subseteq\mathcal{R}_{s,1}\cup\mathcal{R}_{s,2}$, for any $\tau\in S_n$, we have
\begin{equation}\label{Eq19.1.1}
    LIS(\tau|_{\mathcal{R}_s})\leq \mathcal{L}_{s,1}(\tau)+\mathcal{L}_{s,2}(\tau).
\end{equation}

For any $(x,y)\in \mathcal{R}_{s,1}\cap [n]^2$, we have
\begin{equation*}
    y\in ((s-1)L\beta_n^{-1\slash 2},sL\beta_n^{-1\slash 2}]\cap \mathbb{N}^{*}=\bigcup_{t=1}^{L\slash L_0} \mathcal{Y}_t,
\end{equation*}
so there exists some $t\in [L\slash L_0]$ such that $y\in \mathcal{Y}_t$. Note that
\begin{equation*}
    x\geq y> (s-1)L\beta_n^{-1\slash 2}+(t-1)L_0\beta_n^{-1\slash 2}.
\end{equation*}
Hence $(x,y)\in \mathcal{X}_t\times \mathcal{Y}_t=\mathcal{Q}_t$. Therefore, $\mathcal{R}_{s,1}\cap [n]^2\subseteq \bigcup_{t=1}^{L\slash L_0}\mathcal{Q}_t$, which leads to
\begin{equation}
    \mathcal{L}_{s,1}(\sigma)\leq \sum_{t=1}^{L\slash L_0} LIS(\sigma|_{\mathcal{Q}_t}).
\end{equation}
By (\ref{Eq18.1.1}), we have
\begin{equation}\label{Eq18.1.2}
    \mathbb{E}[\mathcal{L}_{s,1}(\sigma)]\leq \sum_{t=1}^{L\slash L_0} \mathbb{E}[LIS(\sigma|_{\mathcal{Q}_t})]\leq CL\beta_n^{-1\slash 4}.
\end{equation}

Note that for $\sigma$ drawn from $\tilde{\mathbb{P}}_{n,\beta_n}$, the distribution of $\sigma^{-1}$ is also given by $\tilde{\mathbb{P}}_{n,\beta_n}$. Moreover, $\mathcal{L}_{s,2}(\sigma)=\mathcal{L}_{s,1}(\sigma^{-1})$. Hence by (\ref{Eq18.1.2}), we have
\begin{equation}\label{Eq19.1.3}
    \mathbb{E}[\mathcal{L}_{s,2}(\sigma)]=\mathbb{E}[\mathcal{L}_{s,1}(\sigma^{-1})]\leq CL\beta_n^{-1\slash 4}.
\end{equation}
By (\ref{Eq19.1.1}), (\ref{Eq18.1.2}), and (\ref{Eq19.1.3}), we conclude that
\begin{equation}\label{Eq20.1.1}
    \mathbb{E}[LIS(\sigma|_{\mathcal{R}_s})]\leq \mathbb{E}[\mathcal{L}_{s,1}(\sigma)]+\mathbb{E}[\mathcal{L}_{s,2}(\sigma)] \leq CL\beta_n^{-1\slash 4}.
\end{equation}

\bigskip

In the following, we show (\ref{Eq15.2.1}) for $s=\lfloor n\beta_n^{1\slash 2} \slash L\rfloor$. Let $\sigma$ be drawn from $\tilde{\mathbb{P}}_{n,\beta_n}$, and let $\bar{\sigma}\in S_n$ be such that $\bar{\sigma}(i)=n+1-\sigma(n+1-i)$ for every $i\in [n]$. For any $\tau\in S_n$, 
\begin{eqnarray*}
 \mathbb{P}(\bar{\sigma}=\tau)&=&\mathbb{P}(\sigma(n+1-i)=n+1-\tau(i)\text{ for every }i\in [n])\nonumber\\
 &=& \mathbb{P}(\sigma(i)=n+1-\tau(n+1-i)\text{ for every }i\in [n])\nonumber\\
 &=& \tilde{Z}_{n,\beta_n}^{-1}\exp\Big(-\beta_n \sum_{i=1}^n (n+1-\tau(n+1-i)-i)^2 \Big)\nonumber\\
 &=& \tilde{Z}_{n,\beta_n}^{-1}\exp\Big(-\beta_n \sum_{i=1}^n (\tau(i)-i)^2 \Big)=\tilde{\mathbb{P}}_{n,\beta_n}(\tau).
\end{eqnarray*}
Hence the distribution of $\bar{\sigma}$ is given by $\tilde{\mathbb{P}}_{n,\beta_n}$. Note that
\begin{eqnarray*}
    n+1-(\lfloor n\beta_n^{1\slash 2}\slash L\rfloor -1)L\beta_n^{-1\slash 2}&\leq& n+1-(n\beta_n^{1\slash 2}\slash L-2)L\beta_n^{-1\slash 2}\nonumber\\
   &=& 2L\beta_n^{-1\slash 2}+1\leq 3L\beta_n^{-1\slash 2},
\end{eqnarray*}
which leads to
\begin{equation*}
    [1,n+1-(\lfloor n\beta_n^{1\slash 2}\slash L\rfloor -1)L\beta_n^{-1\slash 2}]^2\subseteq \Big(\bigcup_{s=1}^3 \mathcal{R}_s\Big) \bigcup \Big(\bigcup_{s=1}^3 \mathcal{R}_s'\Big)\bigcup \Big(\bigcup_{s=1}^3 \mathcal{R}_s''\Big).
\end{equation*}
Hence by Proposition \ref{P5.1} and (\ref{Eq20.1.1}) (for $s\in [\lfloor n\beta_n^{1\slash 2}\slash L\rfloor-1]$), we have
\begin{eqnarray}
  &&   \mathbb{E}\big[LIS \big(\sigma|_{\mathcal{R}_{\lfloor n\beta_n^{1\slash 2} \slash L\rfloor}}\big)\big]\leq\mathbb{E}\big[LIS\big(\bar{\sigma}|_{[1,n+1-(\lfloor n\beta_n^{1\slash 2}\slash L\rfloor -1)L\beta_n^{-1\slash 2}]^2}\big)\big] \nonumber\\
  &\leq& \sum_{s=1}^3 \mathbb{E}[LIS(\bar{\sigma}|_{\mathcal{R}_s})]+\sum_{s=1}^3 \mathbb{E}[LIS(\bar{\sigma}|_{\mathcal{R}_s'})]+\sum_{s=1}^3 \mathbb{E}[LIS(\bar{\sigma}|_{\mathcal{R}_s''  })] \nonumber\\
  &\leq& CL\beta_n^{-1\slash 4}+CL^2\exp(-c\beta_n^{-1\slash 4}).
\end{eqnarray}

\end{proof}

The following proposition gives a more precise bound on $LIS(\sigma|_{\mathcal{R}_s})$ for $\sigma$ drawn from $\tilde{\mathbb{P}}_{n,\beta_n}$ and $s\in[2,\lfloor n\beta_n^{1\slash 2}\slash L\rfloor-1]\cap\mathbb{N}$ that satisfies certain conditions.

\begin{proposition}\label{P5.3}
We denote by $C_1'$ the constant $C_1$ in Proposition \ref{P2.3.2} (with $\delta_0=1\slash 4$ and $K=2L$; note that $C_1'$ only depends on $L$). Let
\begin{equation}\label{Eq30.1.5}
    r_s:=\frac{1}{2} \min \{(s-1)L, (\lfloor n\beta_n^{1\slash 2}\slash L\rfloor-s)L,\log(1+\beta_n^{-1\slash 2})\}
\end{equation}
for any $s\in [2,\lfloor n\beta_n^{1\slash 2}\slash L\rfloor-1]\cap\mathbb{N}$. There exist positive constants $C_L,c_L,C_L',c_L'$ that only depend on $L$ and positive absolute constants $C,C'$ with $C'\geq 1$, such that the following holds.

Assume that $n\beta_n^{1\slash 2}\geq 4L$ and $\beta_n^{-1\slash 2}\geq C'L^{10}e^{60L^2}$, and let $\sigma$ be drawn from $\tilde{\mathbb{P}}_{n,\beta_n}$. Let $\Psi_s:=(1-L^{-1}-C_L r_s^{-1\slash 25})_{+}$ for any $s\in [2,\lfloor n\beta_n^{1\slash 2}\slash L\rfloor-1]\cap\mathbb{N}$. Then for any $s\in [2,\lfloor n\beta_n^{1\slash 2}\slash L\rfloor-1]\cap\mathbb{N}$ such that $r_s\geq C_1'$, we have
\begin{eqnarray}
  &&  \mathbb{E}[|LIS(\sigma|_{\mathcal{R}_s})-2\pi^{-1\slash 4}L\beta_n^{-1\slash 4}|] \nonumber\\
  &\leq& C_L'\beta_n^{-1}\exp(-c_L'\beta_n^{-1\slash 8}\Psi_s^{1\slash 4})+C_L'+C L^{1\slash 2} e^{-4L^2} \beta_n^{-1\slash 4}\nonumber\\
  &&+2\pi^{-1\slash 4}L\beta_n^{-1\slash 4} \max\Big\{1-e^{-6L^{-1}}\Psi_s^{1\slash 2} (1-\max\{c_L\beta_n^{-1 \slash 2}\Psi_s,1\}^{-1\slash 6}), \nonumber\\
  && \quad\quad e^{3L^{-1}}(1+C_L r_s^{-1\slash 25})^{1\slash 2}  (1+\max\{c_L\beta_n^{-1\slash 2}\Psi_s,1\}^{-1\slash 6})-1\Big\}.
\end{eqnarray}
\end{proposition}

\begin{proof}

Let $C_0,c_0,C_2$ be the constants that appear in Proposition \ref{P2.3.2} (with $\delta_0=1\slash 4$ and $K=2L$). We also denote by $C_1'$ the constant $C_1$ in Proposition \ref{P2.3.2} (with $\delta_0=1\slash 4$ and $K=2L$). Note that these constants only depend on $L$. Throughout the proof, we fix an arbitrary $s\in [2,\lfloor n\beta_n^{1\slash 2}\slash L\rfloor-1]\cap\mathbb{N}$ such that $r_s\geq C_1'$. We also assume that $n\beta_n^{1\slash 2} \geq 4L$ and $\beta_n\leq 1\slash 100$. 

We denote by $C_L',c_L'$ positive constants that only depend on $L$. The values of these constants may change from line to line.

In the following, we fix any $T,K_0\in \mathbb{N}^{*}$ such that $\min\{T,K_0\}\geq L^2$, any refined path $\Gamma\in \Pi^{T,T,K_0}$, and any $l\in [2T-1]$. We assume that 
\begin{equation}\label{Eq20.2.1}
     L\beta_n^{-1\slash 2} \geq \max\{8 K_0 T,K_0^2 T^3\}.
\end{equation}

We let
\begin{eqnarray*}
   && Q_{\Gamma,l}:=(x_{l-1}(\Gamma),x_l(\Gamma)]\times ( y_{l-1}(\Gamma), y_l(\Gamma)],\\
   &&Q_{\Gamma,l}':=[ a_{l-1}(\Gamma),   c_l(\Gamma)]\times [ b_{l-1}(\Gamma), d_l(\Gamma)].
\end{eqnarray*}
We also let
\begin{eqnarray*}
 \tilde{Q}_{\Gamma,l} &:=&  ((s-1)L\beta_n^{-1\slash 2},(s-1)L\beta_n^{-1\slash 2})+L\beta_n^{-1\slash 2}Q_{\Gamma,l}\\
   &=& ((s-1)L\beta_n^{-1\slash 2}+L\beta_n^{-1\slash 2}x_{l-1}(\Gamma),(s-1)L\beta_n^{-1\slash 2}+L\beta_n^{-1\slash 2}x_{l}(\Gamma)]\\
   &&\times ((s-1)L\beta_n^{-1\slash 2}+L\beta_n^{-1\slash 2}y_{l-1}(\Gamma),(s-1)L\beta_n^{-1\slash 2}+L\beta_n^{-1\slash 2}y_{l}(\Gamma)],
\end{eqnarray*}
\begin{equation*}
    \tilde{Q}_{\Gamma,l}':=((s-1)L\beta_n^{-1\slash 2},(s-1)L\beta_n^{-1\slash 2})+L\beta_n^{-1\slash 2}Q_{\Gamma,l}'.
\end{equation*}

\paragraph{Step 1}

We start by bounding $LIS(\sigma|_{\tilde{Q}_{\Gamma,l}})$. If
\begin{equation*}
    x_{l-1}(\Gamma)=x_l(\Gamma) \text{ or } y_{l-1}(\Gamma)=y_l(\Gamma),
\end{equation*}
then $Q_{\Gamma,l}=\emptyset$ and $LIS(\tau|_{\tilde{Q}_{\Gamma,l}})=0$ for any $\tau\in S_n$. In the following, we assume that $x_{l-1}(\Gamma)<x_l(\Gamma)$ and $y_{l-1}(\Gamma)<y_l(\Gamma)$. Note that 
\begin{equation}\label{Eq26.3.1}
    (2 K_0 T)^{-1}\leq x_l(\Gamma)-x_{l-1}(\Gamma)\leq T^{-1},\quad  (2 K_0 T)^{-1}\leq y_l(\Gamma)-y_{l-1}(\Gamma)\leq T^{-1},
\end{equation}
which by (\ref{Eq20.2.1}) implies
\begin{equation}\label{Eq30.1.1}
   \min\{L\beta_n^{-1\slash 2}(x_l(\Gamma)-x_{l-1}(\Gamma)),L\beta_n^{-1\slash 2}(y_l(\Gamma)-y_{l-1}(\Gamma))\}\geq \frac{L\beta_n^{-1\slash 2}}{2K_0 T}\geq  4.
\end{equation}
In the following, we assume that
\begin{eqnarray}\label{Eq24.1.1}
 && ((s-1)L\beta_n^{-1\slash 2}+L\beta_n^{-1\slash 2} x_{l-1}(\Gamma), (s-1)L\beta_n^{-1\slash 2}+L\beta_n^{-1\slash 2} x_l(\Gamma)]\cap \mathbb{N}^{*}\nonumber\\
 &&=\{s_1,s_1+1,\cdots,s_2\},\nonumber\\
 && ((s-1)L\beta_n^{-1\slash 2}+L\beta_n^{-1\slash 2} y_{l-1}(\Gamma),(s-1)L\beta_n^{-1\slash 2}+L\beta_n^{-1\slash 2} y_l(\Gamma)]\cap\mathbb{N}^{*}\nonumber\\
 &&=\{s_1',s_1'+1,\cdots,s_2'\}.
\end{eqnarray}

We let
\begin{equation}
    \mathcal{X}_s:=((s-1)L\beta_n^{-1\slash 2}, (s+1)L\beta_n^{-1\slash 2}]\cap\mathbb{N}^{*},
\end{equation}
\begin{equation}
    \alpha_s:=(s-1)L\beta_n^{-1\slash 2}-1.
\end{equation}
We assume that
\begin{equation}\label{Eq24.1.2}
    \mathcal{X}_s=\{s_3,s_3+1,\cdots,s_4\}.
\end{equation}
Note that $s_3\leq \min\{s_1,s_1'\}$ and $s_4\geq \max\{s_2,s_2'\}$. 

We sample $\sigma_0$ from $\tilde{\mathbb{P}}_{n,\beta_n}$, and run the resampling algorithm for the $L^2$ model (as described at the end of Section \ref{Sect.1.5}) with inputs $\sigma_0, \mathcal{X}_s, \mathcal{X}_s, \alpha_s$ to obtain $\sigma$. By Lemma \ref{L2.2}, the distribution of $\sigma$ is given by $\tilde{\mathbb{P}}_{n,\beta_n}$. 

We let $M\in\mathbb{N}$, $I_1,\cdots,I_M,J_1,\cdots,J_M\in\mathcal{X}_s$ be such that $I_1<\cdots<I_M$, $J_1<\cdots<J_M$, and 
\begin{equation}
    \{i\in \mathcal{X}_s:\sigma_0(i)\in\mathcal{X}_s\}=\{I_1,\cdots,I_M\}, \quad \{i\in \mathcal{X}_s:\sigma_0^{-1}(i)\in\mathcal{X}_s\}=\{J_1,\cdots,J_M\}.
\end{equation}
For any $m\in [n]\backslash [M]$, we let $I_m=0$ and $J_m=0$. According to the resampling algorithm for the $L^2$ model, $\sigma$ can be generated as follows:
\begin{itemize}
    \item For each $m\in [M]$, we independently sample $u_m$ from the uniform distribution on $[0,e^{2\beta_n(I_m-\alpha_s)\sigma_0(I_m)}]$, and $b_m=\log(u_m)\slash (2\beta_n(I_m-\alpha_s))$. For each $m\in [n]\backslash [M]$, we let $b_m=0$.
    \item For each $m\in [M]$, let
    \begin{equation}
        N_m=|\{m'\in [M]: b_{m'}\leq J_m\}|-m+1.
    \end{equation}
    Now look at the $N_1$ integers $m'\in [M]$ with $b_{m'}\leq J_1$, and pick $Y_1$ uniformly from these integers; then look at the $N_2$ remaining integers $m'\in [M]$ with $b_{m'}\leq J_2$ (with $Y_1$ deleted from the list), and pick $Y_2$ uniformly from these integers; and so on. In this way we obtain $\{Y_m\}_{m\in [M]}$. For each $m\in [n]\backslash [M]$, we let $N_m=0$ and $Y_m=0$. 
\end{itemize}
We let $\sigma\in S_n$ be the unique permutation that satisfies the following conditions:
\begin{itemize}
    \item For any $m\in [M]$, $\sigma(I_{Y_m})=J_m$.
    \item For any $i\in [n]\backslash \{I_1,\cdots,I_M\}$, $\sigma(i)=\sigma_0(i)$.  
\end{itemize}
Note that
\begin{equation}\label{Eq37.1.1}
    \{i\in \mathcal{X}_s:\sigma(i)\in\mathcal{X}_s\}=\{I_1,\cdots,I_M\}, \quad \{i\in \mathcal{X}_s:\sigma^{-1}(i)\in\mathcal{X}_s\}=\{J_1,\cdots,J_M\}.
\end{equation}

Let 
\begin{equation}
    z:=s L\beta_n^{-1\slash 2}, \quad z':=(s+1)L\beta_n^{-1\slash 2}. 
\end{equation}
Recall Definition \ref{Def2.2}. As $\lceil z\rceil-1,\lfloor z'\rfloor \in [n]$, by (\ref{Eq5.1.1}) and Proposition \ref{P2.2.2},
\begin{equation}\label{Eq22.1.1}
    \mathbb{P}(|\mathcal{D}_{\lceil z \rceil-1}(\sigma_0)|\geq C_1\beta_n^{-1\slash 2})\leq C\exp(-c\beta_n^{-1\slash 2}),
\end{equation}
\begin{equation}\label{Eq22.1.2}
    \mathbb{P}(|\mathcal{D}'_{\lfloor z' \rfloor}(\sigma_0)|\geq C_1\beta_n^{-1\slash 2})\leq C\exp(-c\beta_n^{-1\slash 2}). 
\end{equation}
Let $\mathcal{Z}$ be the event that $|\mathcal{D}_{\lceil z \rceil-1}(\sigma_0)|\leq C_1\beta_n^{-1\slash 2}$ and $|\mathcal{D}'_{\lfloor z' \rfloor}(\sigma_0)|\leq C_1\beta_n^{-1\slash 2}$. By (\ref{Eq22.1.1}), (\ref{Eq22.1.2}), and the union bound, we have 
\begin{equation}\label{Eq22.1.15}
    \mathbb{P}(\mathcal{Z}^c)\leq C\exp(-c\beta_n^{-1\slash 2}). 
\end{equation}
When the event $\mathcal{Z}$ holds, as $\lfloor z' \rfloor-\lceil z\rceil\geq z'-z-2= L\beta_n^{-1\slash 2}-2$, $8C_1=L_0\leq L$, and $L\beta_n^{-1\slash 2}\geq 40$, we have
\begin{eqnarray}\label{Eq22.1.3}
 && |S(\sigma_0)\cap [z,z']^2| \geq |S(\sigma_0)\cap [\lceil z\rceil, \lfloor z' \rfloor]^2| \nonumber\\
 &\geq& |[\lceil z\rceil, \lfloor z' \rfloor]\cap\mathbb{N}^{*}|-|\mathcal{D}_{\lceil z \rceil-1}(\sigma_0)|-|\mathcal{D}'_{\lfloor z' \rfloor}(\sigma_0)| \nonumber\\
 &\geq& \lfloor z' \rfloor-\lceil z\rceil+1-2C_1\beta_n^{-1\slash 2}\geq L\beta_n^{-1\slash 2}-1-\frac{1}{4}L\beta_n^{-1\slash 2} \nonumber \\
 &=& \frac{3}{4}L\beta_n^{-1\slash 2}-1\geq \frac{1}{2}L\beta_n^{-1\slash 2}.
\end{eqnarray}

We let $\mathscr{M}$ be the set of $m\in [M]$ that satisfies $J_m\in \mathcal{I}_{n,s}\cap\mathbb{N}^{*}$. Below we consider any $m\in [n]$. If $m\in\mathscr{M}$, for any $m'\in [M]$ such that $\sigma_0(I_{m'})<J_m$ (note that there are $m-1$ such $m'$), we have $b_{m'}\leq \sigma_0(I_{m'})<J_m$, hence
\begin{equation}\label{Eq20.1.4}
    N_m=\sum_{\substack{m'\in [M]:\\\sigma_0(I_{m'})\geq J_m}} \mathbbm{1}_{b_{m'}\leq J_m}. 
\end{equation}
For any $i\in [n]$ such that $(i,\sigma_0(i))\in [z,z']^2$, we have $(i,\sigma_0(i))\in \mathcal{X}_s\times\mathcal{X}_s$. Hence there exists some $m'\in [M]$, such that $i=I_{m'}$. Let 
\begin{equation}
    \mathcal{M}:=\{m'\in [M]: (I_{m'},\sigma_0(I_{m'}))\in [z,z']^2\}. 
\end{equation}
By (\ref{Eq22.1.3}), when the event $\mathcal{Z}$ holds, we have
\begin{equation}\label{Eq22.1.11}
    |\mathcal{M}|\geq |S(\sigma_0)\cap [z,z']^2|\geq \frac{1}{2}L\beta_n^{-1\slash 2}. 
\end{equation}
If $m\in\mathscr{M}$, then $J_m\in \mathcal{I}_{n,s}$, hence for any $m'\in \mathcal{M}$, $\sigma_0(I_{m'})\geq z\geq J_m$. Hence by (\ref{Eq20.1.4}), we have
\begin{equation}\label{Eq22.1.6}
    N_m\geq \sum_{m'\in\mathcal{M}} \mathbbm{1}_{b_{m'}\leq J_m}.
\end{equation}
Now note that if $m\in\mathscr{M}$, conditional on $\sigma_0$, $\{\mathbbm{1}_{b_{m'}\leq J_m}\}_{m'\in\mathcal{M}}$ are mutually independent, and for each $m'\in\mathcal{M}$, $\mathbbm{1}_{b_{m'}\leq J_m}$ follows the Bernoulli distribution with 
\begin{eqnarray}\label{Eq22.1.5}
 \mathbb{P}(\mathbbm{1}_{b_{m'}\leq J_m}=1|\sigma_0)&=&\mathbb{P}(b_{m'}\leq J_m|\sigma_0)=\mathbb{P}(u_{m'}\leq e^{2\beta_n (I_{m'}-\alpha_s)J_m}|\sigma_0)\nonumber\\
 &=& e^{-2\beta_n(I_{m'}-\alpha_s)(\sigma_0(I_{m'})-J_m)}. 
\end{eqnarray}
If $m\in\mathscr{M}$, for any $m'\in\mathcal{M}$, as $(I_{m'},\sigma_0(I_{m'}))\in [z,z']^2$ and $J_m\in \mathcal{I}_{n,s}$, we have
\begin{equation*}
    I_{m'}-\alpha_s\begin{cases}
     \geq z-\alpha_s\geq 1 \\
     \leq z'-\alpha_s=2L\beta_n^{-1\slash 2}+1\leq 3L\beta_n^{-1\slash 2}
    \end{cases},
\end{equation*}
\begin{equation*}
   0\leq  \sigma_0(I_{m'})-J_{m} \leq z'-J_m\leq 2L\beta_n^{-1\slash 2}, 
\end{equation*}
hence by (\ref{Eq22.1.5}), we have
\begin{equation}\label{Eq22.1.7}
    \mathbb{P}(\mathbbm{1}_{b_{m'}\leq J_m}=1|\sigma_0)\geq e^{-12L^2}. 
\end{equation}
By (\ref{Eq22.1.6}), (\ref{Eq22.1.7}), and Hoeffding's inequality, for any $x\in [0,e^{-12L^2}]$, we have 
\begin{equation*}
    \mathbb{P}(N_m\leq (e^{-12L^2}-x)|\mathcal{M}||\sigma_0)\mathbbm{1}_{m\in\mathscr{M}}\leq e^{-2|\mathcal{M}|x^2} \mathbbm{1}_{m\in\mathscr{M}},
\end{equation*}
which by (\ref{Eq22.1.11}) leads to
\begin{equation*}
     \mathbb{P}(N_m\leq (e^{-12L^2}-x)L\beta_n^{-1\slash 2}\slash 2|\sigma_0)\mathbbm{1}_{m\in\mathscr{M}}\mathbbm{1}_{\mathcal{Z}}\leq e^{-L\beta_n^{-1\slash 2} x^2} \mathbbm{1}_{m\in\mathscr{M}}.
\end{equation*}
Taking $x=e^{-12L^2}\slash 2$, we have
\begin{eqnarray}\label{Eq22.1.8}
   && \mathbb{P}(\{N_m\leq e^{-12L^2}L\beta_n^{-1\slash 2}\slash 4\}\cap \{m\in\mathscr{M}\} \cap\mathcal{Z}|\sigma_0) \nonumber\\
   &=& \mathbb{P}(N_m\leq e^{-12L^2}L\beta_n^{-1\slash 2}\slash 4|\sigma_0) \mathbbm{1}_{m\in\mathscr{M}}\mathbbm{1}_{\mathcal{Z}}\leq e^{-c_L'\beta_n^{-1\slash 2}}\mathbbm{1}_{m\in\mathscr{M}}.
\end{eqnarray}

Let $\mathcal{C}$ be the event that $N_{m}\geq e^{-12L^2}L\beta_n^{-1\slash 2}\slash 4$ for any $m\in\mathscr{M}$. By (\ref{Eq22.1.8}) and the union bound, 
\begin{eqnarray}
&&\mathbb{P}(\mathcal{C}^c\cap\mathcal{Z}|\sigma_0)\leq 
\mathbb{P}\Big(\bigcup_{m=1}^n\big(\{N_{m}\leq e^{-12L^2}L\beta_n^{-1\slash 2} \slash 4\}\cap\{m\in \mathscr{M}\}\cap\mathcal{Z}\big)\Big|\sigma_0\Big)\nonumber\\
&\leq& \sum_{m=1}^n \mathbb{P}(\{N_{m}\leq e^{-12L^2} L \beta_n^{-1\slash 2} \slash 4\}\cap\{m\in \mathscr{M}\}\cap\mathcal{Z}|\sigma_0)\nonumber\\ &\leq& e^{-c_L'\beta_n^{-1\slash 2}}\sum_{m=1}^n \mathbbm{1}_{m\in \mathscr{M}}=|\mathscr{M}| e^{-c_L'\beta_n^{-1\slash 2}}
\leq |\mathcal{X}_s|e^{-c_L'\beta_n^{-1\slash 2}}\nonumber\\
&\leq& (2L\beta_n^{-1\slash 2}+1)e^{-c_L'\beta_n^{-1\slash 2}}\leq CL\beta_n^{-1\slash 2}\exp(-c_L'\beta_n^{-1\slash 2}).
\end{eqnarray}
Hence 
\begin{equation}\label{Eq22.1.12}
    \mathbb{P}(\mathcal{C}^c\cap\mathcal{Z})=\mathbb{E}[\mathbb{P}(\mathcal{C}^c\cap\mathcal{Z}|\sigma_0)]\leq CL\beta_n^{-1\slash 2}\exp(-c_L'\beta_n^{-1\slash 2}).
\end{equation}
By (\ref{Eq22.1.15}), (\ref{Eq22.1.12}), and the union bound, 
\begin{equation}\label{Eq26.2.1}
    \mathbb{P}(\mathcal{C}^c)\leq CL\beta_n^{-1\slash 2}\exp(-c_L'\beta_n^{-1\slash 2}).
\end{equation}

For any $m\in [M]$, we let $\mathscr{B}_{I_m}=b_m$, $\mathscr{N}_{J_{m}}=N_{m}$, and $\mathscr{Y}_{J_{m}}=I_{Y_{m}}$; for any $i\in\mathcal{X}_s\backslash \{I_1,\cdots,I_M\}$, we let $\mathscr{B}_i=n+1$; for any $j\in \mathcal{X}_s\backslash \{J_{1},\cdots,J_{M}\}$, we let $\mathscr{N}_{j}=n$ and $\mathscr{Y}_{j}=0$. We let $\mathcal{B}$ be the $\sigma$-algebra generated by $\sigma_0$ and $\{b_{m}\}_{m\in [n]}$. For any $j\in\mathcal{X}_s$, we let $\mathcal{F}_{j}$ be the $\sigma$-algebra generated by $\sigma_0$, $\{b_{m}\}_{m\in [n]}$, and $\{\mathscr{Y}_{l}\}_{l\in [j-1]\cap \mathcal{X}_s}$.

We assume that the event $\mathcal{C}$ holds. For any $m\in\mathscr{M}$, we have
\begin{equation*}
    \mathscr{N}_{J_m}=N_m\geq e^{-12L^2} L\beta_n^{-1\slash 2}\slash 4.
\end{equation*}
Hence for any $j\in \{J_1,\cdots,J_M\}\cap\mathcal{I}_{n,s}\cap\mathbb{N}^{*}$, $\mathscr{N}_j\geq e^{-12L^2} L\beta_n^{-1\slash 2}\slash 4$. Moreover, for any $j\in \{J_1,\cdots,J_M\}^c \cap\mathcal{I}_{n,s}\cap\mathbb{N}^{*}$, $\mathscr{N}_j=n\geq e^{-12L^2} L\beta_n^{-1\slash 2}\slash 4$ (note that $n\beta_n^{1\slash 2}\geq 4L$). Hence when the event $\mathcal{C}$ holds, for any $j\in \mathcal{I}_{n,s}\cap\mathbb{N}^{*}$, we have
\begin{equation}\label{Eq25.1.12}
    \mathscr{N}_j \geq e^{-12L^2} L\beta_n^{-1\slash 2}\slash 4.
\end{equation}

Recall (\ref{Eq24.1.1}) and (\ref{Eq24.1.2}). We let 
\begin{eqnarray}\label{Eq30.1.22}
&& \mathcal{S}_{1,l} :=\{i\in\{s_1,\cdots,s_2\}\backslash \{\mathscr{Y}_{s_3}, \cdots, \mathscr{Y}_{s_1'-1}\}: \mathscr{B}_i<s_1' \}, \nonumber\\
&& \mathcal{S}_{2,l} :=\{i\in\{s_1,\cdots,s_2\}\backslash \{\mathscr{Y}_{s_3}, \cdots, \mathscr{Y}_{s_1'-1}\}: s_1'\leq \mathscr{B}_i \leq s_2' \}, \nonumber\\
&& \mathcal{S}_l' :=\{i\in\{s_1,\cdots,s_2\}: s_1'\leq \mathscr{B}_i \leq s_2'\}, \quad W_l:=|\mathcal{S}_l'|.
\end{eqnarray}
Note that $\mathcal{S}_{2,l}\subseteq \mathcal{S}_l'$. We also let
\begin{eqnarray}\label{Eq25.1.2}
 && D_l:=|\{i\in [n]: (i,\sigma(i))\in  \tilde{Q}_{\Gamma,l}\}|, \nonumber\\
 &&  D_l':=|\{i\in [n]: (i,\sigma(i))\in  \tilde{Q}_{\Gamma,l}, i\in \mathcal{S}_{2,l}\}|.
\end{eqnarray}

We bound $W_l$ as follows. For any $i\in\{s_1,\cdots,s_2\}$, let $Z_i:=\mathbbm{1}_{s_1'\leq \mathscr{B}_i\leq s_2'}$. Note that $W_l=\sum_{i=s_1}^{s_2}Z_i$, and conditional on $\sigma_0$, $Z_{s_1},\cdots,Z_{s_2}$ are mutually independent. Conditional on $\sigma_0$, for any $i\in\{s_1,\cdots,s_2\}$, if $i\in \{I_1,\cdots,I_M\}$, $Z_i$ follows the Bernoulli distribution with 
\begin{eqnarray*}
  &&  \mathbb{P}(Z_i=1|\sigma_0)=e^{-2\beta_n(i-\alpha_s)(\sigma_0(i)-s_2')_{+}}-e^{-2\beta_n(i-\alpha_s)(\sigma_0(i)-s_1')_{+}} \nonumber\\
  &\leq& 1-e^{-2\beta_n(i-\alpha_s)((\sigma_0(i)-s_1')_{+}-(\sigma_0(i)-s_2')_{+})}\leq 1-e^{-2\beta_n(i-\alpha_s)(s_2'-s_1')} \nonumber\\
  &\leq& 2\beta_n (i-\alpha_s)(s_2'-s_1')\leq 2\beta_n(L\beta_n^{-1\slash 2}+1)(L\beta_n^{-1\slash 2}(y_l(\Gamma)-y_{l-1}(\Gamma)))\nonumber\\
  &\leq& 4L^2(y_l(\Gamma)-y_{l-1}(\Gamma));
\end{eqnarray*}
otherwise $Z_i=0$. Hence by Hoeffding's inequality, for any $t\geq 0$, we have
\begin{equation*}
    \mathbb{P}(W_l\geq (s_2-s_1+1)(4L^2(y_l(\Gamma)-y_{l-1}(\Gamma))+t)|\sigma_0)\leq e^{-2(s_2-s_1+1)t^2}. 
\end{equation*}
Taking $t=L^2(y_l(\Gamma)-y_{l-1}(\Gamma))$, we obtain that
\begin{eqnarray*}
&& \mathbb{P}(W_l\geq 5L^2(s_2-s_1+1)(y_l(\Gamma)-y_{l-1}(\Gamma)))\nonumber\\
&=& \mathbb{E}[\mathbb{P}(W_l\geq 5L^2(s_2-s_1+1)(y_l(\Gamma)-y_{l-1}(\Gamma))|\sigma_0)]\nonumber\\
&\leq& e^{-2L^4(s_2-s_1+1)(y_l(\Gamma)-y_{l-1}(\Gamma))^2}.
\end{eqnarray*}
Let $\mathcal{E}_l$ be the event that
\begin{equation}\label{Eq25.1.10}
    W_l\leq 5L^2(s_2-s_1+1)(y_l(\Gamma)-y_{l-1}(\Gamma)).
\end{equation}
We have
\begin{equation}\label{Eq26.2.2}
    \mathbb{P}(\mathcal{E}_l^c)\leq \exp(-2L^4(s_2-s_1+1)(y_l(\Gamma)-y_{l-1}(\Gamma))^2).
\end{equation}

Let 
\begin{eqnarray}\label{Eq25.1.3}
  L_{1,l}:=LIS(\sigma|_{\mathcal{S}_{1,l}\times ((s-1)L\beta_n^{-1\slash 2}+L\beta_n^{-1\slash 2} y_{l-1}(\Gamma),(s-1)L\beta_n^{-1\slash 2}+L\beta_n^{-1\slash 2} y_l(\Gamma)]}), \nonumber\\
  L_{2,l}:=LIS(\sigma|_{\mathcal{S}_{2,l}\times ((s-1)L\beta_n^{-1\slash 2}+L\beta_n^{-1\slash 2} y_{l-1}(\Gamma),(s-1)L\beta_n^{-1\slash 2}+L\beta_n^{-1\slash 2} y_l(\Gamma)]}).
\end{eqnarray}
Below we show that
\begin{equation}\label{Eq25.1.1}
    L_{1,l}\leq LIS(\sigma|_{\tilde{Q}_{\Gamma,l}}) \leq L_{1,l}+L_{2,l}.
\end{equation}
We denote $LIS(\sigma|_{\tilde{Q}_{\Gamma,l}})$ by $d$. By the definition of $LIS(\sigma|_{\tilde{Q}_{\Gamma,l}})$, there exist indices $i_1,\cdots,i_d\in [n]$, such that $i_1<\cdots<i_d$, $\sigma(i_1)<\cdots<\sigma(i_d)$, and for every $j\in [d]$, $(i_j,\sigma(i_j))\in \tilde{Q}_{\Gamma,l}$ (which leads to $s_1\leq i_j\leq s_2$ and $s_1'\leq \sigma(i_j)\leq s_2'$). For any $j\in [d]$, as $(i_j,\sigma(i_j))\in \tilde{Q}_{\Gamma,l}\cap [n]^2 \subseteq \mathcal{X}_s\times \mathcal{X}_s$, there exists some $m_j\in [M]$ such that $i_j=I_{m_j}$, hence $\mathscr{B}_{i_j}=\mathscr{B}_{I_{m_j}}=b_{m_j}$; according to the resampling algorithm for the $L^2$ model (see Section \ref{Sect.1.5}), $b_{m_j}\leq \sigma(I_{m_j})=\sigma(i_j)\leq s_2'$, hence $\mathscr{B}_{i_j}\leq s_2'$. For any $j\in [d]$, if $i_j=\mathscr{Y}_{r}$ for some $r\in \{s_3,\cdots,s_1'-1\}$, then we have $\mathscr{Y}_r>0$ and $r=\sigma(i_j)\geq s_1'$, which leads to a contradiction; hence $i_j\notin\{\mathscr{Y}_{s_3},\cdots,\mathscr{Y}_{s_1'-1}\}$. Thus we have $i_j\in \mathcal{S}_{1,l}\cup\mathcal{S}_{2,l}$ for any $j\in [d]$. Assume that $\{i_1,\cdots,i_d\}=\{k_1,\cdots,k_q\}\cup \{k_1',\cdots,k_{d-q}'\}$, where $q\in\{0\}\cup  [d]$, $k_1,\cdots,k_q\in\mathcal{S}_{1,l}$, $k_1<\cdots<k_q$, $k_1',\cdots,k_{d-q}'\in\mathcal{S}_{2,l}$, and $k_1'<\cdots<k_{d-q}'$. As $\sigma(k_1)<\cdots<\sigma(k_q)$ and
\begin{eqnarray*}
   && (k_1,\sigma(k_1)),\cdots,(k_q,\sigma(k_q)) \nonumber\\
   &\in&  \mathcal{S}_{1,l}\times ((s-1)L\beta_n^{-1\slash 2}+L\beta_n^{-1\slash 2} y_{l-1}(\Gamma),(s-1)L\beta_n^{-1\slash 2}+L\beta_n^{-1\slash 2} y_l(\Gamma)],
\end{eqnarray*}
we have $L_{1,l}\geq q$. Similarly, $L_{2,l}\geq d-q$. Hence $LIS(\sigma|_{\tilde{Q}_{\Gamma,l}})=d\leq L_{1,l}+L_{2,l}$. The inequality $L_{1,l}\leq LIS(\sigma|_{\tilde{Q}_{\Gamma,l}})$ follows from the fact that
\begin{eqnarray*}
 S_{1,l}\times ((s-1)L\beta_n^{-1\slash 2}+L\beta_n^{-1\slash 2} y_{l-1}(\Gamma),(s-1)L\beta_n^{-1\slash 2}+L\beta_n^{-1\slash 2} y_l(\Gamma)]\subseteq \tilde{Q}_{\Gamma,l}.
\end{eqnarray*}
We conclude that (\ref{Eq25.1.1}) holds. 

In the following, we bound $D_l'$, $L_{2,l}$, $D_l$, $L_{1,l}$ (as defined in (\ref{Eq25.1.2}) and (\ref{Eq25.1.3})) in \textbf{Sub-steps 1.1-1.4}, respectively. 

\subparagraph{Sub-step 1.1}

In this sub-step, we bound $D_l'$. Note that
\begin{equation}\label{Eq25.2.1}
    D_l'\leq \sum_{i=s_1'}^{s_2'}\mathbbm{1}_{\sigma^{-1}(i)\in \mathcal{S}_{2,l}}\leq \sum_{i=s_1'}^{s_2'}\mathbbm{1}_{\sigma^{-1}(i)\in\mathcal{S}_l'}.
\end{equation}
For any $i\in \{s_1',\cdots,s_2'\}$, if $\sigma^{-1}(i)\in \mathcal{S}_l'$, then $(\sigma^{-1}(i),i)\in \mathcal{X}_s\times \mathcal{X}_s$; according to the resampling algorithm for the $L^2$ model, $\mathscr{Y}_{i}=\sigma^{-1}(i)\in\mathcal{S}_l'$. Hence by (\ref{Eq25.2.1}), we have
\begin{equation}\label{Eq25.1.6}
    D_l'\leq \sum_{i=s_1'}^{s_2'} \mathbbm{1}_{\mathscr{Y}_i\in \mathcal{S}_l'}. 
\end{equation}

Conditional on $\mathcal{F}_{s_1'}$, we couple $\{\mathscr{Y}_i\}_{i=s_1'}^{s_2'}$ with mutually independent Bernoulli random variables $\{Y_i'\}_{i=s_1'}^{s_2'}$ with parameters (note that $W_l$ is $\mathcal{F}_{s_1'}$-measurable)
\begin{equation}\label{Eq25.1.7}
    \mathbb{P}(Y_i'=1|\mathcal{F}_{s_1'})=\min\Big\{\frac{W_l}{\mathscr{N}_i},1\Big\}, \quad \forall i\in\{s_1',\cdots,s_2'\}
\end{equation}
as follows. Sequentially for $i=s_1',\cdots,s_2'$, we do the following. If $i\notin \{J_1,\cdots,J_M\}$, we let $\mathscr{Y}_i=0$. Below we assume that $i=J_m$ for some $m\in [M]$. Assume that $\mathscr{Y}_{s_1'},\cdots, \mathscr{Y}_{i-1}$ have been sampled, and 
\begin{equation}\label{Eq25.1.4}
    \mathscr{B}_{\mathscr{Y}_j}\leq j, \quad \forall j\in \{s_3,\cdots, i-1\}\cap \{J_1,\cdots,J_M\}=\{J_1,\cdots,J_{m-1}\}.
\end{equation}
We let
\begin{eqnarray}\label{Eq25.1.5}    
\mathcal{S}_{l,i}''&:=&\mathcal{S}_l'\cap (\{j\in \{s_3,\cdots,s_4\}: \mathscr{B}_j\leq i\}\backslash \{\mathscr{Y}_{s_3},\cdots,\mathscr{Y}_{i-1}\})\nonumber\\
    &=& \{j\in \{s_1,\cdots,s_2\}: s_1'\leq \mathscr{B}_j\leq i\}\backslash \{\mathscr{Y}_{s_3},\cdots,\mathscr{Y}_{i-1}\}.
\end{eqnarray}
As $\mathscr{B}_j=n+1>i$ for any $j\in \{s_3,\cdots,s_4\}\backslash \{I_1,\cdots,I_M\}$ and $i=J_m$, we have
\begin{eqnarray*}
\mathscr{N}_i&=&\mathscr{N}_{J_m}=N_m=|\{m'\in [M]: b_{m'}\leq J_m\}|-m+1\nonumber\\
&=& |\{j\in \{s_3,\cdots,s_4\}:\mathscr{B}_j\leq i\}|-m+1\nonumber\\
&=& |\{j\in \{s_3,\cdots,s_4\}:\mathscr{B}_j\leq i\}\backslash \{\mathscr{Y}_{s_3},\cdots,\mathscr{Y}_{i-1}\}|,
\end{eqnarray*}
where we use (\ref{Eq25.1.4}) in the last line. Hence $|\mathcal{S}''_{l,i}|\leq \mathscr{N}_i$ and 
\begin{eqnarray}
   && |\{j\in \{s_3,\cdots,s_4\}:\mathscr{B}_j\leq i\}\backslash (\{\mathscr{Y}_{s_3},\cdots,\mathscr{Y}_{i-1}\}\cup \mathcal{S}''_{l,i})|\nonumber\\
   && =\mathscr{N}_i-|\mathcal{S}''_{l,i}|\geq \min\{W_l,\mathscr{N}_i\}-|\mathcal{S}''_{l,i}|\geq 0,
\end{eqnarray}
where we use the fact that $|\mathcal{S}''_{l,i}|\leq |\mathcal{S}'_l|=W_l$ in the last inequality. We let $\mathcal{S}'''_{l,i}$ be the set that consists of the smallest $\min\{W_l,\mathscr{N}_i\}-|\mathcal{S}''_{l,i}|$ elements in the set $\{j\in \{s_3,\cdots,s_4\}:\mathscr{B}_j\leq i\}\backslash (\{\mathscr{Y}_{s_3},\cdots,\mathscr{Y}_{i-1}\}\cup \mathcal{S}''_{l,i})$. If $Y_i'=1$, we pick $\mathscr{Y}_i$ uniformly from the set $\mathcal{S}_{l,i}''\cup \mathcal{S}'''_{l,i}$. If $Y_i'=0$, we pick $\mathscr{Y}_i$ uniformly from the set $\{j\in \{s_3,\cdots,s_4\}:\mathscr{B}_j\leq i\}\backslash (\{\mathscr{Y}_{s_3},\cdots,\mathscr{Y}_{i-1}\}\cup \mathcal{S}''_{l,i}\cup \mathcal{S}'''_{l,i} )$. Note that $\mathscr{B}_{\mathscr{Y}_i}\leq i$. 

It can be checked that $\{\mathscr{Y}_i\}_{i=s_1'}^{s_2'}$ has the desired conditional distribution given $\mathcal{F}_{s_1'}$ as specified by the resampling algorithm for the $L^2$ model. Therefore, the above procedure gives a valid coupling between $\{\mathscr{Y}_i\}_{i=s_1'}^{s_2'}$ and $\{Y_i'\}_{i=s_1'}^{s_2'}$ conditional on $\mathcal{F}_{s_1'}$. 

Now for any $i\in \{s_1',\cdots,s_2'\}$ such that $Y_i'=0$, we have $\mathscr{Y}_i\notin \mathcal{S}''_{l,i}$; as $\mathscr{Y}_i\in \{j\in \{s_3,\cdots,s_4\}:\mathscr{B}_j\leq i\}\backslash\{\mathscr{Y}_{s_3},\cdots,\mathscr{Y}_{i-1}\}$, by (\ref{Eq25.1.5}), we have $\mathscr{Y}_i\notin \mathcal{S}_l'$. Hence for any $i\in \{s_1',\cdots,s_2'\}$, we have $\mathbbm{1}_{\mathscr{Y}_i\in\mathcal{S}_l'}\leq Y_i'$. By (\ref{Eq25.1.6}), we have
\begin{equation}\label{Eq25.1.8}
    D_l'\leq \sum_{i=s_1'}^{s_2'} Y_i'. 
\end{equation}

By (\ref{Eq25.1.7}), (\ref{Eq25.1.8}), and Hoeffding's inequality, we obtain that for any $t\geq 0$, 
\begin{equation}\label{Eq25.1.9}
    \mathbb{P}\Big(D_l'\geq \sum_{i=s_1'}^{s_2'} \frac{W_l}{\mathscr{N}_i}+(s_2'-s_1'+1)t\Big|\mathcal{F}_{s_1'}\Big)\leq e^{-2(s_2'-s_1'+1)t^2} 
\end{equation}
Let $\mathcal{D}_l$ be the event that 
\begin{equation}\label{Eq39.1.2}
    D_l'\geq 30 L e^{12L^2} \beta_n^{1\slash 2} (s_2-s_1+1) (s_2'-s_1'+1) (y_l(\Gamma)-y_{l-1}(\Gamma)).
\end{equation}
Taking $t=L \beta_n^{1\slash 2} (s_2-s_1+1)(y_l(\Gamma)-y_{l-1}(\Gamma))$ in (\ref{Eq25.1.9}) and noting (\ref{Eq25.1.12}) and (\ref{Eq25.1.10}), we obtain that
\begin{equation*}
    \mathbb{P}(\mathcal{D}_l\cap \mathcal{C}\cap\mathcal{E}_l|\mathcal{F}_{s_1'})\leq \exp(-2L^2\beta_n (s_2-s_1+1)^2(s_2'-s_1'+1)(y_l(\Gamma)-y_{l-1}(\Gamma))^2 ). 
\end{equation*}
Hence
\begin{eqnarray}\label{Eq26.2.3}
   && \mathbb{P}(\mathcal{D}_l\cap \mathcal{C}\cap\mathcal{E}_l)=\mathbb{E}[\mathbb{P}(\mathcal{D}_l\cap \mathcal{C}\cap\mathcal{E}_l|\mathcal{F}_{s_1'})]\nonumber\\
   &\leq& \exp(-2L^2\beta_n (s_2-s_1+1)^2(s_2'-s_1'+1)(y_l(\Gamma)-y_{l-1}(\Gamma))^2 ).
\end{eqnarray}
By (\ref{Eq26.2.1}), (\ref{Eq26.2.2}), (\ref{Eq26.2.3}), and the union bound, we have
\begin{eqnarray}\label{Eq26.3.2}
    \mathbb{P}(\mathcal{D}_l)&\leq& CL\beta_n^{-1\slash 2}\exp(-c_L'\beta_n^{-1\slash 2})+\exp(-2L^4(s_2-s_1+1)(y_l(\Gamma)-y_{l-1}(\Gamma))^2) \nonumber\\
    &&+\exp(-2L^2\beta_n (s_2-s_1+1)^2(s_2'-s_1'+1)(y_l(\Gamma)-y_{l-1}(\Gamma))^2 ).
\end{eqnarray}
By (\ref{Eq20.2.1}), (\ref{Eq26.3.1}), and (\ref{Eq24.1.1}), we have
\begin{equation}\label{Eq26.3.3}
    s_2-s_1\geq L\beta_n^{-1\slash 2}(x_l(\Gamma)-x_{l-1}(\Gamma))-2\geq \frac{L\beta_n^{-1\slash 2}}{2K_0T}-2\geq \frac{L\beta_n^{-1\slash 2}}{4K_0T},
\end{equation}
\begin{equation}\label{Eq26.3.4}
    s_2'-s_1'\geq L\beta_n^{-1\slash 2}(y_l(\Gamma)-y_{l-1}(\Gamma))-2\geq \frac{L\beta_n^{-1\slash 2}}{2K_0T}-2\geq \frac{L\beta_n^{-1\slash 2}}{4K_0T}.
\end{equation}
By (\ref{Eq26.3.1}) and (\ref{Eq26.3.2})-(\ref{Eq26.3.4}), we have
\begin{equation}\label{Eq43.1.2}
    \mathbb{P}(\mathcal{D}_l)\leq CL\beta_n^{-1\slash 2}\exp(-c_L'\beta_n^{-1\slash 2})+2\exp(-L^5\beta_n^{-1\slash 2}\slash (128 K_0^5 T^5)). 
\end{equation}

\subparagraph{Sub-step 1.2}

In this sub-step, we bound $L_{2,l}$. For any $q\in \mathbb{N}^{*}$, we define 
\begin{equation}\label{Eq27.1.2}
    \Lambda_{l,q}:=\sum_{\substack{i_1<\cdots<i_q,j_1<\cdots<j_q\\ i_1,\cdots,i_q\in \{s_1,\cdots,s_2\} \\j_1,\cdots,j_q\in \{s_1',\cdots,s_2'\}}} \mathbbm{1}_{\sigma(i_1)=j_1,\cdots,\sigma(i_q)=j_q}\mathbbm{1}_{i_1,\cdots,i_q\in\mathcal{S}_{2,l}}.
\end{equation}

Consider any $i_1,\cdots,i_q\in \{s_1,\cdots,s_2\}$ and $j_1,\cdots,j_q\in \{s_1',\cdots,s_2'\}$ such that $i_1<\cdots<i_q$ and $j_1<\cdots<j_q$. For any $l\in [q]$, if $\sigma(i_l)=j_l$, then 
\begin{equation*}
    (i_l,\sigma(i_l))=(i_l,j_l)\in \{s_1,\cdots,s_2\}\times \{s_1',\cdots,s_2'\}\subseteq \mathcal{X}_s\times \mathcal{X}_s,
\end{equation*}
which implies $\mathscr{Y}_{j_l}=i_l$. Hence we have
\begin{eqnarray}\label{Eq27.1.1}
 &&   \mathbb{E}[\mathbbm{1}_{\sigma(i_1)=j_1,\cdots,\sigma(i_q)=j_q}\mathbbm{1}_{i_1,\cdots,i_q\in\mathcal{S}_{2,l}}|\mathcal{F}_{s_1'}]\leq \mathbb{E}[\mathbbm{1}_{\mathscr{Y}_{j_1}=i_1,\cdots,\mathscr{Y}_{j_q}=i_q}\mathbbm{1}_{i_1,\cdots,i_q\in\mathcal{S}_{2,l}}|\mathcal{F}_{s_1'}] \nonumber\\
 &=& \mathbbm{1}_{i_1,\cdots,i_q\in\mathcal{S}_{2,l}} \mathbb{E}[\mathbbm{1}_{\mathscr{Y}_{j_1}=i_1,\cdots,\mathscr{Y}_{j_q}=i_q}|\mathcal{F}_{s_1'}] \nonumber\\
 &=& \mathbbm{1}_{i_1,\cdots,i_q\in\mathcal{S}_{2,l}} \mathbb{E}[\mathbb{E}[\mathbbm{1}_{\mathscr{Y}_{j_q}=i_q}|\mathcal{F}_{j_q}]\mathbbm{1}_{\mathscr{Y}_{j_1}=i_1,\cdots,\mathscr{Y}_{j_{q-1}}=i_{q-1}}|\mathcal{F}_{s_1'}]\nonumber\\
 &\leq& \frac{\mathbbm{1}_{i_1,\cdots,i_q\in\mathcal{S}_{2,l}}}{\mathscr{N}_{j_q}} \mathbb{E}[\mathbbm{1}_{\mathscr{Y}_{j_1}=i_1,\cdots,\mathscr{Y}_{j_{q-1}}=i_{q-1}}|\mathcal{F}_{s_1'}]\leq \cdots\leq \frac{\mathbbm{1}_{i_1,\cdots,i_q\in\mathcal{S}_{2,l}}}{\mathscr{N}_{j_1}\mathscr{N}_{j_2}\cdots \mathscr{N}_{j_q}}.
\end{eqnarray}
By (\ref{Eq25.1.12}) and (\ref{Eq27.1.1}), we have
\begin{eqnarray}\label{Eq28.1.1}
   && \mathbb{E}[\mathbbm{1}_{\sigma(i_1)=j_1,\cdots,\sigma(i_q)=j_q}\mathbbm{1}_{i_1,\cdots,i_q\in\mathcal{S}_{2,l}}|\mathcal{F}_{s_1'}]\mathbbm{1}_{\mathcal{C}\cap\mathcal{E}_l} \nonumber\\
   &\leq& \frac{\mathbbm{1}_{i_1,\cdots,i_q\in\mathcal{S}_{2,l}}\mathbbm{1}_{\mathcal{C}\cap\mathcal{E}_l}}{\mathscr{N}_{j_1}\mathscr{N}_{j_2}\cdots \mathscr{N}_{j_q}} \leq (4e^{12L^2} L^{-1}\beta_n^{1\slash 2})^q \mathbbm{1}_{i_1,\cdots,i_q\in\mathcal{S}_{2,l}}\mathbbm{1}_{\mathcal{E}_l}.
\end{eqnarray} 

By (\ref{Eq25.1.10}), (\ref{Eq27.1.2}), (\ref{Eq28.1.1}), and Lemma \ref{Lemma3.1},
\begin{eqnarray}
 &&   \mathbb{E}[\Lambda_{l,q}|\mathcal{F}_{s_1'}]\mathbbm{1}_{\mathcal{C}\cap\mathcal{E}_l} \leq (4e^{12L^2} L^{-1}\beta_n^{1\slash 2})^q \mathbbm{1}_{\mathcal{E}_l} \sum_{\substack{i_1<\cdots<i_q,j_1<\cdots<j_q\\ i_1,\cdots,i_q\in \{s_1,\cdots,s_2\} \\j_1,\cdots,j_q\in \{s_1',\cdots,s_2'\}}} \mathbbm{1}_{i_1,\cdots,i_q\in\mathcal{S}_{2,l}} \nonumber\\
 &\leq& (4e^{12L^2} L^{-1}\beta_n^{1\slash 2})^q \mathbbm{1}_{\mathcal{E}_l}\binom{|\mathcal{S}_{2,l}|}{q}\binom{s_2'-s_1'+1}{q}\nonumber\\
 &\leq& \mathbbm{1}_{\mathcal{E}_l}\Big(\frac{4e^{12L^2+2} L^{-1}\beta_n^{1\slash 2}|\mathcal{S}_{2,l}|(s_2'-s_1'+1)}{q^2}\Big)^q \nonumber\\
 &\leq& \mathbbm{1}_{\mathcal{E}_l}\Big(\frac{4e^{12L^2+2} L^{-1}\beta_n^{1\slash 2} W_l (s_2'-s_1'+1)}{q^2}\Big)^q \nonumber\\
 &\leq& \Big(\frac{20 e^{12L^2+2} L\beta_n^{1\slash 2}(s_2-s_1+1)(s_2'-s_1'+1)(y_l(\Gamma)-y_{l-1}(\Gamma))}{q^2}\Big)^q.
\end{eqnarray}
Hence
\begin{eqnarray}\label{Eq28.1.3}
  &&  \mathbb{P}(\{\Lambda_{l,q}\geq 1\}\cap \mathcal{C}\cap\mathcal{E}_l)=\mathbb{E}[\mathbb{E}[\mathbbm{1}_{\Lambda_{l,q}\geq 1}|\mathcal{F}_{s_1'}]\mathbbm{1}_{\mathcal{C}\cap\mathcal{E}_l}]\leq \mathbb{E}[\mathbb{E}[\Lambda_{l,q}|\mathcal{F}_{s_1'}]\mathbbm{1}_{\mathcal{C}\cap\mathcal{E}_l}] \nonumber\\
&\leq&  \Big(\frac{20 e^{12L^2+2} L\beta_n^{1\slash 2}(s_2-s_1+1)(s_2'-s_1'+1)(y_l(\Gamma)-y_{l-1}(\Gamma))}{q^2}\Big)^q.\nonumber\\
&&
\end{eqnarray}
Let 
\begin{equation}
    q_0:=10e^{6L^2+1}L^{1\slash 2}\beta_n^{1\slash 4}(s_2-s_1+1)^{1\slash 2}(s_2'-s_1'+1)^{1\slash 2}(y_l(\Gamma)-y_{l-1}(\Gamma))^{1\slash 2}.
\end{equation}
Taking $q=\lceil q_0 \rceil$ in (\ref{Eq28.1.3}), we obtain that
\begin{equation*}
    \mathbb{P}(\{\Lambda_{l,\lceil q_0 \rceil }\geq 1\}\cap \mathcal{C}\cap\mathcal{E}_l) \leq 2^{-q_0},
\end{equation*}
which leads to
\begin{equation}\label{Eq30.1.2}
    \mathbb{P}(\{L_{2,l}\geq q_0+1\}\cap \mathcal{C}\cap\mathcal{E}_l) \leq 2^{-q_0}.
\end{equation}

By (\ref{Eq20.2.1}) and (\ref{Eq26.3.1}), we have
\begin{eqnarray*}
  &&  L^{3\slash 2} \beta_n^{-1\slash 4} T^{-1\slash 2}\sqrt{(x_l(\Gamma)-x_{l-1}(\Gamma))(y_l(\Gamma)-y_{l-1}(\Gamma))}   \nonumber\\
  && \geq \frac{1}{2} L^{3\slash 2}\beta_n^{-1\slash 4}T^{-3\slash 2} K_0^{-1}\geq \frac{1}{2}L \geq 1.
\end{eqnarray*}
Hence by (\ref{Eq26.3.1})-(\ref{Eq24.1.1}), (\ref{Eq26.3.3})-(\ref{Eq26.3.4}), and the AM-GM inequality, we have
\begin{equation}
    q_0\geq 10e^{6L^2+1}L^{1\slash 2}\beta_n^{1\slash 4}\cdot \frac{L\beta_n^{-1\slash 2}}{4K_0T}\cdot (2K_0T)^{-1\slash 2}\geq e^{6L^2} L^{3\slash 2} K_0^{-3\slash 2} T^{-3\slash 2} \beta_n^{-1\slash 4},
\end{equation}
\begin{eqnarray}\label{Eq30.1.3}
    q_0+1&\leq& 20e^{6L^2+1}L^{3\slash 2} \beta_n^{-1\slash 4} T^{-1\slash 2}\sqrt{(x_l(\Gamma)-x_{l-1}(\Gamma))(y_l(\Gamma)-y_{l-1}(\Gamma))}+1 \nonumber\\
    &\leq& 30 e^{6L^2+1}L^{3\slash 2} \beta_n^{-1\slash 4} T^{-1\slash 2}\sqrt{(x_l(\Gamma)-x_{l-1}(\Gamma))(y_l(\Gamma)-y_{l-1}(\Gamma))}\nonumber\\
    &\leq& 15e^{6L^2+1}L^{3\slash 2} \beta_n^{-1\slash 4} T^{-1\slash 2} (x_l(\Gamma)-x_{l-1}(\Gamma)+y_l(\Gamma)-y_{l-1}(\Gamma)).\nonumber\\
    &&
\end{eqnarray}
Let $\mathscr{E}_l$ be the event that 
\begin{equation}\label{Eq45.1.1}
    L_{2,l}\leq 15e^{6L^2+1}L^{3\slash 2} \beta_n^{-1\slash 4} T^{-1\slash 2} (x_l(\Gamma)-x_{l-1}(\Gamma)+y_l(\Gamma)-y_{l-1}(\Gamma)).
\end{equation}
By (\ref{Eq30.1.2})-(\ref{Eq30.1.3}), we have
\begin{equation}\label{Eq30.1.4}
    \mathbb{P}(\mathscr{E}_l^c\cap\mathcal{C}\cap\mathcal{E}_l) \leq \exp(-c e^{6L^2} L^{3\slash 2}\beta_n^{-1\slash 4}\slash (K_0^{3\slash 2} T^{3\slash 2})).  
\end{equation}
By (\ref{Eq26.2.1}), (\ref{Eq26.2.2}), (\ref{Eq30.1.4}), and the union bound, we have
\begin{eqnarray}
    \mathbb{P}(\mathscr{E}_l^c)&\leq& \exp(-c e^{6L^2} L^{3\slash 2}\beta_n^{-1\slash 4}\slash (K_0^{3\slash 2} T^{3\slash 2}))+CL\beta_n^{-1\slash 2}\exp(-c_L'\beta_n^{-1\slash 2})\nonumber\\
    &&+\exp(-2L^4(s_2-s_1+1)(y_l(\Gamma)-y_{l-1}(\Gamma))^2).
\end{eqnarray}
Noting (\ref{Eq26.3.1}) and (\ref{Eq26.3.3}), we obtain that 
\begin{equation}\label{Eq45.1.5}
    \mathbb{P}(\mathscr{E}_l^c)\leq CL\beta_n^{-1\slash 2} \exp(-c_L'\beta_n^{-1\slash 4}\slash(K_0^3T^3)).
\end{equation}

\subparagraph{Sub-step 1.3}

In this sub-step, we bound $D_l$. We let 
\begin{equation}
   t_s:=\lceil (s-1)L\beta_n^{-1\slash 2}\rceil.
\end{equation}
Note that $t_s\in [n]$. Recall the definition of $r_s$ from (\ref{Eq30.1.5}). 
As
\begin{equation*}
    \min\{(s-1)L\beta_n^{-1\slash 2}, (\lfloor n\beta_n^{1\slash 2}\slash L\rfloor -s +1 )L\beta_n^{-1\slash 2}\}\geq L\beta_n^{-1\slash 2}\geq 2,
\end{equation*}
we have
\begin{equation*}
    t_s-1\geq (s-1)L\beta_n^{-1\slash 2}-1\geq \frac{1}{2}(s-1)L\beta_n^{-1\slash 2}\geq r_s\beta_n^{-1\slash 2},
\end{equation*}
\begin{equation*}
  n-t_s\geq (\lfloor n\beta_n^{1\slash 2}\slash L\rfloor -s +1 )L\beta_n^{-1\slash 2}-1 \geq \frac{1}{2}(\lfloor n\beta_n^{1\slash 2}\slash L\rfloor -s +1 ) L\beta_n^{-1\slash 2}\geq r_s\beta_n^{-1\slash 2}.
\end{equation*}
Hence
\begin{equation}\label{Eq30.3.1}
    r_s\beta_n^{-1\slash 2}+1\leq t_s\leq n-r_s\beta_n^{-1\slash 2}. 
\end{equation}
Take $\beta=\beta_n$, $\delta_0=1\slash 4$, $K=2L$, $r=r_s$, $t_0=t_s$ in Proposition \ref{P2.3.2}. As $r_s\geq C_1'$ and $r_s\leq \log(1+\beta_n^{-1\slash 2})\leq \log(1+\beta_n^{-1\slash 2})^4$, noting (\ref{Eq30.3.1}), we obtain that
\begin{equation}\label{Eq30.1.20} \mathbb{P}\Big(\sup_{f\in\mathbb{B}_{2L}}\Big|\int f 
 d\tilde{\mu}_{n,t_s}-\int f d\tilde{\mu}\Big|>C_2(\log{r_s})^{1\slash 4}r_s^{-1\slash 20}\Big) \leq C_0\exp(-c_0 \beta_n^{-3\slash 8}),
\end{equation}
where we recall from Definition \ref{Def2.1} that 
\begin{equation*}
\tilde{\mu}_{n,t_s}=\beta_n^{1\slash 2}\sum_{i=1}^n\delta_{(\beta_n^{1\slash 2}(i-t_s),\beta_n^{1\slash 2}(\sigma(i)-t_s))}, \quad d\tilde{\mu}=\frac{1}{\sqrt{\pi}}e^{-(x-y)^2}dxdy.
\end{equation*}

Below we assume that the event
\begin{equation}\label{Eq31.1.1}
\sup_{f\in\mathbb{B}_{2L}}\Big|\int f 
 d\tilde{\mu}_{n,t_s}-\int f d\tilde{\mu}\Big|\leq C_2(\log{r_s})^{1\slash 4}r_s^{-1\slash 20}
\end{equation}
holds. For any $\mathbf{x}\in \mathbb{R}^2$, we let
\begin{equation*}
    g(\mathbf{x})=\mathbbm{1}_{\tilde{Q}_{\Gamma,l}} ((t_s,t_s)+\beta_n^{-1\slash 2}\mathbf{x}).
\end{equation*}
For any $\delta\in (0,1)$, we let
\begin{eqnarray*}
    \mathscr{R}_{\Gamma,l;\delta}&:=& ((s-1)L\beta_n^{-1\slash 2}+L\beta_n^{-1\slash 2}x_{l-1}(\Gamma)-\delta\beta_n^{-1\slash 2},(s-1)L\beta_n^{-1\slash 2}+L\beta_n^{-1\slash 2}x_{l}(\Gamma)+\delta\beta_n^{-1\slash 2}]\\
   &&\times ((s-1)L\beta_n^{-1\slash 2}+L\beta_n^{-1\slash 2}y_{l-1}(\Gamma)-\delta\beta_n^{-1\slash 2},(s-1)L\beta_n^{-1\slash 2}+L\beta_n^{-1\slash 2}y_{l}(\Gamma)+\delta\beta_n^{-1\slash 2}],
\end{eqnarray*}
\begin{eqnarray*}
  \mathscr{R}_{\Gamma,l;\delta}' &:=& ((s-1)L\beta_n^{-1\slash 2}+L\beta_n^{-1\slash 2}x_{l-1}(\Gamma)+\delta\beta_n^{-1\slash 2},(s-1)L\beta_n^{-1\slash 2}+L\beta_n^{-1\slash 2}x_{l}(\Gamma)-\delta\beta_n^{-1\slash 2}]\\
   &&\times ((s-1)L\beta_n^{-1\slash 2}+L\beta_n^{-1\slash 2}y_{l-1}(\Gamma)+\delta\beta_n^{-1\slash 2},(s-1)L\beta_n^{-1\slash 2}+L\beta_n^{-1\slash 2}y_{l}(\Gamma)-\delta\beta_n^{-1\slash 2}].
\end{eqnarray*}
For any $\delta\in (0,1)$ and $\mathbf{x}\in \mathbb{R}^2$, we let
\begin{equation*}
    g_{1,\delta}(\mathbf{x})=\min\{1,\delta^{-1}\beta_n^{1\slash 2}\mathbbm{1}_{\tilde{Q}_{\Gamma,l}}((t_s,t_s)+\beta_n^{-1\slash 2}\mathbf{x})d((t_s,t_s)+\beta_n^{-1\slash 2}\mathbf{x},\partial \tilde{Q}_{\Gamma,l})\},
\end{equation*}
\begin{equation*}
    g_{2,\delta}(\mathbf{x})=\min\{1,\delta^{-1}\beta_n^{1\slash 2} \mathbbm{1}_{\mathscr{R}_{\Gamma,l;\delta}}((t_s,t_s)+\beta_n^{-1\slash 2}\mathbf{x})d((t_s,t_s)+\beta_n^{-1\slash 2}\mathbf{x},\partial \mathscr{R}_{\Gamma,l;\delta})\},
\end{equation*}
where for any $\mathbf{x}\in\mathbb{R}^2$ and any set $A\subseteq \mathbb{R}^2$, $d(\mathbf{x},A):=\inf_{\mathbf{z}\in A}\|\mathbf{x}-\mathbf{z}\|_2$. In the following, we consider any $\delta\in (0,1)$. It can be checked that $\|g_{1,\delta}\|_{\infty}\leq 1$, $\|g_{2,\delta}\|_{\infty}\leq 1$, $\|g_{1,\delta}\|_{Lip}\leq \delta^{-1}$, and $\|g_{2,\delta}\|_{Lip}\leq \delta^{-1}$. Note that
\begin{eqnarray*}
   && \supp (g_{1,\delta}), \supp (g_{2,\delta}) \subseteq \beta_n^{1\slash 2} \overline{\mathscr{R}_{\Gamma,l;\delta}}-\beta_n^{1\slash 2}(t_s, t_s) \nonumber\\
   &=& [(s-1+x_{l-1}(\Gamma))L-\delta-\beta_n^{1\slash 2}t_s, (s-1+x_l(\Gamma))L+\delta-\beta_n^{1\slash 2}t_s] \nonumber\\
   && \times [(s-1+y_{l-1}(\Gamma))L-\delta-\beta_n^{1\slash 2}t_s, (s-1+y_l(\Gamma))L+\delta-\beta_n^{1\slash 2}t_s].
\end{eqnarray*}
As 
\begin{eqnarray*}
 && (s-1+\min\{x_{l-1}(\Gamma),y_{l-1}(\Gamma)\})L-\delta-\beta_n^{1\slash 2}t_s \nonumber\\
 &\geq& \beta_n^{1\slash 2}((s-1)L\beta_n^{-1\slash 2}-\lceil(s-1)L\beta_n^{-1\slash 2} \rceil)-\delta\geq -\beta_n^{1\slash 2}-\delta\geq -2\geq -L,
\end{eqnarray*}
\begin{eqnarray*}
    && (s-1+\max\{x_{l}(\Gamma),y_{l}(\Gamma)\})L+\delta-\beta_n^{1\slash 2}t_s \nonumber\\
 &\leq& \beta_n^{1\slash 2}((s-1)L\beta_n^{-1\slash 2}-\lceil(s-1)L\beta_n^{-1\slash 2} \rceil)+L+\delta\leq L+\delta\leq 2L,
\end{eqnarray*}
we have $\supp (g_{1,\delta}), \supp (g_{2,\delta}) \subseteq [-2L,2L]^2$. Hence $\delta g_{1,\delta}, \delta g_{2,\delta}\in \mathbb{B}_{2L}$ (recall Definition \ref{Def2.3}). By (\ref{Eq31.1.1}), as $r_s\geq \min\{L,\log(1+\beta_n^{-1\slash 2})\}\slash 2 \geq 1$, we have
\begin{equation}\label{Eq31.1.2}
 \Big|\int g_{1,\delta} 
 d\tilde{\mu}_{n,t_s}-\int g_{1,\delta} d\tilde{\mu}\Big|\leq C_2\delta^{-1}(\log{r_s})^{1\slash 4}r_s^{-1\slash 20}\leq C_3\delta^{-1} r_s^{-1\slash 25},
\end{equation}
\begin{equation}
 \Big|\int g_{2,\delta}
 d\tilde{\mu}_{n,t_s}-\int g_{2,\delta} d\tilde{\mu}\Big|\leq C_2\delta^{-1}(\log{r_s})^{1\slash 4}r_s^{-1\slash 20}\leq C_3\delta^{-1} r_s^{-1\slash 25},
\end{equation}
where $C_3$ is a positive constant that only depends on $L$. It can be checked that 
\begin{equation}
    g_{1,\delta}(\mathbf{x})\leq g(\mathbf{x})\leq g_{2,\delta}(\mathbf{x}) \text{ for any } \mathbf{x}\in \mathbb{R}^2,
\end{equation}
\begin{equation}\label{Eq31.1.3}
    \int g d\tilde{\mu}_{n,t_s}=\beta_n^{1\slash 2}\sum_{i=1}^n \mathbbm{1}_{\tilde{Q}_{\Gamma,l}}((i,\sigma(i)))=\beta_n^{1\slash 2}|S(\sigma)\cap \tilde{Q}_{\Gamma,l}|.
\end{equation}
By (\ref{Eq31.1.2})-(\ref{Eq31.1.3}), 
\begin{equation}\label{Eq30.1.10}
    D_l=|S(\sigma)\cap \tilde{Q}_{\Gamma,l}| \geq \beta_n^{-1\slash 2}\int g_{1,\delta} d\tilde{\mu}-C_3\beta_n^{-1\slash 2}\delta^{-1} r_s^{-1\slash 25},
\end{equation}
\begin{equation}\label{Eq30.1.11}
    D_l=|S(\sigma)\cap \tilde{Q}_{\Gamma,l}|  
\leq  \beta_n^{-1\slash 2}\int g_{2,\delta} d\tilde{\mu}+C_3\beta_n^{-1\slash 2}\delta^{-1} r_s^{-1\slash 25}.
\end{equation}
For any $\mathbf{x}=(x_1,x_2)\in\mathbb{R}^2$, 
\begin{equation}\label{Eq30.1.6}
    g_{1,\delta}(\mathbf{x})\geq \mathbbm{1}_{\mathscr{R}_{\Gamma,l;\delta}'}((t_s,t_s)+\beta_n^{-1\slash 2}\mathbf{x}), \quad g_{2,\delta}(\mathbf{x})\leq \mathbbm{1}_{\mathscr{R}_{\Gamma,l;\delta}}((t_s,t_s)+\beta_n^{-1\slash 2}\mathbf{x}).
\end{equation}
For any $\mathbf{x}=(x_1,x_2)\in \mathbb{R}^2$ such that $(t_s,t_s)+\beta_n^{-1\slash 2}\mathbf{x}\in \mathscr{R}_{\Gamma,l;\delta}$, we have  
\begin{equation}\label{Eq31.1.5}
  (y_{l-1}(\Gamma)-x_l(\Gamma))L-2\delta\leq  x_2-x_1\leq (y_l(\Gamma)-x_{l-1}(\Gamma))L+2\delta,
\end{equation}
hence by (\ref{Eq26.3.1}),
\begin{eqnarray}\label{Eq30.1.7}
   && |x_2-x_1|\nonumber\\
   &\leq& |y_{l-1}(\Gamma)-x_{l-1}(\Gamma)|L+\max\{|y_{l}(\Gamma)-y_{l-1}(\Gamma)|,|x_l(\Gamma)-x_{l-1}(\Gamma)|\}L+2\delta\nonumber\\
    &\leq& |y_{l-1}(\Gamma)-x_{l-1}(\Gamma)|L+L T^{-1}+2\delta;
\end{eqnarray}
moreover, by (\ref{Eq31.1.5}), we have
\begin{eqnarray*}
  &&  x_2-x_1-(y_l(\Gamma)-y_{l-1}(\Gamma))L-2\delta\\
  &\leq& (y_{l-1}(\Gamma)-x_{l-1}(\Gamma))L\leq x_2-x_1+(x_l(\Gamma)-x_{l-1}(\Gamma))L+2\delta,
\end{eqnarray*}
hence by (\ref{Eq26.3.1}),
\begin{eqnarray}\label{Eq30.1.9}
 && |x_2-x_1|\nonumber\\
 &\geq& |y_{l-1}(\Gamma)-x_{l-1}(\Gamma)|L-\max\{|y_l(\Gamma)-y_{l-1}(\Gamma)|,|x_l(\Gamma)-x_{l-1}(\Gamma)|\}L-2\delta\nonumber\\
 &\geq& |y_{l-1}(\Gamma)-x_{l-1}(\Gamma)|L-LT^{-1}-2\delta.
\end{eqnarray}
By (\ref{Eq30.1.6}), (\ref{Eq30.1.7}), and (\ref{Eq30.1.9}), we have
\begin{eqnarray}\label{Eq30.1.12}
    \int g_{1,\delta}d\tilde{\mu} &\geq& \frac{1}{\sqrt{\pi}}\int\mathbbm{1}_{\mathscr{R}_{\Gamma,l;\delta}'}((t_s,t_s)+\beta_n^{-1\slash 2}\mathbf{x})e^{-(x_2-x_1)^2}dx_1dx_2\nonumber\\
    &\geq& \frac{1}{\sqrt{\pi}} 
 \beta_n e^{-(|y_{l-1}(\Gamma)-x_{l-1}(\Gamma)|L+LT^{-1}+2\delta)^2}|\mathscr{R}_{\Gamma,l;\delta}'|\nonumber\\
    &\geq& \frac{1}{\sqrt{\pi}} e^{-(|y_{l-1}(\Gamma)-x_{l-1}(\Gamma)|L+LT^{-1}+2\delta)^2}(L(x_l(\Gamma)-x_{l-1}(\Gamma))-2\delta)_{+}\nonumber\\
    &&\times (L(y_l(\Gamma)-y_{l-1}(\Gamma))-2\delta)_{+},
\end{eqnarray}
\begin{eqnarray}\label{Eq30.1.14}
 \int g_{2,\delta}d\tilde{\mu} &\leq& \frac{1}{\sqrt{\pi}} \int \mathbbm{1}_{\mathscr{R}_{\Gamma,l;\delta}}((t_s,t_s)+\beta_n^{-1\slash 2}\mathbf{x})e^{-(x_2-x_1)^2}dx_1dx_2\nonumber\\
 &\leq& \frac{1}{\sqrt{\pi}}\beta_n e^{-(|y_{l-1}(\Gamma)-x_{l-1}(\Gamma)|L-LT^{-1}-2\delta)_{+}^2}|\mathscr{R}_{\Gamma,l;\delta}|\nonumber\\
 &\leq& \frac{1}{\sqrt{\pi}} e^{-(|y_{l-1}(\Gamma)-x_{l-1}(\Gamma)|L-LT^{-1}-2\delta)_{+}^2}(L(x_l(\Gamma)-x_{l-1}(\Gamma))+2\delta)\nonumber\\
 &&\times(L(y_l(\Gamma)-y_{l-1}(\Gamma))+2\delta).
\end{eqnarray}
Below we take $\delta=1\slash (4K_0T)$. By (\ref{Eq26.3.1}), we have
\begin{equation}\label{Eq30.1.15}
    \min\{x_l(\Gamma)-x_{l-1}(\Gamma),y_l(\Gamma)-y_{l-1}(\Gamma)\}\geq \frac{1}{2K_0T}=2\delta.
\end{equation}
As $\min\{T,K_0\}\geq L^2$, we have $\delta\leq 1\slash (4L^4)$. Hence by (\ref{Eq30.1.10})-(\ref{Eq30.1.11}) and (\ref{Eq30.1.12})-(\ref{Eq30.1.15}), we have
\begin{eqnarray*}
    D_l&\geq& -4C_3K_0T\beta_n^{-1\slash 2} r_s^{-1\slash 25}+\frac{1}{\sqrt{\pi}} L^2 \beta_n^{-1\slash 2}  e^{-(|y_{l-1}(\Gamma)-x_{l-1}(\Gamma)|L+2L^{-1})^2} \nonumber\\
    &&\quad\quad\quad\quad\quad\quad\quad\quad\quad\quad \times (1-L^{-1})^2 (x_l(\Gamma)-x_{l-1}(\Gamma))(y_l(\Gamma)-y_{l-1}(\Gamma)),
\end{eqnarray*}
\begin{eqnarray*}
    D_l&\leq& 4C_3K_0T\beta_n^{-1\slash 2} r_s^{-1\slash 25}+\frac{1}{\sqrt{\pi}} L^2\beta_n^{-1\slash 2} e^{-(|y_{l-1}(\Gamma)-x_{l-1}(\Gamma)|L-2L^{-1})_{+}^2}\nonumber\\
    &&\quad\quad\quad\quad\quad\quad\quad\quad\quad\quad \times (1+L^{-1})^2 (x_l(\Gamma)-x_{l-1}(\Gamma))(y_l(\Gamma)-y_{l-1}(\Gamma)).
\end{eqnarray*}
As $L\geq 4$, we have $1-L^{-1}\geq e^{-2L^{-1}}$ and $1+L^{-1}\leq e^{L^{-1}}$. Hence 
\begin{eqnarray}\label{Eq30.1.19}
    D_l&\geq& -4C_3K_0T\beta_n^{-1\slash 2} r_s^{-1\slash 25}+\frac{1}{\sqrt{\pi}} L^2 \beta_n^{-1\slash 2}  e^{-(|y_{l-1}(\Gamma)-x_{l-1}(\Gamma)|L+2L^{-1})^2-4L^{-1}} \nonumber\\
    &&\quad\quad\quad\quad\quad\quad\quad\quad\quad\quad \times   (x_l(\Gamma)-x_{l-1}(\Gamma))(y_l(\Gamma)-y_{l-1}(\Gamma)),
\end{eqnarray}
\begin{eqnarray}\label{Eq30.1.21}
    D_l&\leq& 4C_3K_0T\beta_n^{-1\slash 2} r_s^{-1\slash 25}+\frac{1}{\sqrt{\pi}} L^2\beta_n^{-1\slash 2} e^{-(|y_{l-1}(\Gamma)-x_{l-1}(\Gamma)|L-2L^{-1})_{+}^2+4L^{-1}}\nonumber\\
    &&\quad\quad\quad\quad\quad\quad\quad\quad\quad\quad \times  (x_l(\Gamma)-x_{l-1}(\Gamma))(y_l(\Gamma)-y_{l-1}(\Gamma)).
\end{eqnarray}

Let $\mathcal{H}_l$ be the event that (\ref{Eq30.1.19}) and (\ref{Eq30.1.21}) hold. By (\ref{Eq30.1.20}) and the above discussion, we have
\begin{equation}\label{Eq43.1.3}
    \mathbb{P}(\mathcal{H}_l^c)\leq C_0\exp(-c_0 \beta_n^{-3\slash 8}). 
\end{equation}

\subparagraph{Sub-step 1.4}

In this sub-step, we bound $L_{1,l}$. Recall the definition of $\mathcal{S}_{1,l}$ in (\ref{Eq30.1.22}). We let
\begin{equation}
    R:=|\{i\in [n]: (i,\sigma(i))\in \mathcal{S}_{1,l}\times \{s_1',s_1'+1,\cdots,s_2'\}\}|.
\end{equation}
We also let $\mathscr{I}_1,\cdots,\mathscr{I}_n\in \{0\}\cup [n]$ and $\mathscr{J}_1,\cdots,\mathscr{J}_n\in \{0\}\cup [n]$ be such that
\begin{equation*}
\mathscr{I}_{R+1}=\cdots=\mathscr{I}_n=0, \quad \mathscr{J}_{R+1}=\cdots=\mathscr{J}_n=0,
\end{equation*}
\begin{equation*}
1\leq \mathscr{I}_1<\cdots<\mathscr{I}_R, \quad 1\leq \mathscr{J}_1<\cdots<\mathscr{J}_R,
\end{equation*}
\begin{equation*}
 \{\mathscr{I}_1,\cdots,\mathscr{I}_R\}=\{i\in [n]: (i,\sigma(i))\in \mathcal{S}_{1,l}\times \{s_1',s_1'+1,\cdots,s_2'\}\},
\end{equation*}
\begin{equation*}
 \{\mathscr{J}_1,\cdots,\mathscr{J}_R\}=\{i\in [n]: (\sigma^{-1}(i),i)\in \mathcal{S}_{1,l}\times \{s_1',s_1'+1,\cdots,s_2'\}\}.
\end{equation*} 
As $\mathcal{S}_{1,l}\times \{s_1',\cdots,s_2'\}\subseteq \mathcal{X}_s\times \mathcal{X}_s$, by (\ref{Eq37.1.1}), we have 
\begin{equation}\label{Eq37.1.3}
\{\mathscr{I}_1,\cdots,\mathscr{I}_R\}\subseteq \{I_1,\cdots,I_M\}, \quad \{\mathscr{J}_1,\cdots,\mathscr{J}_R\}\subseteq \{J_1,\cdots,J_M\}.
\end{equation}

Note that for any $i\in [n]$ such that $(i,\sigma(i))\in \mathcal{S}_{1,l}\times \{s_1',s_1'+1,\cdots,s_2'\}$, we have $(i,\sigma(i))\in \tilde{Q}_{\Gamma,l}$ and $i\notin \mathcal{S}_{2,l}$. Hence we have
\begin{equation}\label{Eq34.1.1}
    R\leq D_l-D_l'. 
\end{equation}
Now consider any $i\in [n]$ such that $(i,\sigma(i))\in \tilde{Q}_{\Gamma,l}$ and $i\notin \mathcal{S}_{2,l}$. Note that $i\in \{s_1,\cdots,s_2\}$ and $\sigma(i)\in \{s_1',\cdots,s_2'\}$. If $i=\mathscr{Y}_{j}$ for some $j\in \{s_3,\cdots,s_1'-1\}$, then $\mathscr{Y}_j>0$ and $\sigma(i)=\sigma(\mathscr{Y}_j)=j<s_1'$, which leads to a contradiction. Hence 
\begin{equation}\label{Eq31.2.1}
    i \notin \{\mathscr{Y}_{s_3}, \cdots,\mathscr{Y}_{s_1'-1}\}.
\end{equation}
Note that $(i,\sigma(i))\in\mathcal{X}_s\times \mathcal{X}_s$; according to the resampling algorithm for the $L^2$ model, we have
\begin{equation}\label{Eq31.2.2}
 \mathscr{B}_{i}=\mathscr{B}_{\mathscr{Y}_{\sigma(i)}}\leq \sigma(i)\leq s_2'.
\end{equation}
As $i\notin \mathcal{S}_{2,l}$, by (\ref{Eq31.2.1}) and (\ref{Eq31.2.2}), we have $\mathscr{B}_i<s_1'$, hence $i\in\mathcal{S}_{1,l}$. Therefore, 
\begin{equation}\label{Eq34.1.2}
    R\geq D_l-D_l'.
\end{equation}
Combining (\ref{Eq34.1.1}) and (\ref{Eq34.1.2}), we conclude that 
\begin{equation}\label{Eq39.1.1}
    R=D_l-D_l'.
\end{equation}

Throughout the rest of the proof, we let $S_0$ be the set that consists solely of the empty mapping $\tau_0:\emptyset\rightarrow\emptyset$, and let $LIS(\tau_0):=0$. If $R\geq 1$, we let $\tau\in S_R$ be such that $\sigma(\mathscr{I}_s)=\mathscr{J}_{\tau(s)}$ for every $s\in [R]$. If $R=0$, we let $\tau$ be the empty mapping. In the following, we condition on $\mathcal{F}_{s_1'}$, and consider any $r\in [n]$, $i_1,\cdots,i_r\in [n]$, and $j_1,\cdots,j_r\in [n]$ such that 
\begin{equation*}
\mathbb{P}(R=r, \mathscr{I}_1=i_1,\cdots,\mathscr{I}_r=i_r,\mathscr{J}_1=j_1,\cdots,\mathscr{J}_r=j_r|\mathcal{F}_{s_1'})>0.
\end{equation*}
Note that by (\ref{Eq37.1.3}), 
\begin{equation}\label{Eq37.1.5}
    \{i_1,\cdots,i_r\}\subseteq \{I_1,\cdots,I_M\}, \quad \{j_1,\cdots,j_r\}\subseteq \{J_1,\cdots,J_M\}.
\end{equation}
According to the resampling algorithm for the $L^2$ model, conditional on $\mathcal{F}_{s_1'}$, the distribution of $\sigma$ is given by the uniform distribution on the following set: 
\begin{eqnarray*}
  &&  \{\kappa\in S_n: \kappa(s)\geq \mathscr{B}_s \text{ for every } s\in\{I_1,\cdots,I_M\}, \nonumber\\
  &&\quad \kappa(s)=\sigma_0(s)\text{ for every }s\in [n]\backslash \{I_1,\cdots,I_M\}, \nonumber\\
  &&\quad \kappa^{-1}(s)=\mathscr{Y}_s\text{ for every }s\in \{s_3,\cdots,s_1'-1\}\cap\{J_1,\cdots,J_M\}\},
\end{eqnarray*}
which has cardinality $\prod_{s\in \{s_1',\cdots,s_4\}\cap\{J_1,\cdots,J_M\} }\mathscr{N}_s$. For any $\eta\in S_r$, let $M_{r,\eta}$ be the following set (recall Definition \ref{Def1.5}):
\begin{eqnarray*}
  &&  \{\kappa\in S_n: \kappa(s)\geq \mathscr{B}_s \text{ for every } s\in\{I_1,\cdots,I_M\}, \nonumber\\
  &&\quad \kappa(s)=\sigma_0(s)\text{ for every }s\in [n]\backslash \{I_1,\cdots,I_M\}, \nonumber\\
  &&\quad \kappa^{-1}(s)=\mathscr{Y}_s\text{ for every }s\in \{s_3,\cdots,s_1'-1\}\cap\{J_1,\cdots,J_M\},\nonumber\\
  &&\quad S(\kappa)\cap (\mathcal{S}_{1,l}\times \{s_1',\cdots,s_2'\})=\{(i_s,j_{\eta(s)}):s\in [r]\}\}.
\end{eqnarray*}
Then for any $\eta\in S_r$, we have 
\begin{eqnarray}\label{Eq38.1.1}
    && \mathbb{P}(\{\tau=\eta\}\cap \{R=r,\mathscr{I}_1=i_1,\cdots,\mathscr{I}_r=i_r,\mathscr{J}_1=j_1,\cdots,\mathscr{J}_r=j_r\}|\mathcal{F}_{s_1'}) \nonumber  \\
   && =\frac{|M_{r,\eta}|}{\prod_{s\in \{s_1',\cdots,s_4\}\cap \{J_1,\cdots,J_M\}}\mathscr{N}_s}.
\end{eqnarray}

Now for any $\eta_1,\eta_2\in S_r$, we define a mapping $\psi_{\eta_1,\eta_2}: M_{r,\eta_1}\rightarrow M_{r,\eta_2}$ as follows. Let $\iota_{\eta_1,\eta_2}\in S_n$ be the unique permutation that maps $j_s$ to $j_{\eta_2\eta_1^{-1}(s)}$ for every $s\in [r]$ and fixes every element in $[n]\backslash \{j_1,\cdots,j_r\}$. Now for every $\kappa\in M_{r,\eta_1}$, we let $\psi_{\eta_1,\eta_2}(\kappa):=\iota_{\eta_1,\eta_2}\kappa$. Below we verify that $\psi_{\eta_1,\eta_2}(\kappa)\in M_{r,\eta_2}$. For every $s\in [n]\backslash \{i_1,\cdots,i_r\}$, we have $\kappa(s)\in [n]\backslash \{j_1,\cdots,j_r\}$. Hence for every $s\in\{I_1,\cdots,I_M\}\backslash \{i_1,\cdots,i_r\}$, 
\begin{equation}\label{Eq36.1.3}
    \iota_{\eta_1,\eta_2}\kappa(s)=\kappa(s)\geq \mathscr{B}_s;
\end{equation}
for every $s\in[n]\cap\{I_1,\cdots,I_M\}^c \cap\{i_1,\cdots,i_r\}^c=[n]\backslash \{I_1,\cdots,I_M\}$ (note (\ref{Eq37.1.5})),
\begin{equation}\label{Eq37.1.7}
    \iota_{\eta_1,\eta_2}\kappa(s)=\kappa(s)=\sigma_0(s). 
\end{equation}
For every $s\in [r]$, we have 
\begin{equation}\label{Eq36.1.1}
    \iota_{\eta_1,\eta_2}\kappa(i_s)=\iota_{\eta_1,\eta_2}(j_{\eta_1(s)})=j_{\eta_2(s)}.
\end{equation}
Note that for any $s\in [r]$, $j_{\eta_2(s)}\in\{s_1',\cdots,s_2'\}$. For any $s\in [r]$, as $i_s\in \mathcal{S}_{1,l}$, by (\ref{Eq36.1.1}), we have $\mathscr{B}_{i_s}<  s_1'\leq j_{\eta_2(s)}=\iota_{\eta_1,\eta_2}\kappa(i_s)$. Combining this with (\ref{Eq36.1.3}), we obtain that for every $s\in \{I_1,\cdots,I_M\}$, 
\begin{equation}
    \iota_{\eta_1,\eta_2}\kappa(s)\geq\mathscr{B}_{s}.
\end{equation} 
For any $s\in \{s_3,\cdots,s_1'-1\}\cap\{J_1,\cdots,J_M\}$, we have $s\notin \{j_1,\cdots,j_r\}$, hence
\begin{equation}
    \iota_{\eta_1,\eta_2}\kappa(\mathscr{Y}_s)=\iota_{\eta_1,\eta_2} (s)=s.
\end{equation}
Moreover, it can be checked that
\begin{equation}\label{Eq37.1.9}
    S(\iota_{\eta_1,\eta_2}\kappa) \cap (\mathcal{S}_{1,l}\times  \{s_1',\cdots,s_2'\}  )=\{(i_s,j_{\eta_2(s)}):s\in [r]\}.
\end{equation}
By (\ref{Eq37.1.7})-(\ref{Eq37.1.9}), $\iota_{\eta_1,\eta_2} \kappa \in M_{r,\eta_2}$. We can also verify that for any $\eta_1,\eta_2\in S_n$,
\begin{equation*}
\psi_{\eta_2,\eta_1}\psi_{\eta_1,\eta_2}=Id_{M_{r,\eta_1}}, \quad \psi_{\eta_1,\eta_2}\psi_{\eta_2,\eta_1}=Id_{M_{r,\eta_2}},
\end{equation*}
where for any set $A$, $Id_A$ denotes the identity map on $A$. We conclude that for any $\eta_1,\eta_2\in S_r$, $\psi_{\eta_1,\eta_2}$ is a bijection from $M_{r,\eta_1}$ to $M_{r,\eta_2}$, hence
\begin{equation}\label{Eq38.1.3}
|M_{r,\eta_1}|=|M_{r,\eta_2}|.
\end{equation}

By (\ref{Eq38.1.1}) and (\ref{Eq38.1.3}), we conclude that for any $\eta\in S_r$,
\begin{eqnarray}
  &&  \frac{\mathbb{P}(\{\tau=\eta\}\cap \{R=r,\mathscr{I}_1=i_1,\cdots,\mathscr{I}_r=i_r,\mathscr{J}_1=j_1,\cdots,\mathscr{J}_r=j_r\}|\mathcal{F}_{s_1'})}{\mathbb{P}(R=r,\mathscr{I}_1=i_1,\cdots,\mathscr{I}_r=i_r,\mathscr{J}_1=j_1,\cdots,\mathscr{J}_r=j_r|\mathcal{F}_{s_1'})}\nonumber\\
  && = \frac{|M_{r,\eta}|}{\sum_{\eta'\in S_r}|M_{r,\eta'}|}=\frac{1}{r!}. 
\end{eqnarray}
Let $\mathcal{B}_l'$ be the $\sigma$-algebra generated by $\sigma_0$, $\{b_m\}_{m\in [n]}$, $\{\mathscr{Y}_{l}\}_{l\in [s_1'-1]\cap \mathcal{X}_s}$, $R$, $\{\mathscr{I}_m\}_{m\in [n]}$, and $\{\mathscr{J}_m\}_{m\in [n]}$. Following the argument between (\ref{Eq3.17}) and (\ref{E3.1.6}), we can deduce that for any $\delta_0\in (0,1\slash 3)$, 
\begin{equation}
    \mathbb{P}(|LIS(\tau)-2\sqrt{R}|>R^{1\slash 2-\delta_0}|\mathcal{B}_l')\leq C_{\delta_0} \exp(-R^{(1-3\delta_0)\slash 2}),
\end{equation}
where $C_{\delta_0}$ is a positive constant that only depends on $\delta_0$. Taking $\delta_0=1\slash 6$ and noting that $L_{1,l}=LIS(\tau)$, we obtain that
\begin{equation}\label{Eq43.1.1}
    \mathbb{P}(|L_{1,l}-2\sqrt{R}|>R^{1\slash 3}|\mathcal{B}_l')\leq C \exp(-R^{1\slash 4}).
\end{equation}
By (\ref{Eq39.1.2}), (\ref{Eq30.1.19}), (\ref{Eq30.1.21}), and (\ref{Eq39.1.1}), when the event $\mathcal{D}_l^c\cap \mathcal{H}_l$ holds,
\begin{eqnarray}\label{Eq39.1.4}
    R&\leq& 4C_3K_0T\beta_n^{-1\slash 2} r_s^{-1\slash 25}+\frac{1}{\sqrt{\pi}} L^2\beta_n^{-1\slash 2} e^{-(|y_{l-1}(\Gamma)-x_{l-1}(\Gamma)|L-2L^{-1})_{+}^2+4L^{-1}}\nonumber\\
    &&\quad\quad\quad\quad\quad\quad\quad\quad\quad\quad \times  (x_l(\Gamma)-x_{l-1}(\Gamma))(y_l(\Gamma)-y_{l-1}(\Gamma)),
\end{eqnarray}
\begin{eqnarray}
    R&\geq& -4C_3K_0T\beta_n^{-1\slash 2} r_s^{-1\slash 25}+\frac{1}{\sqrt{\pi}} L^2 \beta_n^{-1\slash 2}  e^{-(|y_{l-1}(\Gamma)-x_{l-1}(\Gamma)|L+2L^{-1})^2-4L^{-1}} \nonumber\\
    &&\quad\quad\quad\quad\quad\quad\quad\quad\quad\quad \times   (x_l(\Gamma)-x_{l-1}(\Gamma))(y_l(\Gamma)-y_{l-1}(\Gamma))\nonumber\\
    &&-30 L e^{12L^2} \beta_n^{1\slash 2} (s_2-s_1+1) (s_2'-s_1'+1) (y_l(\Gamma)-y_{l-1}(\Gamma)).
\end{eqnarray}
By (\ref{Eq30.1.1}) and (\ref{Eq24.1.1}), we have
\begin{equation*}
    s_2-s_1+1\leq L\beta_n^{-1\slash 2}(x_l(\Gamma)-x_{l-1}(\Gamma))+1\leq 2L\beta_n^{-1\slash 2}(x_l(\Gamma)-x_{l-1}(\Gamma)),
\end{equation*}
\begin{equation*}
    s_2'-s_1'+1\leq L\beta_n^{-1\slash 2}(y_l(\Gamma)-y_{l-1}(\Gamma))+1\leq 2L\beta_n^{-1\slash 2}(y_l(\Gamma)-y_{l-1}(\Gamma)),
\end{equation*}
which by (\ref{Eq26.3.1}) leads to
\begin{eqnarray}
    && 30 L e^{12L^2} \beta_n^{1\slash 2} (s_2-s_1+1) (s_2'-s_1'+1) (y_l(\Gamma)-y_{l-1}(\Gamma))\nonumber\\
    &\leq& 120 \beta_n^{-1\slash 2}   L^3 e^{12L^2} T^{-1} (x_l(\Gamma)-x_{l-1}(\Gamma)) (y_l(\Gamma)-y_{l-1}(\Gamma)) \nonumber\\
    &\leq& \frac{1}{\sqrt{\pi}} L^2 \beta_n^{-1\slash 2}  e^{-(|y_{l-1}(\Gamma)-x_{l-1}(\Gamma)|L+2L^{-1})^2-4L^{-1}} \nonumber\\
    && \times (x_l(\Gamma)-x_{l-1}(\Gamma))(y_l(\Gamma)-y_{l-1}(\Gamma)) \times 300 L e^{20 L^2}T^{-1}.
\end{eqnarray}
Moreover, by (\ref{Eq26.3.1}),
\begin{eqnarray}\label{Eq39.1.5}
  &&  4C_3K_0T\beta_n^{-1\slash 2} r_s^{-1\slash 25} \nonumber\\
  &\leq& \frac{1}{\sqrt{\pi}} L^2 \beta_n^{-1\slash 2}  e^{-(|y_{l-1}(\Gamma)-x_{l-1}(\Gamma)|L+2L^{-1})^2-4L^{-1}} \nonumber\\
    && \times (x_l(\Gamma)-x_{l-1}(\Gamma))(y_l(\Gamma)-y_{l-1}(\Gamma)) \times 4\sqrt{\pi} C_3K_0T L^{-2} e^{5L^2} (2K_0T)^2 r_s^{-1\slash 25}\nonumber\\
    &\leq& \frac{1}{\sqrt{\pi}} L^2 \beta_n^{-1\slash 2}  e^{-(|y_{l-1}(\Gamma)-x_{l-1}(\Gamma)|L+2L^{-1})^2-4L^{-1}} \nonumber\\
    && \times (x_l(\Gamma)-x_{l-1}(\Gamma))(y_l(\Gamma)-y_{l-1}(\Gamma)) \times C_4 K_0^3T^3 r_s^{-1\slash 25},
\end{eqnarray}
where $C_4$ is a positive constant that only depends on $L$. 

By (\ref{Eq39.1.4})-(\ref{Eq39.1.5}), when the event $\mathcal{D}_l^c\cap \mathcal{H}_l$ holds, we have
\begin{eqnarray}\label{Eq42.1.1}
    R&\leq& \frac{1}{\sqrt{\pi}} L^2\beta_n^{-1\slash 2} e^{-(|y_{l-1}(\Gamma)-x_{l-1}(\Gamma)|L-2L^{-1})_{+}^2+4L^{-1}} (x_l(\Gamma)-x_{l-1}(\Gamma))\nonumber\\
&&\times (y_l(\Gamma)-y_{l-1}(\Gamma)) (1+C_4 K_0^3T^3r_s^{-1\slash 25}),
\end{eqnarray}
\begin{eqnarray}\label{Eq41.1.1}
  R&\geq& \frac{1}{\sqrt{\pi}} L^2 \beta_n^{-1\slash 2}  e^{-(|y_{l-1}(\Gamma)-x_{l-1}(\Gamma)|L+2L^{-1})^2-4L^{-1}} (x_l(\Gamma)-x_{l-1}(\Gamma))\nonumber\\
    && \times  (y_l(\Gamma)-y_{l-1}(\Gamma)) (1-C_4 K_0^3 T^3 r_s^{-1\slash 25}-300Le^{20L^2}T^{-1})_{+}.
\end{eqnarray}
Note that (\ref{Eq26.3.1}) and (\ref{Eq41.1.1}) imply 
\begin{equation}\label{Eq42.1.2}
    R\geq \frac{1}{8}\beta_n^{-1\slash 2} L^2 e^{-5L^2} K_0^{-2} T^{-2} (1-C_4 K_0^3 T^3 r_s^{-1\slash 25}-300Le^{20L^2}T^{-1})_{+}.
\end{equation}

We let
\begin{equation}\label{Eq46.1.1}
  \Phi_1:=e^{4L^{-1}} (1+C_4 K_0^3T^3r_s^{-1\slash 25}),
\end{equation}
\begin{equation}
    \Phi_2:=e^{-4L^{-1}}(1-C_4 K_0^3 T^3 r_s^{-1\slash 25}-300Le^{20L^2}T^{-1})_{+},
\end{equation}
\begin{equation}
    \Phi_3:=\max\Big\{\frac{1}{8}\beta_n^{-1\slash 2} L^2 e^{-5L^2} K_0^{-2} T^{-2} (1-C_4 K_0^3 T^3 r_s^{-1\slash 25}-300Le^{20L^2}T^{-1})_{+},1\Big\},
\end{equation}
\begin{equation}\label{Eq46.1.2}
    \Phi_4:=\frac{1}{8}\beta_n^{-1\slash 2} L^2 e^{-5L^2} K_0^{-2} T^{-2} (1-C_4 K_0^3 T^3 r_s^{-1\slash 25}-300Le^{20L^2}T^{-1})_{+}.
\end{equation}
By (\ref{Eq42.1.1})-(\ref{Eq42.1.2}), when the event $\{|L_{1,l}-2\sqrt{R}|\leq R^{1\slash 3}\}\cap\mathcal{D}_l^c\cap \mathcal{H}_l$ holds, 
\begin{eqnarray}
   && L_{1,l}\leq 2\sqrt{R}+2R^{1\slash 3}=2\sqrt{R}(1+\max\{R,1\}^{-1\slash 6}) \nonumber\\
   &\leq& 2\pi^{-1\slash 4}L\beta_n^{-1\slash 4}e^{-(|y_{l-1}(\Gamma)-x_{l-1}(\Gamma)|L-2L^{-1})_{+}^2\slash 2}\nonumber\\
   &&\times \sqrt{(x_l(\Gamma)-x_{l-1}(\Gamma))(y_l(\Gamma)-y_{l-1}(\Gamma))} \Phi_1^{1\slash 2}(1+\Phi_3^{-1\slash 6}),
\end{eqnarray}
\begin{eqnarray}
    && L_{1,l}\geq 2\sqrt{R}-2R^{1\slash 3} =2\sqrt{R}(1-\max\{R,1\}^{-1\slash 6})\nonumber\\
    &\geq& 2\pi^{-1\slash 4}L\beta_n^{-1\slash 4}e^{-(|y_{l-1}(\Gamma)-x_{l-1}(\Gamma)|L+2L^{-1})^2\slash 2}\nonumber\\
    && \times \sqrt{(x_l(\Gamma)-x_{l-1}(\Gamma))(y_l(\Gamma)-y_{l-1}(\Gamma))} \Phi_2^{1\slash 2}(1-\Phi_3^{-1\slash 6}).
\end{eqnarray}
Let $\mathscr{E}_l'$ be the event that
\begin{eqnarray}\label{Eq45.1.3}
   && 2\pi^{-1\slash 4}L\beta_n^{-1\slash 4}e^{-(|y_{l-1}(\Gamma)-x_{l-1}(\Gamma)|L+2L^{-1})^2\slash 2}\nonumber\\
    && \times \sqrt{(x_l(\Gamma)-x_{l-1}(\Gamma))(y_l(\Gamma)-y_{l-1}(\Gamma))} \Phi_2^{1\slash 2}(1-\Phi_3^{-1\slash 6})\nonumber\\
    &\leq& L_{1,l} \leq 2\pi^{-1\slash 4}L\beta_n^{-1\slash 4}e^{-(|y_{l-1}(\Gamma)-x_{l-1}(\Gamma)|L-2L^{-1})_{+}^2\slash 2}\nonumber\\
   &&\times \sqrt{(x_l(\Gamma)-x_{l-1}(\Gamma))(y_l(\Gamma)-y_{l-1}(\Gamma))} \Phi_1^{1\slash 2}(1+\Phi_3^{-1\slash 6}).
\end{eqnarray}
We have $\{|L_{1,l}-2\sqrt{R}|\leq R^{1\slash 3}\}\cap\mathcal{D}_l^c\cap \mathcal{H}_l \subseteq \mathscr{E}_l'\cap\mathcal{D}_l^c\cap \mathcal{H}_l $, which by (\ref{Eq42.1.2}) leads to
\begin{eqnarray}
 (\mathscr{E}_l')^c\cap\mathcal{D}_l^c\cap \mathcal{H}_l &\subseteq &  \{|L_{1,l}-2\sqrt{R}|> R^{1\slash 3}\}\cap\mathcal{D}_l^c\cap \mathcal{H}_l \nonumber\\
 &\subseteq &  \{|L_{1,l}-2\sqrt{R}|> R^{1\slash 3}\}\cap\{R\geq \Phi_4\}. 
\end{eqnarray}
Hence by (\ref{Eq43.1.1}), 
\begin{eqnarray}\label{Eq43.1.4}
    \mathbb{P}((\mathscr{E}_l')^c\cap\mathcal{D}_l^c\cap \mathcal{H}_l) &\leq& \mathbb{P}(\{|L_{1,l}-2\sqrt{R}|> R^{1\slash 3}\}\cap\{R\geq \Phi_4\})\nonumber\\
    &=&  \mathbb{E}[\mathbb{P}(|L_{1,l}-2\sqrt{R}|>R^{1\slash 3}|\mathcal{B}_l') \mathbbm{1}_{R\geq \Phi_4} ]\nonumber\\
    &\leq& C\mathbb{E}[\exp(-R^{1\slash 4})\mathbbm{1}_{R\geq \Phi_4}]\leq C\exp(-\Phi_4^{1\slash 4}).
\end{eqnarray}
By (\ref{Eq43.1.2}), (\ref{Eq43.1.3}), (\ref{Eq43.1.4}), and the union bound, we have
\begin{eqnarray}\label{Eq45.1.7}
   && \mathbb{P}((\mathscr{E}_l')^c) \leq \mathbb{P}((\mathscr{E}_l')^c\cap\mathcal{D}_l^c\cap \mathcal{H}_l)+\mathbb{P}(\mathcal{D}_l)+\mathbb{P}(\mathcal{H}_l^c)\nonumber\\
   &\leq& C\exp(-\Phi_4^{1\slash 4})+CL\beta_n^{-1\slash 2}\exp(-c_L'\beta_n^{-1\slash 2})\nonumber\\
   &&+2\exp(-L^5\beta_n^{-1\slash 2}\slash (128 K_0^5 T^5))+C_0\exp(-c_0 \beta_n^{-3\slash 8}).
\end{eqnarray}

\medskip

Let $\mathscr{C}_{\Gamma,l}$ be the event that 
\begin{eqnarray}\label{Eq45.3.1}
    && 2\pi^{-1\slash 4}L\beta_n^{-1\slash 4}e^{-(|y_{l-1}(\Gamma)-x_{l-1}(\Gamma)|L+2L^{-1})^2\slash 2}\nonumber\\
    && \times \sqrt{(x_l(\Gamma)-x_{l-1}(\Gamma))(y_l(\Gamma)-y_{l-1}(\Gamma))} \Phi_2^{1\slash 2}(1-\Phi_3^{-1\slash 6})\nonumber\\
    &\leq& LIS(\sigma|_{\tilde{Q}_{\Gamma,l}}) \nonumber\\
    &\leq& 15e^{6L^2+1}L^{3\slash 2} \beta_n^{-1\slash 4} T^{-1\slash 2} (x_l(\Gamma)-x_{l-1}(\Gamma)+y_l(\Gamma)-y_{l-1}(\Gamma))\nonumber\\
    &&+2\pi^{-1\slash 4}L\beta_n^{-1\slash 4}e^{-(|y_{l-1}(\Gamma)-x_{l-1}(\Gamma)|L-2L^{-1})_{+}^2\slash 2}\nonumber\\
   &&\quad\times\sqrt{(x_l(\Gamma)-x_{l-1}(\Gamma))(y_l(\Gamma)-y_{l-1}(\Gamma))} \Phi_1^{1\slash 2}(1+\Phi_3^{-1\slash 6}).
\end{eqnarray}
By (\ref{Eq25.1.1}), (\ref{Eq45.1.1}), and (\ref{Eq45.1.3}), we have $\mathscr{E}_l\cap\mathscr{E}'_l\subseteq \mathscr{C}_{\Gamma,l}$. Hence by (\ref{Eq45.1.5}), (\ref{Eq45.1.7}), and the union bound, we have
\begin{eqnarray}
   &&  \mathbb{P}((\mathscr{C}_{\Gamma,l})^c)\leq \mathbb{P}(\mathscr{E}_l^c)+\mathbb{P}((\mathscr{E}'_l)^c) \nonumber\\
   &\leq& C\exp(-\Phi_4^{1\slash 4})+C_L'\beta_n^{-1\slash 2}\exp(-c_L'\beta_n^{-1\slash 4}\slash (K_0^5 T^5)).
\end{eqnarray}

\paragraph{Step 2}

Throughout the rest of the proof, we take
\begin{equation*}
    T=\lceil 300L^2 e^{20L^2}\rceil, \quad K_0=2\lceil L^2\rceil +1.
\end{equation*}
We note that $\min\{T,K_0\}\geq L^2$ and $\max\{8K_0T,K_0^2T^3\}\leq C'L^{10}e^{60L^2}$, where $C'\geq 1$ is an absolute constant. We also assume that $\beta_n^{-1\slash 2}\geq C'L^{10}e^{60L^2}$. Note that this implies (\ref{Eq20.2.1}) and $\beta_n\leq 1\slash 100$. 

Recalling (\ref{Eq46.1.1})-(\ref{Eq46.1.2}), we have
\begin{equation}
    \Phi_1\leq e^{4L^{-1}}(1+C_L r_s^{-1\slash 25}), \quad \Phi_2\geq e^{-4L^{-1}}(1-L^{-1}-C_L r_s^{-1\slash 25})_{+},
\end{equation}
\begin{equation}
    \Phi_3\geq \max\{c_L\beta_n^{-1\slash 2}(1-L^{-1}-C_L r_s^{-1\slash 25})_{+},1\}, 
\end{equation}
\begin{equation}\label{Eq45.4.1}
    \Phi_4\geq c_L\beta_n^{-1\slash 2}(1-L^{-1}-C_L r_s^{-1\slash 25})_{+},
\end{equation}
where $C_L,c_L$ are positive constants that only depend on $L$. In the following, we denote
\begin{equation}
    \Psi_s:=(1-L^{-1}-C_L r_s^{-1\slash 25})_{+}.
\end{equation}

For any $\Gamma\in \Pi^{T,T,K_0}$ and any $l\in [2T-1]$, we let $\mathscr{D}_{\Gamma,l}$ be the event that 
\begin{eqnarray}\label{Eq45.2.4}
      && 2\pi^{-1\slash 4}L\beta_n^{-1\slash 4}\sqrt{(x_l(\Gamma)-x_{l-1}(\Gamma))(y_l(\Gamma)-y_{l-1}(\Gamma))}\nonumber\\
    && \times  e^{-(|y_{l-1}(\Gamma)-x_{l-1}(\Gamma)|L+2L^{-1})^2\slash 2-2L^{-1}}\Psi_s^{1\slash 2}
    (1-\max\{c_L\beta_n^{-1\slash 2}\Psi_s,1\}^{-1\slash 6})\nonumber\\
    &\leq& LIS(\sigma|_{\tilde{Q}_{\Gamma,l}}) \nonumber\\
    &\leq& 15e^{6L^2+1}L^{3\slash 2} \beta_n^{-1\slash 4} T^{-1\slash 2} (x_l(\Gamma)-x_{l-1}(\Gamma)+y_l(\Gamma)-y_{l-1}(\Gamma))+1\nonumber\\
    &&+2\pi^{-1\slash 4}L\beta_n^{-1\slash 4}\sqrt{(x_l(\Gamma)-x_{l-1}(\Gamma))(y_l(\Gamma)-y_{l-1}(\Gamma))}\nonumber\\
   &&\quad\times e^{-(|y_{l-1}(\Gamma)-x_{l-1}(\Gamma)|L-2L^{-1})_{+}^2\slash 2+2L^{-1}}  \nonumber \\
   &&\quad\times  (1+C_L r_s^{-1\slash 25})^{1\slash 2}  (1+\max\{c_L\beta_n^{-1\slash 2}\Psi_s,1\}^{-1\slash 6}).
\end{eqnarray}
By (\ref{Eq45.3.1})-(\ref{Eq45.4.1}), we have
\begin{eqnarray}\label{Eq45.5.1}
    \mathbb{P}((\mathscr{D}_{\Gamma,l})^c)&\leq& C\exp(-\Phi_4^{1\slash 4})+C_L'\beta_n^{-1\slash 2}\exp(-c_L'\beta_n^{-1\slash 4}\slash (K_0^5 T^5)) \nonumber\\ 
    &\leq& C_L'\beta_n^{-1\slash 2}\exp(-c_L'\beta_n^{-1\slash 8}\Psi_s^{1\slash 4}).
\end{eqnarray}

For any $\Gamma\in \Pi^{T,T,K_0}$ and any $l\in [2T-1]$, we let $\mathscr{D}'_{\Gamma,l}$ be the event that
\begin{eqnarray}\label{Eq46.2.1}
      && 2\pi^{-1\slash 4}L\beta_n^{-1\slash 4}\sqrt{(c_l(\Gamma)-a_{l-1}(\Gamma))(d_l(\Gamma)-b_{l-1}(\Gamma))}\nonumber\\
    && \times  e^{-(|b_{l-1}(\Gamma)-a_{l-1}(\Gamma)|L+2L^{-1})^2\slash 2-2L^{-1}}\Psi_s^{1\slash 2}
    (1-\max\{c_L\beta_n^{-1\slash 2}\Psi_s,1\}^{-1\slash 6})\nonumber\\
    &\leq& LIS(\sigma|_{\tilde{Q}'_{\Gamma,l}}) \nonumber\\
    &\leq& 15e^{6L^2+1}L^{3\slash 2} \beta_n^{-1\slash 4} T^{-1\slash 2} (c_l(\Gamma)-a_{l-1}(\Gamma)+d_l(\Gamma)-b_{l-1}(\Gamma))+1\nonumber\\
    &&+2\pi^{-1\slash 4}L\beta_n^{-1\slash 4}\sqrt{(c_l(\Gamma)-a_{l-1}(\Gamma))(d_l(\Gamma)-b_{l-1}(\Gamma))}\nonumber\\
   &&\quad\times e^{-(|b_{l-1}(\Gamma)-a_{l-1}(\Gamma)|L-2L^{-1})_{+}^2\slash 2+2L^{-1}}  \nonumber \\
   &&\quad\times  (1+C_L r_s^{-1\slash 25})^{1\slash 2}  (1+\max\{c_L\beta_n^{-1\slash 2}\Psi_s,1\}^{-1\slash 6}).
\end{eqnarray}
Similarly, we have
\begin{equation}\label{Eq45.5.2}
    \mathbb{P}((\mathscr{D}'_{\Gamma,l})^c)\leq C_L'\beta_n^{-1\slash 2}\exp(-c_L'\beta_n^{-1\slash 8}\Psi_s^{1\slash 4}).
\end{equation}

Now we let 
\begin{equation}
    \mathscr{A}:=\bigcap_{\Gamma\in \Pi^{T,T,K_0}}\bigcap_{l=1}^{2T-1}(\mathscr{D}_{\Gamma,l}\cap \mathscr{D}_{\Gamma,l}').
\end{equation}
By (\ref{Eq45.5.1}), (\ref{Eq45.5.2}), and the union bound, we have
\begin{equation}\label{Eq47.1.1}
    \mathbb{P}(\mathscr{A}^c)\leq C_L'\beta_n^{-1\slash 2}\exp(-c_L'\beta_n^{-1\slash 8}\Psi_s^{1\slash 4}).
\end{equation}

\paragraph{Step 3}

Let $\Gamma_0\in \Pi^{T,T,K_0}$ be 
\begin{equation*}
    (1,1), \frac{K_0+1}{2}, (2,1), \frac{K_0+1}{2}, (2,2), \frac{K_0+1}{2}, \cdots, (T,T-1), \frac{K_0+1}{2}, (T,T). 
\end{equation*}
We have $(x_0(\Gamma_0),y_0(\Gamma_0))=(0,0)$, $(x_{2T-1}(\Gamma_0),y_{2T-1}(\Gamma_0))=(1,1)$. For any $l\in [2T-2]$,
\begin{equation*}
    (x_l(\Gamma_0),y_l(\Gamma_0))=\Big(\frac{l+1}{2T},\frac{l}{2T}\Big).
\end{equation*}
By Lemma \ref{L1}, we have 
\begin{equation}\label{Eq45.2.2}
    LIS(\sigma|_{\mathcal{R}_s})\geq \sum_{l=1}^{2T-1} LIS(\sigma|_{\tilde{Q}_{\Gamma_0,l}}). 
\end{equation}
When the event $\mathscr{A}$ holds, by (\ref{Eq45.2.4}) and (\ref{Eq45.2.2}), we have 
\begin{eqnarray}\label{Eq45.2.3}
 && LIS(\sigma|_{\mathcal{R}_s}) \nonumber\\
 &\geq& 2\pi^{-1\slash 4}L\beta_n^{-1\slash 4} \cdot \frac{2T-3}{2T} \cdot e^{-4L^{-1}}\Psi_s^{1\slash 2} (1-\max\{c_L\beta_n^{-1 \slash 2}\Psi_s,1\}^{-1\slash 6}) \nonumber\\ 
 &\geq& 2\pi^{-1\slash 4}L\beta_n^{-1\slash 4} e^{-6L^{-1}}\Psi_s^{1\slash 2} (1-\max\{c_L\beta_n^{-1 \slash 2}\Psi_s,1\}^{-1\slash 6}),
\end{eqnarray}
where we use the fact that $1-3\slash (2T)\geq 1-L^{-1}\geq e^{-2L^{-1}}$. 

Below we consider any $\Gamma\in \Pi^{T,T,K_0}$. Following the argument in (\ref{Eq2.19.4}), we obtain that
\begin{eqnarray}\label{Eq46.2.2}
   &&  \sum_{l=1}^{2T-1}\sqrt{(c_l(\Gamma)-a_{l-1}(\Gamma))(d_l(\Gamma)-b_{l-1}(\Gamma))}\nonumber\\
   &\leq&  \frac{1}{2}\sum_{l=1}^{2T-1}(c_l(\Gamma)-a_{l-1}(\Gamma)+d_l(\Gamma)-b_{l-1}(\Gamma))\leq 1+L^{-1}.
\end{eqnarray}
When the event $\mathscr{A}$ holds, by (\ref{Eq46.2.1}) and (\ref{Eq46.2.2}), we have 
\begin{eqnarray}\label{Eq46.2.3}
   && \sum_{l=1}^{2T-1} LIS(\sigma|_{\tilde{Q}'_{\Gamma,l}}) \nonumber\\
   &\leq& 15e^{6L^2+1}L^{3\slash 2} \beta_n^{-1\slash 4} T^{-1\slash 2} \sum_{l=1}^{2T-1} (c_l(\Gamma)-a_{l-1}(\Gamma)+d_l(\Gamma)-b_{l-1}(\Gamma)) +2T-1\nonumber\\
   &&  +2\pi^{-1\slash 4}L\beta_n^{-1\slash 4}e^{2L^{-1}}(1+C_L r_s^{-1\slash 25})^{1\slash 2}  (1+\max\{c_L\beta_n^{-1\slash 2}\Psi_s,1\}^{-1\slash 6})\nonumber\\ 
   &&\quad \times \sum_{l=1}^{2T-1} \sqrt{(c_l(\Gamma)-a_{l-1}(\Gamma))(d_l(\Gamma)-b_{l-1}(\Gamma))}\nonumber\\
   &\leq& 2\pi^{-1\slash 4}L\beta_n^{-1\slash 4} e^{3L^{-1}}(1+C_L r_s^{-1\slash 25})^{1\slash 2}  (1+\max\{c_L\beta_n^{-1\slash 2}\Psi_s,1\}^{-1\slash 6})\nonumber\\
   &&+6 L^{1\slash 2} e^{-4L^2} \beta_n^{-1\slash 4}+800L^2 e^{20L^2}.
\end{eqnarray}

By Lemma \ref{L1} and (\ref{Eq46.2.3}), when the event $\mathscr{A}$ holds, we have
\begin{eqnarray}\label{Eq46.2.4}
    &&  LIS(\sigma|_{\mathcal{R}_s})\leq \max_{\Gamma\in \Pi^{T,T,K_0}}\Big\{\sum_{l=1}^{2T-1}LIS(\sigma|_{\tilde{Q}_{\Gamma,l}'})\Big\} \nonumber\\
    &\leq& 2\pi^{-1\slash 4}L\beta_n^{-1\slash 4} e^{3L^{-1}}(1+C_L r_s^{-1\slash 25})^{1\slash 2}  (1+\max\{c_L\beta_n^{-1\slash 2}\Psi_s,1\}^{-1\slash 6})\nonumber\\
   &&+6 L^{1\slash 2} e^{-4L^2} \beta_n^{-1\slash 4}+800L^2 e^{20L^2}.
\end{eqnarray}

By (\ref{Eq45.2.3}) and (\ref{Eq46.2.4}), when the event $\mathscr{A}$ holds, we have
\begin{eqnarray}\label{Eq47.1.2}
  &&  |LIS(\sigma|_{\mathcal{R}_s})-2\pi^{-1\slash 4}L\beta_n^{-1\slash 4}|\leq 6 L^{1\slash 2} e^{-4L^2} \beta_n^{-1\slash 4}+800L^2 e^{20L^2}\nonumber\\   && +2\pi^{-1\slash 4}L\beta_n^{-1\slash 4} \max\Big\{1-e^{-6L^{-1}}\Psi_s^{1\slash 2} (1-\max\{c_L\beta_n^{-1 \slash 2}\Psi_s,1\}^{-1\slash 6}), \nonumber\\
  && \quad\quad e^{3L^{-1}}(1+C_L r_s^{-1\slash 25})^{1\slash 2}  (1+\max\{c_L\beta_n^{-1\slash 2}\Psi_s,1\}^{-1\slash 6})-1\Big\}.
\end{eqnarray}
Note that $LIS(\sigma|_{\mathcal{R}_s})\leq |\mathcal{I}_{n,s}\cap\mathbb{N}^{*}|\leq L\beta_n^{-1\slash 2}+1\leq 2L\beta_n^{-1\slash 2}$. 

Hence by (\ref{Eq47.1.1}) and (\ref{Eq47.1.2}), we have
\begin{eqnarray}
  &&  \mathbb{E}[|LIS(\sigma|_{\mathcal{R}_s})-2\pi^{-1\slash 4}L\beta_n^{-1\slash 4}|] \nonumber\\
  &\leq& (2L\beta_n^{-1\slash 2})(C_L'\beta_n^{-1\slash 2}\exp(-c_L'\beta_n^{-1\slash 8}\Psi_s^{1\slash 4}))+6 L^{1\slash 2} e^{-4L^2} \beta_n^{-1\slash 4}+800L^2 e^{20L^2}\nonumber\\ 
  && +2\pi^{-1\slash 4}L\beta_n^{-1\slash 4} \max\Big\{1-e^{-6L^{-1}}\Psi_s^{1\slash 2} (1-\max\{c_L\beta_n^{-1 \slash 2}\Psi_s,1\}^{-1\slash 6}), \nonumber\\ 
  && \quad\quad\quad e^{3L^{-1}}(1+C_L r_s^{-1\slash 25})^{1\slash 2}  (1+\max\{c_L\beta_n^{-1\slash 2}\Psi_s,1\}^{-1\slash 6})-1\Big\}\nonumber\\
  &\leq& C_L'\beta_n^{-1}\exp(-c_L'\beta_n^{-1\slash 8}\Psi_s^{1\slash 4})+C_L'+C L^{1\slash 2} e^{-4L^2} \beta_n^{-1\slash 4}\nonumber\\
  &&+2\pi^{-1\slash 4}L\beta_n^{-1\slash 4} \max\Big\{1-e^{-6L^{-1}}\Psi_s^{1\slash 2} (1-\max\{c_L\beta_n^{-1 \slash 2}\Psi_s,1\}^{-1\slash 6}), \nonumber\\
  && \quad\quad\quad e^{3L^{-1}}(1+C_L r_s^{-1\slash 25})^{1\slash 2}  (1+\max\{c_L\beta_n^{-1\slash 2}\Psi_s,1\}^{-1\slash 6})-1\Big\}.
\end{eqnarray}

\end{proof}

\subsection{Proof of Theorem \ref{limit_l2_2}}\label{Sect.5.2}

In this subsection, we finish the proof of Theorem \ref{limit_l2_2} based on Propositions \ref{P5.1}-\ref{P5.3}.

\begin{proof}[Proof of Theorem \ref{limit_l2_2}]

Throughout the proof, we fix an arbitrary sequence of positive numbers $(\beta_n)_{n=1}^{\infty}$ such that $\lim_{n\rightarrow\infty}\beta_n=0$ and $\lim_{n\rightarrow\infty} n^2\beta_n=\infty$. For each $n\in \mathbb{N}^{*}$, we let $\gamma_n:=n^{1\slash 2}\beta_n^{1\slash 4}$. Note that
\begin{equation}\label{Eq49.1.9}
    \lim_{n\rightarrow\infty} \gamma_n=\infty, \quad \lim_{n\rightarrow\infty} \frac{\gamma_n}{n \beta_n^{1\slash 2}}=0.
\end{equation}
We denote by $C_1$ the constant $C$ that appears in Proposition \ref{P2.2.2} (with $C_0=1$). Without loss of generality, we assume that $C_1\geq 1$. We let $L_0=8C_1$, and fix any $L\geq 4$ such that $L\slash L_0\in\mathbb{N}^{*}$. 

Let $C_1',C_L,c_L,C_L',c_L',C'$ and $r_s,\Psi_s$ be defined as in Proposition \ref{P5.3}. In the following, we assume that $n\in\mathbb{N}^{*}$ is sufficiently large, so that 
\begin{eqnarray}\label{Eq48.1.1}
    && n\beta_n^{1\slash 2}\geq 20L, \quad \beta_n^{-1\slash 2}\geq C'L^{10}e^{60L^2}, \quad \gamma_n\in [2,n\beta_n^{1\slash 2}\slash (4L)],  \nonumber\\
    && \min\{(\gamma_n-1)L,\log(1+\beta_n^{-1\slash 2})\}\geq 2\max\{(C_L L)^{25}, C_1'\}.
\end{eqnarray}

Let $\mathcal{S}_1:=[\gamma_n,n\beta_n^{1\slash 2}\slash L-\gamma_n]\cap\mathbb{N}$. As
\begin{equation*}
    \gamma_n\geq 2,\quad n\beta_n^{1\slash 2}\slash L-\gamma_n\leq n\beta_n^{1\slash 2}\slash L-2\leq \lfloor n\beta_n^{1\slash 2}\slash L \rfloor-1, 
\end{equation*}
we have $ \mathcal{S}_1 \subseteq [2,\lfloor n\beta_n^{1\slash 2}\slash L\rfloor-1]\cap\mathbb{N}$. Let $\mathcal{S}_2:=[\lfloor n\beta_n^{1\slash 2}\slash L \rfloor]\backslash \mathcal{S}_1$. Note that
\begin{equation}\label{Eq49.1.1}
    |\mathcal{S}_1|\leq n\beta_n^{1\slash 2}\slash L, \quad |\mathcal{S}_1|\geq n\beta_n^{1\slash 2}\slash L-2\gamma_n-1\geq  n\beta_n^{1\slash 2}\slash L-3\gamma_n,
\end{equation}
\begin{equation}\label{Eq49.1.5}
    |\mathcal{S}_2|\leq n\beta_n^{1\slash 2}\slash L-|\mathcal{S}_1|\leq 3\gamma_n.
\end{equation}

By (\ref{Eq48.1.1}), for any $s\in \mathcal{S}_1$, we have 
\begin{equation*}
    r_s\geq \frac{1}{2}\min\{(\gamma_n-1)L,\log(1+\beta_n^{-1\slash 2})\}\geq \max\{(C_L L)^{25},C_1'\},
\end{equation*}
hence $\Psi_s\geq 1-2L^{-1}\geq 1\slash 2$. By Proposition \ref{P5.3}, for any $s\in\mathcal{S}_1$, we have
\begin{eqnarray}\label{Eq49.1.3}
  &&  \mathbb{E}[|LIS(\sigma|_{\mathcal{R}_s})-2\pi^{-1\slash 4}L\beta_n^{-1\slash 4}|] \nonumber\\
  &\leq& C_L'\beta_n^{-1}\exp(-c_L'\beta_n^{-1\slash 8}\slash 2)+C_L'+C L^{1\slash 2} e^{-4L^2} \beta_n^{-1\slash 4}\nonumber\\ 
  &&+2\pi^{-1\slash 4}L\beta_n^{-1\slash 4} \max\Big\{1-e^{-6L^{-1}}(1-2L^{-1})^{1\slash 2} (1-\max\{c_L\beta_n^{-1 \slash 2}\slash 2,1\}^{-1\slash 6}), \nonumber\\
  && \quad\quad\quad\quad\quad\quad\quad e^{3L^{-1}}(1+L^{-1})^{1\slash 2}  (1+\max\{c_L\beta_n^{-1\slash 2}\slash 2,1\}^{-1\slash 6})-1\Big\}.
\end{eqnarray}
By (\ref{Eq48.1.2}) and (\ref{Eq48.1.4}), we have
\begin{eqnarray}\label{Eq49.1.7}
   && \mathbb{E}[|LIS(\sigma)-2\pi^{-1\slash 4}n\beta_n^{1\slash 4}|] \nonumber\\
  &\leq&  \sum_{s\in \mathcal{S}_1}\mathbb{E}[|LIS(\sigma|_{\mathcal{R}_s})-2\pi^{-1\slash 4}L\beta_n^{-1\slash 4}|]+|2\pi^{-1\slash 4}n\beta_n^{1\slash 4}-2\pi^{-1\slash 4}L\beta_n^{-1\slash 4}|\mathcal{S}_1||\nonumber\\
  &&+\sum_{s\in \mathcal{S}_2}\mathbb{E}[LIS(\sigma|_{\mathcal{R}_s})]+\sum_{s=1}^{\lfloor n\beta_n^{1\slash 2}\slash L\rfloor-1}\mathbb{E}[LIS(\sigma|_{\mathcal{R}_s'})]+\sum_{s=1}^{\lfloor n\beta_n^{1\slash 2}\slash L\rfloor-1}\mathbb{E}[LIS(\sigma|_{\mathcal{R}_s''})].\nonumber\\
  && 
\end{eqnarray}
By (\ref{Eq49.1.1}) and (\ref{Eq49.1.3}), 
\begin{eqnarray}
  &&  \sum_{s\in \mathcal{S}_1}\mathbb{E}[|LIS(\sigma|_{\mathcal{R}_s})-2\pi^{-1\slash 4}L\beta_n^{-1\slash 4}|]\nonumber\\
  &\leq&  C_L'n \beta_n^{-1\slash 2}\exp(-c_L'\beta_n^{-1\slash 8}\slash 2)+C_L'n\beta_n^{1\slash 2}+C L^{-1\slash 2} e^{-4L^2} n \beta_n^{1\slash 4}\nonumber\\
  &&+2\pi^{-1\slash 4}n\beta_n^{1\slash 4} \max\Big\{1-e^{-6L^{-1}}(1-2L^{-1})^{1\slash 2} (1-\max\{c_L\beta_n^{-1 \slash 2}\slash 2,1\}^{-1\slash 6}), \nonumber\\
  &&  \quad\quad\quad\quad\quad\quad e^{3L^{-1}}(1+L^{-1})^{1\slash 2}   (1+\max\{c_L\beta_n^{-1\slash 2}\slash 2,1\}^{-1\slash 6})-1\Big\}.
\end{eqnarray}
By (\ref{Eq49.1.1}),
\begin{equation}
    0\leq 2\pi^{-1\slash 4}n\beta_n^{1\slash 4}-2\pi^{-1\slash 4}L\beta_n^{-1\slash 4}|\mathcal{S}_1| \leq CL \gamma_n \beta_n^{-1\slash 4}.
\end{equation}
By Proposition \ref{P5.2}, (\ref{Eq48.1.1}), and (\ref{Eq49.1.5}),
\begin{eqnarray}
    \sum_{s\in \mathcal{S}_2}\mathbb{E}[LIS(\sigma|_{\mathcal{R}_s})]&\leq& CL\beta_n^{-1\slash 4}|\mathcal{S}_2|+CL^2\exp(-c\beta_n^{-1\slash 4})|\mathcal{S}_2| \nonumber\\
    &\leq& CL\gamma_n\beta_n^{-1\slash 4}+CL^2\gamma_n\exp(-c\beta_n^{-1\slash 4})\nonumber\\
    &\leq& CL\gamma_n\beta_n^{-1\slash 4}+CL^2n\beta_n^{1\slash 2}\exp(-c\beta_n^{-1\slash 4}).
\end{eqnarray}
By Proposition \ref{P5.1}, 
\begin{eqnarray}
    \sum_{s=1}^{\lfloor n\beta_n^{1\slash 2}\slash L\rfloor-1}\mathbb{E}[LIS(\sigma|_{\mathcal{R}_s'})]&\leq& (  n\beta_n^{1\slash 2}\slash L)(CL^{1\slash 2}\beta_n^{-1\slash 4}+CL^2\exp(-c\beta_n^{-1\slash 4}))\nonumber\\
    &\leq& CL^{-1\slash 2}n\beta_n^{1\slash 4}+CLn\beta_n^{1\slash 2}\exp(-c\beta_n^{-1\slash 4}),
\end{eqnarray}
\begin{eqnarray}\label{Eq49.1.8}
    \sum_{s=1}^{\lfloor n\beta_n^{1\slash 2}\slash L\rfloor-1}\mathbb{E}[LIS(\sigma|_{\mathcal{R}_s''})]&\leq& (  n\beta_n^{1\slash 2}\slash L)(CL^{1\slash 2}\beta_n^{-1\slash 4}+CL^2\exp(-c\beta_n^{-1\slash 4}))\nonumber\\
    &\leq& CL^{-1\slash 2}n\beta_n^{1\slash 4}+CLn\beta_n^{1\slash 2}\exp(-c\beta_n^{-1\slash 4}).
\end{eqnarray}
By (\ref{Eq49.1.7})-(\ref{Eq49.1.8}), we have
\begin{eqnarray}
  && \frac{\mathbb{E}[|LIS(\sigma)-2\pi^{-1\slash 4}n\beta_n^{1\slash 4}|]}{n\beta_n^{1\slash 4}} \nonumber\\
  &\leq& C_L'\beta_n^{-3\slash 4}\exp(-c_L'\beta_n^{-1\slash 8}\slash 2)+C_L'\beta_n^{1\slash 4}+CL^{-1\slash 2}\nonumber\\
  && +2\pi^{-1\slash 4} \max\Big\{1-e^{-6L^{-1}}(1-2L^{-1})^{1\slash 2} (1-\max\{c_L\beta_n^{-1 \slash 2}\slash 2,1\}^{-1\slash 6}), \nonumber\\
  && \quad\quad\quad\quad \quad\quad  \quad e^{3L^{-1}}(1+L^{-1})^{1\slash 2}  (1+\max\{c_L\beta_n^{-1\slash 2}\slash 2,1\}^{-1\slash 6})-1\Big\}\nonumber\\
  &&+\frac{CL\gamma_n}{n\beta_n^{1\slash 2}}+CL^2\beta_n^{1\slash 4}\exp(-c\beta_n^{-1\slash 4}).
\end{eqnarray}
Hence by (\ref{Eq49.1.9}), 
\begin{eqnarray}
   && \limsup_{n\rightarrow\infty} \Big\{ \frac{\mathbb{E}[|LIS(\sigma)-2\pi^{-1\slash 4}n\beta_n^{1\slash 4}|]}{n\beta_n^{1\slash 4}} \Big\} \nonumber\\
   &\leq&   CL^{-1\slash 2}+2\pi^{-1\slash 4}\max\{1-e^{-6L^{-1}}(1-2L^{-1})^{1\slash 2},e^{3L^{-1}}(1+L^{-1})^{1\slash 2}-1\}. \nonumber\\
   &&
\end{eqnarray}
Taking $L\rightarrow\infty$, we obtain that
\begin{equation}
    \limsup_{n\rightarrow\infty} \Big\{ \frac{\mathbb{E}[|LIS(\sigma)-2\pi^{-1\slash 4}n\beta_n^{1\slash 4}|]}{n\beta_n^{1\slash 4}} \Big\}\leq 0.
\end{equation}
Hence
\begin{equation}
    \lim_{n\rightarrow \infty } \mathbb{E}\Big[\Big|\frac{LIS(\sigma)}{n\beta_n^{1\slash 4}}-2\pi^{-1\slash 4}\Big|\Big]=0, \text{ i.e., }  \frac{LIS(\sigma)}{n \beta_n^{1\slash 4}}\xrightarrow[]{L^1}  2\pi^{-1\slash 4}.
\end{equation}

\end{proof}

\begin{appendices}
  
\section{Proofs of Proposition \ref{Densi.l1} and Lemma \ref{Lemma2}}\label{Appendix.Sect.1}

In this appendix, we give the proofs of Proposition \ref{Densi.l1} and Lemma \ref{Lemma2}. We start with the proof of Proposition \ref{Densi.l1}.

\begin{proof}[Proof of Proposition \ref{Densi.l1}]

By adapting the proof of \cite[Theorem 1.5]{M1}, we obtain that $\nu_{n,\sigma}$ converges weakly in probability to a probability measure $\mu_{\theta}\in\mathcal{M}$, which has density
\begin{equation*}
  R_{\theta}(x,y)=e^{-\theta|x-y|+A_{\theta}(x)+B_{\theta}(y)}, \quad \forall (x,y)\in [0,1]^2,
\end{equation*}
with respect to the Lebesgue measure on $[0,1]^2$, where the functions $A_{\theta}(\cdot),B_{\theta}(\cdot)$ are in $L^1([0,1])$. Moreover, $R_{\theta}(\cdot,\cdot)$ satisfies $\int_0^1 R_{\theta}(x,y)dy=1$ for almost every $x\in [0,1]$ and $\int_0^1 R_{\theta}(x,y)dx=1$ for almost every $y\in [0,1]$. We also note that $A_{\theta}(x)\in\mathbb{R}$ for almost every $x\in [0,1]$ and $B_{\theta}(y)\in\mathbb{R}$ for almost every $y\in [0,1]$. Hence there exist $N,N'\subseteq [0,1]$ with zero Lebesgue measure, such that for any $x\in [0,1]\backslash N$, $A_{\theta}(x)\in \mathbb{R}$ and $\int_0^1 R_{\theta}(x,y) dy=1$; for any $y\in [0,1]\backslash N'$, $B_{\theta}(y)\in \mathbb{R}$ and $\int_0^1 R_{\theta}(x,y)dx=1$.

We pick any $x_0\in [0,1]\backslash N$. We have
\begin{equation*}
    \int_{0}^{1} e^{-\theta|x_0-y|+B_{\theta}(y)} dy=e^{-A_{\theta}(x_0)},
\end{equation*}
and for any $x\in [0,1]$, 
\begin{equation*}
    \int_0^1 e^{-\theta|x-y|+B_{\theta}(y)}dy \in [e^{-A_{\theta}(x_0)-\theta|x-x_0|},e^{-A_{\theta}(x_0)+\theta|x-x_0|}] \subseteq  (0,\infty).
\end{equation*}
Similarly, we can deduce that for any $y\in [0,1]$, $\int_0^1 e^{-\theta|x-y|+A_{\theta}(x)}dx \in (0,\infty)$.

Now for any $x\in [0,1]\backslash N$, we define $a_{\theta}(x):=A_{\theta}(x)$; for any $x\in N$, we define $a_{\theta}(x):=-\log\big(\int_0^1 e^{-\theta|x-y|+B_{\theta}(y)}dy\big)$. For any $y\in [0,1]\backslash N'$, we define $b_{\theta}(y):=B_{\theta}(y)$; for any $y\in N'$, we define $b_{\theta}(y):=-\log\big(\int_0^1 e^{-\theta|x-y|+A_{\theta}(x)}dx\big)$. Note that $a_{\theta}(x)\in \mathbb{R}$ for any $x\in [0,1]$ and $b_{\theta}(y)\in \mathbb{R}$ for any $y\in [0,1]$. We also define $\rho_{\theta}(x,y):=e^{-\theta|x-y|+a_{\theta}(x)+b_{\theta}(y)}$ for any $(x,y)\in [0,1]^2$. Note that $\rho_{\theta}(x,y)=R_{\theta}(x,y)$ for almost every $(x,y)\in [0,1]^2$. Hence $\rho_{\theta}(\cdot,\cdot)$ is also a density of $\mu_{\theta}$. We also note that $\rho_{\theta}(x,y)\in (0,\infty)$ for every $(x,y)\in [0,1]^2$.  

Note that $a_{\theta}(x)=A_{\theta}(x)$ for almost every $x\in [0,1]$ and $b_{\theta}(y)=B_{\theta}(y)$ for almost every $y\in [0,1]$. For any $x\in [0,1]\backslash N$, as $\int_0^1 R_{\theta}(x,y) dy=1$, we have 
\begin{equation*}
    a_{\theta}(x)=A_{\theta}(x)=-\log\Big(\int_0^1 e^{-\theta|x-y|+B_{\theta}(y)} dy \Big);
\end{equation*}
for any $x\in N$, by definition, we have
$a_{\theta}(x)=-\log\big(\int_0^1 e^{-\theta|x-y|+B_{\theta}(y)} dy \big)$. Hence for any $x\in [0,1]$, 
\begin{equation}
    a_{\theta}(x)=-\log\Big(\int_0^1 e^{-\theta|x-y|+B_{\theta}(y)} dy \Big)=-\log\Big(\int_0^1 e^{-\theta|x-y|+b_{\theta}(y)} dy \Big).
\end{equation}
Similarly, for any $y\in [0,1]$, we have 
\begin{equation}
    b_{\theta}(y)=-\log\Big(\int_0^1 e^{-\theta|x-y|+a_{\theta}(x)}dx\Big).
\end{equation}

For any $x_1,x_2\in [0,1]$, we have
\begin{eqnarray*}
  &&  a_{\theta}(x_2)-a_{\theta}(x_1)=\log\Big(\int_0^1 e^{-\theta|x_1-y|+b_{\theta}(y)} dy \Big)-\log\Big(\int_0^1 e^{-\theta|x_2-y|+b_{\theta}(y)} dy \Big) \nonumber\\
  &&\quad\quad \leq \log\Big(e^{\theta|x_1-x_2|} \int_0^1 e^{-\theta|x_2-y|+b_{\theta}(y)} dy \Big)-\log\Big(\int_0^1 e^{-\theta|x_2-y|+b_{\theta}(y)} dy \Big)\nonumber\\
  &&\quad\quad =\theta|x_1-x_2|,
\end{eqnarray*}
and similarly, $a_{\theta}(x_1)-a_{\theta}(x_2)\leq \theta|x_1-x_2|$. Hence $a_{\theta}(\cdot)$ is continuous on $[0,1]$. Similarly, we can deduce that $b_{\theta}(\cdot)$ is continuous on $[0,1]$. Therefore, $\rho_{\theta}(\cdot,\cdot)$ is continuous on $[0,1]^2$. As $\rho_{\theta}(x,y)\in (0,\infty)$ for every $(x,y)\in [0,1]^2$, there exist positive constants $m_{\theta}$ and $M_{\theta}$ that only depend on $\theta$, such that for every $(x,y)\in [0,1]^2$, $m_{\theta} \leq \rho_{\theta}(x,y)\leq M_{\theta}$.

Recall Definition \ref{Def1.5}. Let $\sigma$ be drawn from $\mathbb{P}_{n,\beta_n}$. As $\nu_{n,\sigma}$ converges weakly in probability to the probability measure with density $\rho_{\theta}(\cdot,\cdot)$, $\nu_{n,\sigma^{-1}}$ converges weakly in probability to the probability measure on $[0,1]^2$ with density 
\begin{equation*}
  \phi_{\theta}(x,y)=e^{-\theta|x-y|+b_{\theta}(x)+a_{\theta}(y)}=\rho_{\theta}(y,x), \quad \forall (x,y)\in [0,1]^2. 
\end{equation*}
As $\rho_{\theta}(\cdot,\cdot)$ is continuous on $[0,1]^2$, $\phi_{\theta}(\cdot,\cdot)$ is also continuous on $[0,1]^2$. Noting that the distribution of $\sigma^{-1}$ is also given by $\mathbb{P}_{n,\beta_n}$, we obtain that 
\begin{equation*}
   \rho_{\theta}(x,y)=\phi_{\theta}(x,y), \quad \forall (x,y)\in [0,1]^2.
\end{equation*}
Hence we can take $a_{\theta}(x)=b_{\theta}(x)$ for any $x\in [0,1]$. Therefore, we have
\begin{equation*}
  \rho_{\theta}(x,y)=e^{-\theta |x-y|+a_{\theta}(x)+a_{\theta}(y)}, \quad \forall (x,y)\in [0,1]^2. 
\end{equation*}

We show that $a_{\theta}(x)=a_{\theta}(1-x)$ for any $x\in [0,1]$ as follows. Let $\sigma$ be drawn from $\mathbb{P}_{n,\beta_n}$, and let $\bar{\sigma}\in S_n$ be such that $\bar{\sigma}(i)=n+1-\sigma(n+1-i)$ for every $i\in [n]$. As $\nu_{n,\sigma}$ converges weakly in probability to the probability measure with density $\rho_{\theta}(\cdot,\cdot)$, $\nu_{n,\bar{\sigma}}$ converges weakly in probability to the probability measure on $[0,1]^2$ with density
\begin{equation*}
    \psi_{\theta}(x,y)=e^{-\theta|x-y|+a_{\theta}(1-x)+a_{\theta}(1-y)}=\rho_{\theta}(1-x,1-y), \quad \forall (x,y)\in [0,1]^2.
\end{equation*}
As $\rho_{\theta}(\cdot,\cdot)$ is continuous on $[0,1]^2$, $\psi_{\theta}(\cdot,\cdot)$ is also continuous on $[0,1]^2$. Following the argument in (\ref{Eq1.3.7}), we can deduce that the distribution of $\bar{\sigma}$ is also given by $\mathbb{P}_{n,\beta_n}$, hence $a_{\theta}(x)=a_{\theta}(1-x)$ for any $x\in [0,1]$.


\end{proof}

Now we give the proof of Lemma \ref{Lemma2}.

\begin{proof}[Proof of Lemma \ref{Lemma2}]

Recall the definition of $\mathcal{M}$ from Definition \ref{Defm}. We also let $\mathcal{M}_0$ be the set of Borel probability measures on $[0,1]^2$ and endow it with the weak topology.

Recall that we have fixed $T,K_0\in \mathbb{N}^{*}$ such that $T\geq 4$ in Section \ref{Sect.3.1}. We also fix any $\delta>0$. For any $\Gamma\in \Pi^{T,T,K_0}$ and $l\in [2T-1]$, we let
\begin{equation}
    U_{\Gamma,l,\delta}:=\Big\{\mu\in\mathcal{M}_0: \Big|\mu(Q_l)-\int_{Q_l}\rho_{\theta}(x,y)dxdy\Big|\geq \delta\Big\}.
\end{equation}

Recall from Proposition \ref{Densi.l1} that $\rho_{\theta}(x,y)\leq M_{\theta}$ for any $(x,y)\in [0,1]^2$. Let $\mathcal{B}_{[0,1]^2}$ be the Borel $\sigma$-algebra on $[0,1]^2$, and let $d_{LP}$ be the L\'evy-Prokhorov metric on $\mathcal{M}_0$ ($d_{LP}$ metrizes the weak topology; see e.g. \cite[Chapter 1]{Bil}). Thus for any $\mu,\nu\in\mathcal{M}_0$, 
\begin{eqnarray*}
    d_{LP}(\mu,\nu)&:=&\inf\{\epsilon>0: \mu(A)\leq \nu(A^{\epsilon})+\epsilon \text{ and } \nu(A)\leq \mu(A^{\epsilon})+\epsilon\\
    &&\quad\quad \text{ for any }A\in\mathcal{B}_{[0,1]^2}\},
\end{eqnarray*}
where $A^{\epsilon}:=\{(x,y)\in [0,1]^2: \|(x,y)-(z,w)\|_2<\epsilon\text{ for some }(z,w)\in A\}$. 

Consider any $\mu\in U_{\Gamma,l,\delta}$ and any $\epsilon\leq \delta\slash (4M_{\theta}+2)$. Recall the definition of $\mu_{\theta}$ from Proposition \ref{Densi.l1}. Note that either $\mu(Q_l)\geq \mu_{\theta}(Q_l)+\delta$ or $\mu(Q_l)\leq \mu_{\theta}(Q_l)-\delta$. In the former case, we have
\begin{eqnarray*}
    \mu(Q_l)&\geq& \mu_{\theta}(Q_l)+\delta=\mu_{\theta}(Q_l^{\epsilon})-\mu_{\theta}(Q_l^{\epsilon}\backslash Q_l)+\delta \\
    &\geq& \mu_{\theta}(Q_l^{\epsilon})-4\epsilon M_{\theta}+\delta\geq \mu_{\theta}(Q_l^{\epsilon})+2\epsilon.
\end{eqnarray*}
In the latter case, we have $\mu(Q_l^c)\geq \mu_{\theta}(Q_l^c)+\delta$, hence 
\begin{eqnarray*}
   \mu(Q_l^c) &\geq& \mu_{\theta}(Q_l^c)+\delta = \mu_{\theta}((Q_l^c)^{\epsilon})-\mu_{\theta}((Q_l^c)^{\epsilon}\backslash Q_l^c)+\delta \\
   &\geq& \mu_{\theta}((Q_l^c)^{\epsilon})-4\epsilon M_{\theta}+\delta\geq  \mu_{\theta}((Q_l^c)^{\epsilon})+2\epsilon.
\end{eqnarray*}
Hence $d_{LP}(\mu,\mu_{\theta})\geq \delta\slash (4M_{\theta}+2)$ for any $\mu\in U_{\Gamma,l,\delta}$. Letting
\begin{equation*}
    V_{\delta}:= \Big\{\mu\in\mathcal{M}_0:d_{LP}(\mu,\mu_{\theta})\geq\frac{\delta}{4M_{\theta}+2}\Big\},
\end{equation*}
we have $U_{\Gamma,l,\delta}\subseteq V_{\delta}$. Note that $V_{\delta}$ is a closed subset of $\mathcal{M}_0$ and $\mu_{\theta}\notin V_{\delta}$. 


Let $\tau$ be drawn from the uniform distribution on $S_n$. Recall Definition \ref{Def1.5}. By \cite[Theorem 1.5]{M1}, the sequence of probability measures $\nu_{n,\tau}$ satisfy a large deviation principle on $\mathcal{M}_0$ with the good rate function
\begin{equation*}
    I(\mu):=\begin{cases}
        D(\mu\|u) & \text{ if }\mu \in  \mathcal{M} \\
        +\infty & \text{ otherwise}
    \end{cases},
\end{equation*}
where $u$ is the uniform measure on $[0,1]^2$ and $D(\cdot\|\cdot)$ is the Kullback-Leibler divergence. That is, for any $A\subseteq \mathcal{M}_0$, we have
\begin{eqnarray}\label{LDP}
   && -\inf_{\mu\in A^{\circ}} I(\mu)\leq \liminf_{n\rightarrow\infty} \frac{1}{n}\log{\mathbb{P}(\nu_{n,\tau}\in A)} \nonumber\\
   &&\leq \limsup_{n\rightarrow\infty} \frac{1}{n}\log{\mathbb{P}(\nu_{n,\tau}\in A)} \leq -\inf_{\mu\in\overline{A}} I(\mu),
\end{eqnarray}
where $A^{\circ}$ and $\overline{A}$ are the interior and closure of $A$ respectively.

For any $\mu\in\mathcal{M}_0$, we define $F(\mu):=\int_{[0,1]^2} |x-y| d\mu$ and
\begin{equation*}
    \tilde{F}(\mu):=\begin{cases}
        F(\mu) & \text{ if }\mu\in V_{\delta} \\
        +\infty & \text{ otherwise}
    \end{cases}.
\end{equation*}
Note that $F(\mu)$ is continuous on $\mathcal{M}_0$ and $\tilde{F}(\mu)$ is lower semicontinuous on $\mathcal{M}_0$. Let $\sigma$ be drawn from $\mathbb{P}_{n,\beta_n}$. In the following, we adapt the arguments in \cite[Lemmas 4.3.4 and 4.3.6]{DZ2} to derive an upper bound on $\mathbb{P}(\nu_{n,\sigma}\in V_{\delta})$.

As $\tau$ is drawn from the uniform distribution on $S_n$, we have
\begin{eqnarray}\label{E2.1}
\frac{1}{n}\log\Big(\frac{Z_{n,\beta_n}}{n!}\Big) &=&  \frac{1}{n}\log\Big(\frac{1}{n!}\sum_{\kappa \in S_n}e^{-\beta_n \sum_{i=1}^n |\kappa(i)-i|}\Big)  \nonumber\\
&=& \frac{1}{n}\log{\mathbb{E}[e^{-n^2\beta_n F(\nu_{n,\tau})}]}.
\end{eqnarray}

By the continuity of $F(\mu)$, for any $\delta'>0$, there exists an open set $G\subseteq \mathcal{M}_0$, such that $\mu_{\theta}\in G$ and $|F(\mu)-F(\mu_{\theta})|\leq\delta'$ for any $\mu\in G$. By (\ref{LDP}) and (\ref{E2.1}), 
\begin{eqnarray*}
&& \liminf_{n\rightarrow\infty} \frac{1}{n}\log\Big(\frac{Z_{n,\beta_n}}{n!}\Big) \geq \liminf_{n\rightarrow \infty} \frac{1}{n}\log{\mathbb{E}[e^{-n^2\beta_n F(\nu_{n,\tau})} \mathbbm{1}_{\nu_{n,\tau}\in G}]}  \nonumber \\
&  \geq   &  -\theta F(\mu_{\theta})-\theta\delta'+\liminf_{n\rightarrow\infty}\frac{1}{n}\log{\mathbb{P}(\nu_{n,\tau}\in G)} \nonumber\\
&\geq& -\theta F(\mu_{\theta})-\theta\delta'-\inf_{\mu\in G}I(\mu)\geq -\theta F(\mu_{\theta})-\theta \delta'-I(\mu_{\theta}).
\end{eqnarray*}
Letting $\delta'\rightarrow 0^{+}$, we obtain that
\begin{equation}\label{E2.2}
    \liminf_{n\rightarrow\infty} \frac{1}{n}\log\Big(\frac{Z_{n,\beta_n}}{n!}\Big)\geq -\theta F(\mu_{\theta})-I(\mu_{\theta}).
\end{equation}

Now fix an arbitrary $\alpha\in (0,\infty)$, and let $\Psi_{I}(\alpha):=\{\mu\in\mathcal{M}_0: I(\mu)\leq \alpha\}$. As $I(\mu)$ is a good rate function, $\Psi_I(\alpha)$ is a compact subset of $\mathcal{M}_0$. By the lower semicontinuity of $I(\mu)$ and $\tilde{F}(\mu)$, for any $\delta'>0$, the following holds: For any $\mu\in \Psi_{I}(\alpha)$, there exists an open set $G_{\mu}\subseteq \mathcal{M}_0$, such that $\mu\in     G_{\mu}$, and for any $\nu\in \overline{G_{\mu}}$, $I(\nu)\geq I(\mu)-\delta'$ and $\tilde{F}(\nu)\geq \tilde{F}(\mu)-\delta'$. As $\bigcup_{\mu\in\Psi_I(\alpha)}G_{\mu}$ is an open cover of the compact set $\Psi_I(\alpha)$, we can find $\mu_1,\mu_2,\cdots,\mu_L\in \Psi_I(\alpha)$, such that $\Psi_I(\alpha)\subseteq \bigcup_{j=1}^L G_{\mu_j}$. Hence
\begin{eqnarray*}
&& \frac{1}{n!}\sum_{\kappa\in S_n: \nu_{n,\kappa}\in V_{\delta}}e^{-\beta_n\sum_{i=1}^n|\kappa(i)-i|}=\mathbb{E}[e^{-n^2\beta_n\tilde{F}(\nu_{n,\tau})}]\\
&\leq & \sum_{j=1}^L \mathbb{E}[e^{-n^2\beta_n\tilde{F}(\nu_{n,\tau})} \mathbbm{1}_{\nu_{n,\tau}\in G_{\mu_j}}]+\mathbb{P}\Big(\nu_{n,\tau}\in \Big(\bigcup_{j=1}^L G_{\mu_j}\Big)^c\Big)\\
&\leq& \sum_{j=1}^L e^{-n^2\beta_n(\tilde{F}(\mu_j)-\delta')}\mathbb{P}(\nu_{n,\tau}\in G_{\mu_j})+\mathbb{P}\Big(\nu_{n,\tau}\in \Big(\bigcup_{j=1}^L G_{\mu_j}\Big)^c \Big).
\end{eqnarray*}
Hence by (\ref{LDP}), we have
\begin{eqnarray*}
&& \limsup_{n\rightarrow\infty}\frac{1}{n}\log\Big(\frac{1}{n!}\sum_{\kappa\in S_n: \nu_{n,\kappa}\in V_{\delta}}e^{-\beta_n\sum_{i=1}^n|\kappa(i)-i|}\Big)\\
&\leq& \max\{\max_{j\in [L]}\{-\theta \tilde{F}(\mu_j)+\theta \delta'-\inf_{\nu\in \overline{G_{\mu_j}}}I(\nu)\},-\inf_{\nu\in (\bigcup_{j=1}^L G_{\mu_j})^c}I(\nu)\}\\
&\leq& \max\{\max_{j\in [L]}\{-\theta \tilde{F}(\mu_j)-I(\mu_j)+(\theta+1) \delta'\},-\inf_{\nu\in \Psi_I(\alpha)^c}I(\nu)\}\\
&\leq& \max\{\max_{j\in [L]}\{-\theta \tilde{F}(\mu_j)-I(\mu_j)+(\theta+1) \delta'\},-\alpha\}.
\end{eqnarray*}
Letting $\delta'\rightarrow 0^{+}$ and $\alpha\rightarrow\infty$, we obtain that
\begin{eqnarray}\label{E2.3}
  &&  \limsup_{n\rightarrow\infty}\frac{1}{n}\log\Big(\frac{1}{n!}\sum_{\kappa\in S_n: \nu_{n,\kappa}\in V_{\delta}}e^{-\beta_n\sum_{i=1}^n|\kappa(i)-i|}\Big)\leq \sup_{\mu\in \mathcal{M}_0}\{-\theta\tilde{F}(\mu)-I(\mu)\}\nonumber\\
  &&   =\sup_{\mu\in V_{\delta}\cap\mathcal{M}} \{-\theta F(\mu)-I(\mu)\}.
\end{eqnarray}

Combining (\ref{E2.2}) and (\ref{E2.3}), we have
\begin{eqnarray}\label{E2.4}
&&\limsup_{n\rightarrow\infty}\frac{1}{n}\log{\mathbb{P}(\nu_{n,\sigma}\in V_{\delta})}\nonumber\\
&\leq&\limsup_{n\rightarrow\infty}\frac{1}{n}\log\Big(\frac{1}{n!}\sum_{\kappa\in S_n: \nu_{n,\kappa}\in V_{\delta}}e^{-\beta_n\sum_{i=1}^n|\kappa(i)-i|}\Big)-\liminf_{n\rightarrow\infty} \frac{1}{n}\log\Big(\frac{Z_{n,\beta_n}}{n!}\Big)\nonumber\\
&\leq& \sup_{\mu\in V_{\delta}\cap\mathcal{M}}\{-\theta F(\mu)-I(\mu)\}+\theta F(\mu_{\theta})+I(\mu_{\theta}).
\end{eqnarray}

By the proof of \cite[Theorem 1.5]{M1}, $\mu_{\theta}$ is the unique maximizer of the function $Q(\mu):=-\theta F(\mu)-I(\mu)$ over $\mathcal{M}$. As $V_{\delta}\cap \mathcal{M}$ is a compact set, the supremum of $Q(\mu)$ is attained on $V_{\delta}\cap\mathcal{M}$. As $\mu_{\theta}\notin V_{\delta}\cap \mathcal{M}$, we conclude that
\begin{equation}\label{E2.5}
    \sup_{\mu\in V_{\delta}\cap\mathcal{M}}\{-\theta F(\mu)-I(\mu)\}+\theta F(\mu_{\theta})+I(\mu_{\theta})<0.
\end{equation}
Combining (\ref{E2.4}) and (\ref{E2.5}), we conclude that there exist positive constants $C_{\Gamma,l,\delta}$, $c_{\Gamma,l,\delta}$ that only depend on $T,K_0,\delta,\Gamma,l$ and $\{\beta_n\}$, such that
\begin{equation}
    \mathbb{P}(\nu_{n,\sigma}\in U_{\Gamma,l,\delta})\leq \mathbb{P}(\nu_{n,\sigma}\in V_{\delta})\leq C_{\Gamma,l,\delta}\exp(-c_{\Gamma,l,\delta} n). 
\end{equation}
We take
\begin{equation*}
    C_0:=\sup_{\Gamma\in \Pi^{T,T,K_0}}\sup_{l\in [2T-1]}C_{\Gamma,l,\delta}, \quad c_0:=\inf_{\Gamma\in \Pi^{T,T,K_0}}\inf_{l\in [2T-1]}c_{\Gamma,l,\delta},
\end{equation*}
Note that $C_0$, $c_0$ are positive constants that only depend on $T,K_0,\delta$ and $\{\beta_n\}$. Moreover, for any $\Gamma\in \Pi^{T,T,K_0}$ and any $l\in [2T-1]$, we have 
\begin{equation}
    \mathbb{P}(\nu_{n,\sigma}\in U_{\Gamma,l,\delta})\leq C_0 \exp(-c_0 n).
\end{equation}
By the definition of $U_{\Gamma,l,\delta}$, we obtain the conclusion of the lemma.

\end{proof}

\end{appendices}

\bibliographystyle{acm}
\bibliography{Mallows.bib}
\end{document}